\documentclass[12pt]{amsart}
\allowdisplaybreaks
\usepackage[english]{babel}
\usepackage[pdftex,textwidth=400pt,marginratio=1:1]{geometry}
\usepackage{amsfonts}
\usepackage[dvips]{graphics}
\usepackage[colorinlistoftodos]{todonotes}
\usepackage{amsmath}
\usepackage{amsthm}
\usepackage{amssymb}
\usepackage{bbm}
\usepackage{cancel}
\usepackage{color}
\usepackage[all,cmtip]{xy}
\usepackage{graphicx}
\usepackage{mathabx}
\usepackage{tikz-cd}
\usepackage{epigraph} 
\usepackage{hyperref}
\usepackage{stmaryrd}
\usepackage{enumitem}
\usepackage[mathscr]{eucal}
\usepackage{soul}

\usepackage{xfrac}

\newtheorem{theorem}{Theorem}[section]
\newtheorem{corollary}[theorem]{Corollary}
\newtheorem{proposition}[theorem]{Proposition}
\newtheorem{lemma}[theorem]{Lemma}
\newtheorem{conjecture}[theorem]{Conjecture}

\theoremstyle{definition}
\newtheorem{definition}[theorem]{Definition}
\newtheorem{example}[theorem]{Example}
\newtheorem{remark}[theorem]{Remark}
\newtheorem{question}[theorem]{Question}

\DeclareFontFamily{OT1}{rsfs}{}
\DeclareFontShape{OT1}{rsfs}{n}{it}{<-> rsfs10}{}
\DeclareMathAlphabet{\curly}{OT1}{rsfs}{n}{it}

\newcommand\I{\mathcal I}

\newcommand\LL{\mathbb L}

\renewcommand\O{\mathcal O}
\newcommand\PP{\mathbb P}
\newcommand\cP{\mathcal P}

\newcommand\cV{\mathcal V}

\newcommand\cW{\mathcal W}

\newcommand\EE{\mathbb E}
\newcommand\cA{\mathcal A}
\newcommand\F{\mathcal F}

\newcommand\C{\mathbb C}
\newcommand\cC{\mathcal C}

\newcommand\FF{\mathbb F}
\newcommand\GG{\mathbb G}

\newcommand\II{\mathbb I}

\newcommand\Q{\mathbb Q}
\newcommand\cQ{\mathcal Q}
\newcommand\R{\mathbb R}

\newcommand{\cX}{\mathcal{X}}
\newcommand\Z{\mathbb Z}
\newcommand\cZ{\mathcal Z}

\newcommand\Coh{\mathrm{Coh}}

\newcommand\m{\mathfrak m}

\newcommand\pure{\mathrm{pure}}

\newcommand\vd{\mathrm{vd}}
\newcommand\pt{\mathrm{pt}}
\newcommand\vir{\mathrm{vir}}
\newcommand{\red}{\mathrm{red}}

\newcommand\DT{\mathrm{DT}}
\newcommand\PT{\mathrm{PT}}

\newcommand\td{\mathrm{td}}
\newcommand\rk{\operatorname{rk}}

\newcommand\tr{\operatorname{tr}}
\newcommand\coker{\operatorname{coker}}
\newcommand\im{\operatorname{im}}
\newcommand\ch{\operatorname{ch}}
\newcommand\ev{\operatorname{ev}}

\newcommand\id{\operatorname{id}}

\newcommand\Hom{\operatorname{Hom}}
\renewcommand\hom{\mathcal{H}{\it{om}}}
\newcommand\Ext{\operatorname{Ext}}
\newcommand\ext{\curly Ext}
\newcommand\At{\operatorname{At}}
\newcommand\Aut{\operatorname{Aut}}
\newcommand\ob{\operatorname{ob}}

\newcommand\Proj{\operatorname{Proj}\,}
\newcommand\Spec{\operatorname{Spec}\,}
\newcommand\Hilb{\operatorname{Hilb}}

\newcommand\Sym{\operatorname{Sym}}

\newcommand\mdot{{\scriptscriptstyle\bullet}}

\newcommand\INTO{\ar@{^{(}->}[r]}

\DeclareRobustCommand{\SkipTocEntry}[4]{}

\def\dual{^{\vee}}
\def\fI{\mathfrak{I}}

\def\Supp{\mathrm{Supp}}
\def\cone{\mathrm{cone}}
\def\udot{^{\mdot}}
\def\spl{\mathrm{spl}}
\def\tv{\widetilde{v}}

\def\cM{\mathcal{M}}
\def\cL{\mathcal{L}}
\def\cG{\mathcal{G}}
\def\cB{\mathcal{B}}

\def\Pic{\mathrm{Pic}}
\def\GL{\mathrm{GL}}
\def\PGL{\mathrm{PGL}}
\def\SL{\mathrm{SL}}
\def\rank{\mathrm{rank}}

\def\cPair{{\curly Pair}}
\def\cCoh{{\curly Coh}}
\def\cPerf{{\curly Perf}}
\def\cQuot{{\curly Quot}}

\def\tPair{\widetilde{\cPair}}

\def\TT{\mathbb{T}}

\def\PTqvX{{\curly P}^{(q)}_v(X)}
\def\PTqvXB{{\curly P}^{(q)}_{\tv}(\cX/\cB)}
\def\PTqvXb{{\curly P}^{(q)}_{\tv_b}(\cX_b)}
\def\KS{\mathsf{KS}}
\def\SD{\mathsf{SD}}

\def\sfob{\mathsf{ob}}
\def\tgamma{\widetilde{\gamma}}
\def\fC{\mathfrak{C}}
\def\fq{\mathfrak{q}}
\def\bbA{\mathbb{A}}
\def\sfB{\mathsf{B}}
\def\top{\mathrm{top}}

\def\SR{\mathsf{SR}}

\def\XX{\mathbb{X}}
\def\UU{\mathbb{U}}
\def\HH{\mathbb{H}}

\def\fQ{\mathfrak{Q}}
\def\trunc{\tau^{\geq-1}}
\def\curP{\curly P}

\def\fl{\mathfrak{l}}

\def\cH{\mathcal{H}}

\def\rvd{\mathrm{rvd}}
\def\sX{\mathscr{X}}
\def\KK{\mathbb{K}}

\def\cT{\mathcal{T}}

\def\abs{\mathrm{abs}}

\def\sY{\mathscr{Y}}

\def\tcurP{\widetilde{\curP}}

\def\sfY{\mathsf{Y}}

\def\sZ{\mathscr{Z}}

\def\tsX{\widetilde{\sX}}
\def\tEE{\widetilde{\EE}}
\def\tphi{\widetilde{\phi}}
\def\tSigma{\widetilde{\Sigma}}
\def\talpha{\widetilde{\alpha}}

\def\cU{\mathcal{U}}

\def\YY{\mathbb{Y}}

\def\and{\quad\mathrm{and}\quad}

\newcommand{\sr}{\mathsf{sr}}
\newcommand{\cl}{\mathrm{cl}}

\newcommand{\an}{\mathrm{an}}

\def\dSt{\mathsf{dSt}}
\def\cdga{\mathsf{cdga}}
\def\gr{\mathrm{gr}}
\def\MM{\mathbb{M}}

\def\sfD{\mathsf{D}}
\def\dg{\mathsf{dg}}
\def\DR{\mathsf{DR}}
\def\op{\mathrm{op}}
\def\aff{\mathrm{aff}}
\def\NC{\mathsf{NC}}
\def\CC{\mathsf{CC}}
\def\Map{\mathsf{Map}}
\def\ttalpha{\widetilde{\widetilde{\alpha}}}
\def\HN{\mathsf{HN}}
\def\HochP{\mathsf{HP}}
\def\HC{\mathsf{HC}}
\def\HochH{\mathsf{HH}}

\def\PTVV{\mathrm{PTVV}}
\def\UMap{R\underline{\Map}}
\def\txi{\widetilde{\xi}}
\def\hxi{\widehat{\xi}}

\def\sE{\mathscr{E}}
\def\cPic{{\curly Pic}}
\def\tcM{\widetilde{\cM}}
\def\sp{\mathrm{sp}}
\def\hOmega{\widehat{\Omega}}
\def\PC{\mathsf{PC}}
\def\cY{\mathcal{Y}}
\def\hcY{\widehat{\cY}}
\def\tDelta{\widetilde{\Delta}}
\def\bcA{\overline{\cA}}
\def\cS{\mathcal{S}}
\def\J{\mathcal{J}}

\begin{document}
\title[Counting surfaces on Calabi-Yau 4-folds I]{Counting surfaces on Calabi-Yau 4-folds I: foundations}
\author[Y.~Bae, M.~Kool, H.~Park]{Younghan Bae, Martijn Kool, and Hyeonjun Park}
\maketitle
\begin{abstract}
This is the first part in a series of papers on counting surfaces on Calabi-Yau 4-folds. Besides the Hilbert scheme of 2-dimensional subschemes, we introduce \emph{two} types of moduli spaces of stable pairs. We show that all three moduli spaces are related by GIT wall-crossing and parametrize stable objects in the bounded derived category. 

We construct  \emph{reduced} Oh-Thomas virtual cycles on the moduli spaces via Kiem-Li cosection localization and prove that they are deformation invariant along Hodge loci. As an application, we show that the variational Hodge conjecture holds for any family of Calabi-Yau 4-folds supporting a non-zero reduced virtual cycle.
\end{abstract}

\tableofcontents
\addtocontents{toc}{\protect\setcounter{tocdepth}{1}}

\section{Introduction}
\subsection{Background}
Counting curves on a space $X$ is one of the central topics in modern enumerative geometry. It is a rich subject with links to many other fields of mathematics as well as physics (string theory).
In \cite{K95} Kontsevich introduced the concept of \textit{hidden smoothness} of moduli spaces of stable maps in order to define algebraic Gromov-Witten (GW) invariants of $X$. The mathematical foundation of its virtual cycles was established by Li-Tian \cite{LT} and Behrend-Fantechi \cite{BF}. When $X$ is a \emph{3-fold}, an alternative enumeration can be carried out using Grothendieck's Hilbert schemes. This change of perspective represents the difference between viewing curves on $X$ as being parametrized or cut out by equations. In \cite{Th00}, Thomas then defined virtual cycles on Hilbert schemes of a 3-fold $X$, which yields Donaldson-Thomas (DT) invariants\footnote{\cite{Th00} defines DT invariants for stable sheaves of any rank on Fano and Calabi-Yau $3$-folds.} of $X$. A remarkable connection between GW and DT invariants for 3-folds was discovered by  Maulik-Nekrasov-Okounkov-Pandharipande \cite{MNOP}. This correspondence was further refined by Pandharipande-Thomas via stable pair (PT) invariants \cite{PT1,PT2,PT3}. We refer the reader to \cite{PT14} for an overview.

It was recently found that a similar relation between curve counting invariants exists on \emph{Calabi-Yau 4-folds} $X$. The structure of GW invariants of $X$ was studied by Klemm and Pandharipande \cite{KP}. On the sheaf side, the obstruction theories are not 2-term, and a new discovery was needed. Inspired by work of Donaldson and Thomas \cite{DT}, Cao and Leung \cite{CL} defined such a virtual cycle in special cases. The general case was subsequently established by Borisov and Joyce \cite{BJ} using methods of derived differential geometry. The algebraic counterpart was obtained by Oh and Thomas \cite{OT}. Both constructions \cite{BJ,OT} deeply rely on the derived structure of the moduli space --- the shifted symplectic structures of Pantev-To\"en-Vaqui\'e-Vezzosi \cite{PTVV} and the derived Darboux theorem of Brav-Bussi-Joyce \cite{BBJ}. 
Another feature of the theory is that it requires orientations.
The existence of orientations was shown by Cao-Gross-Joyce \cite{CGJ}. 
An explicit GW-PT correspondence on Calabi-Yau 4-folds was conjectured by Cao, Maulik and Toda \cite{CMT1,CMT2,CT21}. Virtual counts of points and curves on Calabi-Yau 4-folds recently attracted quite some attention. 
In physics, Nekrasov-Piazzalunga's Magnificent Four conjecture \cite{Nek, NP}
provides a formula for counting points on the Calabi-Yau 4-fold $\C^4$. For the case of counting points on $X$ we refer to \cite{CK1,Par,Boj2,KR} and for the case of counting curves on $X$ we refer to \cite{CK2,CKM,CT19,CT20,Par}.

The next natural step in enumerative geometry is counting higher-dimensional objects, such as surfaces on Calabi-Yau 4-folds --- the topic of this paper. It is not clear at present how to define virtual cycles on the stable map side (e.g., see Alexeev \cite{Ale} for the moduli space and Jiang \cite{Jia} for some recent developments).
On the sheaf side, a natural starting point is the virtual cycle of \cite{BJ, OT} on Hilbert schemes of 2-dimensional subschemes. However, we will see that for enumeration of surfaces, this virtual cycle generally does \emph{not} suffice and has to be modified.

\subsection{Counting surfaces}

The aim of this project is to count surfaces on Calabi-Yau 4-folds $X$ via sheaf theory. This goal leads to two immediate challenges.
\begin{enumerate}
    \item[(A)] The Hilbert scheme parametrizes subschemes which are not pure dimensional, i.e.~2-dimensional subschemes may have 0- and 1-dimensional components. Informally: Hilbert schemes allow for free roaming curves and points thereby obscuring the contribution of genuine surfaces.
    \item[(B)] A Hodge class of type $(2,2)$ on $X$ generally does not remain of type $(2,2)$ as the complex structure of $X$ deforms. When this happens, the virtual cycle vanishes.
\end{enumerate}
The main purpose of this paper is to resolve (A) and (B).

Inspired by Pandharipande-Thomas's work on curves on 3-folds \cite{PT1}, we settle (A) by constructing moduli spaces of stable pairs on $X$. We find that for 2-dimensional pairs, there is more structure than for 1-dimensional pairs, and we obtain \emph{two} types which we call $\PT_0$ and $\PT_1$ pairs. The stable pair moduli spaces give smaller compactifications --- roughly speaking, $\PT_0$ pairs prevent free roaming points on $X$ and $\PT_1$ pairs prevent free roaming points \emph{and} curves on $X$. The Hilbert scheme and two stable pair moduli spaces each have desirable features and should be studied together.


We overcome (B) by constructing reduced virtual cycles using the cosection localization method of Kiem and Li \cite{KL13}.
This method was generalized to the case of \emph{isotropic cosections} by Kiem and the third-named author \cite{KP20}.
Instead, in this paper we use a variant, namely \emph{non-degenerate cosections}.
For a class $\gamma\in H^{2}(X,\Omega^2_X)\cap H^4(X,\Q)$, virtual cycles of moduli spaces parametrizing 2-dimensional pairs in class $\gamma$ often vanish.
Similar vanishing for curve counting theories on surfaces/3-folds led to \textit{reduced theories} e.g.~\cite{MPT,KT1} (we refer the reader to the introductions of these papers for more references). From the perspective of curves on surfaces/3-folds $X$, reducing is usually needed when $X$ has a non-zero holomorphic 2-form. However, for compact Calabi-Yau 3-folds satisfying $h^{0,1}(X) = h^{0,2}(X) = 0$, it is of no concern. As we will see, even on compact Calabi-Yau 4-folds satisfying $h^{0,1}(X) = h^{0,2}(X) = h^{0,3}(X) = 0$ reducing is necessary for most interesting surface classes $\gamma$.

\subsection{Moduli spaces of stable pairs}

There are three types of stability conditions on $2$-dimensional pairs.

\begin{definition}[Definition \ref{def:stabilityconditionsofpairs'}]
Let $X$ be a smooth projective variety of dimension $n\geq 2$.  
Let $F$ be a 2-dimensional coherent sheaf on $X$, $s:\O_X\to F$ a section, and $Q:=\coker(s)$. The pair $(F,s)$ is called
\begin{enumerate}
    \item[(i)] {\em $\DT(=\PT_{-1})$-stable} if $s$ is surjective,
    \item[(ii)] {\em $\PT_0$-stable} if $T_0(F)=0$ and $\dim(Q)\leq 0$,
    \item[(iii)] {\em $\PT_1$-stable} if $T_1(F)=0$ and $\dim(Q)\leq 1$.
\end{enumerate}
Here $T_d(F)$ denotes the maximal $d$-dimensional subsheaf of $F$.
\end{definition}

The $\DT$-stable pairs correspond to ideal sheaves. 
The $\PT_1$-stable pairs are natural generalizations of Pandharipande-Thomas's stable pairs \cite{PT1}, which are a special case of Le Potier's coherent systems \cite{Pot1,Pot2}.
The underlying sheaf of a $\PT_0$-stable pair is usually \textit{not} pure and this stability condition appears to be new. 
We introduce $\PT_0$-stability as a bridge between $\DT$- and $\PT_1$-stability.

Suppose $(F,s)$ is a $\PT_q$ pair on $X$ with cokernel $Q$ and \emph{smooth} scheme-theoretic support $S$, which generally has 0-, 1-, and 2-dimensional connected components. 
For $q=0$, $S$ has no 0-dimensional connected components and $Q$ is supported on the 1-dimensional connected components of $S$. For $q=1$, $S$ has no 0- or 1-dimensional components and the pair $(F,s)$ can be equivalently described by a nested triple $Z \subset C \subset S$, where $C$ is an effective divisor and $Z$ is a 0-dimensional subscheme (Corollary \ref{cor:PT1smsupp}). We give various characterizations of $\PT_q$ pairs with \emph{singular} support in Section \ref{sec:pair}. 

\begin{example}[Proposition \ref{Lem.PT0=PT1}]
Suppose $(F,s)$ is a pair which is $\PT_0$- and $\PT_1$-stable, i.e., $F$ is pure and $Q$ is 0-dimensional. 
Then $Q$ is a quotient sheaf of the 0-dimensional sheaf $\ext^{n-1}_X(\O_S,K_X)$, which is supported on the locus where $S$ is not Cohen-Macaulay.
\end{example}




Our first main result (Section \ref{sec:moduli}) is the construction of projective moduli spaces  of $\PT_q$-stable pairs via geometric invariant theory (GIT).
\begin{theorem}[Theorem \ref{Thm:GIT}]\label{thm:1}
Let $X$ be a smooth projective variety and $v\in H^*(X,\mathbb{Q})$. Then, for each $q \in \{-1,0,1\}$, there exists a fine moduli space $\PTqvX$ of $\PT_q$-stable pairs on $X$ with $\ch(F)=v$ as a projective scheme.
Moreover, there exists a projective scheme $\cM$, an action of $\PGL_N$ on $\cM$ for some $N$, and an $\SL_N$-equivariant ample line bundle $\cL^{(q)} \in \Pic^{\SL_N}(\cM)_{\Q}$ for each $q \in \{-1,0,1\}$ such that \[\curP^{(q)}_v(X) = [\cM^{st}(\cL^{(q)})/\PGL_N] = [\cM^{ss}(\cL^{(q)})/\PGL_N].\]
In particular, $\PTqvX$ is the GIT quotient $ \cM/\!/_{\cL^{(q)}}\SL_N$.
\end{theorem}
Theorem \ref{thm:1} shows that $\PT_q$ pair moduli spaces are related by GIT wall-crossing.
Since the underlying sheaves of $\PT_0$ pairs are not {\em pure}, we needed finer arguments than the standard ones for pure sheaves.
This approach is inspired by an analogous result of Stoppa-Thomas on 1-dimensional pairs \cite{ST}.

Following \cite{PT1}, we show that the moduli space of $\PT_q$ pairs consists of open components of the moduli space of perfect complexes. Denote by $\cPerf(X)_{\O_X}^{\mathrm{spl}}$ the moduli space of simple perfect complexes on $X$ with trivial determinant.

\begin{theorem}[Theorem \ref{Thm:PairtoPerf}, Proposition \ref{prop:PTq=Bayerstab}]\label{thm:2}
Let $X$ be a smooth projective variety of dimension $n\geq 4$, $v \in H^*(X,\Q)$, and  $q\in \{-1,0,1\}$.
The canonical map 
\begin{equation}\label{eq:1.1}
\curP^{(q)}_v(X) \to \cPerf(X)_{\O_X}^{\mathrm{spl}}, \quad (F,s) \mapsto I\udot=[\O_X \xrightarrow{s} F]	
\end{equation}
is an open embedding.
Moreover, there is a polynomial Bridgeland stability condition $\rho^{(q)}$ such that the image of \eqref{eq:1.1} is the locus of $\rho^{(q)}$-stable objects. 
\end{theorem}
For $q=1$, Theorem \ref{thm:2} was also shown by Gholampour-Jiang-Lo \cite{GJL}.
The connection to Bayer's polynomial Bridgeland stability \cite{Bay} is inspired by \cite{GJL}.

\subsection{Reduced virtual cycles}

Let $X$ be a {\em Calabi-Yau $4$-fold}, i.e., a smooth projective variety of dimension $4$ with trivial canonical bundle. In particular, $X$ could be a hyperk{\"a}hler $4$-fold or an abelian $4$-fold.
Fix a trivialization $\omega : \O_X \xrightarrow{\cong}\Omega^4_X$.

Throughout this section, let $q \in \{-1,0,1\}$ and let $\curP := \PTqvX$ be the moduli space of $\PT_q$ pairs $(F,s)$ on $X$ with Chern character
\[\ch(F)=v=(0,0,\gamma,\beta,n-\gamma\cdot\td_2(X)) \in H^*(X,\Q).\]
By Theorem \ref{thm:2},  we have an Oh-Thomas virtual cycle \cite{OT}
\[\left[\curP\right]^\vir \in A_{\vd}\left(\curP\right), \qquad \vd = n - \frac12 \gamma^2,\]
which maps to Borisov-Joyce's virtual cycle \cite{BJ} under the cycle class map by \cite{OT2}.\footnote{Here $A_*(-)$ denotes the Chow group with rational coefficients. We also use the convention that $A_{i}(-)=0$ for $i \in \Q\setminus\Z$.}
The Oh-Thomas virtual cycle is constructed from the 
canonical 3-term {\em symmetric obstruction theory}
\[\phi : \EE:=R\hom_{\pi}(\II\udot,\II\udot)_0[3] \to \LL_{\curP}, \qquad \SD : \EE\dual[2] \cong \EE,\]
of Huybrechts-Thomas \cite{HT} and
an {\em orientation}
\[ o  : \O_{\curP} \xrightarrow{\cong} \det(\EE)\]
of Cao-Gross-Joyce  \cite{CGJ}. 
Then the intrinsic normal cone is {\em isotropic} by the derived Darboux theorem of Brav-Bussi-Joyce \cite{BBJ}, which is necessary to define the virtual cycle. We refer to Section \ref{sec:VFC} for the precise statements.

For a generic surface class $\gamma$, the Oh-Thomas virtual cycle $[\curP]^\vir$ vanishes due to the following two heuristics:
\begin{enumerate}
\item [(a)] When deforming $X$, the codimension of the Hodge locus of $\gamma$ in the base is often positive.
Therefore deformation invariance of virtual cycles implies the vanishing.
\item [(b)] There often exist non-degenerate cosections given by semi-regularity maps. 
They give trivial factors $\O$ of $\EE$ which imply the vanishing.
\end{enumerate}
We present a motivating example which reveals the above two phenomena.
\begin{example}\label{Ex:1}
Let $X \subseteq \PP^5$ be a smooth sextic hypersurface which contains a plane $S\subseteq X$. Let $v=\ch(\O_S) \in H^*(X,\Q)$ and $q\in\{-1,0,1\}$. 
\begin{itemize}
\item The virtual dimension is $-\frac{19}{2}$.
\item The codimension of the Hodge locus of $\gamma=v_2=[S]$ is $19$ \cite{CGG}. 
\item There are $19$ non-degenerate cosections.
\end{itemize}
\end{example}

We first discuss the codimension of the Hodge locus in (a).
By Bloch \cite{Blo} (see Proposition \ref{prop:ob=srdual}), the obstruction to deforming the Hodge class $\gamma \in H^2(X,\Omega^2_X) \subseteq H^4(X,\C)$ is governed by the symmetric bilinear form
\begin{equation*}\label{eq:1.2}
\mathsf{B}_\gamma : H^1(X,T_X) \otimes H^1(X,T_X) \to \C, \quad \xi_1 \otimes \xi_2 \mapsto \int_X \iota_{\xi_1}\iota_{\xi_2}\gamma \cup \omega\,,	
\end{equation*}
where $\iota$ denotes the contraction map. In particular, its rank
\begin{equation*}\label{eq:1.2.1}
    \rho_\gamma :=\rk (\mathsf{B}_\gamma)
\end{equation*}
is the codimension of the Hodge locus of $\gamma$.\footnote{Let $f:\cX\to\cB$ be a smooth proper morphism of smooth analytic spaces such that the fibre $\cX_0$ over a point $0 \in \cB$ is $X^\an$ and $\cB$ is contractible. Then there exists a unique lift $\tgamma$ of $\gamma \in H^4(X,\Q)$ as a section of the local system $R^4f_* \Q$. The {\em Hodge locus} of $\gamma$ is defined as the subset $\mathrm{Hd}_\gamma:=\{b\in\cB : \tgamma_b \in H^{2,2}(\cX_b)\}$ and has a natural analytic scheme structure (see \cite{Voi}). If the Kodaira-Spencer map $\KS : T_{\cB,0} \to H^1(X,T_X)$ is an isomorphism, then $\rho_\gamma = \dim(T_{\cB,0}) -\dim(T_{\mathrm{Hd}_\gamma,0})$ by \cite{Blo}.
We will not use the Hodge locus in this paper except for heuristic arguments.\label{footnote:codHdg}}

We now discuss the cosections in (b). 
By Buchweitz-Flenner \cite{BF03}, a perfect complex $[I\udot]\in \curP$ gives rise to a {\em semi-regularity map}
\begin{equation}\label{eq:sr1}
\sr : \Ext^2_X(I\udot,I\udot)_0 \to H^3(X,\Omega_X^1)\cong H^1(X,T_X)\dual, \quad e\mapsto \tr(\At(I\udot)\circ e),
\end{equation}
where $\At(I\udot)$ is the Atiyah class.
This semi-regularity map relates the obstructions of deforming the perfect complex $I\udot$ and the Hodge class $\gamma=\ch_2(I\udot)$ via the commutative triangle 
\begin{equation}\label{eq:t1} 
\xymatrix{ H^1(X,T_X) \ar[rr]^-{\sfob=\sr\dual} \ar[rd]_-{\mathsf{B}_\gamma} && \Ext^2_{X}(I\udot,I\udot)_0  \ar[ld]^-{\sr} \\ & H^1(X,T_X)\dual}.
\end{equation}
The symmetry of the above triangle given by $\sfob=\sr\dual$ is a special feature of Calabi-Yau 4-folds $X$ and $(2,2)$-classes $\gamma \in H^2(X,\Omega^2_X)$. 
The semi-regularity maps $\sr$ can be extended to cosections on the moduli space,
\[\SR : \EE\dual[1]=R\hom_{\pi}(\II\udot,\II\udot)_0[2] \to H^1(X,T_X)\dual \otimes \O_{\curP}.\]
Then $\SR^2 := \SR\circ\SR\dual = \sfB_\gamma\otimes 1_{\curP}$, which can be shown by extending the triangle \eqref{eq:t1}. See Section \ref{sec:VFC} for details.

Choose a non-degenerate subspace $V \subseteq H^1(X,T_X)$ with respect to the bilinear form $\sfB_\gamma$. Then $V \cong V^\vee$.
Assume that $V$ is {\em maximal}, i.e., $\dim(V)=\rho_\gamma$.

\begin{enumerate}
\item [(a')] From the perspective of Hodge theory, $V$ is a transversal slice to the Hodge locus of $\gamma$. Heuristically, we want to constrain virtual cycles to be deformation invariant only on the Hodge locus of $\gamma$.
\item [(b')] From the perspective of cosections, $V$ induces \emph{non-degenerate} cosections
\[\Sigma : \EE\dual[1] \xrightarrow{\SR} H^1(X,T_X)\dual \otimes \O_{\curP} \to V\dual \otimes \O_{\curP}.\]
Here non-degeneracy means that the square $\Sigma^2:=\Sigma\circ\Sigma\dual$ is an isomorphism.
In particular, we have a direct sum decomposition
\[\EE = \EE^\red \oplus (V\otimes \O_{\curP})[1]\]
and we want to remove the trivial factors from the obstruction theory.
\end{enumerate}

The second main result in this paper achieves (a') and (b').
\begin{theorem}[Theorem \ref{Thm:RVFC}, Proposition \ref{prop:redSOT}]\label{thm:3}
Let $X$ be a Calabi-Yau 4-fold, $v \in H^*(X,\Q)$, and $q\in \{-1,0,1\}$. For any maximal non-degenerate subspace $V \subseteq H^1(X,T_X)$ with respect to $\sfB_{\gamma}$, the composition
\[\phi^\red :\EE^\red \xrightarrow{(1,0)} \EE \xrightarrow{\phi} \LL_{\PTqvX}\]
is a symmetric obstruction theory such that the reduced closed substack of the intrinsic normal cone is isotropic (see Section \ref{sec:VFC} for the precise definition).
Consequently, given orientations of $\EE$ and $H^1(X,T_X)_\gamma:=H^1(X,T_X)/\ker(\sfB_\gamma)$, there exists a canonical reduced virtual cycle
\[\left[\PTqvX\right]^\red \in A_{\rvd}\left(\PTqvX\right), \qquad \rvd = n- \frac12 \gamma^2 + \frac12 \rho_\gamma,\]
which is independent of the choice of $V$.\footnote{The subspace $V$ depends on a choice, but $H^1(X,T_X)_\gamma$ is canonically defined. These spaces are identified via the map $V\subset H^1(X,T_X)\to H^1(X,T_X)_\gamma$.}
\end{theorem}

We prove that $\phi^\red$ is an obstruction theory via the {\em algebraic twistor family} \cite{KT1} associated to the transversal slice $V$ and prove the isotropic condition via a variation of the {\em cosection localization} in \cite{KL13,KP20} for the non-degenerate cosections $\Sigma$. In contrast to the {\em isotropic} cosections in \cite{KP20}, the {\em non-degenerate} cosections considered in this paper allow us to handle cases where the virtual dimension $n-\frac12\gamma^2$ is not an integer, but the reduced virtual dimension $n-\frac12\gamma^2+\frac12\rho_\gamma$ is an integer, which actually occurs quite often (e.g. Example \ref{Ex:1}).

We note that the algebraic twistor method alone does not produce a reduced virtual cycle
since the isotropic condition for $\phi$ does not imply it for $\phi^\red$.
The cosection localization method alone produces the reduced virtual cycle as a short-cut, but does not give a reduced obstruction theory.
Combining the two methods enables us to directly apply all the known tools for handling Oh-Thomas virtual cycles (torus localization \cite{OT}, cosection localization \cite{KP20}, virtual pullback \cite{Par}) to the reduced virtual cycle. 
We will indeed apply these tools in the sequels to this paper.

We observe that the reduced virtual cycle is closely related to semi-regularity.
In \cite{BF03}, a perfect complex $[I\udot] \in \curP$ is called {\em semi-regular} if the semi-regularity map $\sr$ in \eqref{eq:sr1} is injective.

\begin{theorem}[Theorem \ref{Thm:SR=smoothofrvd}]\label{thm:4}
Let $X$ be a Calabi-Yau 4-fold, $v \in H^*(X,\Q)$, and $q\in \{-1,0,1\}$. 
For any $\PT_q$ pair $[(F,s)] \in \PTqvX$, the associated complex $I\udot$ is semi-regular if and only if the moduli space $\PTqvX$ is smooth of dimension $\rvd=n-\frac12\gamma^2+\frac12\rho_\gamma$ at  $[I\udot]$.
\end{theorem}

As a direct corollary of Theorem \ref{thm:4}, if all $[(F,s)] \in \PTqvX$ are semi-regular, the reduced virtual cycle equals the fundamental cycle (for a certain choice of orientations).
Thus, heuristically, we view the semi-regular case as the {\em ideal situation}. Put differently, we view the reduced virtual cycle as the correct {\em virtual generalization} of the fundamental cycle in the semi-regular case.

For Calabi-Yau 4-folds, Theorem \ref{thm:4} generalizes the results of Bloch \cite{Blo} and Buchweitz-Flenner \cite{BF03}. Indeed, it is shown in loc.~cit.~that semi-regularity implies smoothness, whereas Theorem \ref{thm:4} also provides a formula for the dimension and a converse.
\begin{example}
We revisit  Example \ref{Ex:1}.
All planes $S \cong \PP^2$ in a smooth sextic $4$-fold $X$ are semi-regular and the moduli space ${\curly P}^{(-1)}_v(X) \cong {\curly P}^{(0)}_v(X) \cong {\curly P}^{(1)}_v(X)$, with $v = \ch(\O_S)$, consists of isolated reduced points (Corollary \ref{cor:cicy4}). 
Hence the reduced virtual dimension is zero and the contribution of any plane $S\subset X$ to the reduced virtual cycle  is $\pm[\{S\}]$. 
\end{example}

More generally, we give a simple description of the semi-regularity map for local complete intersection surfaces. This identification is important for computing $\rho_\gamma$ and the reduced virtual dimension.
Determining $\rho_\gamma$ is an interesting new aspect of the surface counting theory. For example, $\rho_\gamma$ is only sub-additive in $\gamma$, e.g., $\rho_{d \gamma} = \rho_{\gamma}$ for any $d \in \Z_{>0}$. 

\begin{proposition}[Proposition~\ref{prop:4.6.1}]
Let $i: S\hookrightarrow X$ be a local complete intersection surface inside a Calabi-Yau 4-fold $X$. Then there exists an isomorphism 
\[\theta_S: H^1(S,N_{S/X})\to \Ext^2_X(I_{S/X},I_{S/X})_0\]
preserving non-degenerate symmetric bilinear forms induced by Serre duality. Under this isomorphism, the sheaf theoretic obstruction and semi-regularity maps \eqref{eq:t1} match with the ones from the deformation  theory of subschemes.
\end{proposition}
We identify the dual of the semi-regularity map \eqref{eq:sr1} with the obstruction map $\ob_{S/X}:H^1(X,T_X)\to H^1(S,N_{S/X})$ for Hilbert schemes, which is easier to determine. This leads to the following corollary:
\begin{corollary}[Corollary~\ref{cor:cicy4}, \ref{cor:4.6.5}]\label{cor:6} 
\hfill
\begin{enumerate}
\item[$(1)$] Let $X\subset \PP^N$ be a smooth complete intersection Calabi-Yau 4-fold. Suppose $S\subset X$ is a complete intersection surface in $\PP^N$. Then $I_{S/X}$ is semi-regular. Moreover if $v=\ch(\O_S)$ and $q \in \{-1,0,1\}$, then $\PTqvX$ is smooth of dimension $\rvd=h^0(N_{S/X})$ at $I_{S/X}$.
\item[$(2)$] Let $X$ be a Calabi-Yau 4-fold with non-degenerate holomorphic $2$-form. Suppose $i:S\hookrightarrow X$ is a smooth lagrangian surface and $v=\ch(\O_S)$. Then $\rvd=h^{1,0}(S) - \frac{1}{2} \dim \coker(i^*)$ where $i^*: H^1(X,\Omega_X^1)\to H^1(S,\Omega_S^1)$.
\end{enumerate}
\end{corollary}

For Corollary \ref{cor:6} (1) and $X$ a sextic 4-fold, we show that $\rvd$ is often zero, but can be positive in some cases.
For Corollary \ref{cor:6} (2), we find that lagrangian planes are semi-regular with $\rvd = 0$, but we also consider Schoen surfaces \cite{Sch} embedded in their Albanese 4-fold which are semi-regular with $\rvd = 4$. 
Furthermore, for Corollary \ref{cor:6} (2), we also consider non-semi-regular cases on hyperk\"ahler 4-folds with lagrangian fibration and the Hilbert square $K3^{[2]}$.
Finally, we show that $\rho_\gamma$ vanishes when  $h^{2,0}(X)=0$ and the class $\gamma$ is a product of two $(1,1)$-classes. In this case, non-reduced virtual cycles contain non-trivial information. 
See Section~\ref{ss:rvd} for a precise list of examples.

\subsection{Deformation invariance}
Let
\[f:\cX\to \cB\]
be a {\em family of Calabi-Yau 4-folds}, i.e., a smooth projective morphism of relative dimension $4$ to a smooth connected affine scheme $\cB$ such that $\omega_{\cX/\cB} \cong \O_{\cX}$ and with connected fibres. Consider a horizontal section 
\[\tv\in \bigoplus_{p=2}^4 F^p H^{2p}_{DR}(\cX/\cB),\] 
where $F^p H^{2p}_{DR}(\cX/\cB)$ is the $p$th filtered piece of the Hodge bundle $H^{2p}_{DR}(\cX/\cB)$. 
Then $\tv_b \in H^*(\cX_b,\C)$ is a locally constant family of cohomology classes on the fibres $\cX_b:=\cX\times_{\cB}\{b\}$ over $b\in \cB$.
The component $\tv_2$ induces a section  $\tgamma$ of $R^2f_*\Omega^2_{\cX/\cB}$. 
The bilinear forms $\sfB_{\tgamma_b}$ on $H^1(\cX_b,T_{\cX_b})$ for $b \in \cB$ extend to a bilinear form $\sfB_{\tgamma}$ on the vector bundle $R^1f_*(T_{\cX/\cB})$. 

Let $\PTqvXB$ be the moduli space of $\PT_q$-stable pairs on the fibres of $f:\cX\to\cB$. This relative moduli space fits into the fibre diagram
\begin{equation*}
\xymatrix{
\PTqvXb \ar@{^{(}->}[r]^{} \ar[d] & \PTqvXB \ar[d]^{} \\
\{b\} \ar@{^{(}->}[r]^{i_b} & \cB
}	
\end{equation*}
for each $b\in \cB$. We have a relative obstruction theory $\phi_{\cX/\cB} :\EE_{\cX/\cB} \to \LL_{\PTqvXB/\cB}$.

Our third main result (Section \ref{sec:deformationinvariance}) is the deformation invariance of the reduced virtual cycle in Theorem \ref{thm:3}.

\begin{theorem}[Theorem \ref{Thm:DefInv}] \label{thm:5}
Let $f:\cX\to\cB$ be a family of Calabi-Yau 4-folds.
Let $\tv \in \bigoplus_p F^pH^{2p}_{DR}(\cX/\cB)$ be a horizontal section, $\tgamma \in \Gamma(\cB, R^2f_*\Omega^2_{\cX/\cB})$ the induced section, and  $q\in \{-1,0,1\}$.
Assume that the function $b\in \cB \mapsto \rho_{\tgamma_b}$ is constant. Given orientations 
\[\O_{\PTqvXB} \cong \det(\EE_{\cX/\cB}),\qquad \O_{\cB} \cong \det(R^1f_*(T_{\cX/\cB})/\ker(\sfB_{\tgamma})),\]
there exists a cycle class
\[\left[\PTqvXB\right]^\red \in A_{\rvd+\dim(\cB)}\left(\PTqvXB\right)\]
such that, for all $b \in \cB$, we have
\[\left[\PTqvXb\right]^\red = i_b^!\left[\PTqvXB\right]^\red \in A_{\rvd}\left(\PTqvXb\right)\]
for the induced orientation.
Here $i_b^!$ denotes the refined Gysin pullback.
\end{theorem}

In general, reduced virtual cycles given by cosection localization \cite{KL13} are deformation invariant only under an additional assumption. In Theorem \ref{thm:5}, this assumption follows from the fact that the Hodge locus of $\tgamma$ in $\cB$ is $\cB$ itself.

The existence of a relative orientation is required in Theorem \ref{thm:5}. We conjecture (Conjecture \ref{conj:familyorientation}) that a relative orientation exists after a finite {\'e}tale base change of the base $\cB$ and we prove its topological counterpart (Corollary \ref{Cor:Or.2}). The existence of a relative orientation is crucial for deformation invariance of numerical invariants, but is \emph{not} needed in the main application below.

\subsubsection*{Variational Hodge Conjecture}
The Hodge conjecture predicts that for any smooth projective variety $X$, all rational $(p,p)$-classes on $X$ are algebraic.
In \cite{Gr66} Grothendieck introduced a variant of the Hodge conjecture.

\begin{conjecture}[variational Hodge conjecture]\label{conj:1}
Let $X$ be a smooth projective variety and $\gamma$ an algebraic $(p,p)$-class on $X$. For any smooth projective morphism $f:\cX\to \cB$ to a smooth connected scheme $\cB$ and a horizontal section $\tv_p$ of $F^p\cH^{2p}_{DR}(\cX/\cB)$ 
such that $\cX_0 \cong X$ and $(\tv_p)_0=\gamma$ for some closed point $0 \in \cB$, the cohomology classes $(\tv_p)_b$ are algebraic for all closed points $b\in\cB$. 
\end{conjecture}

Our fourth main result is an affirmative answer to Conjecture \ref{conj:1} when there exists a non-zero reduced virtual cycle. 
\begin{theorem}[Theorem \ref{Thm:VHC}]\label{thm:6}
Let $X$ be a Calabi-Yau 4-fold and let $\gamma$ be a $(2,2)$-class on $X$.
If for some $v \in H^*(X,\Q)$ with $v_2=\gamma$, $q \in \{-1,0,1\}$, and choice of orientation on $\PTqvX$ we have
\[[\PTqvX]^\red \neq 0 \in A_{\rvd}(\PTqvX)\]
then Conjecture \ref{conj:1} holds for $X$ and $\gamma$.
\end{theorem}

We prove Theorem \ref{thm:6} via the deformation invariance in Theorem \ref{thm:5}. Here we do not require the two technical assumptions in Theorem \ref{thm:5}. 
The function $b\in\cB\mapsto \rho_{\tgamma_b}\in \Z$ is lower semi-continuous and the reduced virtual cycle vanishes if $\rho_{\tgamma_b}$ drops at $0 \in \cB$  (Proposition \ref{Prop:vanishingforjumpingrho}). 
The assumption on orientations can be removed by working with virtual {\em cycles} instead of {\em invariants}.

We view Theorem \ref{thm:6} as a {\em virtual} generalization of a result of Buchweitz-Flenner \cite{BF03} (see also \cite{Blo}) for Calabi-Yau 4-folds. In fact, \cite{BF03} implies that Conjecture \ref{conj:1} holds if there exists a semi-regular $\PT_q$ pair $I\udot=[\O_X\to F]$ on $X$ with $\ch_2(F) = \gamma$.
Theorem \ref{thm:6} recovers the result of \cite{BF03} for Calabi-Yau 4-folds, because the reduced virtual cycle equals the fundamental cycle near a semi-regular point (Theorem \ref{Thm:SR=smoothofrvd}).

This application to the variational Hodge conjecture shows that it can be useful to work with the virtual cycle instead of invariants, because one can restrict to each connected component of the moduli space. 

\begin{example}[Example~\ref{ex:k3k3}]
We give a non-semi-regular application of Theorem~\ref{thm:6}. Let $X=S_1\times S_2$ be a product of two K3 surfaces. Let $\beta_1\in H^{1,1}(S_1)$, $\beta_2\in H^{1,1}(S_2)$ be two irreducible effective classes with $\beta_1^2\geq 2$ and $\beta_2^2=0$. 
Let $\gamma=\beta_1\cup \beta_2$.
Then, for any $q \in \{-1,0,1\}$, we have $\PTqvX=|\O_{S_1}(\beta_1)|\times \PP^1$ for some $v\in H^*(X,\mathbb{Q})$ with $v_2=\gamma$. In this case, the reduced virtual cycle is $\pm h^0(\O_{S_1}(\beta_1))\cdot[\pt\times\PP^1]$. Therefore Conjecture \ref{conj:1} holds for such $\gamma$. 
This is a non-semi-regular case since the reduced virtual cycle differs from the fundamental cycle.
\end{example}


\subsection{Moduli stack of semi-stable sheaves}

In Appendix \ref{Appendix:Artinstacks}, we generalize the construction of reduced virtual cycles to {\em Artin} stacks (which are global quotient stacks). In particular, we consider the moduli stack $\cM^{H,ss}_v(X)$ of Gieseker semi-stable sheaves on $X$ (with respect to a polarization $H$) with Chern character $v \in H^*(X,\Q)$, rigidified by the action of $B\GG_m$. 
The Chern character $v$ is arbitrary so this includes moduli of $2$-dimensional \emph{torsion} sheaves and moduli of \emph{higher rank} sheaves.

\begin{theorem}[Theorem \ref{Thm:semistable.rigidified}]\label{thm:7}
Let $(X,H)$ be a polarized Calabi-Yau 4-fold and $v \in H^*(X,\Q)$. 
For any choice of orientation on $\cM^{H,ss}_v(X)$, there exists a canonical reduced virtual cycle
\[[\cM^{H,ss}_v(X)]^\red \in A_{1- \frac12\chi(v,v)+\frac12\rho_\gamma}(\cM^{H,ss}_v(X))\]
where $\gamma:=v_2$ and $\chi(v,v):=\int_Xv\dual \cdot v \cdot \td(X)$.
Moreover, if $[\cM^{H,ss}_v(X)]^\red \neq 0$, then Conjecture \ref{conj:1} holds for $X$ and $\gamma$.
\end{theorem}

We construct the reduced virtual cycle as a class in Kresch's Chow group \cite{Kre} by extending Oh-Thomas's construction \cite{OT} and Kiem-Li's cosection localization \cite{KL13} to Artin stacks (which are global quotients).
If there exist strictly semi-stable sheaves, this virtual cycle does not provide numerical invariants since there is no push-forward map for $\cM^{H,ss}_v(X) \to \Spec(\C)$. Nonetheless, this virtual cycle is sufficient for applications to the variational Hodge conjecture.



There is a ``fixed determinant version'' of Theorem \ref{thm:7} under the assumption $v_0\neq 0$ and $v_1=0$ (Theorem \ref{Thm:semistable.fixeddet}).

\subsection{Derived algebraic geometry}

In Appendix \ref{Appendix:ReductionviaDAG}, we revisit Theorem \ref{thm:3} via derived algebraic geometry.
The main result is the following:

\begin{theorem}[Corollary \ref{Cor:ReducedOTviaDAG}]
In the situation of Theorem \ref{thm:3}, there exists a $(-2)$-shifted symplectic derived enhancement of $\PTqvX$ inducing the reduced symmetric obstruction theory $\phi^\red$.
\end{theorem}

We first introduce general reduction methods via $(-1)$-shifted {\em closed}/{\em exact} 1-forms (Theorem \ref{Thm:DerivedCL}).
In order to apply this to the moduli space $\PTqvX$, we show that the semi-regularity maps can be enhanced to $(-1)$-shifted closed $1$-forms (Proposition \ref{Prop:SRclosed}).
This is achieved by generalizing the integration map of \cite{PTVV} (Theorem \ref{Prop:Integration}).

As a byproduct, we obtain a new proof of Bloch's conjecture for surfaces on Calabi-Yau $4$-folds (Remark \ref{Rem:BlochConj}). The general case of Bloch's conjecture was shown by Pridham \cite{Pri} and Bandiera-Lepri-Manetti \cite{BLM} (and by Iacono-Manetti \cite{IM} when the normal bundle is extendable).


\subsection{Relations to physics}
Finding connections between surface counting on Calabi-Yau 4-folds and string theory is an interesting direction. We briefly point out relations to two (disjoint) parts of the physics literature:
\begin{itemize}
    \item In M-/F-theory, the complex moduli of a Calabi-Yau 4-fold $X$ are fixed by fixing a rational 4-cycle class called the $G_4$-flux.
    Then Gukov-Vafa-Witten \cite{GVW} introduced an associated superpotential giving rise to a symmetric bilinear form which (up to change of coordinates) coincides with $\sfB_\gamma$. Explicit bounds on the rank $\rho_\gamma$ of $\sfB_\gamma$ appear to be of interest to physicists \cite{BCV,BV,BBGL,GGHHP,LVWX}.\footnote{We thank Grimm for pointing us to the physics literature, specifically \cite{BCV}.}
    \item In the series of papers starting from \cite{Nek1}, Nekrasov studied the moduli space of (non-commutative) instantons on origami spacetimes in $\mathbb{C}^4$. Our choice of $K$-theoretic insertions for toric Calabi-Yau 4-folds in the sequels (and the choice in \cite{CKM, KR}) are motivated by Nekrasov-Piazzalunga's study of SUSY Yang-Mills theory on $\C^4$ \cite{Nek,NP}.
\end{itemize}

\subsection{Sequels}
We briefly mention the content of the sequels to this paper:
\begin{itemize}
    \item In Part II we study the $\DT/\PT_0$ correspondence. We conjecture a general correspondence and provide evidence on toric Calabi-Yau 4-folds via a vertex formalism, and on compact Calabi-Yau 4-folds via a virtual Lefschetz principle.
    \item In Part III we study the $\PT_0/\PT_1$ correspondence in special cases. We consider virtual projective bundles and Weierstrass 4-folds.
    \item In Part IV we compute $\PT_1$ invariants in some examples including local surfaces and Calabi-Yau 4-folds with a nondegenerate holomorphic 2-form.
\end{itemize}
We will also further explore the link between the variational Hodge conjecture and the reduced virtual cycles.

\subsection{Related work}
Reduced virtual cycles for {\em curves} on hyperk{\"a}hler $4$-folds were studied by Cao-Oberdieck-Toda in \cite{COT1,COT2}.
For the moduli space of $\PT_1$-pairs on $X$, Theorem \ref{thm:2} was independently proven by Gholampour-Jiang-Lo \cite{GJL}. 
Correspondences for $\PT_q$ invariants in the context of Joyce's wall-crossing framework \cite{GJT,Joy,J21} will be an interesting direction to study.

There are two types of generalizations of cosection localization \cite{KL13} to Donaldson-Thomas theory of Calabi-Yau $4$-folds \cite{BJ,OT}. In \cite{Sav}, Savvas showed vanishing of Borisov-Joyce virtual cycles with a nowhere vanishing isotropic cosection.
In \cite{KP20}, Kiem and the third-named author studied cosection localization for Oh-Thomas virtual cycles \cite{OT}.

\subsection*{Acknowledgements}
We thank Rahul Pandharipande for encouraging us to study surfaces on Calabi-Yau 4-folds, which prompted this project.
The second-named author wishes to thank Richard Thomas for previous collaborations on reduced virtual cycles for curves on surfaces, which obviously impacted this paper. He also thanks Yalong Cao for their collaborations which led him to study Calabi-Yau 4-folds in the first place.
The third-named author thanks Young-Hoon Kiem for teaching him the cosection localization and a previous collaboration on its generalization to Calabi-Yau $4$-folds.
We are also grateful to Markus Upmeier for conversations on orientation problems. We thank Dhyan Aranha, Arkadij Bojko, Amin Gholampour, Thomas Grimm, Yunfeng Jiang, Dominic Joyce, Woonam Lim, Jason Lo, Sergej Monavari, Georg Oberdieck, Jeongseok Oh, J{\o}rgen Rennemo, and Cumrun Vafa for many interesting conversations.
Y.B.~is supported by ERC Grant ERC-2017-AdG-786580-MACI and Korea Foundation for Advanced Studies (KFAS).
The project has received funding from the European Research Council (ERC) under the European Union Horizon 2020 research and innovation
program (grant agreement No.~786580).
M.K.~is supported by NWO grant VI.Vidi.192.012.

\subsection{Notation and conventions}
We use the following notation and conventions throughout the paper.
\begin{itemize}
\item All schemes and algebraic stacks are assumed to be of finite type over the field $\C$, unless stated otherwise. 
A {\em variety} is a separated integral scheme.
A {\em point} $x\in X$ in a scheme or an algebraic stack $X$ means a $\C$-valued point. 
\item A {\em Calabi-Yau $n$-fold} is a smooth projective variety $X$ of dimension $n$ with $K_X \cong \O_X$. 
\item For a coherent sheaf $F$ on a smooth quasi-projective variety $X$, we use the following notation for the various duals:
\begin{itemize} 
\item $F^\vee := R\hom_X(F,\O_X)$;
\item $F^D:=\ext^c_X(F,K_X)$, where $c$ is the codimension of $F$ in $X$.
\end{itemize}
\item For a smooth projective variety $X$, we denote by 
\begin{itemize}
\item $\Coh(X)$ the abelian category of coherent sheaves on $X$;
\item $\cCoh(X)$ the moduli stack of coherent sheaves on $X$;
\item $D^b_{\mathrm{coh}}(X)$ the bounded derived category of coherent sheaves on $X$.
\end{itemize}
\item For a perfect complex $\FF$ on a scheme $\sX$ of tor-amplitude $(-\infty,0]$, denote by $\fC(\FF):=h^1/h^0(\FF\dual)$ the associated abelian cone stack \cite[Prop.~2.4]{BF}.
\item An {\em orthogonal bundle} on a scheme $\sX$ is a pair of a vector bundle $E$ and an isomorphism $q: E\to E\dual$ such that $q=q\dual$. By abuse of notation, we also denote by $q:E \to \bbA^1_{\sX}$ the corresponding quadratic function. An {\em orientation} of an orthogonal bundle $E$ is an isomorphism $o : \O_{\sX} \cong \det(E)$ such that $\det(q) = (-1)^{\frac{r(r-1)}{2}}(o \circ o\dual)^{-1}$ where $r=\rank(E)$.
\item A {\em symmetric complex} on a scheme $\sX$ is a pair of a perfect complex $\EE$ on $\sX$ of tor-amplitude $[-2,0]$ and an isomorphism $\theta: \EE\dual[2] \cong \EE$ such that $\theta\dual[2]=\theta$.  An {\em orientation} of a symmetric complex $\EE$ is an isomorphism $o:\O_{\sX}\cong \det(\EE)$ such that $\det(\theta)=  (-1)^{\frac{r(r-1)}{2}}o \circ o\dual$  where $r=\rank(\EE)$.\footnote{This convention is slightly different from \cite{Par}. In loc.~cit., all symmetric complexes are assumed to be equipped with an orientation.}
\item For a morphism $f : \sX \to \sY$ of algebraic stacks, we denote by $\LL_{\sX/\sY}$ the full cotangent complex \cite{Ill} (\cite{Ols}) and by $\tau^{\geq -1}\LL_{\sX/\sY}$ the truncated cotangent complex. 
\item For any algebraic stack $\sX$, we denote by $A_*(\sX)$ the Chow group of Kresch \cite{Kre} with $\Q$-coefficients. For any $d \in \Q \setminus \Z$, we let $A_d(\cX):=0$.
\item For a morphism $f:\sX\to\sY$ and a coherent sheaf (resp.~a perfect complex) $E$ on $\sY$, we denote the pullback $f^*E$ (resp.~the derived pullback $L f^*E$) by $E|_{\sX}$. 
\end{itemize}

\section{Stable pairs}\label{sec:pair}

In this section, we introduce the $\PT_q$-stability conditions on 2-dimensional pairs for $q\in\{-1,0,1\}$, and present basic properties which will be used throughout the paper and its sequels. The $\PT_{-1}$-stable pairs coincide with $\DT$-stable pairs (i.e.~ideal sheaves), $\PT_1$-stable pairs are $\PT$-stable pairs (i.e.~Le Potier stable pairs \cite{Pot1,Pot2,PT1}), and $\PT_0$-stable pairs is a new intermediate notion
\begin{equation*}\label{eq:DTPT0PT1wallcrossing}
\DT:=\PT_{-1} \leadsto \PT_0 \leadsto \PT_{1}=:\PT.
\end{equation*}


Any $\PT_q$ pair $(F,s)$ on a smooth variety $X$ has an associated complex
\[I\udot := [\O_X \xrightarrow{s} F]\]
in $D^b _{\mathrm{coh}}(X)$ --- the bounded derived category of coherent sheaves on $X$.
We study $\PT_q$ pairs through these associated complexes. We provide descriptions of these associated complexes under some various assumptions on the singularities of the support of $F$. 
The results in this section provide some insights into the geometric meaning of $\PT_q$-stability. They will also play an important role in the sequels to this paper.

Throughout this section, we fix a smooth projective variety $X$ of dimension $n$ over the complex field $\C$.

\subsection{$\PT_q$-stability}

For any integer $q \geq -1$, consider the two full subcategories
\begin{align*}
\Coh_{\leq q}(X) &:= \{F \in \Coh(X) : \dim(F)\leq q\},\\
\Coh_{\geq q+1}(X) &:=\{ F \in \Coh(X) : T_q(F)=0\}.
\end{align*}

\begin{definition}\label{def:stabilityconditionsofpairs'}
Let $F$ be a coherent sheaf on $X$ with 2-dimensional support and let $s: \O_X \to F$ be a section.
For $q \in \{-1,0,1\}$, we say that the pair $(F,s)$ is {\em $\PT_q$-stable} if
\[F \in \Coh_{\geq q+1}(X) \and Q:=\coker(\O_X \xrightarrow{s} F) \in \Coh_{\leq q}(X).\]
\end{definition}


Abbreviating, we also refer to $\PT_q$-stable pairs as {\em $\PT_q$ pairs}.

\begin{example}
The $\PT_q$-stability conditions for the following two extreme cases are well-known.
	\begin{enumerate}
		\item [DT)] A pair $(F,s)$ is $\PT_{-1}$-stable if and only if $s:\O_X \to F$ is surjective. Hence $\PT_{-1}$ pairs on $X$ correspond to 2-dimensional closed subschemes of $X$ (which may have irreducible components of dimension 0 and 1). Thus we refer to $\PT_{-1}$-stability as $\DT$-stability.
		\item [PT)] A pair $(F,s)$ is $\PT_{1}$-stable if and only if $F$ is pure and $\dim(Q)\leq 1$. Hence $\PT_{1}$ pairs are exactly the 2-dimensional stable pairs in the sense of Le Potier \cite[Def.~4.2]{Pot1}.\footnote{Equivalently, $\PT_{1}$ pairs are the 2-dimensional $\alpha$-stable pairs in the sense of \cite[Def.~4.4]{Pot2} for sufficiently big $\alpha(t) \in \Q[t]$.} Their 1-dimensional analogs
		are used in \cite{PT1,CMT2} (and many other places). Thus we sometimes refer to $\PT_{1}$-stability as $\PT$-stability.
	
	\end{enumerate}
\end{example}

We present some elementary properties of $\PT_q$ pairs. We begin with an observation on the supports.

\begin{lemma}\label{lem:support}
Given a $\PT_q$ pair $(F,s)$ on $X$, the scheme theoretic support of the coherent sheaf $F$ is the scheme theoretic support of the section $s$.
\end{lemma}
\begin{proof} 
The proof is analogous to \cite[Lem.~1.6]{PT1}. 
Let $\mathrm{Ann}(s) \subseteq \O_X$ be the ideal sheaf of the support of $s$. Away from the support of $Q=\coker(s)$, the sheaf $F$ is generated by the section $s$. Hence $\mathrm{Ann}(s)\cdot F \subseteq F$ is set-theoretically supported in the support of $Q$, and thus $\mathrm{Ann}(s)\cdot F \in \Coh_{\leq q}(X)$. Since $F \in \Coh_{\geq q+1}(X)$, we obtain $\mathrm{Ann}(s)\cdot F=0$.
\end{proof}

We note that the $\PT_q$ pairs have limit descriptions. 

\begin{proposition} \label{lem:limPTq} Let $W \subseteq Z \subseteq X$ be closed subschemes such that 
$\O_Z \in \Coh_{\geq q+1}(X)$ and $\O_W \in \Coh_{\leq q}(X)$. Suppose $W$ is reduced 
and denote its ideal by $\fI \subseteq \O_Z$. 
Then $\PT_q$ pairs $(F,s)$ with support $Z$ and cokernel $Q$ satisfying $\Supp(Q)^{\mathrm{red}} \subseteq W$ are equivalent to coherent subsheaves of $\varinjlim \hom_X(\fI^r,\O_Z) / \O_Z$.
\end{proposition}

\begin{proof}
The proof is analogous to \cite[Prop.~1.8]{PT1}. Choose a $\PT_q$ pair $(F,s)$. By Lemma \ref{lem:support}, we can form a short exact sequence
\begin{equation}\label{SES.OFQ}
\xymatrix{
0 \ar[r] & \O_Z \ar[r] & F \ar[r] & Q \ar[r] & 0,
}\end{equation}
where $Z$ is the support of $F$ and $Q$ is the cokernel of $s$. Since $\O_Z \in \Coh_{\geq q+1}(X)$ and $Q \in \Coh_{\leq q}(X)$, we have $\hom_X(Q,\O_Z)=0$. Hence the short exact sequence \eqref{SES.OFQ} gives us an inclusion \[\hom_X(F,\O_Z) \hookrightarrow \hom_X(\O_Z,\O_Z)=\O_Z,\] which is an isomorphism away from the support of $Q$. Hence $\hom_X(F,\O_Z)$ is an ideal such that $\fI^r \subseteq \hom_X(F,\O_Z)$ for some integer $r$. Since $F \in \Coh_{\geq q+1}(X)$, we have \[F \subseteq \hom_X(\hom_X(F,\O_Z),\O_Z) \subseteq \hom_X(\fI^r,\O_Z)\subseteq \varinjlim_{r \to \infty}\hom_X(\fI^r,\O_Z),\] where the last inclusion map follows from $\hom_X(\frac{\fI^r}{\fI^{r+1}},\O_Z)=0$.

On the other hand, choose a coherent subsheaf $F \subseteq \varinjlim \hom_X(\fI^r , \O_Z)$ containing $\O_Z$. Since $F$ is coherent, we have $F \subseteq \hom_X(\fI^r , \O_Z)$ for some sufficiently large $r$. Let $s: \O_X \to \O_Z \hookrightarrow F$ be the composition. Consider a subsheaf $G\subseteq F$ of dimension $\leq q$. By adjunction, the inclusion map $G \hookrightarrow \hom_X(\fI^r,\O_Z)$ corresponds to a map $G \otimes \fI^r \to \O_Z$, which is zero since $\O_Z \in \Coh_{\geq q+1}(X)$. Hence $G=0$. This means that $F \in \Coh_{\geq q+1}(X)$. Therefore $(F,s)$ is a $\PT_q$-stable pair.
\end{proof}


\begin{remark} \label{rem:generalizePTq}
Some results in this paper hold for $d$-dimensional pairs. For a coherent sheaf $F$ on $X$ with a $d$-dimensional support, a section $s: \O_X \to F$, and any $q \in \{-1,0,\ldots, d-1\}$, we define the pair $(F,s)$ to be {\em $\PT_q$-stable} if
\[F \in \Coh_{\geq q+1}(X) \and Q:=\coker(\O_X \xrightarrow{s} F) \in \Coh_{\leq q}(X).\]
Then Lemma \ref{lem:support} and Proposition \ref{lem:limPTq} hold for $d$-dimensional $\PT_q$ pairs as well (as is immediately obvious from the proofs). However, most parts of the paper require $d=2$, so we restrict to this case.
\end{remark}

\subsection{Special cases}

Given a $\PT_0$ pair $(F,s)$ on an $n$-dimensional smooth projective variety $X$, we may consider the following two special cases:
\begin{enumerate}
\item $(F,s)$ is additionally $\DT$-stable, i.e.~$s:\O_X \to F$ is surjective.
\item $(F,s)$ is additionally $\PT_1$-stable, i.e.~$F$ is pure.
\end{enumerate}
The first case (1) is very clear. It consists of 2-dimensional subschemes $Z \subset X$ with no (possibly embedded) $0$-dimensional components. The second case (2) is more interesting. This case depends on the singularities of the support. 


\begin{proposition}\label{Lem.PT0=PT1}
Let $S$ be a \emph{pure} 2-dimensional subscheme of $X$. Then the following three conditions are equivalent:
\begin{enumerate}
\item[$\mathrm{(i)}$] $S$ is not Cohen-Macaulay.
\item[$\mathrm{(ii)}$] $\ext^{n-1}_X(\O_S,K_X) \neq 0$.
\item[$\mathrm{(iii)}$] There exists a $\PT_0$ pair $(F,s)$ with support $S$ and non-zero cokernel $Q$.
\end{enumerate}
Furthermore, for any $m \geq 0$, the $\PT_0$ pairs $(F,s)$ with support $S$ and $\chi(Q) = m$ are in bijective correspondence with the closed points of the Quot scheme
$$
{\curly Quot}_X(\ext^{n-1}_X(\O_S,K_X),m)
$$
parametrizing length $m$ quotients of the 0-dimensional sheaf $\ext^{n-1}_X(\O_S,K_X)$.
\end{proposition}

\begin{proof}
Since $S$ is pure, $\ext^i_X(\O_S,K_X)=0$ for $i \neq n -2, n -1$ and $\ext^{n-1}_X(\O_S,K_X)$ is 0-dimensional \cite[Prop.~1.1.6, 1.1.10]{HL}. Let $i: S \hookrightarrow X$ denote the inclusion map. Since $R\hom_X(\O_S,K_X) = R i_* i^! K_X$, we have $\ext^{n-1}_X(\O_S,K_X)=0$ if and only if $S$ is Cohen-Macaulay. 

For any 0-dimensional sheaf $Q$, we have
\begin{equation} \label{eqn:ext1iso}
\Ext^1_X(Q,\O_S) \cong \Hom_X(\ext^{n-1}_X(\O_S,K_X), Q^D)
\end{equation}
where $Q^D:=\ext^n_X(Q,K_X)$. Hence (iii) implies (ii). On the other hand, if $\ext^{n-1}_X(\O_S,K_X)\neq 0$, we can always find a closed point $P\in X$ in the support of $\ext^{n-1}_X(\O_S,K_X)$ and a non-zero map $\ext^{n-1}_X(\O_S,K_X) \to \O_P$. By \eqref{eqn:ext1iso}, we find a non-split extension in $\Ext^1_X(\O_P^D,\O_S)\neq 0$ from which we obtain a $\PT_0$ pair $(F,s)$ with non-zero cokernel. By Lemma \ref{lem:support}, the support of this $\PT_0$ pair is $S$. This proves (ii)$\implies$(iii). 

Finally, it follows from \eqref{eqn:ext1iso} that there exists a bijection between the collections
$$
\Big( \bigsqcup_{Q} \Ext^1_X(Q,\O_S) \Big) / \cong, \quad \quad \Big( \bigsqcup_Q \Hom_X(\ext^{n-1}_X(\O_S,K_X), Q) \Big) / \cong
$$
where the union is over all 0-dimensional sheaves on $X$ and we divide out by isomorphisms. This bijection is obtained by dualizing the short exact sequence $0 \to \O_S \to F \to Q \to 0$. Since $F$ is pure if and only if $\ext^{n-1}_X(F,K_X)$ is 0-dimensional and $\ext^{n}_X(F,K_X) = 0$ \cite[Prop.~1.1.10]{HL}, it follows that the short exact sequences with $F$ \emph{pure} correspond precisely to the \emph{surjective} morphisms. The result follows.
\end{proof}

Now consider a $\PT_1$ pair $(F,s)$. We can similarly consider the following two special cases:
\begin{enumerate}
\item [(3)] $(F,s)$ is additionally $\PT_0$-stable, i.e.~$Q$ is 0-dimensional.
\item [(4)] $F$ is additionally reflexive.
\end{enumerate}
Case (3) is exactly case (2) above. Case (4) is a ``dual'' version of case (1) above, in the sense of the following proposition:

\begin{proposition}\label{lem:S2.1}
Let $(F,s)$ be a pair of a coherent sheaf $F$ on $X$ and a section $s: \O_X \to F$. 
\begin{enumerate}
\item[$(1)$] 
If $(F,s)$ is $\PT_0$-stable and $\chi(F)$ is minimal among all $\PT_0$ pairs with Chern characters $\ch_{n-2}(F)$ and $\ch_{n-1}(F)$, then $(F,s)$ is $\DT$-stable.
\item[$(2)$] 
If $(F,s)$ is $\PT_1$-stable and $\chi(F)$ is maximal among all $\PT_1$ pairs with Chern characters $\ch_{n-2}(F)$ and $\ch_{n-1}(F)$, then $F$ is reflexive.
\end{enumerate}	
\end{proposition}

In Proposition \ref{prop:boundofchiforPT}, we will show that $\chi(F)$ is bounded below for $\DT$ and $\PT_0$ pairs $(F,s)$ with fixed Chern characters  $\ch_{n-2}(F)$ and $\ch_{n-1}(F)$, and bounded above in the $\PT_1$ case.
\begin{proof}[Proof of Proposition \ref{lem:S2.1}]
(1) Form the short exact sequence
\[\xymatrix{ 0 \ar[r] & \O_Z \ar[r] & F \ar[r] & Q \ar[r] & 0,}\]
where $Z$ is the support of $F$ and $Q$ is the cokernel of $s$. Then we have
\[\ch_{n-2}(\O_Z)=\ch_{n-2}(F), \quad \ch_{n-1}(\O_Z)=\ch_{n-1}(F), \and \chi(\O_Z)=\chi(F)-\chi(Q).\]
Since $(\O_Z,1:\O_X \twoheadrightarrow \O_Z$) is a $\PT_0$ pair, and
$\chi(F)$ is minimal, we have $Q=0$. Therefore $(F,s)$ is a $\DT$ pair.


(2) Consider the canonical map
\[\theta_F : F \to F^{DD}.\]
The purity of $F$ implies $\ext^{n-1}_X(F,K_X) \in \Coh_{\leq 0}(X)$ and $\ext^n_X(F,K_X)=0$ by \cite[Prop.~1.1.10]{HL}. Moreover, the spectral sequence
\[\ext^p_X(\ext^{-q} _X(F,K_X),K_X) \Rightarrow F \]
in \cite[Lem.~1.1.8]{HL} gives us a short exact sequence
\[\xymatrix{0 \ar[r] & F \ar[r] & F^{DD} \ar[r] & \ext_X^{n-1}(F,K_X)^D \ar[r] & 0,}\]
where $\ext_X^{n-1}(F,K_X)^D \in \Coh_{\leq 0}(X)$. Hence $(F^{DD}, \theta_F \circ s)$ is also a $\PT_1$ pair such that
\[\chi(F^{DD}) = \chi(F) +\chi(\ext^{n-1}_X(F,K_X)^D) \geq \chi(F).\]
Since $\chi(F)$ is maximal, we have $F=F^{DD}$.
\end{proof}

\subsection{Associated complexes}

Let $X$ be a smooth projective variety of dimension $n$. Given a $\PT_q$ pair $(F,s)$ on $X$, we can form an associated complex
\[I\udot:=[\O_X \xrightarrow{s} F],\]
with $\O_X$ in degree zero. We view $I\udot$ as an element of $D^b_{\mathrm{coh}}(X)$ --- the bounded derived category of coherent sheaves on $X$. In this subsection, we analyze $\PT_q$ pairs in terms of the associated complex $I\udot$.

We first note that a $\PT_q$ pair is equivalent to its associated complex.

\begin{lemma}\label{lem:injectivity'}
Assume $X$ is a smooth projective variety of dimension $n \geq 4$.
Consider two $\PT_q$ pairs $(F_1,s_1)$ and $(F_2,s_2)$. Then the two associated complexes $I_1\udot:=[\O_X \xrightarrow{s_1} F_1]$ and $I_2\udot:=[\O_X \xrightarrow{s_2} F_2]$ are quasi-isomorphic if and only if the two $\PT_q$ pairs $(F_1,s_1)$ and $(F_2,s_2)$ are isomorphic.
\end{lemma}

\begin{proof}
The proof is analogous to \cite[Prop.~1.21]{PT1}. Since 
\[\Ext^{\leq 1}_X(F_1,\O_X)\cong H^{\geq n-1}(X,F_1\otimes K_X)^*=0,\]
by $F_1 \in \Coh_{\leq2}(X)$, the map
\[\C=\Hom_X(\O_X,\O_X) \to \Hom_X(I^\mdot_1,\O_X)\]
is an isomorphism. Therefore, any quasi-isomorphism $\phi:I^\mdot_1 \to I^\mdot_2$ can be completed to an isomorphism of distinguished triangles
\[\xymatrix{
I^\mdot_1 \ar[r] \ar[d]^{\phi} & \O_X \ar[r]^{s_1} \ar@{.>}[d]^c & F_1 \ar[r] \ar@{.>}[d]^\psi &  \\
I^\mdot_2 \ar[r]  & \O_X \ar[r]^{s_2} & F_2 \ar[r] & 
}\]
for some dotted arrows $c$ and $\psi$. Since $c\in \Gamma(X,\O_X^*)=\C^*$ is a non-zero constant, $c^{-1}\cdot\psi :F_1 \to F_2$ is an isomorphism of sheaves with $\psi(s_1)=s_2$.
\end{proof}

We show that a $\PT_0$ pair, for which the pure part of its scheme theoretic support is Cohen-Macaulay, can be expressed by a map from a $\DT$ pair to a $1$-dimensional pure sheaf.


\begin{proposition} 
For any $\PT_0$ pair $(F,s)$ on $X$ such that $S=\Supp(F)^{\pure}$ is Cohen-Macaulay, we have a natural quasi-isomorphism
\begin{equation}\label{Eq.PT0.ItoG}
I\udot \cong [I_{S/X} \to T_1(F)],  
\end{equation}
where $T_1(F)$ is zero or pure 1-dimensional and $I_{S/X}$ denotes the ideal sheaf of $S$ in $X$.

In particular, when $\dim(X) = 3$, we have $I_{S/X} = \O_X(-S)$ and 
\[I\udot \cong [\O_X \to T_1(F)(S)](-S)\]
which, for $T_1(F) \neq 0$, is a 1-dimensional stable pair on $X$ as in \cite{PT1}.
\end{proposition}

\begin{proof}
We may assume $T_1(F) \neq 0$, because otherwise the result is immediate from Proposition \ref{Lem.PT0=PT1}. Note that $F/T_1(F)$ is a pure 2-dimensional sheaf and $T_1(F)$ is a pure 1-dimensional sheaf. Hence the support of the $\PT_0$-stable pair $\O_X \to F/T_1(F)$ is also pure. Let $Z = \Supp(F)$, so $S = Z^{\pure}$. Since $T_1(\O_Z)=I_{S/Z}$ by definition, 
the support of the $\PT_0$ pair $\O_X \to F/T_1(F)$ is $S$. By Proposition \ref{Lem.PT0=PT1} above, we obtain the desired isomorphism $\O_S \cong F/T_1(F)$. Moreover, since the commutative square
\[\xymatrix{
I_{S/X} \ar[r] \ar[d] & T_1(F) \ar[d] \\
\O_X \ar[r] & F
}\]
is cartesian and cocartesian, we have the desired quasi-isomorphism \eqref{Eq.PT0.ItoG}.
\end{proof}


Dually, $\PT_1$ pairs (with Cohen-Macaulay support) can be expressed by maps between $1$-dimensional pure sheaves and $\DT$ pairs.

\begin{proposition} \label{lem:PT1CMsupp}
For any $\PT_1$ pair $(F,s)$ on $X$ with Cohen-Macaulay support $S$, we have a natural quasi-isomorphism
\[I\udot \cong \cone(Q[-2] \to I_{S/X}), \]
where the cokernel $Q$ is zero or pure 1-dimensional. 
\end{proposition}

\begin{proof}
The quasi-isomorphism follows from the canonical distinguished triangle
\[\xymatrix{ I_{S/X} \ar[r] & I^\mdot \ar[r] & Q[-1] \ar[r] & I_{S/X}[1]. }\]
Thus it suffices to assume $Q \neq 0$ and show it is pure 1-dimensional.
Applying $\Hom_X(T_0(Q),-)$ to the short exact sequence $0 \to \O_S \to F \to Q \to 0$, we obtain a long exact sequence
\[\xymatrix{ \cdots \ar[r] & \Hom_X(T_0(Q),F) \ar[r] & \Hom_X(T_0(Q),Q) \ar[r] & \Ext^1_X(T_0(Q), \O_S) \ar[r] & \cdots }.\]
By purity of $F$ and Proposition \ref{Lem.PT0=PT1}, we deduce $T_0(Q)=0$, which implies that $Q$ is pure 1-dimensional.
\end{proof}

For $\PT_1$ pairs with Gorenstein support, we also have the following characterization.

\begin{proposition} \label{lem:PT1PTcv}
For any $\PT_1$ pair $(F,s)$ on $X$ with Gorenstein support $S$ and non-zero cokernel $Q$, we have a natural quasi-isomorphism
\begin{equation}\label{Eq.PT1toPTc}
F\dual \otimes K_X[n-2] \cong [\O_S \to Q^D \otimes K_S\dual] \otimes K_S
\end{equation}
where $Q^D:=\ext^{n-1}_X(Q,K_X)$ is a pure 1-dimensional sheaf supported on $S$.
Moreover, this induces an equivalence between $\PT_1$ pairs supported on $S$ with non-zero cokernel and 1-dimensional $\PT$ pairs on $S$. 
\end{proposition}

\begin{proof}
Since $S$ is Gorenstein, we have $R\hom_X(\O_S,K_X) \cong K_S[-(n-2)]$. 
Moreover, $Q$ is pure 1-dimensional by Proposition \ref{lem:PT1CMsupp}.
Applying the functor $R\hom_X(-,K_X)$ to the short exact sequence $0 \to \O_S \to F \to Q \to 0$, we obtain a distinguished triangle
\[\xymatrix{ F\dual \otimes K_X[n-2] \ar[r] & K_S \ar[r] & Q^D \ar[r] & F\dual \otimes K_X[n-1], }\]
which implies the desired quasi-isomorphism \eqref{Eq.PT1toPTc}. 

Since $Q$ is pure 1-dimensional, we have $Q \cong Q^{DD}$ and $Q^D$ is pure 1-dimensional as well. Moreover $F$ is pure, hence the cokernel $\ext_X^{n-1}(F,K_X)$ of the map $K_S \to Q^D$ is 0-dimensional. It is easy to see that all these operations are invertible thereby establishing an equivalence.
\end{proof}

The equivalence in Proposition \ref{lem:PT1PTcv} implies that $\PT_1$ pairs on Gorenstein (and smooth) supports have simple descriptions. Fix a 2-dimensional pure subscheme $S$ of $X$. 

\begin{corollary}\label{cor:PT1Gorsupp}
Assume that $S$ is Gorenstein. A $\PT_1$ pair $(F,s)$ on $X$ with support $S$ and non-zero cokernel is equivalent to a pair $(C,t)$ where $C \subset S$ is a pure 1-dimensional subscheme and $t : \omega_C \twoheadrightarrow R$ is a 0-dimensional quotient.
\end{corollary}

\begin{proof}
By the previous lemma, we may view a $\PT_1$ pair on $X$ with support $S$ and non-zero cokernel as a 1-dimensional $\PT$ pair on $S$. The support of this stable pair is a 1-dimensional pure subscheme $C \subset S$ and the stable pair is determined by a short exact sequence
$$
0 \to \O_C \to G \to Q \to 0,
$$
where $G$ is pure 1-dimensional and $Q$ is 0-dimensional. Applying $R \hom_S(-, K_S)$ and using $\ext^2_S(G,K_S) = 0$ (purity of $G$), we obtain a quotient
$$
t : \omega_C = \ext_S^1(\O_C,K_S) \twoheadrightarrow \ext^2_S(Q,K_S).
$$
Note that applying $R \hom_S(-, K_S)$ again, and using the fact that $\O_C^{DD} \cong \O_C$ and $\O_C^D$ is pure 1-dimensional, we can recover the original 1-dimensional $\PT$ pair on $S$, and hence $\PT_1$ pair on $X$. 
\end{proof}

\begin{corollary}\label{cor:PT1smsupp}
Assume that $S$ is smooth. A $\PT_1$ pair $(F,s)$ on $X$ with support $S$ and non-zero cokernel is equivalent to a pair $(C,Z)$ of an (non-zero) effective divisor $C \subseteq S$ and a 0-dimensional subscheme $Z \subseteq C$.
\end{corollary}

\begin{proof}
Since any pure 1-dimensional subscheme $C$ of a smooth surface $S$ is a local complete intersection, $C$ is Gorenstein and $K_C$ is invertible. Hence the proof follows directly from Corollary \ref{cor:PT1Gorsupp}.
\end{proof}

\subsection{Numerical bounds}

Let $X$ be a smooth projective variety of dimension $n$ with fixed very ample line bundle $\O_X(1)$. For any coherent sheaf $F$ on $X$, we denote its Hilbert polynomial by
\[P_F(t) = \sum_{i\geq0} \alpha_i(F) \cdot \frac{t^i}{i!},\]
where $\alpha_0(F)=\chi(F)$ is the Euler characteristic. In this subsection, we derive boundedness results on $\alpha_i(F)$ for $\PT_q$ pairs $(F,s)$ which are crucial in the GIT construction in subsection \ref{ss:GIT}.


\begin{proposition}\label{prop:boundofchiforPT}
Fix rational numbers $\alpha_2>0$ and $\alpha_1$.
\begin{enumerate}
\item[$(1)$] For each $q\in \{-1,0,1\}$, the set
\[\{\alpha_1(F) : \text{$\PT_q$ pairs $(F,s)$ such that $\alpha_2(F)=\alpha_2$}\}\]
is bounded below.
\item[$(2)$] The two sets
\[\{\chi(F) : \text{$\DT$ pairs $(F,s)$ such that $\alpha_2(F)=\alpha_2$, $\alpha_1(F)=\alpha_1$}\}\]
\[\{\chi(F) : \text{$\PT_0$ pairs $(F,s)$ such that $\alpha_2(F)=\alpha_2$, $\alpha_1(F)=\alpha_1$}\}\]
are bounded below, and the set
\[\{\chi(F) : \text{$\PT_1$ pairs $(F,s)$ such that $\alpha_2(F)=\alpha_2$, $\alpha_1(F)=\alpha_1$}\}\]
is bounded above.
\end{enumerate}
\end{proposition}

\begin{proof}
(1) The $\DT$ case is shown in \cite[Cor.~2.13]{Pot1}. The $\PT_0$ case follows directly from the $\DT$ case by considering the short exact sequence $0 \to \O_Z \to F \to Q \to 0$ since $\alpha_2(F)=\alpha_2(\O_Z)$ and $\alpha_1(F)=\alpha_1(\O_Z)$. Finally, the $\PT_1$ case also follows directly from the $\DT$ case by considering the short exact sequence $0 \to \O_S \to F \to Q \to 0$ since $\alpha_2(F)=\alpha_2(\O_S)$ and $\alpha_1(F)\geq \alpha_1(\O_S)$. \\

\noindent (2) ($\DT$) Choose a $\DT$ pair $(\O_Z,1:\O_X \twoheadrightarrow \O_Z)$ such that $\alpha_2(\O_Z) = \alpha_2$, $\alpha_1(\O_Z) = \alpha_1$, and $\chi(\O_Z)<0$. Consider the map
\begin{equation}\label{2}
(X\setminus Z)^{[-\chi(\O_Z)]} \to \Hilb(X,\tfrac12\alpha_2 t^2+ \alpha_1 t) : W \mapsto Z \sqcup W    
\end{equation}
from the Hilbert scheme of points to the Hilbert scheme of closed subschemes with Hilbert polynomial $v=\tfrac12\alpha_2 t^2+ \alpha_1 t$. Since the map \eqref{2} is injective and the Hilbert scheme $\Hilb(X,v)$ is of finite type, the Euler characteristic $\chi(\O_Z)$ has a lower bound which only depends on $X$, $\alpha_2$, and $\alpha_1$. \\

\noindent ($\PT_0$) Given a $\PT_0$ pair $(F,s)$, we can form a short exact sequence
\[\xymatrix{0 \ar[r] & \O_Z \ar[r] & F \ar[r] & Q \ar[r] & 0,}\]
where $Z$ is the support of $F$ and $Q$ is the cokernel of $s$. Since
\[\alpha_2(F)=\alpha_2(\O_Z), \quad \alpha_1(F)=\alpha_1(\O_Z), \and \chi(F) =\chi(\O_Z)+\chi(Q),\]
a lower bound in the case ($\DT$) also gives us a lower bound for the case ($\PT_0$). \\

\noindent ($\PT_1$). Choose a $\PT_1$ pair $(F,s)$ and form a short exact sequence
\[\xymatrix{0 \ar[r] & \O_S \ar[r] & F \ar[r] & Q \ar[r] & 0,}\]
where $S$ is the support of $F$ and $Q$ is the cokernel of $s$. By Lemma \ref{lem:boundofchi}(2) below, it suffices to show that $\mu_{\min}(F)$ is bounded below. Note that there exists a surjective map $F \twoheadrightarrow F_{\min}$ to a 2-dimensional $\mu$-semistable sheaf $F_{\min}$ such that $\mu_{\min}(F)=\mu(F_{\min})$. We can assume that $F_{\min}$ is pure. We can form an induced short exact sequence
\[\xymatrix{0 \ar[r] & \O_Z \ar[r] & F_{\min} \ar[r] & R \ar[r] & 0}\]
for a subscheme $Z \subseteq S$ and a quotient sheaf $Q \twoheadrightarrow R$. Note that we have
\[\mu_{\min}(\O_S) \leq \mu_{\min}(\O_Z) \leq \mu(F_{\min}) = \mu_{\min}(F)\]
(cf. \cite[Lem.~1.3.3]{HL}). Since $S$ is pure, the minimal slope $\mu_{\min}(\O_S)$ is bounded below by \cite[Cor.~2.13]{Pot1}. Therefore $\mu_{\min}(F)$ is also bounded below.
\end{proof}

We need the following lemma to complete the proof of Proposition \ref{prop:boundofchiforPT}(2) for $\PT_1$ pairs.

\begin{lemma}
\label{lem:boundofchi}
Let $X$ be a smooth projective variety with a fixed very ample line bundle $\O_X(1)$ and let $\alpha_2$, $\alpha_1$ be rational numbers such that $\alpha_2>0$.
\begin{enumerate}
\item[$(1)$] The set of integers
\[\left\{
\alpha_0 :
\begin{matrix}
\text{$\mu$-stable reflexive sheaves $F$ on $X$ such that} \\ 
\text{$P_F(t) = \tfrac{1}{2} \alpha_2 t^2 +\alpha_1 t +\alpha_0$}
\end{matrix}\right\}\]
is bounded above.
\item[$(2)$] For a fixed rational number $C$, the set of integers
\[\left\{
\alpha_0 :
\begin{matrix}
\text{pure sheaves $F$ on $X$ such that} \\ 
\text{$P_F(t) = \tfrac{1}{2} \alpha_2 t^2 +\alpha_1 t +\alpha_0$ and $\mu_{\min}(F)\geq C $}
\end{matrix}\right\}\]
is bounded above.
\end{enumerate}
\end{lemma}

\begin{proof} 
(1) The proof is similar to \cite[Prop.~3.6(2)]{GK}. Choose a $\mu$-stable reflexive sheaf $F$ with $P_F(t) = \tfrac{1}{2} \alpha_2 t^2 +\alpha_1 t +\alpha_0$. Assume that $\alpha_0>0$. Let $Z$ be the support of $F$. Let $U \subseteq Z^{\mathrm{red}}$ be a non-empty smooth connected open subscheme of the reduced subscheme $Z^{\mathrm{red}}$ of $Z$ . After shrinking $U$ to a possibly smaller non-empty open subscheme, we may assume that $F|_U \cong \O_U ^{\oplus r}$ is a trivial vector bundle on $U$.

We first claim that there is a map
\[\Phi : U^{[\alpha_0]} \to M^{\mu}(X,\tfrac{1}{2} \alpha_2 t^2 +\alpha_1 t)\]
from the Hilbert scheme of points on $U$ to the moduli space of $\mu$-stable pure sheaves on $X$ such that there is a short exact sequence 
\begin{equation}\label{eq:2}
\xymatrix{0 \ar[r] & \Phi(W) \ar[r] & F \ar[r] & \O_W \ar[r] & 0,}
\end{equation}
for all $W \in U^{[\alpha_0]}$. Indeed, let $\cW \subseteq U^{[\alpha_0]}\times U \subseteq U^{[\alpha_0]}\times X$ be the universal family. Then we have a surjective map $\O_{U^{[\alpha_0]}} \boxtimes F|_U \twoheadrightarrow \O_{U^{[\alpha_0]} \times U} \twoheadrightarrow \O_{\cW}$ induced by some projection map $\O_U^{\oplus r} \to \O_U$. Hence we can form a short exact sequence
\begin{equation}\label{eq:3}
\xymatrix{0 \ar[r] & \F' \ar[r] & \O_{U^{[\alpha_0]}} \boxtimes F \ar[r] & \O_{\cW} \ar[r] & 0}
\end{equation}
of sheaves on $U^{[\alpha_0]}\times X$, which are flat over $U^{[\alpha_0]}$. Hence $\F'$ defines the map $\Phi$ and by taking the fibres of \eqref{eq:3} over $W \in U^{[\alpha_0]}$, we obtain \eqref{eq:2}. 
From \eqref{eq:2} and the fact that $F$ is reflexive, it follows that $\Phi(W)$ is pure 2-dimensional and $\Phi(W)^D \cong F^D$.
Moreover, $\Phi(W) \cong F \in \Coh_{2,1}(X):=\Coh_{\leq2}(X)/\Coh_{\leq0}(X)$ so that $\Phi(W)$ is $\mu$-stable.

We then claim that the map $\Phi$ is injective. Choose two 0-dimensional closed subschemes $W_1, W_2 \subseteq U$ and let $F_i=\Phi(W_i)$ for both $i=1,2$. Assume that there is an isomorphism $\phi : F_1 \xrightarrow{\cong} F_2$ of sheaves. Then we can form a commutative square
\[\xymatrix{ F_1 \ar[d]^{\phi} \ar@{^{(}->}[r] & F_1^{DD} \ar[d]^{\phi^{DD}} & F \ar[l]_-{\cong} \ar@{.>}[d]^{\psi}\\ F_2 \ar@{^{(}->}[r] & F_2^{DD} & F \ar[l]_-{\cong} }\]
since $F_1^{DD} \cong  F^{DD} \cong F_2^{DD}$ and $F$ is reflexive. Since $F$ is $\mu$-stable, $F$ is stable, and thus simple, i.e.~$\Hom_X(F,F)=\C$. Therefore by replacing $\phi$ by $c\cdot\phi$ for some constant $c \in \C^*$, we may assume that the dotted arrow $\psi$ is the identity map $\id_F$.
Hence $F_1 \cong F_2$ as subsheaves of $F$, and thus the cokernels are isomorphic, $\O_{W_1} \cong \O_{W_2}$. Therefore $W_1=W_2$ and $\Phi$ is injective.

Since $\Phi$ is injective and $\dim(U^{[\alpha_0]}) \geq \dim(U) \cdot\alpha_0$, the Euler characteristic $\alpha_0$ is bounded above by the dimension of the finite type scheme $M^{\mu}(X,\tfrac12 \alpha_2 t^2 + \alpha_1 t)$. This proves (1). \\

\noindent (2) {\em Case 1.} We first show that 
\[
\{\alpha_0: \text{ $\mu$-stable pure sheaves $F$ on $X$ such that $P_F(t) = \tfrac{1}{2} \alpha_2 t^2 +\alpha_1 t +\alpha_0$}\}
\] 
is bounded above. Let $F$ be a $\mu$-stable pure sheaf with $P_F(t) = \tfrac{1}{2} \alpha_2 t^2 +\alpha_1 t +\alpha_0$. Since $F$ is pure, we can form a short exact sequence 
\[\xymatrix{0 \ar[r] & F \ar[r] & F^{DD} \ar[r] & R \ar[r] & 0 }\]
for some 0-dimensional sheaf $R$ by \cite[Prop. 1.1.10]{HL}. Then $F^{DD}$ is a reflexive $\mu$-stable sheaf with $P_{F^{DD}}(t) = \tfrac{1}{2} \alpha_2 t^2 +\alpha_1 t + (\alpha_0+\chi(R))$. Since $\alpha_0(F^{DD})= \alpha_0 + \chi(R)$ is bounded above by Lemma \ref{lem:boundofchi}(1), $\alpha_0(F)$ is also bounded above. \\

\noindent {\em Case 2.} We then show that 
\[
\{\alpha_0 : \text{ $\mu$-semistable pure sheaves $F$ on $X$ such that $P_F(t) = \tfrac{1}{2} \alpha_2 t^2 +\alpha_1 t +\alpha_0$}\}
\] 
is bounded above. Choose a $\mu$-semistable pure sheaf $F$ with $P_F(t) = \tfrac{1}{2} \alpha_2 t^2 +\alpha_1 t +\alpha_0$. Consider the Jordan-H{\"o}lder filtration
\[0=F_0 \subseteq F_1 \subseteq \cdots \subseteq F_l =F\]
of $F$ with respect to $\mu$-stability \cite[Thm.~1.6.7(ii)]{HL}. The length $l$ of the filtration is bounded from above because the leading term $\alpha_2$ of the Hilbert polynomial is fixed.  Since $F$ is pure, we can assume that all $F_i/F_{i-1}$ are also pure. Since each $F_i/F_{i-1}$ is a $\mu$-stable pure sheaf with the slope $\mu(F_i/F_{i-1})=\mu(F)$, Case 2 follows from Case 1. \\

\noindent {\em Case 3.} Finally, we consider the general case. Let $F$ be a pure sheaf with $P_F(t)= \tfrac12 \alpha_2 t^2 + \alpha_1 t + \alpha_0$. Choose a Harder-Narasimhan filtration
\[0=F_0 \subseteq F_1 \subseteq \cdots \subseteq F_l =F\]
of $F$ with respect to the $\mu$-stability \cite[Thm.~1.6.7(i)]{HL}. Since $F$ is pure, we may assume that $F_i/F_{i-1}$ are $\mu$-semistable pure sheaves. If $\alpha_2(F)$, $\alpha_1(F)$ are fixed and $\mu_{\min}(F)$ is bounded below, then there are only finitely many possible $\alpha_1(F_i/F_{i-1})$ (and $\alpha_2(F_i/F_{i-1})$). Therefore
\[\alpha_0(F) = \sum_i \alpha_0(F_i/F_{i-1})\]
is also bounded above by the result in Case 2. 
\end{proof}

\section{Moduli spaces} \label{sec:moduli}

In this section, we construct moduli spaces of $\PT_q$ pairs via geometric invariant theory and prove that they are open subschemes of the moduli space of perfect complexes parametrizing polynomially Bridgeland stable objects.

\subsection{Moduli stack of pairs}\label{ss:modulistackofpairs}

In this subsection, we introduce the moduli stack of all pairs and prove that the $\PT_q$-stability conditions are bounded and open conditions.

\begin{definition}
Let $X$ be a smooth projective variety over $\C$.
\begin{enumerate}
\item We define the {\em moduli stack of pairs on $X$} as a 2-functor
\[\tPair(X) : \mathrm{Sch}^{\mathrm{op}}_{\C} \to \mathrm{Groupoid}\]
sending a scheme $T$ to the groupoid of (not necessarily 2-dimensional) pairs $(F,s)$ of $T$-flat coherent sheaves $F$ on $X \times T$ and sections $s \in \Gamma(X\times T, F)$. The morphisms $(F_1,s_1) \to (F_2,s_2)$ in the groupoid are isomorphisms $\phi :F_1 \to F_2$ of sheaves such that $\phi(s_1)=s_2$.
\item We define the {\em moduli stack of non-zero pairs on $X$} as the substack
\[\cPair(X) \subseteq \tPair(X)\]
consisting of the pairs $(F,s)$ such that $s_t \neq 0 \in \Gamma(X,F_t)$ for all $t \in T$.
\end{enumerate}
\end{definition}



The moduli stacks $\tPair(X)$ and $\cPair(X)$ are algebraic stacks. Moreover, they have explicit descriptions as unions of open substacks which are global quotient stacks. We will now review the proofs of these facts.

Let us first fix some notation. Let $X$ be a smooth projective variety of dimension $n$ with a fixed very ample line bundle $\O_X(1)$. Fix $v \in H^*(X,\Q)$ and let $P_v(t) \in \Q[t]$ be the associated Hilbert polynomial. Choose an integer $m$. Fix a vector space $V$ of dimension $P_v(m)$. Let 
\begin{equation}\label{eq:def.quot}
\cQ :=\cQuot(X,V \otimes \O_X(-m),v)    
\end{equation}
be the Quot scheme of quotients $\rho : V\otimes \O_X(-m) \twoheadrightarrow F$ with $\ch(F)=v$. 
Let 
\begin{equation}\label{eq:quot.universal}
V \otimes \O_{\cQ} \otimes \O_X(-m) \twoheadrightarrow \FF    
\end{equation}
be the universal quotient on $\cQ\times X$ and let $\pi : \cQ \times X \to \cQ$ be the projection map. Consider the open subscheme
\begin{equation}\label{eq:def.quot0}
\cQ^{\circ} := \Big\{(F,\rho) \in \cQ : V \stackrel{\cong}{\to} H^0(X,F(m)) \text{ and } H^{i}(X,F(m))=0 \, \forall i>0\Big\},\    
\end{equation}
where $V \to H^0(X,F(m))$ is the map induced by $\rho$. Then the canonical map 
\begin{equation}\label{eq:S2.4}
V \otimes \O_{\cQ} \to R\pi_*(\FF(m))
\end{equation}
induced by \eqref{eq:quot.universal} is an isomorphism over $\cQ^{\circ}$.

\begin{proposition}\label{prop:modulistackofpairs}
The moduli stack $\cPair(X)$ has the following properties:
\begin{enumerate}
\item[$(1)$] We have a decomposition 
\[\cPair(X) = \bigsqcup_{v \in H^*(X,\Q)} \cPair(X,v)\]
by disjoint open substacks $\cPair(X,v)$ consisting of the pairs $(F,s)$ on $X$ with $\ch(F)=v \in H^*(X,\Q)$. 
\item[$(2)$] We have
\[\cPair(X,v) = \bigcup_{m \geq 0}\cPair(X,v)_m\]
for open substacks $\cPair(X,v)_m$ consisting of the pairs $(F,s)$ on $X$ such that $F(m)$ is globally generated and $H^{>0}(X,F(m))=0$.
\item[$(3)$] We have a global quotient stack presentation
\[\cPair(X,v)_m = \left[\PP_{\cQ^{\circ}}(R \pi_* \FF)/\PGL(V)\right],\]
where $\PP_{\cQ^{\circ}}(R\pi_*\FF):=\Proj \Sym h^0((R \pi_* \FF)\dual)$ denotes the projective cone. 
\end{enumerate}
Consequently, the moduli stack $\cPair(X)$ is an algebraic stack (locally of finite type).
\end{proposition}


\begin{proof}[Proof of Proposition \ref{prop:modulistackofpairs}]
Let $\cCoh(X)$ denote the moduli stack of coherent sheaves on $X$. Then 
\[\cCoh(X) = \bigsqcup_{v \in H^*(X,\Q)} \cCoh(X,v) \and \cCoh(X,v) = \bigcup_{m \geq 0} \cCoh(X,v)_m\]
for open substacks $\cCoh(X,v)_m \subseteq \cCoh(X,v) \subseteq \cCoh(X)$
such that
\begin{align*}
\cCoh(X,v)(\C) &= \{F \in \cCoh(X)(\C) : \ch(F)=v \in H^*(X,\Q) \},\\
\cCoh(X,v)_m(\C) &= \left\{F \in \cCoh(X,v)(\C) : \begin{matrix} F(m) \text{ globally generated and } \\  H^i(X,F(m))=0 \text{ for all } i>0 \end{matrix} \right\}.
\end{align*}
Let $\FF^C$ be the universal family on $\cCoh(X) \times X$ and let $\pi^C : \cCoh(X) \times X \to \cCoh(X)$ be the projection map. Then $R\pi^C_*(\FF^C(m))$ is a vector bundle over $\cCoh(X,v)_m$ and the associated $\GL_{P_v(m)}$ bundle is exactly $\cQ^{\circ}$. Hence we have a global quotient stack presentation (see also \cite[Thm.~4.6.2.1]{LMB})
\[\cCoh(X,v)_m = \left[\cQ^{\circ} / \GL(V) \right].\]

By definition, the moduli stack of pairs is the abelian cone
\[\tPair(X) = C(\ext^n_{\pi^C}(\FF^C,K_X)) = \Spec \Sym  \ext^n_{\pi^C}(\FF^C,K_X) \]
associated to the coherent sheaf $\ext^n_{\pi^C}(\FF^C,K_X)=h^0((R\pi^C_*\FF^C)\dual)$ on $\cCoh(X)$ and the moduli stack of non-zero pairs is the complement
\[\cPair(X) = \tPair(X) \setminus 0_{\cCoh(X)} \subseteq \tPair(X)\]
of the zero section $0_{\cCoh(X)}: \cCoh(X) \hookrightarrow \tPair(X)$. Hence, we have  (1) and (2) for the open substacks $\cPair(X,v):=\cPair(X)\times_{\cCoh(X)}\cCoh(X,v)$ and $\cPair(X,v)_m:=\cPair(X)\times_{\cCoh(X)}\cCoh(X,v)_m$. Moreover, we have a global quotient stack presentation
\[\cPair(X,v)_m = \left[(C_{\cQ^{\circ}}(\ext^n_\pi(\FF,K_X))\setminus 0_{\cQ^{\circ}})/\GL(V) \right] = \left[\PP_{\cQ^{\circ}}(R\pi_*\FF)/\PGL(V)\right]\]
since $\PGL(V) = \GL(V)/\GG_m$. 
\end{proof}





\begin{proposition}\label{prop:repPTq}
There exists an open substack $\cPair^{(q)}(X,v)$ of $\cPair(X,v)$, for any $q \in \{-1,0,1\}$, such that the $\C$-points are given by
\[\cPair^{(q)}(X,v)(\C) = \{(F,s)\in \cPair(X,v) : \text{$(F,s)$ is $\PT_q$-stable}\}.\]
Moreover, the moduli stack $\cPair^{(q)}(X,v)$ is an algebraic space of finite type.
\end{proposition}


We will prove Proposition \ref{prop:repPTq} through  three steps. We first prove that $\PT_q$ pairs are {\em simple} pairs.

\begin{lemma}\label{lem:splpair}
For any $\PT_q$ pair $(F,s)$ on $X$, we have $$\Aut_{\cPair(X)}(F,s)=\{1\}.$$
\end{lemma}

\begin{proof}
Consider the short exact sequence of sheaves
\[\xymatrix{0 \ar[r] & \O_Z \ar[r] & F \ar[r] & Q \ar[r] & 0}\]
where $Z$ is the support of $F$ and $Q$ is the cokernel of $s$. By the $\PT_q$-stability condition, we have $\Hom_X(Q,F) =0.$ Hence we have an injective map
\begin{equation}\label{eq:S2.15}
\Hom_X(F,F) \longrightarrow \Hom_X(\O_Z,F) = \Gamma(X,F) : \phi \mapsto \phi(s).
\end{equation}
Since $\Aut_{\cPair(X)}(F,s)$ is the fibre of the above map \eqref{eq:S2.15} over the point $s \in \Gamma(X,F)$, the automorphism group $\Aut_{\cPair(X)}(F,s)$ is trivial.
\end{proof}

Secondly, we prove that the collection of $\PT_q$ pairs is {\em bounded}.

\begin{lemma}\label{lem:boundedness}
The collection of sheaves
\[\{F \in \cCoh(X,v) (\C): \text{$\exists s\in \Gamma(X,F)$ such that $(F,s)$ is $\PT_q$-stable}\}\]
for fixed $X$ and $v$ is bounded in the sense of \cite[Def.~1.1]{Gro}.
\end{lemma}

\begin{proof}
The boundedness of the $\DT=\PT_{-1 }$ pairs follows from \cite[Thm.~2.1]{Gro} and the boundedness of the $\PT=\PT_{1}$ pairs follows from \cite[Thm.~4.11]{Pot1}. Thus it suffices to consider the $\PT_0$ pairs.  

Given a $\PT_0$ pair $(F,s)$ with $\ch(F) = v$, let $Z$ be the support of $F$ and $Q$ the cokernel of $s$. By Proposition \ref{prop:boundofchiforPT}(2) for $\DT$-stability, $\chi(\O_Z)$ is bounded below. On the other hand $\chi(\O_Z)$ is bounded above since $\chi(Q)\geq 0$. Moreover, $\ch_i(\O_Z)=\ch_i(F)$ for $i<n$. Hence there are only finitely many possible Chern characters $\ch(\O_Z)$, and the collection of all such $\O_Z$ is bounded by \cite[Thm.~2.1]{Gro}. 
Therefore, there exists an integer $m$ such that for all $\O_Z$ as above, we have $H^i(X,\O_Z(m-i)) = 0$ for all $i>0$. For any $\PT_0$ pair $(F,s)$ with $\ch(F) = v$, we then also have $H^i(X,F(m-i)) = 0$ for all $i>0$ ($H^{i}(X,Q(m-i))$ is automatically zero since $Q$ is 0-dimensional). By \cite[Lem.~1.7.6]{HL}, we deduce the required boundedness for $\PT_0$ pairs. 
\end{proof}


The proof of Lemma \ref{lem:boundedness} also proves the following lemma, which will be used in the proof of Proposition \ref{prop:GIT/PTq} in the next subsection.

\begin{lemma}\label{lem2}
The collection of sheaves 
\[\{F \in \cCoh(X,v)(\C) : \text{$\exists s\in \Gamma(X,F)$ such that $\coker(\O_X \xrightarrow{s} F)\in \Coh_{\leq0}(X)$}\}\]
for fixed $X$ and $v$ is bounded.
\end{lemma}

Lastly, we show that the $\PT_q$-stability conditions are {\em open}.

\begin{lemma}\label{lem:openness}
Let $(F,s)$ be a $T$-flat family of pairs on $X$ parametrized by a scheme $T$ of finite type over $\C$. Then 
\[\{t \in T(\C) : \text{the fibre $(F_t,s_t)$ of $(F,s)$ is $\PT_q$-stable}\}\]
is a Zariski open subset of $T(\C)$.
\end{lemma}

\begin{proof}
We first show that the condition 
\[Q_t:=\coker(\O_X \xrightarrow{s_t} F_t) \in \Coh_{\leq q}(X)\]
 is an open condition.
By \cite[Prop.~1.2(i)]{Gro}, the set $\{P_{Q_t} : t \in T(\C)\}$ of Hilbert polynomials is finite. Hence the result follows from upper semi-continuity of Hilbert polynomials of non-flat families of coherent sheaves.

Now we consider the second condition
\[F_t \in \Coh_{\geq q+1}(X).\]
Firstly, if $q=-1$, then this condition is always satisfied. Secondly, if $q=1$, this condition is open by the openness of purity \cite[Prop.~2.3.1]{HL}. Finally, assume that $q=0$. We show that the set of the Hilbert polynomials
\begin{equation} \label{eqn:PF/T0F}
\{P_{{F_t}/{T_0(F_t)}} : t \in T(\C)\}
\end{equation}
is finite. It suffices to show that $\chi(T_0(F_t))$ is bounded above. Note that the short exact sequence $0 \to T_0 (F_t)\to F_t \to F_t/T_0(F_t) \to 0$
induces a short exact sequence
\[\xymatrix{0 \ar[r] & \ext^n_X(F_t/T_0(F_t), K_X) \ar[r] & \ext^n_X(F_t,K_X) \ar[r] & \ext^n_X(T_0(F_t),K_X) \ar[r] & 0}\]
of 0-dimensional sheaves by \cite[Prop.~1.1.6]{HL}. Then we have
\[\chi(T_0(F_t)) = \chi(\ext^n_X(T_0(F_t),K_X)) \leq \chi(\ext^n_X(F_t,K_X)) \leq \dim_\C \Ext^n_X(F_t,K_X)\]
where $\pi^T : X \times T \to T$ is the projection map, $i_t :\{t\} \hookrightarrow T$ is the inclusion map, and the last inequality follows from the local-to-global spectral sequence. Since $\Ext^n_X(F_t,K_X)$ are the fibres of the coherent sheaf  $\ext^n_{\pi^T} (F,K_X)$ on a scheme $T$ of finite type, their dimensions are bounded above.

Since, for any $t \in T(\C)$, the relative Quot schemes
\[\cQuot(X\times T/T , F, P_{F_t/T_0(F_t)})  \to T\]
are proper over $T$, their images are closed. The claim follows from finiteness of \eqref{eqn:PF/T0F}.
\end{proof}

Proposition \ref{prop:repPTq} follows directly from the above results.

\begin{proof}[Proof of Proposition \ref{prop:repPTq}]
By Lemma \ref{lem:openness}, we can form $\cPair^{(q)}(X,v)$ as an open substack of $\cPair(X,v)$. Then Lemma \ref{lem:splpair} proves that $\cPair^{(q)}(X,v)$ is an algebraic space and Lemma \ref{lem:boundedness} proves that $\cPair^{(q)}(X,v)$  is of finite type.
\end{proof}


\begin{remark}\label{rmk:repforDTPT1}
By \cite{Gro,Pot1}, the moduli spaces $\cPair^{(q)}(X,v)$ of $\PT_q$ pairs are projective schemes for $q=-1$ and $q=1$. In the next subsection, we will show that the space $\cPair^{(0)}(X,v)$ is also a projective scheme.
\end{remark}

We end this subsection with a remark on the natural derived enhancement of the moduli stack of pairs. This derived structure will not be used in the rest of the paper.

\begin{remark}\label{rem:derivedinterpretation}
There is a derived stack  $R\tPair(X)$ whose classical truncation is $\tPair(X)$ and the tangent complex at a point $(F,s)$ is 
\[\TT_{R\tPair(X)}|_{(F,s)} = R\Hom_X(I^\mdot,F)\]
where $I^\mdot=[\O_X \xrightarrow{s} F]$. We sketch the construction here. Recall that there exists a derived moduli stack $R\cCoh(X)$ of coherent sheaves as an open substack of the derived moduli stack $R\cPerf(X)$ of perfect complexes. Let $\FF$ denote the universal family on $R\cCoh(X)$ and let $\pi:R\cCoh(X)\times X\to R\cCoh(X)$ denote the projection map. Define the derived moduli stack of pairs as the total space
\[R\tPair(X):=\mathrm{Tot}_{R\cCoh(X)}(R\pi_* \FF)\]
of the perfect complex $R\pi_*\FF$ on $R\cCoh(X)$.
Moreover, the derived loop stack
\[\cL_{R\tPair(X),(F,s)}:=\Spec(\C)\times^h_{(F,s),R\tPair(X),(F,s)}\Spec(\C)\]
is an open substack of the fibre of the map between derived linear stacks
\[\mathrm{Tot}(R\Hom_X(F,F)) \xrightarrow{s} \mathrm{Tot}(R\Hom_X(\O_X,F))\]
at the point $s \in \Hom_X(\O_X,F)$. From the distinguished triangle
\[\xymatrix{R\Hom_X(I\udot,F)[-1] \ar[r] &R\Hom_X(F,F) \ar[r]^{s} & R\Hom_X(\O_X,F)}\]
the cotangent complex can be computed as
\[\TT_{R\tPair(X)}|_{(F,s)} \cong \TT_{\cL_{R\tPair(X),(F,s)}}|_{0}[1] \cong R\Hom_X(I\udot,F). \]
\end{remark}

\subsection{Geometric invariant theory}\label{ss:GIT}

In this subsection, we construct moduli spaces of $\PT_q$ pairs as projective schemes via geometric invariant theory. Although we largely follow \cite{Pot1}, our arguments differ since the underlying sheaves $F$ of $\PT_q$-stable pairs $(F,s)$ are not {\em pure} for $q=-1,0$. Our approach is inspired by \cite{ST}.

\begin{theorem}\label{Thm:GIT} 
Let $X$ be a smooth projective variety and $v \in H^*(X,\Q)$. 
Then there exists a projective scheme $\cM$, an action of $\PGL_N$ on $\cM$ for some $N$, and an $\SL_N$-equivariant ample line bundle $\cL^{(q)} \in \Pic^{\SL_N}(\cM)_{\Q}$ for each $q \in \{-1,0,1\}$ such that \[\cPair^{(q)}(X,v) = [\cM^{st}(\cL^{(q)})/\PGL_N] = [\cM^{ss}(\cL^{(q)})/\PGL_N].\] 
In particular, the moduli spaces $\cPair^{(q)}(X,v)$ are the GIT quotients $ \cM/\!/_{\cL^{(q)}}\SL_N$ and hence projective schemes.\end{theorem}


The rest of this subsection is devoted to the proof of Theorem \ref{Thm:GIT}. By the boundedness result in Lemma \ref{lem:boundedness}, we have $\cPair^{(q)}(X,v) \subseteq \cPair(X,v)_m$ for sufficiently large $m$. Denoting the Hilbert polynomial determined by $v$ by $P_v(t)$, we then define
$$
V:=\C^{P_v(m)}, \quad R_m := \Gamma(X,\O_X(m)).
$$
We use the notation from subsection \ref{ss:modulistackofpairs} (e.g., the Quot scheme $\cQ$ in \eqref{eq:def.quot}, the universal family \eqref{eq:quot.universal}, and the open subscheme $\cQ^{\circ}$ in \eqref{eq:def.quot0}).  Moreover, by Proposition \ref{prop:modulistackofpairs}, we have a presentation of the moduli stack as a global quotient stack, 
\[\cPair(X,v)_m=[\PP_{\cQ^{\circ}}(R\pi_*\FF)/\PGL(V)].\] 
We would like to use geometric invariant theory on $\PP_{\cQ^{\circ}}(R\pi_*\FF)$. Thus we need to consider a $\PGL(V)$-equivariant compactification of  $\PP_{\cQ^{\circ}}(R\pi_*\FF)$ and linearizations on it.


We define the {\em parameter space} $\cM$ as follows:

\begin{definition}
Define a projective scheme
\[\cM:=\overline{\PP_{\cQ^{\circ}}(R\pi_*\FF)} \subseteq \cQ \times \PP(V \otimes R_m^*)\]
as the closure of $\PP_{\cQ^{\circ}}(R\pi_*\FF)$ in $\cQ \times \PP(V \otimes R_m^*)$ with respect to the embedding
\begin{equation}\label{3}
\xymatrix{
\PP_{\cQ^{\circ}}(R\pi_*\FF) \ar@{^{(}->}[r] & \PP_{\cQ^{\circ}}(R\pi_*(\FF(m)) \otimes R_m^*)  & \\
& \PP_{\cQ^{\circ}}(\O_{\cQ^{\circ}}\otimes V\otimes R_m^*) \ar@{^{(}->}[u]_{\cong} \ar@{^{(}->}[r] & \cQ \times \PP(V\otimes R_m^*).
}    
\end{equation}
Here the first map in \eqref{3} is the closed embedding 
induced by the canonical map
\[\Gamma(X,\O_X(m)) \otimes R\pi_*\FF \to R\pi_*(\FF(m)),\]
the second map in \eqref{3} is an isomorphism induced by the canonical map \eqref{eq:S2.4}, and the last map in \eqref{3} is the open embedding induced by $\cQ^{\circ} \hookrightarrow \cQ$.
\end{definition}

\begin{remark}\label{rmk:parameterspace} The parameter space $\cM$ in this paper is larger than the parameter space of Le Potier in \cite{Pot1,Pot2}. Le Potier considered $\overline{\cQ}_{\mathrm{pure}} \times \PP(V\otimes R_m ^*)$ where $\overline{\cQ}_{\mathrm{pure}} \subseteq \cQ$ is the closure of the open subscheme $\cQ_{\mathrm{pure}}\subseteq \cQ$. The difference arises from the fact that we want to realize all three moduli spaces $\cPair^{(q)}(X,v)$ for $q=-1,0,1$, as GIT quotients of \emph{the same} scheme $\mathcal{M}$ by the same group $\PGL(V)$ (with respect to different linearizations).
\end{remark}

We consider the {\em linearizations} $\cL_{g(t),l}$ on $\cM$ defined as follows:

\begin{definition}
For a sufficiently large $l$, consider the standard Grothendieck embedding into the Grassmannian
\[\cQ \hookrightarrow \cG:={\curly Grass}(V\otimes R_{l-m}, P_v(l)) : (F,\rho) \mapsto (V\otimes R_{l-m} \twoheadrightarrow \Gamma(X,F(l))).\]
Fix a polynomial $g(t) \in \Q[t]$ with positive leading coefficient. Define the $\SL(V)$-equivariant very ample line bundle
\[\cL_{g(t),l} := \O_{\cG}(1)|_{\cQ}\boxtimes \O_{\PP(V\otimes R_m^*)}(g(l)) \in \Pic^{\SL(V)}(\cM)_\Q,\]
where $\O_{\cG}(1)$ is the polarization on $\cG$ induced by the standard Pl\"ucker embedding of $\cG$.
\end{definition}

We describe the $\cL_{g(t),l}$-stable points via the Hilbert-Mumford criterion.

\begin{lemma}\label{Lem:HilbMum} 
A point $(F,\rho,s)$
in $\cQ \times \PP(V\otimes R_m^*)$ is $\cL_{g(t),l}$-semistable if and only if the following two conditions are satisfied:
\begin{enumerate}
\item[$(1)$] for all non-zero proper subspaces $V' \subseteq V$, we have
\begin{equation}\label{eq:ineq1}
\frac{P_{F'}(l)}{\dim(V')} \geq \frac{P_F(l) -g(l)}{\dim(V)}
\end{equation}
where $F':=\rho(V'\otimes \O_X(-m)) \subseteq F$, and
\item[$(2)$] for all non-zero proper subspaces $V' \subseteq V$ with $s \in V' \otimes R_m^*$, we have
\begin{equation}\label{eq:ineq2}
\frac{P_{F'}(l) - g(l)}{\dim(V')} \geq \frac{P_F(l) -g(l)}{\dim(V)},
\end{equation}
where again $F':=\rho(V'\otimes \O_X(-m)) \subseteq F$.
\end{enumerate}
Moreover, $(F,\rho,s)$ is $\cL_{g(t),l}$-stable if and only if the above two inequalities hold strictly.
\end{lemma}

\begin{proof}
This is standard by \cite[Lem.~4.4.3, Lem.~4.4.4]{HL} (see also \cite[Lem.~3.12]{ST}). For completeness, we briefly go through the argument. 


Note that $\cQ$ has a right $\SL(V)$-action and $\PP(V\otimes R_m^*)$ has a left $\SL(V)$-action. Here we will use the left-action convention, i.e., $\phi \cdot (F,\rho) = (F,\rho \circ \phi^{-1})$ for $\phi \in \SL(V)$. 

Choose a 1-parameter torus $\lambda : \GG_m \to \SL(V)$. Equivalently, choose a weight decomposition \[V = \bigoplus_{i \in \Z} V_i,\] where the determinant $1$ condition implies $\sum_{i \in \Z } i \dim(V_i)=0$. Let $V_{\leq i}:= \bigoplus_{j\leq i}V_j$, $F_{\leq i}:=\rho(V_{\leq i}\otimes \O_X(-m))$, $F_i:=\frac{F_{\leq i}}{F_{\leq i-1}}$, and $\rho_i : V_i \otimes \O_X(-m) \twoheadrightarrow F_i$ be the induced surjective map. Let $s=\sum_i s_i$ such that $s_i \in V_i \otimes R_m ^*$ and let $i_{\max}$ be the maximal integer $i$ such that $s_i \neq 0$.

As in \cite[Lem.~4.4.3]{HL}, the limit point is \[\lim_{t \to \infty} \lambda(t) \cdot (F,\rho,s) = (\bigoplus_i F_i, \bigoplus_i \rho_i, s_{i_{\mathrm{max}}}).\] Since the fibre of the line bundle $\cL_{g(t),l}$ over a point $(F,\rho,s)$ is given as \[\cL_{g(t),l}|_{(F,\rho,s)} = \det(H^0(X,F(l))) \otimes \C\langle s\rangle ^{\otimes-g(l)}\] 
where $\C\langle s\rangle \subseteq V\otimes R_m^*$ is the subspace generated by $s$, we can easily compute the Hilbert-Mumford weight \begin{align*} \mu^{\lambda^{-1}}_{\cL_{g(t),l}} (F,\rho,s) &= \sum_i \left(-i\cdot P_{F_i}(l) + i_{\mathrm{max}} \cdot g(l)\right)\\ &= \sum_i\left(  \left(P_{F_{\leq i}}(l) - \dim(\Gamma_{\leq i}) \cdot g(l) \right) -  \frac{\dim(V_{\leq i})}{\dim(V)} \cdot (P_F(l) -g(l)) \right) \end{align*} where $\Gamma_{\leq i}:=\C\langle s \rangle \cap (V_{\leq i}\otimes R_m^*)$.

By \cite[Thm.~2.1]{MFK}, the point $(F,\rho,s)$ is semistable if and only if \[\frac{P_{F'}(l) - \dim(\Gamma') \cdot g(l)}{\dim(V')} \geq \frac{P_F(l) -g(l)}{\dim(V)}\] for all proper subspaces $V' \subseteq V$, where $F':=\rho(V'\otimes \O_X(-m))$ and $\Gamma':=\C\langle s \rangle \cap (V'\otimes R_m^*)$. This completes the proof.
\end{proof}

We will focus on the following three types of linearizations:

\begin{definition}\label{def:linearizationDTPT0PT1}
Recall that $P_v(t) \in \Q[t]$ is the Hilbert polynomial associated to $v \in H^*(X,\Q)$.
\begin{enumerate}
\item [($\DT$)] For any rational number $a>0$, we define a polynomial
 \[g^{(-1)}(t):=P_v(t) + a\cdot t^2 .\]
\item [($\PT_0$)] For any rational number $0< \epsilon <1$, we define a polynomial
\[g^{(0)}(t):=P_v(t) -\epsilon \cdot t.\]
\item [($\PT_1$)] For any rational number $0< \delta <\tfrac12$, we define a polynomial
\[g^{(1)}(t):=P_v(t) - \delta \cdot t^2.\]
\end{enumerate}
For each integer $q \in \{-1,0,1\}$, we define a linearization
\[\cL^{(q)}:= \cL_{g^{(q)}(t),l}\]
for sufficiently large $l$.
\end{definition}

\begin{remark}
We list the linearizations previously considered in the literature \cite{Pot1,Pot2,ST}. \begin{enumerate} \item In \cite{Pot1}, Le Potier considered $\O_{\cG}(c) \boxtimes \O_{\PP(V\otimes R_m^*)}(\frac{l^d}{d!})$ for $c>1$. \item In \cite{Pot2}, Le Potier considered $\O_{\cG}(1+ \frac{\alpha(m)}{P_v(m)}) \boxtimes \O_{\PP(V\otimes R_m^*)}(P_v(l))$ for $\alpha(t) \in \Q[t]$. \item In \cite{ST}, Stoppa-Thomas considered $\O_{\cG}(c) \boxtimes \O_{\PP(V\otimes R_m^*)}(l)$ for $c>0$ (for 1-dimensional pairs). 
\end{enumerate} The linearizations $\cL_{g(t),l}$ in this paper are 2-dimensional versions of the linearizations in \cite{ST}, which contain the linearizations in \cite{Pot1} as a special case.
However our parameter space is larger than the one in \cite{Pot1}, as explained in Remark \ref{rmk:parameterspace}.
\end{remark}	

We now compare GIT stability for the line bundles $\cL^{(q)}$ in Definition \ref{def:linearizationDTPT0PT1} and $\PT_q$-stability introduced in Definition \ref{def:stabilityconditionsofpairs'}.

\begin{proposition}\label{prop:GIT/PTq}
There exists an integer $M$ which only depends on $X$ and $v$, and integers $L_m$ which depend on $X$, $v$, and $m$ such that for any $m \geq M$, $l \geq L_m$, and $(F,\rho,s) \in \cM$, we have the following:
\begin{enumerate}
\item[$(1)$] $(F,\rho,s)$ is $\cL^{(q)}$-semistable $\implies$ $(F,\rho) \in \cQ^{\circ}$ and $(F,s)$ is $\PT_q$-stable;
\item[$(2)$] $(F,\rho) \in \cQ^{\circ}$ and $(F,s)$ is $\PT_q$-stable $\implies$ $(F,\rho,s)$ is $\cL^{(q)}$-stable.
\end{enumerate}
\end{proposition}

We will consider the three stability conditions separately. We start with $\PT_0$-stability.

\begin{proof}[Proof of Proposition \ref{prop:GIT/PTq} for $\PT_0$-stability]

\hfill
\\
\hfill

\noindent (1) Assume that $(F,\rho,s)$ is $\cL^{(0)}$-semistable.
We first claim that the induced map
\[\Gamma(\rho(m)) : V \to \Gamma(X,F(m))\]
is injective. Let $V'$ be the kernel of the above map $\Gamma(\rho(m))$. Then the induced sheaf $F':=\rho(V'\otimes \O_X(-m)) \subseteq F$ is zero. If $V'\neq 0$, then the $\cL^{(0)}$-semistability of $(F,\rho,s)$ implies an inequality
\[\frac{0}{\dim(V')} \geq \frac{\epsilon \cdot l}{\dim(V)}\]
by Lemma \ref{Lem:HilbMum}(1). This leads to a contradiction for $l>0$. Hence $V'=0$.

We claim that $(F,s) \in \cPair(X,v)$. More precisely, there exists a section $\O_X \to F$ that fits into the commutative diagram
\[\xymatrix@C+2pc{
\C \ar[d]^s \ar@{.>}[r]& \Gamma(X,F) \ar@{^{(}->}[d] \\
V \otimes R_m^* \ar@{^{(}->}[r]^-{\Gamma(\rho(m))} & \Gamma(X,F(m)) \otimes R_m^*
}\]
determined by the dotted arrow. Indeed, consider the open subscheme
\[\cQ^{\star}:=\{(F,\rho) \in \cQ : V \xrightarrow{\Gamma(\rho(m))} \Gamma(X,F(m)) \text{ is injective}\}\]
of $\cQ$. Note that the restriction $\cM|_{\cQ^{\star}}$ of $\cM$ is the closure of $\PP_{\cQ^{\circ}}(R\pi_*\FF)$ in $\cQ^{\star}\times \PP(V\otimes R_m ^*)$. Hence we can form a commutative diagram
\[\xymatrix{
\cM|_{\cQ^{\star}} \ar@{.>}[r]^{\exists} \ar@{^{(}->}[d] &  \PP_{\cQ^{\star}}(R\pi_*\FF) \ar@{^{(}->}[d]\\
\cQ^{\star}\otimes \PP(V\otimes R_m^*) \ar@{^{(}->}[r] & \PP_{\cQ^{\star}}(R\pi_*(\FF(m)) \otimes R_m^*)
}\]
of closed embeddings of projective schemes over $\cQ^{\star}$. By the result in the previous paragraph, we have $(F,\rho,s) \in \cM|_{\cQ^{\star}}$, which proves the claim.

Next, we claim that $Q:=\coker(\O_X \xrightarrow{s} F) \in \Coh_{\leq0}(X)$. Let
\[V':= V \cap \Gamma(X,\O_Z(m)) \quad\text{in}\quad \Gamma(X,F(m)),\]
where $\O_Z:=\im(\O_X \xrightarrow{s} F)$. Then the induced sheaf $F':=\rho(V'\otimes \O_X(-m)) \subseteq F$ is contained in $\O_Z \subseteq F$. Note that we have a commutative diagram
\[\xymatrix{
V'\otimes R_m^* \ar@{^{(}->}[r] \ar@{^{(}->}[d]  & \Gamma(X,\O_Z(m)) \otimes R_m^* \ar@{^{(}->}[d]^s & \Gamma(X,\O_Z) \ar@{_{(}->}[l] \ar@{^{(}->}[d]^s\\
V \otimes R_m^* \ar@{^{(}->}[r]^-{\rho(m)} & \Gamma(X,F(m)) \otimes R_m^* & \Gamma(X,F) \ar@{_{(}->}[l]
}\]
where the left square is cartesian. Hence we have $s \in V' \otimes R_m^*$. By Lemma \ref{Lem:HilbMum}(2), we have inequalities
\begin{equation}\label{eq:S2.7}    
\frac{P_{\O_Z}(l) - g^{(0)}(l)}{\dim(V')} \geq\frac{P_{F'}(l) - g^{(0)}(l)}{\dim(V')} \geq \frac{\epsilon \cdot l }{\dim(V)},
\end{equation}
for sufficiently large $l$, where the two collections
\[\{\O_Z:=\im(\O_X \xrightarrow{s} F)\} \and \{F' := \im(V'\otimes \O_X(-m) \xrightarrow{\rho} F)\}\]
for fixed $X$, $v$, $m$ are bounded by \cite[Prop.~1.2(i)]{Gro}. Hence the set $\{P_{\O_Z}(t)\}$ of Hilbert polynomials is finite. Since the leading coefficient of the left-hand side of \eqref{eq:S2.7}, as a polynomial in $l$, should be positive, we can take $l$ sufficiently large and deduce $\alpha_2(\O_Z)=\alpha_2(F)$ and $\alpha_1(\O_Z)>\alpha_1(F) - \epsilon$.
Since $\epsilon<1$, we have $\alpha_1(\O_Z)=\alpha_1(F)$, and thus $Q$ is 0-dimensional.

We now show that $(F,\rho) \in \cQ^{\circ}$. By Lemma \ref{lem2}, the collection of coherent sheaves $F$ with $\ch(F)=v$ and $\coker(s) \in \Coh_{\leq 0}(X)$ is bounded for fixed $X,v$.
Hence we have $H^{>0}(X,F(m))=0$ for sufficiently large $m$, as desired. 

We finally claim that $F \in \Coh_{\geq 1}(X)$. Let $T_0(F) \subseteq F$ be the 0-dimensional torsion sheaf and let
\[V' := \Gamma(X, T_0(F)(m)) \subseteq \Gamma(X,F(m))\xleftarrow{\cong} V,\]
where the last isomorphism follows from the result in the previous paragraph. Then the induced sheaf $F':=\rho(V' \otimes \O_X(-m)) \subseteq F$ is $T_0(F)$ since the 0-dimensional sheaf $T_0(F)$ is globally generated. If $V' \neq 0$, then we have an inequality
\[1= \frac{\chi(T_0(F))}{\chi(T_0(F)(m))} = \frac{\chi(F')}{\dim(V')}\geq \frac{\epsilon \cdot l }{\dim(V)}\]
by Lemma \ref{Lem:HilbMum}(1). This leads to contradiction for big enough $l$. Therefore $V'=0$ and $F'=T_0(F)=0$. \\

\noindent (2) Assume that $(F,\rho) \in \cQ^{\circ}$ and $(F,s)$ is $\PT_0$-stable. We will prove the $\cL^{(0)}$-stability of $(F,\rho,s)$ through Lemma \ref{Lem:HilbMum}. Choose a proper non-zero subspace $V' \subseteq V$ and let $F' := \rho(V' \otimes \O_X(-m)) \subseteq F$ be the induced subsheaf. Then we have a commutative square
\begin{equation}\label{10}
\xymatrix{
V' \ar@{^{(}->}[r] \ar@{^{(}->}[d] & \Gamma(X, F'(m)) \ar@{^{(}->}[d] \\
V \ar[r]^<<<<{\cong} & \Gamma(X,F(m)).
}
\end{equation}
We first prove the inequality \eqref{eq:ineq1} in Lemma \ref{Lem:HilbMum}(1). Note that the collection 
\[\{F' = \im(V' \otimes \O_X(-m) \xrightarrow{\rho} F) : (F,\rho) \in \cQ^{\circ} \and V' \subseteq V\}\]
for fixed $X,v,m$ is bounded by \cite[Prop.~1.2.(i)]{Gro}. Hence the set $\{P_{F'}(t)\}$ of the corresponding Hilbert polynomials is finite. Since $F' \in \Coh_{\geq 1}(X)$ by $\PT_0$-stability and $\chi(F')$ has a lower bound, independent of the choices of $F$, $V'$, $l$, we have an inequality
\[\frac{P_{F'}(l)}{\dim(V')} > \frac{l + \chi(F')}{\dim(V)}> \frac{\epsilon \cdot l}{\dim(V)}\]
for sufficiently large $l$. This proves \eqref{eq:ineq1}.

We then prove the second inequality \eqref{eq:ineq2} in Lemma \ref{Lem:HilbMum}(2) when $s \in V' \otimes R_m^*$. Indeed, consider the canonical commutative diagram
\begin{equation}\label{12}
\xymatrix{
\Gamma(X,F') \ar[r] \ar@{^{(}->}[d] & \Gamma(X,F) \ar[r] \ar@{^{(}->}[d] & \Gamma(X, F/F') \ar@{^{(}->}[d] \\
\Gamma(X,F'(m)) \otimes R_m^* \ar[r] & \Gamma(X,F(m)) \otimes R_m^* \ar[r] & \Gamma(X,F/F'(m)) \otimes R_m^*
}    
\end{equation}
induced by the short exact sequence $0 \to F' \to F \to F/F' \to 0$. By \eqref{10}, we have $s \in \Gamma(X,F'(m)) \otimes R_m^*$. 
By a diagram chasing argument, we can deduce that the left square in \eqref{12} is cartesian. Hence $s \in \Gamma(X,F')$ and we have $\O_Z:=\im(s:\O_X \to F) \subseteq F'$ as subsheaves of $F$. Since $(F,s)$ is $\PT_0$-stable, we have $\alpha_2(F')=\alpha_2(F)$ and $\alpha_1(F')=\alpha_1(F)$. Since $\chi(F')$ has a lower bound, independent of the choices of $F$, $V'$, $l$, we have an inequality
\[\frac{\epsilon\cdot l + \chi(F') - \chi(F)}{\dim(V')} > \frac{\epsilon \cdot l}{\dim(V)}\]
for sufficiently large $l$. This proves \eqref{eq:ineq2}.
\end{proof}

By \cite{Gro,Pot1}, the moduli spaces $\cPair^{(-1)}(X,v)$ and $\cPair^{(1)}(X,v)$ are projective schemes  (see Remark \ref{rmk:repforDTPT1}). We will use these results to prove Proposition \ref{prop:GIT/PTq} for the $\DT$-stability and the $\PT_1$-stability.\footnote{Proposition \ref{prop:GIT/PTq} for the $\PT_1$-stability is not a consequence of the results in \cite{Pot1}, since the parameter space $\cM$ in this paper is {\em larger} than the parameter space in \cite{Pot1} (see Remark \ref{rmk:parameterspace}).} 


\begin{proof}[Proof of Proposition \ref{prop:GIT/PTq} for $\DT$-stability] 

\hfill
\\
\hfill

\noindent (2) Assume that $(F,\rho) \in \cQ^{\circ}$ and $(F,s)$ is a $\DT$ pair. Choose a proper non-zero subspace $V' \subseteq V$ and let $F'=\rho(V'\otimes \O_X(-m)) \subseteq F$ be the induced subsheaf.

We first prove the inequality \eqref{eq:ineq1}. Note that the leading coefficient of the right-hand side of \eqref{eq:ineq1}, as a polynomial of $l$, is negative. On the other hand, the leading coefficient of the left-hand side of \eqref{eq:ineq1} is always positive. Since the collection $\{F'\}$ for fixed $X$, $v$, $m$ is bounded by \cite[Prop.~1.2(i)]{Gro}, the set $\{P_{F'}(t)\}$ is finite. Hence we have \eqref{eq:ineq1} for sufficiently large $l$.

We then prove the second inequality \eqref{eq:ineq2} when $s \in V' \otimes R_m^*$. Indeed, we can form the commutative square \eqref{10} and the commutative diagram \eqref{12} where the left square is cartesian. Hence we have $s \in \Gamma(X,F')$.  Since $(F,s)$ is a $\DT$ pair, the coherent sheaf $F$ is generated by the section $s$. Therefore $F'=F$. We claim that $V'=V$, which then leads to a contradiction. Indeed, since $s \in V' \otimes R_m^*$, we have a commutative diagram
\[\xymatrix{
\C\cong\Gamma(X,\O_X) \ar[r] \ar@{.>}[d] & \Gamma(X,F) \ar@{^{(}->}[d] \\
V' \otimes R_m^* \ar@{^{(}->}[r] & \Gamma(X,F(m)) \otimes R_m^* 
}\]
for some dotted arrow. By adjunction, we can form a commutative diagram
\[\xymatrix{
R_m \ar[r] \ar@{.>}[d] & \Gamma(X,F) \otimes R_m \ar[d] \\
V' \ar@{^{(}->}[r] & \Gamma(X,F(m)).
}\]
The collection $\{\I_{Z/X} : \ch(\O_Z)=v\}$ is bounded, so for $m$ sufficiently large (with lower bound only depending on $X,v$), we may assume that the map
\[R_m = \Gamma(X,\O_X(m)) \to \Gamma(X,F(m))\]
is surjective. Therefore, $V' \hookrightarrow V=\Gamma(X,F(m))$ is also surjective, which proves the claim. \\

\noindent (1) We first claim that 
\[(F,\rho,s) \in \cM^{ss}(\cL^{(-1)})|_{\cQ^{\circ}} \implies  \text{$(F,s)$ is a $\DT$ pair}\]
for sufficiently large $m$ and $l$. Indeed, choose $(F,\rho,s) \in \cM^{ss}(\cL^{(-1)})|_{\cQ^{\circ}}$ and let $Z$ be the support of $s$ and $Q$ be the cokernel of $s$. Let $V'= \Gamma(X,\O_Z(m)) \subseteq V\cong\Gamma(X,F(m))$ and let $F'=\rho(V'\otimes \O_X(-m)) \subseteq \O_Z \subseteq F$. Then clearly $s\in V' \otimes R_m^*$. Thus we have the inequality \eqref{eq:ineq2}. Also note 
\begin{equation} \label{eq:dimVdimV'}
-\frac{a}{\dim(V')} < - \frac{a}{\dim(V)}.
\end{equation}
As above, for fixed $X,v,m$ the collection of $\{F'\}$ is bounded, so there are finitely many $\{P_{F'}(t)\}$. Taking $l$ sufficiently large, \eqref{eq:ineq2} contradicts \eqref{eq:dimVdimV'}. Hence $V'=V$ and thus $F'=F$. Therefore $\O_Z=F$ and $(F,s)$ is a $\DT$ pair.

By the results in (2) and the previous paragraph, we have open embeddings
\[\left[\cM^{ss}|_{\cQ^{\circ}}(\cL^{(-1)})/\PGL(V)\right] \subseteq \cPair^{(-1)}(X,v) \subseteq \left[ \cM^{st}(\cL^{(-1)})/\PGL(V)\right] \]
which induces an open dense embedding
\[\cPair^{(-1)}(X,v) \hookrightarrow \cM^{st}(\cL^{(-1)})/\!/\PGL(V)\hookrightarrow \cM^{ss}(\cL^{(-1)})/\!/\PGL(V).\]
Since $\cPair^{(-1)}(X,v)$ is proper and $\cM^{ss}(\cL^{(-1)})/\!/\PGL(V)$ is separated, the above open embedding is an isomorphism. Hence $\cM^{st}(\cL^{(-1)})=\cM^{ss}(\cL^{(-1)})$ and $\cPair^{(-1)}(X,v)=[\cM^{st}(\cL^{(-1)})/\PGL(V)]$. 
\end{proof}

\begin{proof}[Proof of Proposition \ref{prop:GIT/PTq} for $\PT_1$-stability]

\hfill
\\
\hfill

\noindent (2) Assume that $(F,\rho) \in \cQ^{\circ}$ and $(F,s)$ is a $\PT_1$ pair. Choose a proper non-zero subspace $V' \subseteq V$ and let $F'=\rho(V'\otimes \O_X(-m)) \subseteq F$ be the induced sheaf. 

We first prove the inequality \eqref{eq:ineq1}. Indeed, since $F$ is pure, the subsheaf $F'$ is also 2-dimensional. Since $\delta<\frac12$ and $\dim(V')<\dim(V)$, we have an inequality
\[\frac{\alpha_2(F')}{2\dim(V')} > \frac{\delta}{\dim(V)},\]
between the leading coefficients of the both sides of \eqref{eq:ineq1}, as polynomials of $l$. For fixed $X,v,m$, the collection $\{F'\}$ is bounded by \cite[Prop.~1.2(i)]{Gro}, and hence the set $\{P_{F'}(t)\}$ is finite. Therefore we have \eqref{eq:ineq1} for sufficiently large $l$.

We then prove the inequality \eqref{eq:ineq2} when $s \in V' \otimes R_m^*$. Indeed, we can form the commutative square \eqref{10}. We can also form the commutative diagram \eqref{12} where the left square is cartesian. Hence we have $s \in \Gamma(X,F')$. 
Since $(F,s)$ is a $\PT_1$ pair, $\alpha_2(F)=\alpha_2(F')$. Since $\dim(V')< \dim(V)$, we have an inequality
\[\frac{\delta}{\dim(V')}> \frac{\delta}{\dim(V)}\]
between the leading coefficients of the both sides of \eqref{eq:ineq2}, as polynomials of $l$. Since the collection $\{F'\}$ is bounded, we have \eqref{eq:ineq2} for sufficiently large $l$. \\

\noindent (1) We will show that 
\begin{equation}\label{eq:S2.3}
(F,\rho,s) \in \cM^{ss}(\cL^{(1)})|_{\cQ^{\circ}} \implies  \text{$(F,s)$ is a $\PT_1$ pair}	
\end{equation}
for sufficiently large $m$ and $l$. Choose $(F,\rho,s) \in \cM^{ss}(\cL^{(1)})|_{\cQ^{\circ}}$. Let $Q=\coker(s: \O_X \to F)$ be the cokernel and $\O_S = \im(s: \O_X \to F)$ the image.

We first claim that $Q \in \Coh_{\leq 1}(X).$ Let $V'= \Gamma(X,\O_S(m)) \subseteq \Gamma(X,F(m)) \cong V$ and let $F'=\rho(V'\otimes \O_X(-m)) \subseteq \O_S \subseteq F$. Then clearly $s\in V' \otimes R_m^*$ and in particular $V' \neq 0$. Hence we have inequality \eqref{eq:ineq2}. If $\alpha_2(F')< \alpha_2(F)$, then the leading term of the left-hand side of \eqref{eq:ineq2} is negative since $\delta<\frac{1}{2}$. Since the collection $\{F'\}$ is bounded by \cite[Prop.~1.2(i)]{Gro} for fixed $X,v,m$, this leads to a contradiction for sufficiently large $l$. Hence $\alpha_2(F')=\alpha_2(F)$ and thus $\alpha_2(\O_S)=\alpha_2(F)$. This proves the claim.

We then claim that $F$ is pure. Consider the torsion subsheaf $TF\subseteq F$. Let $V'= \Gamma(X, TF(m)) \subseteq V$ and let $F'=\rho(V'\otimes \O_X(-m))\subseteq TF \subseteq F$. If $V'\neq 0$, then we have the inequality \eqref{eq:ineq1}. Since the degree of the left-hand side of \eqref{eq:ineq1}, as polynomial of $l$, is $\leq 1$ and the degree of the right-hand side of \eqref{eq:ineq1} is 2, and the collection $\{F'\}$ is bounded, the inequality \eqref{eq:ineq1} leads to a contradiction for sufficiently large $l$. Hence we deduce that 
\begin{equation}\label{25}
\Gamma(X, TF(m))=0.	
\end{equation}
Consider the short exact sequence
\[\xymatrix{ 0 \ar[r] & TF \ar[r] & F \ar[r] & F/TF \ar[r] & 0.}\]
Define $\O_Z = \mathrm{im}(\O_S \to F \to F/TF)$.
Then we obtain a short exact sequence
\begin{equation} \label{eqn:sesPT1pairR}
\xymatrix{ 0 \ar[r] & \O_Z \ar[r] & F/TF \ar[r] & R \ar[r] & 0,}
\end{equation}
where $R=\coker(\O_Z \to F/TF) \in \Coh_{\leq 1}(X)$. Then both $F/TF$ and $\O_Z$ are pure 2-dimensional sheaves such that $\alpha_2(F/TF)=\alpha_2(\O_Z)=\alpha_2(F)$. In particular, \eqref{eqn:sesPT1pairR} defines a $\PT_1$ pair. Moreover, $\alpha_1(F/TF)$ is bounded above by $\alpha_1(F)$ and bounded below $\alpha_1(\O_{Z})$, which is bounded below by a constant depending on $X,\alpha_2(F)$ but not on $m$ by \cite[Cor.~2.13]{Pot1}. In particular, $\alpha_1(F/TF)$ takes finitely many values.
By Proposition \ref{prop:boundofchiforPT}(2),  $\chi(F/TF)$ is bounded above by a constant depending on $X,v$ but not on $m$. Since $\chi(TF)=\chi(F)-\chi(F/TF)$, we have
\begin{equation}\label{26}
	\chi(TF)\geq C(X,v)
\end{equation}
for a constant $C(X,v)$ depending on $X,v$ but independent of $m$. By the Riemann-Roch formula, we have
\[\chi(TF(m)) = r \cdot m + \chi(TF)\]
where $r = \alpha_1(F) - \alpha_1(F/TF) \geq 0$ since $TF \in \Coh_{\leq 1}(X)$. Moreover, $r \geq 0$ takes finitely many values independent of $m$. We deduce from \eqref{25} and \eqref{26} that, for all $F$ as above, we have $r=0$ for sufficiently large $m$, which means $TF$ is 0-dimensional. Hence $TF$ is globally generated and thus \eqref{25} implies $TF=0$. Therefore $F$ is pure.

Combining the results in (2) and \eqref{eq:S2.3}, we have open embeddings
\[\left[\cM^{ss}|_{\cQ^{\circ}}(\cL^{(1)})/\PGL(V)\right] \subseteq \cPair^{(1)}(X,v) \subseteq \left[ \cM^{st}(\cL^{(1)})/\PGL(V)\right]\]
which induces an open dense embedding
\[\cPair^{(1)}(X,v) \hookrightarrow \cM^{st}(\cL^{(1)})/\!/\PGL(V)\hookrightarrow \cM^{ss}(\cL^{(1)})/\!/\PGL(V).\]
Since $\cPair^{(1)}(X,v)$ is proper and $\cM^{ss}(\cL^{(1)})/\!/\PGL(V)$ is separated, the above map is an isomorphism. Consequently we have $\cM^{st}(\cL^{(1)})=\cM^{ss}(\cL^{(1)})$ and $\cPair^{(1)}(X,v)=[\cM^{st}(\cL^{(1)})/\PGL(V)]$.
\end{proof}

Theorem \ref{Thm:GIT} is a direct corollary of Proposition \ref{prop:GIT/PTq}.

\begin{proof}[Proof of Theorem \ref{Thm:GIT}]
Since the stacks
\[\cPair^{(q)}(X,v), \quad [\cM^{st}(X,\cL^{(q)})/\PGL(V)],\quad  [\cM^{ss}(X,\cL^{(q)})/\PGL(V)]\]
are open substacks of $\cPair(X,v)_m = [\cM|_{\cQ^{\circ}}/\PGL(V)] = [\PP_{\cQ^{\circ}}(R\pi_*\FF)/\PGL(V)]$, it suffices to compare the $\C$-points. Proposition \ref{prop:GIT/PTq} completes the proof.
\end{proof}

We end this subsection with the following question:

\begin{question}
Does Theorem \ref{Thm:GIT} holds for $d$-dimensional $\PT_q$ pairs with $d>2$ as defined in Remark \ref{rem:generalizePTq}?
\end{question}

\subsection{Moduli of complexes}\label{ss:openembeddingtoPerf}

In this subsection, we prove that the moduli spaces of $\PT_q$ pairs are open subschemes of the moduli space of perfect complexes. We basically follow the reasoning in \cite{PT1}, but the actual proof requires a finer argument.  

Let $X$ be a smooth projective variety and let $v \in H^*(X,\Q)$.
Recall \cite{Ina,Lie} that there exists a moduli space 
\[\cPerf(X,v)_{\O_X}^{\spl}\]
 of simple perfect complexes $I\udot$ on $X$ with $\det(I\udot)\cong \O_X$ and $\ch(I\udot) = v$, as an algebraic space. 
 


\begin{theorem}\label{Thm:PairtoPerf}
Let $X$ be a smooth projective variety of dimension $n\geq 4$, $v \in H^*(X,\Q)$, and $q\in \{-1,0,1\}$. Then we have an open embedding
\begin{equation}\label{eq:S2.1}
\cPair^{(q)}(X,v) \hookrightarrow \cPerf(X,v')_{\O_X}^{\mathrm{spl}} : (F,s) \mapsto I\udot:=[\O_X \to F]	
\end{equation}
where $v':= 1-v \in H^*(X,\Q)$.
\end{theorem}


The rest of this subsection is devoted to the proof of Theorem \ref{Thm:PairtoPerf}. Throughout this subsection, we use the notation in Theorem \ref{Thm:PairtoPerf}.

To any $\PT_q$ pair $(F,s)$ on $X$, we can associate a perfect complex 
\[I\udot:=[\O_X \xrightarrow{s} F]\]
with $\det(I\udot) \cong \O_X$. We first show that $I\udot$ is {\em simple}.

\begin{lemma}
Given a $\PT_q$ pair $(F,s)$ on $X$, we have
\[\Ext^{<0}_X(I\udot,I\udot)=0 \and \Hom_X(I\udot,I\udot)=\C\]
for the associated perfect complex $I\udot:=[\O_X \xrightarrow{s} F]$.
\end{lemma}

\begin{proof}
The proof is analogous \cite[Lem.~1.15]{PT1}. Consider the two canonical distinguished triangles
\[\xymatrix{I\udot \ar[r] & \O_X \ar[r] & F  } \and \xymatrix{\I_Z \ar[r] & I\udot \ar[r] & Q[-1]}\]
where $Z$ is the support of $F$ and $Q$ is the cokernel of $s$. We can form the following three induced distinguished triangles
\begin{align}
&\xymatrix{R\Hom_X(F,\O_X) \ar[r] & R\Hom_X(\O_X,\O_X) \ar[r] & R\Hom_X(I^\mdot,\O_X)},
\label{eq:S2.dt1}\\
&\xymatrix{R\Hom_X(Q,F)[1] \ar[r] & R\Hom_X(I^\mdot,F) \ar[r] & R\Hom_X(\I_Z,F)},
\label{eq:S2.dt2}\\
&\xymatrix{R\Hom_X(I^\mdot,I^\mdot) \ar[r] & R\Hom_X(I^\mdot,\O_X) \ar[r] &R\Hom_X(I^\mdot,F)}.
\label{eq:S2.dt3}
\end{align}
The distinguished triangle \eqref{eq:S2.dt1} implies
\begin{equation}\label{Eq7}
\Hom_X(I^\mdot,\O_X)=\C \and \Ext^{<0}_X(I^\mdot,\O_X)=0
\end{equation}
since $\Ext^{\leq 1}_X(F,\O_X)=H^{\geq n-1}(X,F)^*=0$. On the other hand, the distinguished triangle \eqref{eq:S2.dt2} implies
\begin{equation}\label{Eq8}
\Ext^{<0}_X(I^\mdot,F)=0 
\end{equation}
since $\Hom_X(Q,F)=0$ by $\PT_q$-stability. Then the distinguished triangle \eqref{eq:S2.dt3} and equations \eqref{Eq7} and \eqref{Eq8} prove that $I^\mdot$ is simple.
\end{proof}

Since $\cPair^{(q)}(X,v)$ is a {\em fine} moduli space, there exists a universal family 
\[\II\udot = [\O_{\cPair^{(q)}(X,v)\times X} \xrightarrow{} \FF],\]
which defines a map 
\begin{equation}\label{eq:S2.2}
\cPair^{(q)}(X,v) \xrightarrow{\II\udot} \cPerf(X,v)_{\O_X}^{\spl}	
\end{equation}
between the moduli spaces by Lemma \ref{lem:det} below. 

\begin{lemma}\label{lem:det}
Let $f:\sX \to \sY$ be a smooth projective morphism of schemes. Let $\mathcal{F}$ be a $\sY$-flat coherent sheaf on $\sX$. If $\mathrm{codim}(\Supp(\mathcal{F}_y)) \geq 2$ for all $y \in \sY$, then $\det(\mathcal{F}) \cong \O_{\sX}$.
\end{lemma}

\begin{proof}
This follows directly from \cite[Prop.~3.5]{HK}.
\end{proof}

We show that the above map \eqref{eq:S2.2} satisfies the {\em infinitesimal lifting property}.




\begin{proposition}\label{prop:infinitesimallifting}
Let $B_0 \subseteq B$ be an nilpotent extension of 0-dimensional schemes with a unique closed point $b \in B_0$. Let $I\udot$ be a perfect complex on $X \times B$ with trivial determinant. 
Let $i:X \times B_0 \hookrightarrow X\times B$ be the inclusion map.
If $I\udot|_{X\times B_0} \cong [\O_{X\times B_0} \xrightarrow{s_0} F_0]$ for a $B_0$-flat family $(F_0,s_0)$ of $\PT_q$ pairs over $B_0$, then there exists a $B$-flat family $(F,s)$ of $\PT_q$ pairs over $B$ such that $I\udot\ \cong [\O_{X\times B} \xrightarrow{s} F]$ and $(F,s)|_{X\times B_0}=(F_0,s_0)$.
\end{proposition}

\begin{proof}
We first take $q = 0, 1$. The (easier) case $q=-1$ is discussed at the end of the proof. Let $(F_b,s_b)$ be the fibre of $(F_0,s_0)$ over the point $b \in B_0 $. Since $F_b \in \Coh_{\geq1}(X)$, we have $\mathrm{depth}_{\O_{X,x}}((F_b)_x) \geq 1$ for all closed points $x \in \Supp(F_b)$. 
By the Auslander-Buchsbaum formula, we have
\[\mathrm{pd}_{\O_{X,x}}((F_b)_x) \leq n-1\]
for all $x \in X$.
Then $I^{\mdot}|_{X\times\{b\}}\cong[\O_X \xrightarrow{s_b} F_b]$ has tor-amplitude $[-(n-2),1]$, and thus we have a quasi-isomorphism with a complex of locally free sheaves
\begin{equation} \label{eqn:resIb}
I^\mdot|_{X \times \{b\}}\cong[F^{-(n-2)}\to \cdots \to F^0 \to F^1].
\end{equation}
Since $X\times B$ is quasi-projective, $I^\mdot$ is quasi-isomorphic to a bounded complex of locally free sheaves. Using Nakayama's lemma, and the fact that a morphism between locally free sheaves is injective on the fibres if and only if it is injective with locally free cokernel, we see from \eqref{eqn:resIb} that there exists a quasi-isomorphism with a complex of locally free sheaves
\[I^\mdot\cong[E^{-(n-2)}\to \cdots \to E^0 \to E^1].\]
Moreover since $I\udot$ has amplitude $[0,1]$ by \cite[Lem.~2.1.4]{Lie}, we have
\[I^\mdot \cong [A \to E^1]\]
where $A=\coker(E^{-1} \to E^0)$ is $B$-flat.

We claim that the canonical map
\[h^0(I^\mdot)\to h^0(I^\mdot)^{**}\]
is injective. Indeed, let $K$ be the kernel of the above map $h^0(I^\mdot)\to h^0(I^\mdot)^{**}$. Note that $h^0(I^\mdot)$ is a vector bundle away from the support of $F_0$ by \cite[Lem.~2.1.4]{Lie}. 
Hence $\dim(K) \leq 2 < n$.

The support $W$ of $h^1(I^\mdot)$ is of dimension $\leq q$. This follows from the fact that $i^*h^1(I^\mdot)\cong h^1(I_0^\mdot)$ 
and Nakayama's lemma. 
Away from $W$, the coherent sheaf $h^0(I^\mdot)$ is $B$-flat because $A$ is $B$-flat.
Since the ideal sheaf $h^0(I_b^\mdot)$ is torsion free, all non-zero subsheaves of $h^0(I^\mdot)|_{X \times B \setminus W}$ 
are $n$-dimensional by Lemma \ref{lem:infinitesimalextensionofpuresheaf} below. Therefore $K$ is (set-theoretically) supported in $W$. Thus we have $\dim(K)\leq  \dim(W) \leq q.$

Consider the local-to-global spectral sequence 
\[H^i(X\times B,\ext^j_{X\times B}(A,K\otimes \omega_X))\Rightarrow \Ext^{i+j}_{X\times B}(A,K\otimes \omega_X).\]
Since 
$\dim(K\otimes \omega_X) \leq q \leq 1$
and there is a resolution $A\cong[E^{-(n-2)}\to \cdots \to E^0]$ of length $n-2$, we have
\[\Ext^n_{X\times B}(A,K\otimes \omega_X)=0.\]
By Serre duality $R\hom_{\pi}(K,A)\cong R\hom_{\pi}(A, K \otimes \omega_X))^{\vee}[-n],$ we have
\[\Hom_{X\times B}(K,A)^*\cong\Ext^n_{X\times B}(A,K\otimes \omega_X)=0,\]
where the dual is taken as a $\Gamma(B,\O_B)$-module.
Since $K \subseteq A$, we deduce $K=0$.
This proves the claim.

Since the rest of the proof is exactly the same as \cite{PT1}, we will only sketch the remaining arguments. By \cite[Lem.~6.13]{Kol}, $h^0(I^\mdot)^{**}$ is a vector bundle away from the support of $Q=h^1(I^\mdot)$. 
In the proof of \cite[Lem.~6.13]{Kol}, it is also shown that a reflexive rank 1 sheaf which is locally free outside a subset of codimension $n-1 \geq 3$, is in fact locally free everywhere. Therefore $h^0(I^\mdot)^{**}$ is locally free. 

Next, we have $\det(I^\mdot) \cong \O_{X \times B}$ and $W \ \subset \mathrm{Supp}(F_0)$, thus
$$
h^0(I^\mdot)^{**}|_{(X\times B)\setminus \mathrm{Supp}(F_0)} \cong h^0(I\udot)|_{(X\times B)\setminus \mathrm{Supp}(F_0)} \cong \O_{(X\times B) \setminus \mathrm{Supp}(F_0)}.
$$
Since $\mathrm{Supp}(F_0)$ has codimension $n-2\geq 2$, we find $h^0(I\udot)^{**}\cong \O_{X\times B}$ globally and $\I_Z \cong h^0(I^\mdot)$ is an ideal sheaf of a closed subscheme $Z \subseteq X \times B$.
The distinguished triangle 
\[\xymatrix{\I_Z \ar[r] & I^\mdot \ar[r] & Q[-1]}\]
defines an element
\[\alpha\in \Ext^2_{X\times B}(Q,\I_Z)\cong\Ext^1_{X\times B}(Q,\O_Z)\]
since $\Ext^1_{X\times B}(Q,\O_{X\times B})=\Ext^2_{X\times B}(Q,\O_{X\times B})=0$ (recall $q \leq 1$ and $n \geq 4$). Hence we have a short exact sequence
\[\xymatrix{0 \ar[r] & \O_Z \ar[r] & F \ar[r] & Q \ar[r] & 0}\]
of coherent sheaves on $X\times B$. Therefore we have an isomorphism
\[I^\mdot \cong [\O_{X\times B} \xrightarrow{s} F]\]
of perfect complexes. Since $F$ and $F_0$ are 2-dimensional, $\Hom_{X\times B}(I^\mdot,\O_{X\times B}) \cong \Gamma(X\times B,\O_{X\times B}) \cong \Gamma(B,\O_B)$ and $\Hom_{X\times B_0}(I_0^\mdot,\O_{X\times B_0}) \cong \Gamma(B_0,\O_{B_0})$. As in \cite{PT1}, we therefore have $L i^*F \cong F_0$. Hence $F$ is flat over $B$ by \cite[Lem.~2.5]{PT1}, which completes the proof.

For $q=-1$, $I^\mdot$ is only quasi-isomorphic to a bounded complex of locally free sheaves $[E^{-(n-1)} \to \cdots \to E^1]$. However, this time $W = \emptyset$, so $h^0(I^\mdot)\to h^0(I^\mdot)^{**}$ is obviously injective (Lemma \ref{lem:infinitesimalextensionofpuresheaf}). The rest of the proof proceeds as in the previous paragraph.
\end{proof}

We need the following lemma to complete the proof of Proposition \ref{prop:infinitesimallifting}.

\begin{lemma}\label{lem:infinitesimalextensionofpuresheaf}
Let $(A,\m)$ be a local Artinian $\C$-algebra and let $X$ be a quasi-projective scheme. Consider an $A$-flat family $F$ of sheaves on $X\times \Spec(A)$ whose fibre $F_0=F \otimes_A A/\m$ is a pure sheaf. Then $F$ is also a pure sheaf, i.e., for all non-zero subsheaves $G \subseteq F$, we have $\dim(G)=\dim(F)$.
\end{lemma}

\begin{proof}
Since $A$ is Artinian, we may take a sequence of ideals \cite[Prop.~3.7]{Eis}
\[0=\m_N\subset \cdots \subset\m_2\subset \m_1 =\m,\quad\text{such that}\quad\m_i/\m_{i+1} \cong A/\m\,.\]
We let $\m_0=A$. Suppose we have a subsheaf $G\subset F$ such that $\dim(G)<\dim(F)$. Since $F$ is flat over $A$, we have a short exact sequence
\begin{equation}\label{Eq1}
\xymatrix{0 \ar[r]& \m_{i+1}\cdot F \ar[r]& \m_i \cdot F \ar[r]& F_0 \ar[r]& 0}
\end{equation}
of sheaves on $X \times \Spec(A)$, where $\m_i \otimes_A F=\m_i\cdot F$. Consider the induced morphism 
\[\xymatrix{
0 \ar[r] & G \cap \m\cdot F \ar[r] \ar@{^{(}->}[d] & G \ar[r] \ar@{^{(}->}[d] & \frac{G}{G \cap \m\cdot F} \ar[r] \ar@{^{(}->}[d] & 0 \\
0 \ar[r] & \m\cdot F \ar[r] & F \ar[r] & F_0 \ar[r] & 0}\]
of short exact sequences. Note that
\[\dim(\tfrac{G}{G\cap \m \cdot F}) \leq \dim(G)<\dim(F)=\dim(F_0).\]
Since $F_0$ is pure, we have $\frac{G}{G\cap \m \cdot F} = 0$. Therefore we deduce that
\[G \subseteq \m \cdot F.\]
Moreover, applying this argument to \eqref{Eq1}, we deduce that 
\[G \subseteq \m_i \cdot F \implies G \subseteq \m_{i+1} \cdot F.\] 
Therefore we obtain $G \subseteq \m_N \cdot F = 0$ by induction on $i$.
\end{proof}

Now we can simply deduce Theorem \ref{Thm:PairtoPerf} from Lemma \ref{lem:injectivity'} and Proposition \ref{prop:infinitesimallifting} as follows:

\begin{proof}[Proof of Theorem \ref{Thm:PairtoPerf}]
By Lemma \ref{lem:injectivity'}, the map
\[\cPair^{(q)}(X,v) \xrightarrow{\II\udot} \cPerf(X,v)_{\O_X}^{\spl}\]
in \eqref{eq:S2.2} is injective on $\C$-points. Since $\C$ is algebraically closed, the above map is universally injective. In order to verify that the above map is formally smooth it is enough to check the lifting property along extensions of local Artinian $\C$-algebras \cite[02HW]{Sta} which is done in Proposition \ref{prop:infinitesimallifting}. Since $\cPair^{(q)}(X,v)$ is of finite type, the above map is {\'e}tale. Moreover the above map is universally injective and {\'e}tale, thus it is an open embedding (\cite[tag 02LC]{Sta}).
\end{proof}

\begin{remark}
The open embedding \eqref{eq:S2.1} in Theorem \ref{Thm:PairtoPerf} induces a derived enhancement of the moduli space $\cPair^{(q)}(X,v)$ of $\PT_q$ pairs. The tangent complex at $(F,s) \in \cPair^{(q)}(X,v)$ is given as
\[\TT|_{(F,s)} = R\Hom_X(I\udot,I\udot)_0[1]\]
where $I\udot=[\O_X \xrightarrow{s} F]$. This derived structure (of perfect complexes) is {\em different} from the derived structure (of pairs) in Remark \ref{rem:derivedinterpretation}.
\end{remark}

\begin{remark}
For $\PT_1$-stability, Theorem \ref{Thm:PairtoPerf} was independently proved by Gholampour-Jiang-Lo in \cite[Prop.~2.18]{GJL} using \cite[Lem.~5.12]{Lo}.
\end{remark}

\begin{remark}
In the proof of Proposition \ref{prop:infinitesimallifting} we only used the fact that $T_0(F)=0$ and $Q\in \Coh_{\leq 1}(X)$, and the statement can be generalized as follows.
Let $\cPerf(X)_{\O_X}$ be the moduli stack of universally gluable perfect complexes $I\udot$ on $X$ with $\det(I\udot)\cong \O_X$ which is an algebraic stack locally finite type over $\mathbb{C}$ \cite[Thm. 4.2.1]{Lie}. By the proof of Lemma \ref{lem:openness} we have an open substack
\[\cPair^{[0,1]}(X):= \Big\{(F,s) \in \cPair(X) : T_0(F)=0 \and Q \in \Coh_{\leq 1}(X) \Big\}\]
inside $\cPair(X)$. For $q=0,1$, we have 
\[\xymatrix{
\cPair^{(q)}(X) \ar@{^{(}->}[r] \ar@{^{(}->}[rd] & \cPair^{[0,1]}(X) \ar@{^{(}->}[r] \ar[d]_\epsilon & \cPair(X)\,, \\
& \cPerf(X)_{\O_X}
}\]
where the horizontal morphisms are open embedding and the morphism $\epsilon$ is formally smooth by Proposition \ref{prop:infinitesimallifting}. 
We claim that the map 
\[\epsilon : \cPair^{[0,1]}(X) \to \cPerf(X)_{\O_X}\] 
is an open embedding.
Indeed, it suffices to show that the induced map $\epsilon(\C)$ on the groupoids of $\C$-valued points is fully faithful. By the same argument as for Lemma \ref{lem:injectivity'}, the induced map $\pi_0(\epsilon(\C))$ on the sets of isomorphism classes is injective. By \eqref{eq:S2.dt3} (in fact by \eqref{eqn:4.6.1}), we have an exact sequence
\[0\to \Ext_X^{-1}(I\udot,F)\to \Hom_X(I\udot,I\udot)_0\to \Ext^1_X(F,\O_X)=0\]
for $(F,s)\in\cPair^{[0,1]}(X)$ and hence $\epsilon(\C)$ preserves stabilizers. This proves the claim.
\end{remark}

\subsection{Polynomial Bridgeland stability}
In this subsection, we prove that the $\PT_q$-stability conditions can be realized as Bayer's polynomial Bridgeland stability conditions \cite{Bri07,Bay} on the derived category. 
The results of this subsection, which are inspired by \cite{GJL}, are not used elsewhere in the paper. We use a different torsion pair compared to \cite{GJL}, which allows us to incorporate 2-dimensional (Gieseker) stable sheaves as well.

Let $X$ be a smooth projective variety of dimension $n \geq 4$. Let $D^b _{\mathrm{coh}}(X)$ be the bounded derived category of coherent sheaves on $X$.
We consider the perversity function
$$
p : \{0,1,\ldots, n\} \to \Z, \quad p(0)=p(1)=p(2)=0, \quad p(3)=\ldots =p(n)=-1.
$$
The perversity function defines a bounded $t$-structure on $D^b_{\mathrm{coh}}(X)$ \cite{AB} and this $t$-structure coincides with the one associated to the torsion pair
$$
(\Coh_{\leq2} (X), \Coh_{\geq3} (X)),
$$
\cite[Cor.~2.2]{HRS}. We denote the heart by
$$
\cA := \cA^p := \langle \, \Coh_{\leq2} (X), \Coh_{\geq3} (X) [1]\, \rangle\,.
$$
The abelian category $\cA$ consists of the objects $E\udot \in D^{[-1,0]}_{\mathrm{coh}}(X)$ satisfying
\[h^{-1}(E^\mdot) \in \Coh_{\geq3}(X) \quad \textrm{and} \quad  h^0(E^\mdot) \in \Coh_{\leq 2}(X).\]
A sequence
$0\to A\udot\to E\udot\to B\udot \to 0$  is a short exact sequence in $\cA$ if and only if the associated sequence in $D^b_{\mathrm{coh}}(X)$ is a distinguished triangle. For any $E\udot\in \cA$, the following sequence is exact in $\cA$
\[0\to h^{-1}(E\udot)[1]\to E\udot\to h^0(E\udot)\to 0.\]

For any $\PT_q$ pair $I^\mdot$ on $X$, we have $I^\mdot[1] \in \cA$. Throughout this subsection, we fix $v\in H^*(X,\mathbb{Q})$ such that
$v_{\leq n-3} = 0$ and $v_{n-2} \neq 0$.


Denote the semi-closed upper half plane by
$$
\mathbb{H} = \{ z \in \C \, : \, z \in \R_{>0} \cdot e^{i \pi \phi(z)} \quad 0 < \phi(z) \leq 1 \},
$$
where $\phi(z)$ denotes the phase of $z$. Denote by $\mathbb{H}^\circ \subset \mathbb{H}$ the open upper half plane. Recall \cite{Bay} that a Bayer stability condition (also known as a polynomial Bridgeland stability condition) is determined by a {\em stability vector}, i.e., a sequence
\[\rho = (\rho_0,\ldots , \rho_n)\in (\C^*)^{n+1}\]
such that $\rho_d/\rho_{d+1} \in \mathbb{H}^\circ$ for $0\leq d \leq n-1$ and $(-1)^{p(d)}\rho_d\in\mathbb{H}$ for all $d$. We fix a polarization $H$ of $X$.
Define the central charge $Z\colon K^0(X)\to \C [x]$ as
\[Z(E^\mdot) (x) = \sum_{d=0}^n  x^d  \rho_d \int_X H^d \ch(E^\mdot)\td(X),\]
for any $E^\mdot \in K^0(X)$. 
For $E\udot\in\cA$, let $\phi$ be the polynomial phase function, i.e.,
\[Z(E^\mdot)(x) \in \R_{>0} \cdot e^{\pi i \phi(E^\mdot)(x)}, \quad x\gg 0,\]
for some continuous function germ $\phi \in (\phi_0,\phi_0+1]$ for some $\phi_0$. We say $\phi<\psi$ if and only if $\phi(x)<\psi(x)$ for all $0\ll x<+\infty$, and similarly for $\phi \leq \psi$.

An object $E\udot$ in $\cA$ is called {\em Bayer stable} (resp.~{\em Bayer semi-stable}) with respect to the stability vector $\rho$ if for all proper non-zero subobjects $D\udot\subset E\udot$, we have $\phi(D\udot)< \phi(E\udot)$ (resp. $\phi(D\udot)\leq \phi(E\udot)$).
\begin{definition}
Consider stability vectors $\rho^{(q)}$ for $q \in \{-1,0,1\}$ satisfying the following:
\begin{figure}[htb]
    \begin{equation}\label{Eq.Bayerstabilityvectors}
    {\renewcommand{\arraystretch}{1}
    \begin{tabular}{c|c}
      $\rho^{(-1)}$  & $\phi(-\rho_n)>\phi(\rho_0) >\phi(\rho_1) >\phi(\rho_2)$ \\
      \hline $\rho^{(0)}$  & $\phi(\rho_0)>\phi(-\rho_n) >\phi(\rho_1) >\phi(\rho_2)$\\
      \hline $\rho^{(1)}$ & $\phi(\rho_0)>\phi(\rho_1)>\phi(-\rho_n) >\phi(\rho_2)$.
    \end{tabular} }        
    \end{equation}
\end{figure}   	
\end{definition}
From the condition on the stability vector, we have 
\begin{equation} \label{eqn:Bayerstabaux}
\phi(-\rho_3),\ldots, \phi(-\rho_{n-1}) > \phi(-\rho_n).
\end{equation}
The following proposition is a minor modification of \cite[Prop.~6.1.1]{Bay}.
\begin{proposition} \label{prop:PTq=Bayerstab}
Fix $v \in H^*(X,\Q)$ such that $v_{\leq n-3}=0$ and $v_{n-2}\neq0$.
Let $E^\mdot \in D^b_{\mathrm{coh}}(X)$ with $\ch(E^{\mdot}) = v-1$ and $\det(E^\mdot) \cong \O_X$. Let $q\in\{-1,0,1\}$. Then  $E^\mdot$ is a $\rho^{(q)}$-semistable object of $\cA$ if and only if $E^\mdot \cong I^\mdot[1]$ for some $\PT_q$ pair $(F,s)$ where $I^\mdot=[\O_X \xrightarrow{s}F]$ is the associated complex. Moreover, $E^\mdot$ cannot be strictly $\rho^{(q)}$-semistable.
\end{proposition}
\begin{proof}
Let $E\udot$ be a $\rho^{(q)}$-semistable object of $\cA$ satisfying the above conditions. We first show that $h^{-1}(E\udot)$ is a rank one torsion free sheaf. Since $h^{0}(E\udot)\in \Coh_{\leq2}(X)$,  $h^{-1}(E\udot)$ has rank $1$ and $h^{-1}(E\udot)\in \Coh_{\geq 3}(X)$. Suppose $h^{-1}(E\udot)$ has a non-trivial torsion subsheaf. Let $T\subset h^{-1}(E\udot)$ be the maximal torsion subsheaf. Then $0\to T[1]\to h^{-1}(E\udot)[1]\to (h^{-1}(E\udot)/T)[1]\to 0$ is a short exact sequence in $\cA$ and hence $T[1]$ is a subobject of $E\udot$ in $\cA$. Then \eqref{eqn:Bayerstabaux} leads to a contradiction. Therefore $h^{-1}(E\udot)$ is a rank one torsion free sheaf with trivial determinant. 
Next we show that $h^0(E\udot)\in \Coh_{\leq 1}(X)$. By assumption $h^0(E\udot)$ has dimension at most two. Because $E\udot$ is $\rho^{(q)}$-semistable, $h^0(E\udot)$ has dimension at most one, otherwise, $E\udot\twoheadrightarrow h^0(E\udot)$ destabilizes $E\udot$. 
By a standard argument (see the last two paragraphs of the proof of Proposition \ref{prop:infinitesimallifting}), we have $E\udot \cong I\udot[1]$, $I\udot=[\O_X\to F]$ for some 2-dimensional coherent sheaf $F$.
For $q=1$, we show that $F$ is pure.  Let $T\subset F$ be the maximal torsion subsheaf of $F$. Then $0\to T\to I\udot[1]\to [\O_X\to F/T][1]\to 0$ is a short exact sequence in $\cA$ and hence $T$ is a subobject of $E\udot$. 
By the stability vector \eqref{Eq.Bayerstabilityvectors}, we have $T=0$. For $q=0$, a similar argument shows that $F$ has no non-trivial 0-dimensional subsheaf and $h^0(E^\mdot)$ is 0-dimensional. For $q=-1$ we have $h^0(E^\mdot) = 0$.

Conversely, suppose $E^\mdot \cong I^\mdot[1]$ for some $\PT_q$ pair $(F,s)$ where $I^\mdot=[\O_X \xrightarrow{s}F]$. Then $E\udot\in\cA$. We prove $E\udot$ is a $\rho^{(q)}$-stable object. 
Take any short exact sequence
$$
0 \to A\udot \to E\udot \to B\udot \to 0
$$
in $\cA$ 
with $A^\mdot, B^\mdot \neq 0$.
Consider its long exact sequence
$$
0 \to h^{-1}(A^\mdot) \to h^{-1}(E^\mdot) \to h^{-1}(B^\mdot) \to h^{0}(A^\mdot) \to h^{0}(E^\mdot) \to h^{0}(B^\mdot) \to 0.
$$
We distinguish three cases: \\

\noindent \textbf{Case 1:} $0 \neq h^{-1}(A^\mdot) \subsetneq h^{-1}(E^\mdot)$. Then $h^{-1}(B^\mdot)$ has rank 0. If $h^{-1}(B^\mdot) \neq 0$, then it has dimension in $\{3, \ldots, n-1\}$ so we obtain 
$\phi(E^\mdot) < \phi(B^\mdot)$ from \eqref{eqn:Bayerstabaux}. Suppose $h^{-1}(B^\mdot) = 0$, then $B^\mdot \cong h^0(B^\mdot)$ and since $h^0(E\udot)$ is at most 1-dimensional, $h^0(B\udot)$ has dimension $\leq 1$. For $q=1$, \eqref{Eq.Bayerstabilityvectors} implies that $\phi(E^\mdot) < \phi(B^\mdot)$. For $q=0$, we have $h^0(B^\mdot)$ is 0-dimensional so we also find $\phi(E^\mdot) < \phi(B^\mdot)$. Finally, for $q=-1$ we have $h^0(B^\mdot) = 0$ contradicting the assumption $B^\mdot \neq 0$. \\

\noindent \textbf{Case 2:} $h^{-1}(A^\mdot) = h^{-1}(E^\mdot)$. Then $h^{-1}(B^\mdot) \subset h^0(A^\mdot)$ so $h^{-1}(B^\mdot) = 0$ and $B^\mdot \cong h^0(B^\mdot)$. Then $\phi(E^\mdot) < \phi(B^\mdot)$ follows as in Case 1. \\

\noindent \textbf{Case 3:} $h^{-1}(A\udot) = 0$. Then $A\udot\cong h^0(A\udot)$ has dimension $\leq 2$ and hence $\Hom(A^\mdot, E^\mdot) \cong \Hom(A^\mdot,F)$. Since $A^\mdot$ is a non-zero subobject of $E^\mdot$ and $(F,s)$ is a $\PT_q$ pair, we deduce that $A^\mdot \cong h^0(A^\mdot)$ is 2-dimensional. Then \eqref{Eq.Bayerstabilityvectors} implies $\phi(A^\mdot) < \phi(E^\mdot)$.
\end{proof}


We characterize rank zero $\rho^{(q)}$-semistable objects in $\cA$.
\begin{proposition}
Let $q \in \{-1,0,1\}$ and $v \in H^*(X,\Q)$ such that $v_{\leq n-3}=0$ and $v_{n-2}\neq0$.
An object $E\udot \in \cA$ with $\ch(E\udot) = v$ is $\rho^{(q)}$-stable (resp.~$\rho^{(q)}$-semistable) if and only if $E\udot \cong F$ for some Gieseker stable (resp.~semistable) sheaf $F$ on $X$ with $\ch(F)=v$.
\end{proposition}
\begin{proof}
Let $E\udot\in\cA$ be a $\rho^{(q)}$-(semi)stable object with $\ch(E\udot) = v$. Since $h^{-1}(E\udot)\in\Coh_{\geq 3}(X)$, we have $h^{-1}(E\udot)=0$ and hence $E\udot\cong h^0(E\udot)$. Then $h^0(E\udot)$ is Gieseker (semi)stable \cite[Sect.~2.1]{Bay}. Conversely, it is easy to show that any Gieseker (semi)stable sheaf $F$ on $X$ with $\ch(F)=v$ determines a $\rho^{(q)}$-(semi)stable object in $\cA$ \cite[Sect.~2.1]{Bay}.
\end{proof}

\subsection{Relative version}\label{ss:relativemoduli}

We now generalize the main results in this section to the relative setting.

We first fix some notation.
Let
$f : \cX \to \cB$
be a smooth projective morphism of schemes of relative dimension $n$. 
Consider the {\em algebraic de Rham complex} 
\[\Omega^\bullet_{\cX/\cB}:=\left[\O_\cX \xrightarrow{d_{DR}} \Omega^1_{\cX/\cB} \xrightarrow{d_{DR}} \Omega^2_{\cX/\cB} \xrightarrow{d_{DR}} \cdots \xrightarrow{d_{DR}} \Omega^n_{\cX/\cB}\right],\]
as in \cite{Gr66,D68}. Define the {\em Hodge bundles} as
\begin{equation}\label{Eq:HdgHbdl}
\cH^k_{DR}(\cX/\cB):=R^kf_*\left(\Omega_{\cX/\cB}^\bullet\right).	
\end{equation}
When $\cB$ is an affine scheme, we write $H_{DR}^k(\cX/\cB):=\Gamma(\cB,\cH^k_{DR}(\cX/\cB)).$

Note that there is a a natural spectral sequence
\[E_1^{pq}=R^qf_*\Omega^p_{\cX/\cB} \Rightarrow \cH^{p+q}_{DR}(\cX/\cB)\]
given by the stupid truncations
\[\Omega^{\geq p}_{\cX/\cB}:= \left[0 \to \cdots \to 0 \to \Omega^p_{\cX/\cB} \xrightarrow{d_{DR}} \Omega^{p+1}_{\cX/\cB} \xrightarrow{d_{DR}}  \cdots \xrightarrow{d_{DR}} \Omega^n_{\cX/\cB}\right].\]
By \cite{D68}, the spectral sequence degenerates at $E_1$. Hence we have a filtration 
\[F^p\cH^{p+q}_{DR}(\cX/\cB):=\im\left(R^{p+q}f_*\big(\Omega^{\geq p}_{\cX/\cB}\big) \to R^{p+q}f_*\left(\Omega^{\bullet}_{\cX/\cB}\right)\right)\]
of the Hodge bundles $\cH^{p+q}_{DR}(\cX/\cB)$ such that
\begin{equation}\label{Eq:HdgFilt}
F^p\cH^{p+q}_{DR}(\cX/\cB)\cong R^{p+q}f_*\left(\Omega^{\geq p}_{\cX/\cB}\right)  \and \frac{F^p\cH^{p+q}_{DR}(\cX/\cB)}{F^{p+1}\cH^{p+q}_{DR}(\cX/\cB)} \cong R^qf_*\Omega^p_{\cX/\cB}.	
\end{equation}
Moreover, the Hodge bundles $\cH^{p+q}_{DR}(\cX/\cB)$, the filtrations $F^p\cH^{p+q}_{DR}(\cX/\cB)$, and the quotients $R^qf_*\Omega^p_{\cX/\cB}$ are vector bundles and stable under base change, by \cite[Thm.~5.5]{D68}.

Consider the {\em Gauss-Manin connection}
\[\nabla : \cH^k_{DR}(\cX/\cB) \to \cH^k_{DR}(\cX/\cB) \otimes \Omega_\cB^1.\]
As in \cite{KO68}, the Gauss-Manin connection $\nabla$ can be defined via the canonical filtration of $\bigwedge^p \Omega^1_\cX$ given by the short exact sequence
\[\xymatrix{0 \ar[r] & f^*\Omega_\cB^1 \ar[r] & \Omega_\cX^1 \ar[r] & \Omega_{\cX/\cB}^1 \ar[r] & 0}.\]
Indeed, the first two terms of the filtration give us a map 
\[\Omega^p_{\cX/\cB} \to \Omega^{p-1}_{\cX/\cB} \otimes f^*\Omega^1_{\cB}[1].\]
Then by gathering the above maps for all $p$ and taking the derived pushforward $R^qf_*$, we obtain a flat connection $\nabla$ called the Gauss-Manin connection.

Following \cite{Blo}, we define the notion of {\em horizontal sections} of Hodge bundles.

\begin{definition}[{cf. \cite[Rem.~3.9]{Blo}}]\label{Def:horizontalsection}
Let $f:\cX\to\cB$ be a smooth projective morphism of schemes of constant relative dimension.
We say that a section $\tv \in \Gamma(\cB,\cH^k_{DR}(\cX/\cB))$ is {\em horizontal} if
\[\nabla(\tv) =0 \in \Gamma(\cB, \cH^k_{DR}(\cX/\cB)\otimes \Omega^1_{\cB}).\]
\end{definition}

Roughly speaking, the horizontal sections $\tv$ of $\cH^k_{DR}(\cX/\cB)$ are locally constant families of cohomology classes $\tv_b \in H^*(\cX_b,\C)$ on the fibres $\cX_b:=\cX\times_\cB\{b\}$. The following remark justifies this interpretation.


\begin{remark}\label{Rem:horizontal}
Let $f:\cX\to\cB$ be a smooth projective morphism of schemes of constant relative dimension.
\begin{enumerate}
\item By considering the associated analytic spaces, we have
\[\cH_{DR}^k(\cX/\cB)^\an \cong R^kf^\an_*\C \otimes_{\C} \O_{\cB^\an} \and \nabla^\an \cong \id \otimes d_{\cB^\an}\]
when $\cB$ is smooth \cite{KO68}.
Thus a section $\tv$ of $\cH^k_{DR}(\cX/\cB)$ is horizontal if and only if
$\tv^\an = v\otimes 1 $
for some $v \in \Gamma(\cB, R^k f^\an_*\C)$.
\item 
\begin{enumerate}
\item Firstly, assume that $\cB= \Spec(A)$ for a local Artinian $\C$-algebra $A$. Then we have isomorphisms
\[H^k_{DR}(\cX/\cB) \cong H^k_{DR}(X)\otimes_\C A\]
by \cite[Prop.~3.8]{Blo}, where $X$ is the fibre of $f:\cX\to\cB$ over the unique closed point and $H^k_{DR}(X):=H^k_{DR}(X/\Spec(\C))$.
Then the horizontal sections are exactly the sections of the form
\[v\otimes 1 \in H^k_{DR}(X)\otimes_\C A\]
for some $v \in H^k_{DR}(X) \cong H^k(X,\C)$.
\item Secondly, for general $\cB$, a section $\tv \in \Gamma(\cB,\cH^k_{DR}(\cX/\cB))$ is horizontal if and only if its pullbacks 
\[\tv|_{\cA} \in \Gamma(\cA, \cH^k_{DR}(\cX_\cA/\cA))\]
are horizontal for all morphisms $\cA=\Spec(A) \to \cB$ from local Artinian $\C$-algebras $A$. Here $\cX_\cA:=\cX\times_\cB \cA$.
This can be shown by the Krull intersection theorem.
\end{enumerate}
\end{enumerate}
\end{remark}


We have a relative moduli space of $\PT_q$ pairs on the fibres of $f: \cX \to \cB$.

\begin{theorem}\label{Thm:RelativeModuli}
Let $\cB$ be a quasi-projective scheme and $f: \cX \to \cB$ a smooth projective morphism of relative dimension $n$. Let $\tv \in \bigoplus_p \Gamma(\cB,\cH^{2p}_{DR}(\cX/\cB))$ be a horizontal section
and $q\in \{-1,0,1\}$. 
Let $\cPair^{(q)}(\cX/\cB,\tv)$ be the moduli stack consisting of triples $(b,F,s)$ of a morphism $b: T \to \cB$, a $T$-flat coherent sheaf $F$ on $\cX_T:=\cX\times_{\cB} T$, and a section $s\in \Gamma(\cX_T,F)$ such that for all $t \in T(\C)$ the pair $(F_t,s_t)$ is a $\PT_q$ pair with $\ch(F_t)=\tv_t \in H^*(\cX_t,\C)$. Then we have the following properties:
\begin{enumerate}
\item[$(1)$] The canonical map
\[\cPair^{(q)}(\cX/\cB,\tv) \to \cB, \quad (b,F,s) \mapsto b\]
is a projective morphism of schemes.
\item[$(2)$] If $n \geq 4$, then the canonical map
\[\cPair^{(q)}(\cX/\cB,\tv) \to \cPerf(\cX/\cB)_{\O_{\cX}}^\spl, \quad (b,F,s) \mapsto (b, [\O_{\cX_T} \xrightarrow{s} F])\]
is an open embedding.
\end{enumerate}
\end{theorem}

\begin{proof}
Fix a polynomial $P(t) \in \Q[t]$ and let $\cPair^{(q)}(\cX/\cB,P)$ be the moduli stack of $\PT_q$ pairs on the fibres of $f:\cX\to\cB$ with Hilbert polynomial $P(t)$ with respect to some relative very ample line bundle.
Note that we can find a closed embedding $\cX \hookrightarrow \PP^N_{\cB}$ over $\cB$ since $f:\cX\to\cB$ is projective and $\cB$ is quasi-projective.
Consider the canonical maps
\[\cPair^{(q)}(\cX/\cB,P) \hookrightarrow \cPair^{(q)}(\PP^N_{\cB}/\cB,P) \xrightarrow{\cong}\cPair^{(q)}(\PP^N,P)\times \cB\]
where the first map is a closed embedding, the second map is an isomorphism, and the last term is a projective scheme over $\cB$ by Theorem \ref{Thm:GIT}.
Indeed, as a functor over $\cPair^{(q)}(\PP^N_{\cB}/\cB,P)$, $\cPair^{(q)}(\cX/\cB,P)$ is the zero locus of the canonical map
\[\I_{\cX/\PP^N_{\cB}} \otimes \FF \to \FF\]
where $\I_{\cX/\PP^N_{\cB}}$ is the ideal sheaf of $\cX\subseteq \PP^N_{\cB}$ and $(\FF,s)$ is the universal pair on $\cPair^{(q)}(\PP^N_{\cB}/\cB,P)\times_{\cB} \PP^N_{\cB}$.
By \cite[Cor.~7.7.8]{EGA3}, the map $\I_{\cX/\PP^N_{\cB}} \otimes \FF \to \FF$ can be represented by the abelian cone of a coherent sheaf on $\cPair^{(q)}(\PP^N_{\cB}/\cB,P)$ since $\FF$ is flat.
Hence the desired map is a closed embedding since the abelian cone is separated over the base. Therefore we have (1) since $\cPair^{(q)}(\cX/\cB,\tv)$ is an open and closed subscheme of the projective scheme $\cPair^{(q)}(\cX/\cB,P)$ (for appropriate $P$). This proves (1).

All the arguments in subsection \ref{ss:openembeddingtoPerf} also works in the relative setting. This proves (2). 
\end{proof}


We end this subsection with a technical proposition stating that the relative Chern character of a perfect complex is horizontal. This fact will be used several times in section \ref{sec:VFC} and \ref{sec:deformationinvariance}.

Let $f:\cX \to\cB$ be a smooth projective morphism of schemes of constant relative dimension.
Let $I\udot$ be a perfect complex on $\cX$. Let
\[\At_{\cX/\cB}(I\udot) : I\udot \to I\udot \otimes \Omega^1_{\cX/\cB}[1]\]
be the (relative) Atiyah class of $I\udot$.
Let
\[\At_{\cX/\cB}^p(I\udot) :=\At_{\cX/\cB}(I\udot)\wedge\cdots\wedge\At_{\cX/\cB}(I\udot) : I\udot \to I\udot \otimes \Omega_{\cX/\cB}^p[p]\]
be the induced map. 
We define the (relative) {\em Chern character} $\ch_p(I\udot)$ of $I\udot$ as
\begin{equation}\label{Eq:ch}
H^p(\cX,\Omega^p_{\cX/\cB}) \to \Gamma(\cB,R^pf_*\Omega^p_{\cX/\cB}), \quad \frac{1}{p!}\tr(\At_{\cX/\cB}^p(I\udot)) \mapsto \ch_p(I\udot).
\end{equation}

The Chern character $\ch_p(I\udot)$ in \eqref{Eq:ch} is horizontal. More precisely, we have the following:

\begin{proposition}\label{lem:chishorizontal}
Let $f:\cX\to\cB$ be a smooth projective morphism of schemes of constant relative dimension.
Let $I\udot$ be a perfect complex on $\cX$.
Then there exists a section
\[\widetilde{\ch_p}(I\udot) \in \Gamma(\cB,F^p\cH_{DR}^{2p}(\cX/\cB))\]
satisfying the following properties:
\begin{enumerate}
\item[$(1)$] Under the projection map 
$F^p\cH^{2p}_{DR}(\cX/\cB) \twoheadrightarrow R^pf_*\Omega^p_{\cX/\cB}$ given by \eqref{Eq:HdgFilt},
the section $\widetilde{\ch_p}(I\udot)$ maps to the Chern character $\ch_p(I\udot)$ in \eqref{Eq:ch}.
\item[$(2)$] Under the inclusion map $F^p\cH_{DR}^{2p}(\cX/\cB) \hookrightarrow \cH^{2p}_{DR}(\cX/\cB)$ given by \eqref{Eq:HdgFilt}, 
the section $\widetilde{\ch_p}(I\udot)$ maps to a horizontal section.
\item[$(3)$] Under the restriction map $\Gamma(\cB,\cH_{DR}^{2p}(\cX/\cB)) \to H_{DR}^{2p}(\cX_b/\Spec(\C))$ for any $\C$-point $b \in \cB$, the section $\widetilde{\ch_p}(I\udot)$ maps to the usual Chern character $\ch_p(I\udot_b) \in H^{2p}(\cX_b,\C)$.
\end{enumerate}
The section $\widetilde{\ch_p}(I\udot)$ is uniquely determined by properties $(2)$ and $(3)$.
\end{proposition}

\begin{remark}
We note that an analogous statement for the fundamental class $[\cZ]$ of a local complete intersection $\cZ \subseteq \cX$ was shown in \cite[Prop.~5.6]{Blo}.
\end{remark}

We briefly summarize how we prove Proposition \ref{lem:chishorizontal}.
We use the theory of Hochschild homology and cyclic homology, see Loday \cite{Lod} for the case of rings and Weibel \cite{W3,W4} for the case of schemes.\footnote{By Keller \cite{Kel}, the cyclic homologies defined in \cite{W4} coincide with the cyclic homologies defined in McCarthy \cite{McC} via the dg category of perfect complexes.}
By the Hochschild-Kostant-Rosenberg (HKR) theorem \cite{HKR}, we may identify the Hodge bundles, the filtrations, and the associated quotients in \eqref{Eq:HdgHbdl} and \eqref{Eq:HdgFilt} with the periodic cyclic homology, the negative cyclic homology, and the Hochschild homology, respectively. Moreover, the Chern character in \eqref{Eq:ch} can be identified with the Dennis trace map. 
Then the key to proving Proposition \ref{lem:chishorizontal} is Goodwillie's lift of the Dennis trace map to the negative cyclic homology \cite{Goo2}.\footnote{In the language of \cite{PTVV}, Chern characters lift to $p$-shifted {\em closed} $p$-forms \cite{TV11,TV15}. By \cite{HSSS}, the Chern character of \cite{TV15} matches with the classical Dennis trace.}

The proof was inspired by Pridham \cite{Pri}.

\begin{proof}[Proof of Proposition \ref{lem:chishorizontal}]

Note that the uniqueness of the section $\widetilde{\ch_p}(I\udot)$ follows immediately from the properties (2) and (3) in Proposition \ref{lem:chishorizontal} and Remark \ref{Rem:horizontal}(2). Hence it suffices to consider the statement locally on the base $\cB$. Thus we may assume that $\cB$ is an affine scheme. Then \eqref{Eq:ch} is an isomorphism.

By the HKR theorem in \cite[Thm.~3.3]{W4}, we have natural isomorphisms
\begin{align}
\HN^i(\cX/\cB) &\cong \bigoplus_p F^pH_{DR}^{2p+i}(\cX/\cB),
& \HochP^i(\cX/\cB) &\cong \bigoplus_p H^{2p+i}_{DR}(\cX/\cB), \label{Eq:39} \\ 
\HochH^i(\cX/\cB) &\cong \bigoplus_p R^{p+i}f_*(\Omega^p_{\cX/\cB}),
& \HC^i(\cX/\cB) &\cong \bigoplus_p \frac{H^{2p+i}_{DR}(\cX/\cB)}{F^{p+1}H^{2p+i}_{DR}(\cX/\cB)},\nonumber
\end{align}
where $\HN$, $\HochP$, $\HochH$, and $\HC$ denote the negative cyclic homology, the periodic cyclic homology, the Hochschild homology, and the cyclic homology, respectively, with the cohomological convention. More specifically, the isomorphisms \eqref{Eq:39} are induced by the left inverse ($\pi$ in \cite[Lem.~1.3.14]{Lod} and $e$ in \cite[Cor.~9.4.4]{Wei}) of the antisymmetrization map.

The relative Chern character $\ch_{\cX/\cB}(I\udot):=(\ch_p(I\udot))_p$ in \eqref{Eq:ch} comes from the absolute Chern character. More precisely, we have a commutative diagram
\begin{equation}\label{Eq:38}
\xymatrix@C+1pc{
\O_{\cX} \ar[r]^-{\ch_{\cX}(I\udot)} \ar@{=}[d]  & \bigoplus_p \LL^p_{\cX}[p] \ar[d] \\
\O_{\cX} \ar[r]_-{\ch_{\cX/\cB}(I\udot)}& \bigoplus_p \Omega^p_{\cX/\cB}[p]
}	
\end{equation}
in the derived category of $\cX$, where $\LL_{\cX}$ denotes Illusie's cotangent complex \cite{Ill} of $\cX$, $\LL^p_{\cX}:= \bigwedge^p \LL_{\cX}:=\Sym^p(\LL_{\cX}[1]])[-p]$ denotes the derived wedge product of $\LL_{\cX}$, and $\ch_{\cX}(I\udot)$ denotes the absolute Chern character defined as the trace of the exponential of the absolute Atiyah class $\At_{\cX}(I\udot) : I\udot \to I\udot \otimes \LL_{\cX}[1]$ \cite{Ill}.

Under the HKR ismorphism, the Chern character $\ch_{\cX}(I\udot)$ can be identified with the Dennis trace $\ch^D(I\udot)$. 
More precisely, we have a commutative diagram
\begin{equation}\label{Eq:31}
\xymatrix{
&\O_{\cX} \ar[ld]_{\ch^D(I\udot)} \ar[rd]^{\ch_{\cX}(I\udot)} & \\
L\tDelta^*\tDelta_*\O_{\cX} \ar[rr]_{\cong}\ar[d] & & \bigoplus_p \LL^p_{\cX}[p] \ar[d]\\
L\Delta^*\Delta_*\O_{\cX}  \ar[rr]_{\cong}^{\mathrm{HKR}} & & \bigoplus_p \Omega^p_{\cX/\cB}[p]
}	
\end{equation}
in the derived category of $\cX$ where $\tDelta : \cX \to \cX\times\cX$ and $\Delta : \cX \to \cX\times_{\cB}\cX$ are the diagonal maps.
The horizontal isomorphisms are given by the exponentials of the universal Atiyah classes as in Buchweitz-Flenner \cite[Thm.~6.2.1]{BF08b} (see Markarian \cite{Mar} and C\u{a}ld\u{a}raru \cite{Cal2} for the smooth case).
The hypercohomology of the left column of \eqref{Eq:31} computes the Hochschild homology,
\[\HH^i(\cX, L\tDelta^*\tDelta_*\O_{\cX}) \cong \HochH^i(\cX),\qquad
\HH^i(\cX, L\Delta^*\Delta_*\O_{\cX}) \cong \HochH^i(\cX/\cB)\]
by Yekutieli \cite[Prop.~3.3]{Yek}.
The lower horizontal isomorphism in \eqref{Eq:31} 
is compatible with the corresponding isomorphism in \eqref{Eq:39},
by the arguments in \cite[Prop.~4.4]{Cal2} since $\cX \to\cB$ is smooth.
The map $\ch^D(I\udot)$ in \eqref{Eq:31} is the Euler map of $I\udot$ in the sense of \cite[4.3]{BNT}, which agrees with the Dennis trace map for the dg category $\mathrm{Perf}(\cX)$ \cite{McC}, as desribed in \cite[4.5,~4.7]{BNT}.


As explained in \cite[3.8]{W4}, the Dennis trace map $\ch^D$ can be extended to the Goodwillie Chern character $\ch^G$ \cite{Goo2} via descent \cite{WG},
\begin{equation}\label{Eq:34}
\xymatrix{
& K^0(\cX) \ar[ld]_{\ch^D} \ar[d]^{\ch^G} \ar[rd] & \\
\HochH^0(\cX) & \ar[l] \HN^0(\cX) \ar[r] & \HochP^0(\cX). \\ 
}\end{equation}
Consider the canonical commutative diagram
\begin{equation}\label{Eq:35}
\xymatrix{
\HochH^0(\cX) \ar[d] & \HN^0(\cX) \ar[l] \ar[r] \ar[d] & \HochP^0(\cX) \ar[d] \\
\HochH^0(\cX/\cB) \ar[d]_{\cong}^{\mathrm{HKR}} & \HN^0(\cX/\cB) \ar@{->>}[l] \ar@{^{(}->}[r] \ar[d]_{\cong}^{\mathrm{HKR}}& \HochP^0(\cX/\cB) \ar[d]_{\cong}^{\mathrm{HKR}}\\
\bigoplus_p R^pf_* \Omega^p_{\cX/\cB} & \bigoplus_p F^p H^{2p}_{DR}(\cX/\cB) \ar@{->>}[l] \ar@{^{(}->}[r] & \bigoplus_p H^{2p}_{DR}(\cX/\cB).
}	
\end{equation}
Denote by
\begin{equation}\label{Eq:40}
\widetilde{\ch_p}(I\udot) \in F^p H^{2p}_{DR}(\cX/\cB)
\end{equation}
the image of $\ch^G(I\udot) \in HN_0(\cX)$ under the vertical maps in \eqref{Eq:35}. 
We now show that $\widetilde{\ch_p}(I\udot)$ in \eqref{Eq:40} satisfies the three properties in Lemma \ref{lem:chishorizontal} separately.
\begin{enumerate}
\item 
Clearly $\widetilde{\ch_p}(I\udot)$ maps to the Chern character $\ch_p(I\udot)$ in \eqref{Eq:ch}. This follows from the commutative diagrams \eqref{Eq:38}, \eqref{Eq:31}, \eqref{Eq:34}, and \eqref{Eq:35}. 

\item
We now claim that $\widetilde{\ch_p}(I\udot)$ is horizontal. Indeed, we may assume that $\cB=\Spec(A)$ for a local Artinian $\C$-algebra $A$ (Remark \ref{Rem:horizontal}). Let $X:=\cX_0$ be the fibre over the $\C$-point $0 \in \cB$.
By the Feigin-Tsygan theorem \cite[Thm.~5]{FT} (see also \cite[Thm.~2.2]{Emm}, \cite[Thm.~3.4]{W4}), we have a commutative square with bijective horizontal arrows
\begin{equation}\label{Eq:FT1}
\xymatrix{
\HochP^0(\cX) \ar[r]^-{\cong} \ar[d] & \bigoplus_p\HH^{2p}(\widehat{\cY},\hOmega_{\cY}^\bullet) \ar[d]\\
\HochP^0(\cX/\cB) \ar[r]^-{\cong} & \bigoplus_p\HH^{2p}(\cX,\Omega_{\cX/\cB}^\bullet)
}    
\end{equation}
where $\bigoplus_p\HH^{2p}(\widehat{\cY},\hOmega_{\cY}^\bullet) $ is Hartshorne's algebraic de Rham cohomology \cite{Har} given by an embedding $\cX \hookrightarrow \cY$ into a smooth scheme $\cY$.
Here the lower isomorphism in \eqref{Eq:FT1} coincides with the corresponding isomorphism in \eqref{Eq:39} since they both are given by the same left inverse ($\mu$ in \cite{Emm} and $e$ in \cite{W4}) of the antisymmetrization map.
By \cite[Thm.~IV.1.1]{Har} (and \cite[Lem.~5.5.3]{D68}), we have a commutative diagram
\[\qquad\xymatrix@C-1pc{
\HH^{2p}(\hcY,\hOmega_{\cY}^\bullet) \ar[r]^-{\cong} \ar[d] & \HH^{2p}(\hcY^\an,\hOmega_{\cY^\an}^\bullet) \ar[d] & H^{2p}(\cX^\an,\C) \ar[l]_-{\cong} \ar[d] \ar@{=}[r] & H^{2p}(X^\an,\C) \ar[d] \\
\HH^{2p}(\cX,\Omega_{\cX/\cB}^\bullet) \ar[r]^-{\cong} & \HH^{2p}(\cX^\an,\Omega_{\cX^\an/\cB^\an}^\bullet) &  H^{2p}(\cX^\an,A) \ar[l]_-{\cong} \ar@{=}[r] & H^{2p}(X^\an,A)
}\]
where the last vertical map is induced by the canonical map of constant sheaves $\C \to A$ on $X^\an$. 
Hence the image of the vertical map consists of the horizontal sections (see Remark \ref{Rem:horizontal}), which proves the claim.
\item 
We finally claim that the restriction $\widetilde{\ch_p}(I\udot)$ to the fibre of a point $b \in \cB$ is the usual Chern character. Indeed, we may assume that $\cB=\Spec(\C)$ since Chern characters are compatible with base change. Then the claim follows from \cite[Prop.~3.8.1]{W4}. 
\end{enumerate}
This completes the proof.
\end{proof}

\section{Virtual cycles}\label{sec:VFC}

In this section, we construct reduced virtual cycles for the moduli spaces of $\PT_q$ pairs on Calabi-Yau 4-folds. We use Kiem-Li's cosection localization technique \cite{KL13,KP20} to reduce the Oh-Thomas virtual cycles \cite{OT} on the moduli spaces.
The semi-regularity map of Buchweitz-Flenner \cite{BF03} will provide the necessary cosections.

Throughout this section, we fix a (smooth projective) Calabi-Yau 4-fold $X$ and a Calabi-Yau 4-form $\omega : \O_X \xrightarrow{\cong}\Omega^4_X$.

\subsection{Non-reduced virtual cycles}

In this preliminary subsection, we construct {\em non-reduced} virtual cycles for the moduli spaces of stable pairs.
We review Oh-Thomas's construction of virtual cycles in \cite{OT} for the reader's convenience. Instead of following \cite{OT} directly, we will use the setup introduced in \cite{Par}. 

Let $\curP:=\PTqvX$ be the moduli space of $\PT_q$ pairs $(F,s)$ on $X$ satisfying
\[\ch(F)= v = \left(0,0,\gamma,\beta, n-  \gamma \cdot \td_2(X)\right) \in H^*(X,\Q).\]
By Theorem \ref{Thm:PairtoPerf}, we have an open embedding
\[\curP \hookrightarrow \cPerf(X,v)_{\O_X}^{\spl}\]
into the moduli space of perfect complexes with fixed trivial determinant. Thus, by \cite{OT}, we can construct an Oh-Thomas virtual cycle on $\curP$.

\begin{theorem}[\cite{OT}]\label{Thm:NRVFC}
There exists a virtual cycle
\[\left[\curP\right]^\vir \in A_{\vd}\left(\curP\right)\]
of virtual dimension
\[\vd= n-\frac12 \gamma^2,\]
which depends\footnote{A different choice of orientation changes the {\em signs} of the virtual cycles on the connected components.} on the choice of an orientation \eqref{Eq:orientation} below.
\end{theorem}

We call the virtual cycle $[\curP]^\vir$ in Theorem \ref{Thm:NRVFC} the {\em non-reduced virtual cycle}. We use this terminology since we will construct a canonical {\em reduced virtual cycle} $[\curP]^\red$ for any Calabi-Yau 4-fold $X$ and a surface class $\gamma$ in the next subsection (see Theorem \ref{Thm:RVFC}).

There are three key ingredients for defining the Oh-Thomas virtual cycle (in the language of \cite{Par}):

\begin{enumerate}
\item {\bf Symmetric obstruction theory.}
The moduli space $\curP$ has a (3-term) {\em symmetric obstruction theory} 
\begin{equation}\label{Eq:SOT}
\phi : \EE:=R\hom_{\pi}(\II\udot,\II\udot)_0[3] \xrightarrow{\At(\II\udot)}  \trunc\LL_{\curP},\quad \theta: \EE\dual[2] \xrightarrow{\SD} \EE
\end{equation}
induced by the Atiyah class $\At(\II\udot)$ of the universal complex $\II\udot$ and the relative Serre duality for the projection map $\pi: \curP \times X \to \curP$. Here $\phi$ and $\theta$ satisfy the following properties:
\begin{enumerate}
\item $\phi:\EE\to\trunc\LL_{\curP}$ is an {\em obstruction theory}, \cite{BF} i.e., $h^0(\phi)$ is bijective and $h^{-1}(\phi)$ is surjective;
\item $\EE$ is a {\em symmetric complex} (of tor-amplitude $[-2,0]$), i.e., $\theta$ is an isomorphism and $\theta=\theta\dual[2]$.
\end{enumerate}
The symmetric obstruction theory \eqref{Eq:SOT} is constructed by Huybrechts-Thomas \cite{HT}.\footnote{Alternatively, we can also use the derived enhancement of $\cPerf(X)^\spl_{\O_X}$ in \cite{ToVa,STV}.}
\item {\bf Orientation.}
There exists an {\em orientation} for the virtual cotangent complex $\EE$, i.e., an isomorphism of line bundles
\begin{equation}\label{Eq:orientation}
o : \O_{\curP} \xrightarrow{\cong} \det(\EE)	
\end{equation}
such that $o^2=(-1)^{\frac{(2\vd)(2\vd-1)}{2}}\cdot\det(\theta):\det(\EE\dual) \to \det(\EE)$. The existence of an orientation was shown by Cao-Gross-Joyce \cite{CGJ}.
\item {\bf Isotropic condition.} The intrinsic normal cone $\fC_{\curP}$ is {\em isotropic} in the virtual normal cone $\fC(\EE):=h^1/h^0(\EE\dual)$ with respect to the canonical quadratic function $\fq_\EE:\fC(\EE) \to \bbA^1$ (defined in Remark \ref{Rem:quadraticfunction} below), i.e., the restriction
\begin{equation*}\label{Eq:IsotropicCondition}
\fq_{\EE}|_{\fC_{\curP}} : \fC_{\curP} \hookrightarrow \fC(\EE) \xrightarrow{\fq_\EE} \bbA^1_{\curP}
\end{equation*}
vanishes. This isotropic condition can be shown by the Darboux theorem of Brav-Bussi-Joyce \cite{BBJ} for the $(-2)$-shifted symplectic derived enhancement \cite{PTVV,STV}. Indeed, by \cite[Lem.~1.11]{Par}, the arguments in \cite[Prop.~4.3]{OT} proves the isotropic condition.\footnote{One does not need the full proof of  \cite[Prop.~4.3]{OT}, but only the part involving MacPherson's graph construction \cite[Rem.~5.11]{Ful}.}
\end{enumerate}

We recall the definition of the canonical quadratic function $\fq_{\EE}$ from \cite{Par}.

\begin{remark}\label{Rem:quadraticfunction}
Given a symmetric complex $\EE$ on a quasi-projective scheme $\sX$, there is a canonical quadratic function 
\[\fq_\EE:\fC(\EE) \to \bbA^1_{\sX}\] 
induced by the symmetric form of $\EE$. Indeed, if $\EE=[B\to E \to B\dual]$ is a symmetric resolution, then $\fC(\EE)=[C(D)/B]$ for $D=\coker(B\to E)$, and the restriction $\fq_E|_{C(D)}:C(D) \hookrightarrow E \xrightarrow{\fq_E} \bbA^1$ of the quadratic function $\fq_E$ of $E$ is $B$-invariant. The quadratic function $\fq_\EE$ is defined via descent of $\fq_E|_{C(D)}$, which is independent of the choice of the symmetric resolution. See \cite[Prop.~1.7]{Par} for details.
\end{remark}

Note that from the above data, we can form a commutative diagram
\begin{equation}\label{Eq12}
\xymatrix{& \fQ(\EE) \ar[r] \ar@{^{(}->}[d] & \curP \ar@{^{(}->}[d]^{0} \\
\fC_{\curP} \ar@{^{(}->}[r] \ar@{.>}[ru] & \fC(\EE) \ar[r]^{\fq_\EE} & \bbA^1_{\curP}}
\end{equation}
where $\fQ(\EE)$ is the {\em quadratic cone stack} associated to the symmetric complex $\EE$, defined as the zero locus of the quadratic function $\fq_\EE$ on $\fC(\EE)$. 

We can construct the Oh-Thomas virtual cycle $[\curP]^\vir$ from the closed embedding  $\fC_{\curP} \hookrightarrow \fQ(\EE)$ in \eqref{Eq12}.

\begin{proof}[Proof of Theorem \ref{Thm:NRVFC}]
We first construct a {\em square root Gysin pullback}
\[\sqrt{0^!_{\fQ(\EE)}} : A_*(\fQ(\EE)) \to A_*(\curP)\]
for the zero section $0_{\fQ(\EE)} : \curP \to \fQ(\EE)$. Note that we can find a symmetric resolution 
\begin{equation}\label{Eq:SymRes}
\EE\cong \left[B\xrightarrow{d} E \cong E\dual\xrightarrow{d\dual} B\dual\right]
\end{equation}
for some special orthogonal bundle $E$ and a vector bundle $B$. Indeed, this can be shown by the proof of \cite[Prop.~4.1]{OT}, see also \cite[Prop.~1.3]{Par}. 
Here an orientation of $\EE$ induces an orientation of $E$ via the canonical isomorphism
\[\det(\EE) \cong \det(B\dual)\otimes \det(E\dual)\dual\otimes \det(B) \cong \det(E)\]
of line bundles. By considering the stupid truncation
\[\left[ 0 \to  E\dual \xrightarrow{d\dual} B\dual \right] \to \left[B\xrightarrow{d} E \cong E\dual\xrightarrow{d\dual} B\dual\right],\]
we can form a fibre diagram
\begin{equation}\label{Eq11}
\xymatrix{
Q \ar@{^{(}->}[rr] \ar[d]^{r} & 
& E \ar[d] \\
\fQ(\EE) \ar@{^{(}->}[r] & \fC(\EE) \ar@{^{(}->}[r] & [E/B]
}	
\end{equation}
for some cone $Q$. The horizontal maps in \eqref{Eq11} are closed embeddings and the vertical maps in \eqref{Eq11} are $B$-torsors. Then $Q$ is isotropic subcone of $E$ since
\[\fq_{E}|_Q = \fq_{\fC(\EE)}|_Q = 0\]
by \cite[Prop.~1.7(3)]{Par}.
Consider a factorization
\[\xymatrix{
& E|_Q \ar[d]\\
\curP \ar@{^{(}->}[r]^-{0_Q} \ar[rd]_{0_{\fQ(\EE)}} & Q \ar@/_0.4cm/[u]_{\tau} \ar[d]^r\\
& \fQ(\EE)
}\]
of the zero section $0_{\fQ(\EE)}$, where $\tau$ is the tautological section induced by the embedding $Q \hookrightarrow E$ in \eqref{Eq11}. Note that the zero section $0_Q : \curP \hookrightarrow Q$ is the zero locus of the tautological section $\tau \in \Gamma(Q,E|_Q)$. We define the square root Gysin pullback as the composition
\begin{equation}\label{Eq:squarerootGysin}
\sqrt{0^!_{\fQ(\EE)}}:= \sqrt{e}(E|_Q,\tau) \circ r^*: A_* (\fQ(\EE)) \xrightarrow{} A_{*+b} (Q) \xrightarrow{} A_{*+b-\frac12e} (\curP)	
\end{equation}
where $b:=\rk(B)$, $e:=\rk(E)$, $\vd=b -\frac12 e$, and $\sqrt{e}(E|_Q,\tau)$ denotes Oh-Thomas's localization \cite[Def.~3.2]{OT} of Edidin-Graham's square root Euler class \cite{EG1}.\footnote{See also \cite[Def.~4.1]{KP20} for an alternative construction of the localized square root Euler class $\sqrt{e}(E|_C,\tau)$ using the blowup method.}
The arguments in \cite[pp.~31--34]{OT} show that the square root Gysin pullback in \eqref{Eq:squarerootGysin} is independent of the choice of the symmetric resolution \eqref{Eq:SymRes}.

We then define the Oh-Thomas virtual cycle as
\begin{equation}\label{Eq:Def.NRVFC}
\left[\curP\right]^\vir := \sqrt{0^!_{\fQ(\EE)}}\left[\fC_{\curP}\right] \in A_\vd \left(\curP\right)	
\end{equation}
where 
$[\fC_{\curP}] \in A_0 (\fQ(\EE))$
is the fundamental class of the intrinsic normal cone with respect to the closed embedding $\fC_{\curP} \hookrightarrow \fQ(\EE)$ in \eqref{Eq12}. 

We can easily compute the virtual dimension via the Hirzebruch-Riemann-Roch formula \cite[Cor.~15.2.1]{Ful}. Indeed, we have
\begin{align*}
2\cdot\vd &= \rank(R\Hom_X(I\udot,I\udot)_0[1])\\
&= -\int_X \ch(I\udot)\dual\ch(I\udot) \td(X) + \int_X \td(X)\\	
&=-(-2n + 2 \td_2(X)\gamma +\td_4(X) - 2 \td_2(X)\gamma + \gamma^2) + \td_4(X) \\
& = 2n - \gamma^2.
\end{align*}
It completes the proof.
\end{proof}

We provide a simple analogy between the above description \eqref{Eq:Def.NRVFC} of the Oh-Thomas virtual cycles \cite{OT} and the Behrend-Fantechi virtual cycles \cite{BF}.

\begin{remark}\label{Rem4}
Fix an arbitrary (separated) Deligne-Mumford stack $\sX$.

Recall that a perfect obstruction theory $\psi:\KK \to \trunc \LL_{\sX}$ of tor-amplitude $[-1,0]$ is equivalent to a closed embedding of cone stacks\footnote{A morphism of cone stacks is always assumed to be $\bbA^1$-equivariant as in \cite[Def.~1.8]{BF}.} $\fC_{\sX} \hookrightarrow \fC(\KK)$, for some vector bundle stack $\fC(\KK)$. In \cite{Kre}, Kresch constructed a Gysin pullback 
\[0^!_{\fC(\KK)}:A_*(\fC(\KK)) \to A_*(\sX)\]
for any perfect complex $\KK$ of tor-amplitude $[-1,0]$. In \cite{BF,Kre}, the Behrend-Fantechi virtual cycle for a closed embedding $\fC_{\sX} \hookrightarrow \fC(\KK)$ was defined as
\[[\sX]^\vir_{BF} := 0^!_{\fC(\KK)}[\fC_\sX].\]

Analogously,	 a symmetric obstruction theory $\phi:\EE \to \trunc \LL_{\sX}$ of tor-amplitude $[-2,0$] satisfying the isotropic condition is equivalent to a closed embedding of cone stacks $\fC_{\sX} \hookrightarrow \fQ(\EE)$, for some quadratic cone stack $\fQ(\EE)$ associated to a symmetric complex $\EE$. Note that we can define the square root Gysin pullback 
\[\sqrt{0^!_{\fQ(\EE)}} : A_*(\fQ(\EE)) \to A_*(\sX)\]
for any symmetric complex $\EE$ of tor-amplitude $[-2,0]$ and an orientation $\O_{\sX} \xrightarrow{\cong} \det(\EE)$ as in the proof of Theorem \ref{Thm:NRVFC}.\footnote{The proof of Theorem \ref{Thm:NRVFC} works when $\sX$ is a quasi-projective scheme. Using the Kimura sequence for Artin stacks \cite{BP}, this can be generalized to the case when $\sX$ is a separated Deligne-Mumford stack, see \cite[A]{Par}.} Thus the Oh-Thomas virtual cycle for a closed embedding $\fC_{\sX} \hookrightarrow \fQ(\EE)$ and an orientation can be defined as
\[[\sX]^\vir _{OT}:=\sqrt{0^!_{\fQ(\EE)}}[\fC_{\sX}].\]
\end{remark}

\subsection{Reduced virtual cycles}

Recall that the standard curve counting invariants (GW/PT) on K3 surfaces vanish due to the following two reasons:
\begin{enumerate}
\item There exists a surjective cosection \cite{KL13};
\item the codimension of the Noether-Lefschetz locus is positive \cite{KT1,MPT}.
\end{enumerate}
To obtain non-zero invariants, one should consider {\em reduced} virtual cycles. 

An analogous phenomenon occurs when we want to count surfaces on Calabi-Yau 4-folds. Given a surface class $\gamma$, there are as many natural cosections as the codimension $\rho_\gamma$ of the Hodge locus. Thus in a generic situation, the non-reduced virtual cycle in Theorem \ref{Thm:NRVFC} vanishes and we want to construct a {\em reduced} virtual cycle. This can be achieved by the DT4 version of the cosection localization technique developed in \cite{KP20}. In this subsection, we present the main results on the reduced virtual cycles and we provide the proofs in the subsequent subsections.


We first introduce an important bilinear form $\sfB_\gamma$ on $H^1(X,T_X)$ associated to each surface class $\gamma$ on $X$. We also fix some notation.

\begin{definition}
Let $\gamma \in H^2(X,\Omega^2_X)$ be a (2,2)-class on $X$.
\begin{enumerate}
\item We define a symmetric bilinear form $\sfB_\gamma$ on $H^1(X,T_X)$ as
\begin{equation*}\label{Eq:Bgamma}
\mathsf{B}_\gamma : H^1(X,T_X) \otimes H^1(X,T_X) \to \C, \quad \xi_1 \otimes \xi_2 \mapsto \int_X \iota_{\xi_1}\iota_{\xi_2}\gamma \cup \omega	
\end{equation*}
where $\iota_{\xi_1}$ and $\iota_{\xi_2}$ denote the contraction maps and $\omega$ is the (chosen) Calabi-Yau 4-form.\footnote{For $\xi_1\in H^i(X,T_X)$ and $\xi_2\in H^j(X,T_X)$, we have $\iota_{\xi_2}\iota_{\xi_1}=(-1)^{(1+i)(1+j)}\iota_{\xi_1}\iota_{\xi_2}$ (see for example \cite[Lem.~2.5]{BBJ}). }
\item We denote the rank of the bilinear form $\sfB_\gamma$ by
\[\rho_\gamma :=\rank(\sfB_\gamma).\]
\item For the induced non-degenerate quadratic space of dimension $\rho_\gamma$, we write 
\begin{equation*}\label{Eq:H1TXgamma}
H^1(X,T_X)_\gamma:=\frac{H^1(X,T_X)}{\ker(\sfB_\gamma)}.
\end{equation*}
\end{enumerate}
\end{definition}

The bilinear form $\sfB_\gamma$ is closely related to the Hodge theory of the (2,2)-class $$\gamma\in H^2(X,\Omega^2_X) \subseteq H^4(X,\C).$$ 
Essentially, the rank $\rho_\gamma$ is equal to the codimension of the Hodge locus of $\gamma$ \cite{Blo} (see footnote \eqref{footnote:codHdg}).

We now state our main result.
\begin{theorem}\label{Thm:RVFC}
Let $X$ be a Calabi-Yau 4-fold, $v = (0,0,\gamma,\beta,n-\gamma\cdot\td_2(X)) \in H^*(X,\Q)$, and $q \in \{-1,0,1\}$.
Then there exists a canonical reduced virtual cycle
\[\left[\PTqvX\right]^\red \in A_{\rvd}\left(\PTqvX\right)\]
of reduced virtual dimension
\[\mathrm{rvd} = n- \frac12 \gamma^2 + \frac12 \rho_\gamma\]
depending on the choices of orientation of the symmetric complex $\EE$ and the quadratic space $H^1(X,T_X)_\gamma$. Moreover, $[\PTqvX]^\red$ is deformation invariant in the sense of Theorem \ref{Thm:DefInv}. \end{theorem}

The reduced virtual cycle $[\PTqvX]^\red$ 
has a relation to the notion of {\em semi-regularity} of perfect complexes $I\udot$ associated to pairs $[(F,s)] \in  \PTqvX$ (Definition \ref{Def:SR}(2)).

\begin{theorem}\label{Thm:SR=smoothofrvd}
Let $X$ be a Calabi-Yau 4-fold, $v = (0,0,\gamma,\beta,n-\gamma\cdot\td_2(X)) \in H^*(X,\Q)$, and $q \in \{-1,0,1\}$.
For any $\PT_q$ pair $[(F,s)] \in \PTqvX$, the associated perfect complex $I\udot:=[\O_X \xrightarrow{s} F]$ is semi-regular if and only if 
$\PTqvX$ is smooth of dimension $\rvd=n-\frac12\gamma^2+\frac12\rho_\gamma$ at  $[(F,s)]$.
\end{theorem}

By Theorem \ref{Thm:SR=smoothofrvd}, if all point in $\PTqvX$ are semi-regular, then it is smooth and the reduced virtual cycle is equal to the fundamental cycle
\[\left[\PTqvX\right]^\red=\left[\PTqvX\right] \in A_\rvd\left(\PTqvX\right).\]	

\begin{remark}
The role of $\rho_\gamma$ in the  surface counting theory is an interesting new feature compared to known enumerative geometry. Indeed $\rho_\gamma$ is not a topological datum of $\gamma\in H^4(X)$ but rather it is a Hodge theoretic datum of $\gamma\in H^{2,2}(X)$.
In particular rank of $\sfB_\gamma$ is not additive in $\gamma$ in general, but it is only sub-additive, i.e., $\rho_{\gamma_1+\gamma_2}\leq \rho_{\gamma_1}+\rho_{
\gamma_2}$. For example $\rho_\gamma$ remains the same when taking multiples of $\gamma$. 
\end{remark}
%
We now explain the main ideas in the proof of Theorem \ref{Thm:RVFC}. We will consider three different, but interrelated, approaches to the reduced theory. 

\begin{enumerate}
\item {\bf Cosection localization.} In subsection \ref{ss:VFC.CL}, we will prove Theorem \ref{Thm:RVFC} using the {\em cone reduction} lemma of Kiem-Li \cite{KL13}. The {\em semi-regularity map} of Buchweitz-Flenner \cite{BF03} will provide the necessary cosections. This cosection localization method is the simplest way to construct the reduced virtual cycle. 
\item {\bf Algebraic twistor family.} In subsection \ref{ss:VFC.ATF}, we will construct a {\em reduced obstruction theory} using the algebraic twistor family introduced by the second-named author and Thomas \cite{KT1}. 
The main idea is to consider an infinitesimal deformation of $X$ such that the Hodge locus of $\gamma$ is just the origin. 
The algebraic twistor family method 
alone does not provide a reduced virtual cycle without the cone reduction lemma due to the isotropic condition. See Remark \ref{Rem2}.
\item {\bf Reduction by $(-1)$-shifted $1$-forms.} In Appendix \ref{Appendix:ReductionviaDAG}, we will construct a {\em reduced $(-2)$-shifted symplectic derived enhancement}. 
The crucial part is that the cosections arising from the semi-regularity map have natural extensions to $(-1)$-shifted {\em closed} $1$-forms.
\end{enumerate}

\subsection{Semi-regularity map}

In this subsection, we revisit the {\em semi-regularity map} of Buchweitz-Flenner \cite{BF03} and show that the key commutative triangle in \cite{BF03} is {\em symmetric} for Calabi-Yau 4-folds.

We first recall the definition of the semi-regularity map.\footnote{Here we consider the fixed-determinant version of the semi-regularity map. For strict Calabi-Yau 4-folds, this is identical to the 2-semi-regularity map in \cite{BF03}. Also we use a different sign convention compared to \cite{BF03}.}

\begin{definition}[cf.~{\cite{Blo,BF03}}]\label{Def:SR}
Let $I\udot$ be a perfect complex on $X$.
\begin{enumerate}
\item We define the {\em semi-regularity map} by
\begin{equation}\label{eq:sr}
\sr : \Ext^2_X(I\udot,I\udot)_0 \to H^3(X,\Omega_X^1), \quad e\mapsto \tr(\At(I\udot)\circ e),
\end{equation}
where $\At(I\udot) : I\udot \to I\udot \otimes \Omega^1_X[1]$ denotes the Atiyah class of $I\udot$.
\item We say that $I\udot$ is {\em semi-regular} if $\sr$ is injective. 
\end{enumerate}
\end{definition}

The semi-regularity map connects the obstruction to deforming a perfect complex $I\udot$ and its Chern character $\ch(I\udot)$ by the key commutative triangle in \cite[Cor.~4.3]{BF03}. For Calabi-Yau 4-folds, this commutative triangle has a {\em symmetric} description.

\begin{proposition}\label{prop:ob=srdual}
Let $I\udot$ be a perfect complex on $X$ such that $\ch_1(I\udot)=0 \in H^1(X,\Omega^1_X)$ and $\ch_2(I\udot)=\gamma \in H^2(X,\Omega_X^2)$. 
Then the obstruction map
\begin{equation}\label{eq:ob}
\sfob : H^1(X,T_X) \to \Ext^2_X(I\udot,I\udot)_0, \quad \xi \mapsto \iota_{\xi} \At(I\udot)	
\end{equation}
fits into the commutative diagram 
\begin{equation}\label{eq:triangle.obsrBgamma'} \xymatrix{ H^1(X,T_X) \ar[rr]^-{\sfob} \ar@/_0.6cm/[rdd]_-{B_\gamma} \ar[rd]^-{\iota_{-}(\gamma)} && \Ext^2_{X}(I\udot,I\udot)_0 \ar@/^0.6cm/[ldd]^-{\sfob\dual} \ar[ld]_-{\sr} \\ & H^3(X,\Omega_{X}^1) \ar[d]^{\mathsf{SD}}_{\cong} &\\ & H^1(X,T_{X})\dual & }\end{equation}
where the vertical isomorphism $\SD$
is induced by the Serre duality pairing
\[H^3(X,\Omega_X^1) \otimes H^1(X,T_X) \to \C, \quad \alpha \otimes \xi \mapsto \int_X \iota_{\xi}(\alpha) \cup \omega\]
and we identified $(\Ext^2_X(I\udot,I\udot)_0)\dual = \Ext^2_X(I\udot,I\udot)_0$ via the Yoneda pairing
\[\sfY: \Ext^2_X(I\udot,I\udot)_0 \otimes \Ext^2_X(I\udot,I\udot)_0 \to \C, \quad  u \otimes v \mapsto \int_X \tr(u \circ v) \cup \omega.\]
\end{proposition}


\begin{proof}
Note that the obstruction map $\sfob$ in \eqref{eq:ob} is well-defined since 
\[\tr \circ \sfob (\xi) = \iota_\xi \ch_1(I\udot)=0.\]
The commutativity of the three triangles in \eqref{eq:triangle.obsrBgamma'} can be shown as follows:
\begin{enumerate}
\item (top triangle) We have
$\sr \circ \sfob = \iota_{-}(\gamma)$
by \cite[Cor.~4.3]{BF03}. 
\item (left triangle) For any $\xi_1,\xi_2 \in H^1(X,T_X)$, we have
\[\SD(\iota_{\xi_1}(\gamma))(\xi_2) = \int_X \iota_{\xi_2}\iota_{\xi_1}\gamma \cup \omega = \int_X \iota_{\xi_1}\iota_{\xi_2}\gamma \cup \omega = \sfB_\gamma(\xi_1,\xi_2).\]
\item (right triangle) For any $\xi \in H^1(X,T_X)$ and $e \in \Ext^2_X(I\udot,I\udot)$, we have
\[\sfob\dual(e)(\xi) = \int_X \tr(\iota_\xi \At(I\udot) \circ e) \cup \omega = \int_X \iota_\xi \tr(\At(I\udot)\circ e) \cup \omega = \SD(\sr(e))(\xi).\]
\end{enumerate}
This completes the proof.
\end{proof}

The commutative triangle \eqref{eq:triangle.obsrBgamma'} says that the map of vector spaces
\[\sfob=\sr\dual : \left(H^1(X,T_X),\sfB_\gamma\right) \longrightarrow \left(\Ext^2_X(I\udot,I\udot)_0, \sfY\right)\]
is compatible with the corresponding symmetric bilinear forms. Equivalently, we may write $\sfB_\gamma = \sfob^*(\sfY)$. 

In the following subsections, we will develop the reduced theory with respect to a non-degenerate subspace $V\subseteq H^1(X,T_X)$. We provide a simple observation on such non-degenerate subspaces.

\begin{corollary}\label{Cor1}
Let $I\udot$ be a perfect complex on $X$ with $\ch_1(I\udot)=0$ and $\ch_2(I\udot)=\gamma$.
Choose a non-degenerate subspace
\[V \subseteq H^1(X,T_X)\]
with respect to the bilinear form $\sfB_\gamma$.
Then the restriction
\[ (\sfob=\sr\dual)|_V : \left(V,\, \sfB_\gamma|_V\right) \hookrightarrow \left(\Ext^2_X(I\udot,I\udot)_0,\, \sfY \right)\]
is an injective map of non-degenerate quadratic spaces.
\end{corollary}

\begin{proof}
If $\sfob(v)=0$ for some $v \in V$, then we have
\[\sfB_\gamma(v,-) = \int_X \tr\big(\sfob(v) \circ \sfob(-)\big) = 0 \in H^1(X,T_X)\dual.
\]
Thus the non-degeneracy of $\sfB_\gamma|_V$ implies that $v=0$.
\end{proof}

For reader's convenience, we review how the maps $\ob$ and $\sfB_\gamma$ in the commutative triangle \eqref{eq:triangle.obsrBgamma'} determine the obstructions of deforming the perfect complex $I\udot$ and the Hodge class $\gamma \in H^2(X,\Omega^2_X)$, respectively.
For simplicity, we only consider the first-order deformations here.
We refer to \cite[Prop.~4.2]{Blo} and \cite[Prop.~IV.3.1.5]{Ill} (cf. \cite[Cor.~3.4]{HT}) for the general case.

\begin{lemma}\label{Lem:Obstructions}
Consider a first-order deformation of $X$, i.e., a fibre diagram
\[\xymatrix{
X \ar@{^{(}->}[r] \ar[d] & \cX \ar[d]^f\\
\{o\} \ar@{^{(}->}[r] & \cA
}\]
where $f:\cX\to\cA$ is a smooth projective morphism and
$\cA$ is a square zero extension of $\Spec(\C)$.
Let $o \in \cA$ be the unique point.
\begin{enumerate}
\item[$(1)$] Given a $(2,2)$-class $\gamma \in H^2(X,\Omega^2_X) \subseteq H^4(X,\C)$,
there exists a horizontal section $\tgamma \in F^2H^4_{DR}(\cX/\cA)$ such that $\tgamma_o=\gamma$ if and only if the composition
\begin{equation*}\label{Eq:64}
T_{o,\cA} \xrightarrow{KS_{\cX/\cA}} H^1(X,T_X) \xrightarrow{\sfB_{\gamma}} H^1(X,T_X)\dual	
\end{equation*}
vanishes.
\item[$(2)$] Given a perfect complex $I\udot$ on $\cX$ with $\ch_1(I\udot)=0$,
there exists a perfect complex $\widetilde{I}\udot$ on $\cX$ with $\det(\widetilde{I}\udot)\cong \O_{\cX}$ such that $\widetilde{I}\udot_o\cong I\udot$ if and only if the composition
\begin{equation}\label{Eq:65}
T_{o,\cA} \xrightarrow{KS_{\cX/\cA}} H^1(X,T_X) \xrightarrow{\sfob} \Ext^2_X(I\udot,I\udot)_0 	
\end{equation}
vanishes.
\end{enumerate}
\end{lemma}

\begin{proof}
(1) follows directly from \cite[Prop.~4.2]{Blo}.  
We would like to deduce (2) from \cite[Prop.~IV.3.1.5]{Ill} (or \cite[Cor.~3.4]{HT}). Consider a diagram
\[\xymatrix@C+2pc{
& \LL_X \cong \LL_{\cX/\cA}|_{X} \ar@<.8ex>[d] \ar[rd]^{\KS_{\cX/\cA}|_X} & \\
R\hom_{X}(I\udot,I\udot)[-1] \ar@{.>}[r]|-{\At_{X/\cA}(I\udot)} \ar@{.>}[ru]^-{\At_{X}(I\udot)} \ar@{.>}[rd]_-{\At_{X/\cX}(I\udot)} & \LL_{X/\cA} \ar[r]|-{\KS_{X/\cA}} \ar@{.>}@<.8ex>[u]^-{\mathrm{pr}_2}\ar@{.>}@<-.8ex>[d]_-{\mathrm{pr}_1}
 & \LL_{\cA}|_{X}[1]\\
& \LL_{X/\cX} \cong \LL_{o/\cA}|_{X} \ar@<-.8ex>[u] \ar[ru]_{\KS_{o/\cA}|_X}
}\]
where the left two triangles commute with the dotted arrows and the right two triangles commute with the solid arrows. Moreover, the composition of the middle arrows is zero because $\At_{X/\cA}(I^\mdot)$ factors through $\At_{X}(I^\mdot)$. 
Hence the total big square commutes (up to sign). Therefore, the map \eqref{Eq:65} corresponds to the composition
\[I\udot \xrightarrow{\At_{X/\cX}(I\udot)} I\udot \otimes \LL_{X/\cX}[1] \to I\udot \otimes \I_{X/\cX}[2] = I\udot \otimes \I_{o/\cA}[2]\]
via adjunction. 
\end{proof}


\subsection{Cosection localization}\label{ss:VFC.CL}

In this subsection, we construct the {\em reduced virtual cycle} and prove Theorem \ref{Thm:RVFC} via the {\em cosection localization} technique of Kiem-Li \cite{KL13}, as in \cite{KP20}.
We also prove Theorem \ref{Thm:SR=smoothofrvd} as a corollary.

Let $\curP:=\PTqvX$ be the moduli space of $\PT_q$ pairs on $X$ with fixed Chern character $v = (0,0,\gamma,\beta,n-\gamma\cdot\td_2(X)) \in H^*(X,\Q)$. Let $\phi:\EE\to\trunc\LL_{\curP}$ be the standard symmetric obstruction theory of $\curP$ in \eqref{Eq:SOT}.


We first observe that the semi-regularity maps provide {\em cosections}. 

\begin{proposition}\label{prop:cosection}
There exists a natural map
\[\SR : \EE\dual[1] \to H^1(X,T_X)\dual\otimes \O_{\curP}\]
satisfying the following properties:
\begin{enumerate}
\item[$(1)$] The fibre of the map $\SR$ over  $[(F,s)] \in \curP$ is the composition\footnote{\label{footnote:decomposition}The first map is given by the canonical quasi-isomorphism $C^\mdot \cong \bigoplus_i h^i(C^\mdot)[-i]$ for any complex of {\em vector spaces} (\cite[Prop.~III.2.4]{GM}).}
\[\SR|_{(F,s)} : R\Hom_X(I\udot,I\udot)_0[2] \to \Ext^2_X(I\udot,I\udot)_0 \xrightarrow{\sr} H^1(X,T_X)\dual,\]
where $\sr$ is the semi-regularity map (Definition \ref{Def:SR}(1)) for the associated complex $I\udot := [\O_X \xrightarrow{s}F]$.
\item[$(2)$] The square 
of the map $\SR$ is
\[\SR^2 = \sfB_\gamma \otimes 1 : H^1(X,T_X)\otimes \O_{\curP} \to H^1(X,T_X)\dual \otimes \O_{\curP}.\]
\end{enumerate}
Here, we identified $H^3(X,\Omega^1_X) = H^1(X,T_X)\dual$ via Serre duality and the Calabi-Yau 4-form, as in Proposition \ref{prop:ob=srdual}.
\end{proposition}

\begin{proof}
Let $\II\udot$ denote the universal complex.
Consider the map
\[R\hom_\pi(\II\udot,\II\udot)_0[2] \xrightarrow{\At_{\curP\times X/\curP}(\II\udot)} R\hom_\pi(\II\udot,\II\udot\otimes \Omega^1_X)[3]\xrightarrow{\tr^3}  H^3(X,\Omega^1_X)\otimes \O_{\curP} ,\]
where the third trace map is defined as (see footnote \ref{footnote:decomposition})
$$
\tr^3 : R\hom_\pi(\II\udot,\II\udot\otimes \Omega^1_X)[3] \stackrel{\tr}{\to} R\Gamma(X,\Omega_X^1) \otimes \O_{\curP}[3] \to H^3(X,\Omega^1_X) \otimes \O_{\curP}.
$$
The map $\SR$ is then defined by composing the above with the isomorphism $H^3(X,\Omega^1_X) \cong H^1(X,T_X)^\vee$. The functoriality of the Atiyah classes ensures that the fibres of $\SR$ are the semi-regularity maps. This proves (1). 

Next, we claim that  
the square 
\[\SR^2 : H^1(X,T_X) \otimes \O_{\curP} \xrightarrow{\SR\dual} \EE[-1] \cong \EE\dual[1] \xrightarrow{\SR} H^1(X,T_X)\dual \otimes \O_{\curP}\]
is given by contraction with the Chern character
\[\ch_2(\II\udot) := \frac{1}{2} \tr (\At^2(\II\udot)) \in   H^2(X,\Omega_X^2) \otimes \Gamma(\curP, \O_{\curP}).\]
Statement (2) then follows because Proposition \ref{lem:chishorizontal} implies that the Chern character $\ch_2(\II\udot)$ is a horizontal section, i.e.,
\[\ch_2(\II\udot) = \gamma \otimes 1.\]

In order to prove the claim, it suffices to take hypercohomology $\mathbb{H}^0(\curP,-)$ of the composition $\SR^2$ and show that this equals contraction with $\ch_2(\II\udot)$. 
By \cite[Prop.~4.2]{BF03}, the following triangle commutes 
\[\xymatrix{
H^1(\curP \times X, T_{\curP \times X / \curP}) \ar[rr] \ar[rd]_{\iota_{\ch_2(\II\udot)}} && \Ext^2_{\curP \times X}(\II\udot,\II\udot)_0 \ar[ld] \\
& H^3(\curP \times X, \Omega^1_{\curP \times X / \curP}),&
}\]
where the top map is induced by $\SR\dual$ and the right map by $\SR$. Using the decomposition (see footnote \ref{footnote:decomposition})
$$
H^m(\curP \times X, T_{\curP \times X / \curP}) = \bigoplus_{i+j=m} H^i(X,T_X) \otimes H^j(\curP,\O_{\curP}) 
$$
and similarly for $\Omega^1_{\curP \times X / \curP}$, we obtain the commutative diagram
\[\xymatrix{
H^1(X,T_X) \otimes H^0(\curP,\O_{\curP}) \ar^{\mathbb{H}^0(\curP,\SR^\vee)}[rr] \ar[rd]_{\iota_{\ch_2(\II^\mdot)}} && \Ext^2_{\curP \times X}(\II\udot,\II\udot)_0 \ar^{\mathbb{H}^0(\curP,\SR)}[ld] \\
& H^3(X,\Omega^1_X) \otimes H^0(\curP,\O_{\curP})&
}\]
as desired.
\end{proof}

\begin{remark}
The map $\SR$ in Proposition \ref{prop:cosection} is the collection of all cosections associated to $(3,1)$-forms in \cite[Ex.~9.2]{KP20}. More precisely, if we choose a $(3,1)$-form $\xi \in H^1(X,T_X) \cong H^1(X,\Omega^3_X)$, then the composition 
\[\EE\dual[1] \xrightarrow{\SR} H^1(X,T_X)\dual \otimes \O_{\curP} \xrightarrow{\xi} \O_{\curP}\]
is exactly the cosection $\sigma^{\xi}$ associated to the $(3,1)$-form $\xi$ considered in \cite{KP20}.
\end{remark}

We sometimes call the symmetric complex $\EE$ in the obstruction theory $\phi:\EE \to \trunc \LL_{\curP}$ the {\em virtual cotangent complex}. Following this perspective, we define the {\em reduced virtual cotangent complex} $\EE^\red_V$ as follows:

\begin{definition}\label{Def:EEredV}
Choose a non-degenerate subspace $V\subseteq H^1(X,T_X)$ with respect to the bilinear form $\sfB_\gamma$.
\begin{enumerate}
\item We denote by
\begin{equation*}\label{Eq:SRV}	
\SR_V : \EE\dual[1] \xrightarrow{\SR} H^1(X,T_X)\dual\otimes \O_{\curP} \twoheadrightarrow V\dual \otimes \O_{\curP}
\end{equation*}
the composition.
Then the square
\[\SR^2_V =  \sfB_\gamma|_V \otimes 1 : V \otimes \O_{\curP} \xrightarrow{\SR_V\dual} \EE[-1] \cong \EE\dual[1] \xrightarrow{\SR_V} V\dual \otimes \O_{\curP}\]
is an isomorphism by Proposition \ref{prop:cosection}(2).
\item We define $\EE^\red_V:=\mathrm{cone}(\SR_V\dual[1])$. Since the natural map $\SR_V\dual[1]$ has a left inverse, we have a splitting 
\begin{equation}\label{eq:EEredV}
\EE \cong \EE^\red_V \oplus \left(V\otimes \O_{\curP}[1]\right).
\end{equation}
Moreover, $(\EE^\red_V)\dual[2] \to \EE\dual[2] \xrightarrow{\theta} \EE \to \EE^\red_V$ induces a symmetric form on $\EE^\red_V$ and \eqref{eq:EEredV} is a splitting of symmetric complexes. 
\item We denote by
\[\fQ(\EE_V^\red) \subseteq \fC(\EE_V^\red)\]
the quadratic cone stack induced by the above symmetric form.
\end{enumerate}
\end{definition}

The main result in this subsection is the following {\em cone reduction} property:

\begin{proposition}\label{prop:conereduction}
Given a non-degenerate subspace $V \subseteq H^1(X,T_X)$ with respect to $\sfB_\gamma$, 
there exists a commutative diagram
\[\xymatrix{
 & \fQ(\EE^\red_V) \ar@{^{(}->}[d] \ar@{^{(}->}[r] &\fC(\EE^\red_V)  \ar@{^{(}->}[d]\\
(\fC_{\curP})_\red \ar@{^{(}->}[r] \ar@{.>}[ru] & \fQ(\EE) \ar@{^{(}->}[r] & \fC(\EE)
}\]
for a unique dotted arrow and where the square is cartesian.
Here the closed embedding $(\fC_{\curP})_{\red} \hookrightarrow \fQ(\EE)$ is given by \eqref{Eq12}.\footnote{When we use $\red$ as a subscript, it denotes the associated reduced substack.}
\end{proposition}

\begin{proof}
We will form a commutative diagram
\begin{equation}\label{Eq2}
\xymatrix{
& \fQ(\EE^\red_V) \ar@{^{(}->}[r] \ar@{^{(}->}[d] &\fC(\EE^\red_V) \ar[r] \ar@{^{(}->}[d] & \curP \ar@{^{(}->}[d]^{0}
\\
(\fC_{\curP})_\red \ar@{^{(}->}[r]^{} \ar@{.>}[ru] & \fQ(\EE) \ar@{^{(}->}[r] \ar[d] &\fC(\EE) \ar[r]^-{\fl_{\SR_V}} \ar[d]^{\fq_\EE} & V|_{\curP}\dual\\
& \curP \ar@{^{(}->}[r]^0 & \bbA^1_{\curP}
}	
\end{equation}
and claim that all the squares are Cartesian. Here, $\fl_{\SR_V}:=\fC(\SR_V\dual[1])$ is the linear function induced by $\SR_V\dual : V\otimes \O_{\curP} \to \EE[-1]$ and $\fq_{\EE}$ is the quadratic function induced by $\theta : \O_{\curP} \to (\EE\otimes\EE)[-2]$.
Indeed, the lower square in \eqref{Eq2} is cartesian since the quadratic cone stack $\fQ(\EE)$ is defined as the zero locus of the quadratic function $\fq_\EE$. The upper right square in \eqref{Eq2} is cartesian since $\EE^\red_V$ is defined as the cone of $\SR_V\dual[1] : V\otimes \O_{\curP} [1] \to \EE$, see the proof of Lemma \ref{Lem:KLconereduction}. 
The middle upper square is cartesian since \eqref{eq:EEredV} is a direct sum decomposition of {\em symmetric complexes}. More precisely, we have
\begin{equation}\label{Eq21}
\fq_{\EE} = \fq_{\EE^\red_V}\circ p_1 + \fq_V|_{\curP} \circ p_2 : \fC(\EE) = \fC(\EE^\red_V) \times_{\curP} V|_{\curP}\dual \to \bbA^1_{\curP}     
\end{equation}
where $p_1$, $p_2$ are the projection maps and $\fq_V: V\times V \to \C$ denotes the quadratic function induced by $\sfB_\gamma|_V$. Then the formula \eqref{Eq21} implies that the triangle
\[\xymatrix{
\fC(\EE^\red_V) \ar[rd]^{\fq_{\EE^\red_V}} \ar@{^{(}->}[d] & \\
\fC(\EE) \ar[r]_{\fq_\EE} & \bbA^1_{\curP}
}\]
commutes. Hence, the middle upper square in \eqref{Eq2} is cartesian, as claimed. Finally, by Kiem-Li's cone reduction lemma \cite{KL13} in the form of Lemma \ref{Lem:KLconereduction} below, we can find the dotted arrow that fits into the commutative diagram \eqref{Eq2}.
\end{proof}

We rephrase Kiem-Li's cone reduction lemma as follows. This lemma is required to complete the proof of Proposition \ref{prop:conereduction}.

\begin{lemma}[cf. {\cite[Prop.~4.3]{KL13}}]\label{Lem:KLconereduction}
Let $\phi:\EE \to \trunc\LL_{\sX}$ be an obstruction theory
for an arbitrary Deligne-Mumford stack $\sX$. 
Let $\Sigma: \EE\dual[1] \to F$ be a map for a vector bundle $F$.
Then there exists a commutative diagram
\begin{equation}\label{Eq:CL.3}
\xymatrix{
 & \fC(\EE_{\Sigma}) \ar[r] \ar@{^{(}->}[d] & \sX \ar@{^{(}->}[d]^{0_F}\\
(\fC_{\sX})_\red  \ar@{.>}[ru] \ar@{^{(}->}[r] & \fC(\EE) \ar[r]^-{\fl_{\Sigma}} & F
}    
\end{equation}
with a unique dotted arrow and where the square is cartesian.
Here the closed embedding $(\fC_{\sX})_{\red} \hookrightarrow \fC(\EE)$ is induced by the obstruction theory $\phi$,
the linear function $\fl_{\Sigma}:=\fC(\Sigma\dual[1]):\fC(\EE) \to F$ is induced by the map $\Sigma$, and
the perfect complex $\EE_{\Sigma}$ is the cone of the map $\Sigma\dual[1]:F\dual[1] \to \EE$.
\end{lemma}


\begin{proof}

Note that the statement is local on $\sX$.
Thus we may assume that $\sX$ is a quasi-projective scheme and $F=\O_{\sX}^{\oplus r}$ is a trivial vector bundle. 
By the resolution property of $\sX$, we can find a chain complex of vector bundles that represents $\EE$
\[\EE \cong [E^{-e} \xrightarrow{} \cdots \to E^{-2} \xrightarrow{d^{-1}} E^{-1} \xrightarrow{d^0} E^0]\]
and a morphism of chain complexes that represents $\Sigma$
\[\xymatrix{
\EE\dual[1] \ar[d]^{\Sigma} \ar@{}[rd]|-{\cong} & E_0 \ar[r]^{}\ar[d] & E_1 \ar[r]^{}\ar[d]^{\widetilde{\Sigma}} & E_2 \ar[r] \ar[d]& \cdots \ar[d] \ar[r] & E_e \ar[d] \\
F  & 0 \ar[r] &F \ar[r] & 0 \ar[r] & \cdots \ar[r] & 0
}\]
where we denoted $E_i:=(E^{-i})\dual$. 
Then we also have
\[\EE_{\Sigma} \cong [E^{-e} \xrightarrow{} \cdots \to E^{-2} \oplus F\dual \xrightarrow{(d^{-1},\tSigma\dual)} E^{-1} \xrightarrow{d^0} E^0].\]

Firstly, we will show that the square in \eqref{Eq:CL.3} is cartesian. Indeed, note that the cokernels
\[D:=\coker(d^{-1} : E^{-2} \to E^{-1}),\quad
D_{\tSigma}:=\coker((d^{-1},\tSigma\dual) : E^{-2}\oplus F\dual \to E^{-1})\]
fit into a right exact sequence of coherent sheaves
\[\xymatrix{F\dual \ar[r] & D \ar[r] & D_{\tSigma} \ar[r] & 0}.\]
By taking the associated abelian cones, we obtain a cartesian square
\[\xymatrix{
C(D_{\tSigma}) \ar[r] \ar@{^{(}->}[d] & \bbA^1_{\sX} \ar@{^{(}->}[d]^{0_F}\\
C(D) \ar[r]^-{\tSigma|_{C(D)}} & F.
}\]
Since $\fC(\EE) = [C(D)/E_0]$ and $\fC(\EE_{\Sigma}) = [C(D_{\tSigma})/E_0]$, we obtain the desired cartesian square via descent.

Secondly, we will find a dotted arrow that fits into the commutative diagram \eqref{Eq:CL.3}. The uniqueness is obvious.
Consider the induced (2-term) perfect obstruction theory 
\[\phi^\tau : \EE^\tau:= [E^{-1} \xrightarrow{d^0} E^0] \xrightarrow{t} \EE \xrightarrow{\phi} \trunc\LL_{\sX}\]
and an induced cosection $\Sigma^\tau$ that fits into the commutative diagram
\[\xymatrix{
\EE\dual[1] \ar[d]_-{t\dual[1]} \ar[r]^-{\Sigma} & F\\
(\EE^\tau)\dual[1] \ar@{.>}[ru]_-{\Sigma^\tau}.
}\]
Since $F$ is a trivial vector bundle, we may write
\[\Sigma^\tau=(\sigma_1,\cdots,\sigma_r) : (\EE^\tau)\dual[1] \to F \cong \O_{\sX}^{\oplus r}.\] 
Then $\sigma_i : (\EE^\tau)\dual[1] \to \O_{\sX}$ are cosections for the perfect obstruction theory $\phi^\tau$. Hence by \cite[Prop.~4.3]{KL13}, we have 
\[(\fC_{\sX})_\red \subseteq \fC((\EE^\tau)_{\sigma_i})\]
as substacks of $\fC(\EE^{\tau})$, where $(\EE^\tau)_{\sigma_i}:=\cone(\sigma_i\dual[1])$.
Here $ \fC((\EE^\tau)_{\sigma_i})$ are set-theoretically the kernel cone stacks in \cite{KL13}.
Hence we have a commutative diagram
\[\xymatrix{
& \fC(\EE_\Sigma) \ar@{^{(}->}[r] \ar@{^{(}->}[d] & \bigcap_i\fC((\EE^\tau)_{\sigma_i}) \ar[r] \ar@{^{(}->}[d] & \sX \ar@{^{(}->}[d]^{0_F}\\
(\fC_{\sX})_\red \ar@{^{(}->}[r] \ar@{-->}[rru] \ar@{.>}[ru]& \fC(\EE) \ar@{^{(}->}[r]^{\fC(t)} \ar@/_0.5cm/[rr]_{\fl_{\Sigma}} & \fC(\EE^\tau) \ar[r]^{\fl_{\Sigma^{\tau}}} & F
}\]
where the squares are cartesian by the result in the previous paragraph. Hence the dotted arrow in the above diagram gives us the desired dotted arrow in \eqref{Eq:CL.3}.
\end{proof}


Now Theorem \ref{Thm:RVFC} is a direct corollary of Proposition \ref{prop:conereduction}. 

\begin{proof}[Proof of Theorem \ref{Thm:RVFC}]

We first construct a reduced virtual cycle with respect to a maximal non-degenerate subspace $V \subseteq H^1(X,T_X)$. Here maximal means that $\dim(V)=\rank(\sfB_\gamma)=\rho_\gamma$. Thus the composition
\begin{equation}\label{Eq3}
V \hookrightarrow H^1(X,T_X) \twoheadrightarrow H^1(X,T_X)_\gamma	
\end{equation}
is an isomorphism.

As in the proof of Theorem \ref{Thm:NRVFC} (see also Remark \ref{Rem4}), we can construct a square root Gysin pullback
\[\sqrt{0^!_{\fQ(\EE^\red_V)}} : A_*\left(\fQ(\EE^\red_V)\right) \to A_*\left(\curP\right)\]
for the symmetric complex $\EE^\red_V$.
Here the orientation of $\EE^\red_V$ is determined by the orientations of $\EE$ and $H^1(X,T_X)_\gamma$ via the canonical isomorphisms
\[\det(\EE) \cong \det(\EE^\red_V) \otimes \det(V) \cong \det(\EE^\red_V) \otimes \det(H^1(X,T_X)_\gamma)\]
of line bundles, induced by the decomposition \eqref{eq:EEredV} and the isomorphism \eqref{Eq3}.

By the cone reduction property in Proposition \ref{prop:conereduction}, we may consider the fundamental class of the intrinsic normal cone $\fC_{\curP}$ as a cycle class in $\fQ(\EE^\red_V)$,
\[[\fC_{\curP}] \in A_0\left(\fQ(\EE^\red_V)\right).\]
Indeed, this can be defined as the pushforward of the fundamental class $[\fC_{\curP}] \in A_*((\fC_{\curP})_\red) = A_*(\fC_{\curP})$ under the closed embedding $(\fC_{\curP})_\red \hookrightarrow \fQ(\EE^\red_V)$.

Analogous to the definition of the non-reduced virtual cycle in \eqref{Eq:Def.NRVFC}, we define the {\em reduced virtual cycle} as
\begin{equation}\label{Eq6}
\left[\curP\right]^{\red}_V := \sqrt{0^!_{\fQ(\EE^\red_V)}} \left[\fC_{\curP}\right] \in A_{\mathrm{rvd}}\left(\curP\right).
\end{equation}
Then the reduced virtual dimension is
\[\rvd = \vd + \frac12\dim(V) = n- \frac12 \gamma^2 + \frac12 \rho_\gamma.\]

We now claim that the reduced virtual cycle $[\curP]^\red_V$ in \eqref{Eq6} is independent of the choice of a maximal non-degenerate subspace $V \subseteq H^1(X,T_X)$. Note that a choice of a maximal isotropic subspace $V$ of $H^1(X,T_X)$ is equivalent to a choice of a section of the surjection
$H^1(X,T_X) \twoheadrightarrow H^1(X,T_X)_\gamma.$
Thus, if we fix one splitting
\[H^1(X,T_X) =  H^1(X,T_X)_\gamma \oplus \ker(\sfB_\gamma),\]
then any other splitting corresponds to a map $H^1(X,T_X)_\gamma \to \ker(\sfB_\gamma).$ Therefore, any two splittings are connected via a family of splittings parametrized by $\bbA^1$. A deformation argument proves the claim.\footnote{Generally, we need to be careful when we apply the cone reduction in the relative setting. But in this case, the map $\curP \times \bbA^1 \to \bbA^1$ is flat, so we can detect the cone reduction property fibrewise.}
\end{proof}

\begin{remark}\label{Rem:non-deg/isotropic.cosection}
We view the cosection localization (Proposition \ref{prop:conereduction}) in this paper as an {\em orthogonal} version to the one in \cite{KP20}.
Indeed, there are two types of cosections $\Sigma : \EE\dual[1] \to F$ in DT4 theory:
\begin{enumerate}
\item {\em isotropic} cosections, i.e., $\Sigma^2=0$,
\item {\em non-degenerate} cosections, i.e., $\Sigma^2$ is an isomorphism.
\end{enumerate}
In \cite{KP20}, cosection localization for the first case is studied, whereas we introduced the second case in this paper.

In the second case, if the orthogonal bundle $(F,\Sigma^2)$ is of even rank and has a maximal isotropic subbundle, then we can reduce to the first case. Then both approaches give the same reduced virtual cycle.

However, for surfaces on Calabi-Yau 4-folds, there are \emph{many} examples (e.g. Example \ref{Ex:1}) that have an {\em odd} number of cosections. Therefore we indeed need cosection localization as in the second case for counting surfaces.
\end{remark}





\begin{remark}
In general, the cone reduction property in Proposition \ref{prop:conereduction} does not imply the existence of a reduced obstruction theory, see \cite[Appendix A]{MPT}. 
However, in our situation for $\PTqvX$, we will see in the next subsection that there indeed exists a reduced obstruction theory.
\end{remark}

We now prove Theorem \ref{Thm:SR=smoothofrvd}. The implication $(\Rightarrow)$ can be shown using the infinitesimal lifting criterion as in \cite[Thm.~7.3]{Blo}. We provide an alternative proof, more in line with DT4 theory, using the cone reduction property (Proposition \ref{prop:conereduction}).

\begin{proof}[Proof of Theorem \ref{Thm:SR=smoothofrvd}]
($\Rightarrow$) Let $[(F,s)] \in \curP$ be a $\PT_q$ pair for which the associated perfect complex $I\udot:=[\O_X \xrightarrow{s}F]$ is semi-regular. 

Choose a maximal non-degenerate subspace $V \subseteq H^1(X,T_X)$ with respect to $\sfB_\gamma$. We first claim that the restriction of the obstruction map $\sfob$ in \eqref{eq:ob} to $V$
\[\sfob|_V : V \hookrightarrow H^1(X,T_X) \twoheadrightarrow \Ext^2_X(I\udot,I\udot)_0\]
is an isomorphism. Indeed, by Corollary \ref{Cor1}, the above map $\sfob|_V$ is injective. On the other hand, the semi-regularity of $I\udot$ implies that the obstruction map $\sfob=\sr\dual:H^1(X,T_X) \to \Ext^2_X(I\udot,I\udot)_0$ is surjective. Hence we have 
\[\dim(V) = \rho_\gamma = \rank(\sfB_\gamma) = \rank(\sfob^*(\sfY))= \rank(\sfY)= \dim_{\C}\Ext^2_X(I\udot,I\udot)_0.\]
Therefore the above map $\sfob|_V$ is bijective. This proves the claim.

By the derived Darboux theorem \cite[Thm.~5.18(ii)]{BBJ},
we can find a diagram
\[\xymatrix{
& E \ar[d]\\
U(t) \ar@{^{(}->}[r] & U, \ar@/_0.4cm/[u]_{t}
}\]
where $U$ is a smooth affine scheme, $E$ is an orthogonal bundle over $U$, $t \in \Gamma(U,E)$ is an isotropic section, and $U(t)$ is the zero locus of $t$ in $U$, with a Zariski open embedding $U(t) \hookrightarrow \curP$ whose image contains $[I\udot] \in \curP$. Also, 
the standard symmetric obstruction theory $\phi:\EE \to \tau^{\geq -1} \LL_{\curP}$ in \eqref{Eq:SOT} can be written as
\[\xymatrix{
\EE|_{U(t)} \ar[d]^{\phi} \ar@{}[rd]|-{=} & T_U|_{U(t)} \ar[r]^-{dt}\ar[d] & E|_{U(t)}\dual \ar[r]^-{dt} \ar[d]^t & \Omega_U|_{U(t)} \ar@{=}[d]\\
\trunc\LL_{\curP}|_{U(t)}  & 0 \ar[r] & \I/\I^2 \ar[r]^d & \Omega_U|_{U(t)}
}\]
where $\I=\I_{U(t)/U}$ is the ideal sheaf of $U(t) \subseteq U$.
Moreover, 
by choosing a minimal Darboux chart (in the sense of \cite[Def.~2.13]{BBJ}), 
we can further assume that the fibre of the map $dt: T_U|_{U(t)} \to E|_{U(t)}\dual$ at $[I\udot] \in U(t)$ is zero.

Consider the map $\SR_V$ in Definition \ref{Def:EEredV}(1). Then we have a diagram
\begin{displaymath}
\xymatrix
{
\Omega_U|_{U(t)} \ar[r] \ar@{-->}[dr] & \EE|_{U(t)} \ar^{\SR_V}[d] \ar[r] & [T_U|_{U(t)} \to E|_{U(t)}^\vee \to 0] \\
& V^\vee \otimes \O_{U(t)}[1] &
}
\end{displaymath}
where the last term is the stupid truncation and the dotted arrow is the composition. Since $U(t)$ is affine, the dotted arrow is zero and we obtain a morphism of cochain complexes
\[\xymatrix{
\EE|_{U(t)} \ar[d]^{\SR_V} \ar@{}[rd]|-{=} & T_U|_{U(t)} \ar[r]^-{dt}\ar[d] & E|_{U(t)}\dual \ar[r]^-{dt} \ar[d]^{\widetilde{\SR_V}} & \Omega_U|_{U(t)} \ar[d]\\
V\dual \otimes \O_{U(t)}[1]  & 0 \ar[r] & V\dual \otimes \O_{U(t)} \ar[r] & 0
}\]
for some map $\widetilde{\SR_V}$.
By Proposition \ref{prop:cosection}(1), and the fact that the fibre of $dt$ at $[(F,s)]$ is zero,
the fibre of the map $\widetilde{\SR_V}$ over $[(F,s)] \in \curP$ can be identified with the composition
\begin{equation}\label{Eq5}
\widetilde{\SR_V}|_{(F,s)} : \Ext^2(I\udot,I\udot)_0 \xrightarrow{\sr} H^1(X,T_X)\dual \twoheadrightarrow V\dual.
\end{equation}
Since $\sfob|_V$ is an isomorphism, Proposition \ref{prop:ob=srdual} implies that the map \eqref{Eq5} is an isomorphism. Therefore, by shrinking $U$ further, we may assume that the map $\widetilde{\SR_V}$ is an isomorphism.

Consider $U(t)$ as the zero section $U(t) \subset E|_{U(t)}$.
By the cone reduction property in Proposition \ref{prop:conereduction}, 
we have 
\[\left[\frac{(C_{U(t)/U})_\red}{T_{U}|_{U(t)}} \right] = \left(\fC_{U(t)}\right)_\red \subseteq \fC(\EE^\red_V) = \left[\frac{U(t)}{T_U|_{U(t)}}\right]\]
as substacks of $\fC(\EE)\subseteq [E|_{U(t)}/T_U|_{U(t)}]$ since $\ker(\widetilde{\SR_V})=0$.
Therefore, we have
\[(C_{U(t)/U})_\red = U(t).\]
Equivalently, the ideal 
$\bigoplus_{n>0}\frac{\I^n}{\I^{n+1}} \subseteq \bigoplus_{n\geq 0}\frac{\I^n}{\I^{n+1}}$
is nilpotent. Hence, we have $\frac{\I^n}{\I^{n+1}} =0$
for some $n$, and Nakayama's lemma implies
$(1-i) \cdot \I^n =0 $
for some $i \in \I$. Since $U$ is integral and $U(t)$ is nonempty, we have $\I=0$ and thus $t=0$. Therefore, $U(t)=U$ is smooth of dimension
\[
\dim(U) = \left(\dim(U) - \frac12 \rank(E)\right) + \frac12 \rank(E) = \vd + \frac12 \rho_\gamma = \rvd.\]
This completes the proof of ($\Rightarrow$) since $U(t)$ is an open neighborhood of $[I\udot] \in \curP$.

($\Leftarrow$) Assume that $\curP$ is smooth of dimension $\rvd$ at $[I\udot] \in \curP$. Choose a maximal non-degenerate subspace $V$ of $H^1(X,T_X)$ with respect to $\sfB_\gamma$. By Corollary \ref{Cor1}, the restriction of the obstruction map to $V$
\[\sfob|_V : V \hookrightarrow H^1(X,T_X) \xrightarrow{\sfob} \Ext^2_X(I\udot,I\udot)_0\]
is injective. On the other hand, since $\curP$ is smooth at $[I\udot] \in \curP$, we have
\[\dim(V) = \rho_\gamma = 2\cdot(\rvd - \vd) = 2 \cdot\left(\mathrm{ext}^1 - (\mathrm{ext}^1 -  \frac12 \mathrm{ext}^2)\right)=\mathrm{ext}^2\]
where $\mathrm{ext}^1:=\dim_\C \Ext^1_X(I\udot,I\udot)_0$ and $\mathrm{ext}^2:=\dim_\C \Ext^2_X(I\udot,I\udot)_0$. Hence the above map $\sfob|_V$ is bijective. Then the obstruction map $\sfob$ is surjective, and its dual $\sr$ is injective (Proposition \ref{prop:ob=srdual}). Therefore $[I\udot] \in \curP$ is semi-regular.
\end{proof}


\subsection{Algebraic twistor family}\label{ss:VFC.ATF}

In this subsection, we construct a {\em reduced obstruction theory} via the {\em algebraic twistor family} introduced by the second-named author and Thomas in \cite{KT1}.

We use the notation from the previous subsections. Let $\PTqvX$ be the moduli space of $\PT_q$ pairs $(F,s)$ on a Calabi-Yau 4-fold $X$ with fixed cohomological Chern character $\ch(F)=v=(0,0,\gamma,\beta,n-\gamma\cdot\td_2(X)) \in H^*(X,\Q)$. 
The main result in this subsection is the following:

\begin{proposition}\label{prop:redSOT}
Given a non-degenerate subspace $V \subseteq H^1(X,T_X)$ with respect to $\sfB_\gamma$, the composition
\begin{equation}\label{eq:S5.redSOT}
\phi^\red_V : \EE^\red_V \to \EE \xrightarrow{\phi} \tau^{\geq -1} \LL_{\PTqvX}	
\end{equation}
is also a symmetric obstruction theory, i.e., $h^0(\phi^\red_V)$ is bijective and $h^{-1}(\phi^\red_V)$ is surjective. Here the first map in \eqref{eq:S5.redSOT} is given by the decomposition \eqref{eq:EEredV} and the second map is the standard symmetric obstruction theory in \eqref{Eq:SOT}.
\end{proposition}

The rest of this subsection is devoted to the proof of Proposition \ref{prop:redSOT}. By the deformation theory of algebraic varieties, any subspace $V\subseteq H^1(X,T_X)$ corresponds to a fibre diagram
\[\xymatrix{
X \ar@{^{(}->}[r] \ar[d] & \cX_V \ar[d]^{f_V} \\
\Spec(\C) \ar@{^{(}->}[r] & \cA_V
}\]
where $\cX_V \to \cA_V$ is a smooth projective morphism over the spectrum of the local Artinian $\C$-algebra $A_V:= \O_V / \mathfrak{m}^2 = \C \oplus V\dual$, where $\m \subset \O_V$ denotes the maximal ideal of functions vanishing at the origin. Here $\Omega^1_{\cA_V} = V\dual \otimes \O_{\cA_V}$ and the (restriction) of the (classical) Kodaira-Spencer map 
\[\KS_{\cX_V/\cA_V}|_X : \Omega^1_X \to V\dual\otimes \O_X[1]\]
corresponds to the given inclusion $V \subseteq H^1(X,T_X)$ via adjunction.
Let 
\[\tv=v\otimes 1 \in H^*_{DR}(\cX_V/\cA_V) \cong H^*_{DR}(X)\otimes_\C A_V\]
be the unique horizontal lift of $v \in H^*_{DR}(X)$ (see subsection \ref{ss:relativemoduli} for the definitions of the de Rham cohomology and horizontal sections). Form a fibre diagram
\begin{equation}\label{eq:S5.4}
\xymatrix{
\PTqvX \ \ar@{^{(}->}[r] \ar[d] & {\curly P}^{(q)}_{\tv}(\cX_V/\cA_V)\ar[d] \\
\Spec(\C) \ar@{^{(}->}[r] & \cA_V
}	
\end{equation}
where ${\curly P}^{(q)}_{\tv}(\cX_V/\cA_V)$ is the relative moduli space of $\PT_q$ pairs on the fibres of $\cX_V \to \cA_V$. The key observation is the following:

\begin{proposition}\label{prop:P(X)=P(X/A)}
If $V\subseteq H^1(X,T_X)$ is a non-degenerate subspace, then the canonical inclusion map
\[\PTqvX \hookrightarrow {\curly P}^{(q)}_{\tv}(\cX_V/\cA_V)\]
is an isomorphism of schemes.
\end{proposition}

\begin{proof}
The proof is similar to \cite[Lem.~2.1]{KT1} and \cite[Prop.~2.2]{KT1}. We first claim that the {\em Hodge locus} of the horizontal section
\[\gamma \otimes 1 \in H^4_{DR}(\cX_V/\cA_V) \cong H^4_{DR}(X)\otimes A_V\]
is just the unique closed point in $\cA_V$. More precisely, the scheme-theoretical zero locus of the induced section
\[\overline{\gamma\otimes1} \in \frac{H^4_{DR}(\cX_V/\cA_V)}{F^2H^4_{DR}(\cX_V/\cA_V)}\]
is the closed point in $\cA_V$. Indeed, choose a subspace $W \subseteq V$ and let 
\[\cA_W:=\Spec(\C\oplus W\dual) \subseteq \cA_V\]
be the corresponding subscheme. Assume that the section $\overline{\gamma\otimes1}$ vanishes on $\cA_W$, i.e.,
\[\overline{\gamma\otimes1}|_{\cA_W} =0 \in \frac{H^4_{DR}(\cX_W/\cA_W)}{F^2H^4_{DR}(\cX_W/\cA_W)}.\]
Equivalently, the unique horizontal lift of $\gamma$ over $\cA_W$ is a Hodge class.


By \cite[Prop.~4.2]{Blo} (cf. Lemma \ref{Lem:Obstructions}(1)), we deduce that the composition
\[\sfB_\gamma|_W : W \hookrightarrow H^1(X,T_X) \xrightarrow{\iota_{-}(\gamma)} H^3(X,\Omega^1_X) \cong H^1(X,T_X)\dual\]
is zero. Hence we have
\[ W \subseteq V \cap \ker(\sfB_\gamma) =0\]
by the non-degeneracy of $V$. This proves the claim.

We will now show that any morphism
\begin{equation}\label{eq:S5.3}
t : T \to {\curly P}^{(q)}_{\tv}(\cX_V/\cA_V) 	
\end{equation}
from an affine scheme $T$ factors through $\PTqvX$. Indeed, let 
\[I\udot=[\O_{\cX_T} \xrightarrow{s} F]\]
be the family of $\PT_q$ pairs that corresponds to the map \eqref{eq:S5.3}, where
\[\xymatrix{
\cX_T \ar[r] \ar[d] & \cX_V \ar[d]\\
T \ar[r]^t & \cA_V
}\]
is the induced fibre diagram. By Proposition \ref{lem:chishorizontal}, we have 
\[t^*(\gamma\otimes 1) = \widetilde{\ch_2(I\udot)} \in F^2H^4_{DR}(\cX_T/T) \]
where $\widetilde{\ch_2(I\udot)} \in F^2 H^4_{DR}(\cX_T/T)$ is the unique horizontal lift of the Chern character $\ch_2(I\udot):=\tfrac{1}{2} \tr(\At^2(I\udot)) \in H^2(\cX_T,\Omega_{\cX_T/T}^2)$. Thus we have
\[t^*(\overline{\gamma\otimes 1}) = 0 \in \frac{H^4_{DR}(\cX_T/T)}{F^2H^4_{DR}(\cX_T/T)}.\]
Hence the map $t:T \to \cA_V$ factors through the zero locus of $\overline{\gamma\otimes1}$, which is the unique closed point in $\cA_V$ by the result in the previous paragraph. Now the fibre diagram \eqref{eq:S5.4} proves the desired property.

Finally, since the map $\PTqvX \hookrightarrow {\curly P}^{(q)}_{\tv}(\cX_V/\cA_V)$ is a closed embedding, the result in the previous paragraph proves that the map is an isomorphism.
\end{proof}

By Proposition \ref{prop:P(X)=P(X/A)}, the reduced obstruction theory of $\PTqvX$ can be obtained as the obstruction theory of ${\curly P}^{(q)}_{\tv}(\cX_V/\cA_V)$.

\begin{proof}[Proof of Proposition \ref{prop:redSOT}]
By Proposition \ref{prop:P(X)=P(X/A)}, the fibre diagram \eqref{eq:S5.4} can be expressed as
\[\xymatrix{
 & \cA_V \ar[d]\\
\PTqvX \ar[r] \ar[ru] & \Spec(\C) \ar@/^0.4cm/[u]
}\]
where the two triangles commute. Hence we have a canonical direct sum decomposition of the cotangent complexes
\begin{equation}\label{eq:S5.5}
\xymatrix{
\LL_{\PTqvX} \ar@<.4ex>[r] & \LL_{\PTqvX/\cA_V} \ar@<.4ex>[l] \ar@<.4ex>[r] & \LL_{\cA_V}|_{\PTqvX}[1] \ar@<.4ex>[l]
}	
\end{equation}
where the upper arrows and the lower arrows are the canonical distinguished triangles (cf. \cite[Lem.~2.10]{Qu}). 
Let
\[\phi_{\cX_V/\cA_V} : \EE_{\cX_V/\cA_V} \to \trunc \LL_{\PTqvX/\cA_V}\]
be the standard {\em relative} obstruction theory for ${\curly P}_{\tv}^{(q)}(\cX_V/\cA_V) \to \cA_V$ where $\EE_{\cX_V/\cA_V} = \EE$ by Proposition \ref{prop:P(X)=P(X/A)}.
We will consider the induced {\em absolute} obstruction theory. Indeed, form a morphism of distinguished triangles
\begin{equation}\label{eq:S5.6}
\xymatrix{
\EE^\red_V \ar@{.>}[d]^{\phi^\red_V} \ar[r] & \EE \ar[r]^-{\SR_V[1]} \ar[d]^{\phi_{\cX_V/\cA_V}} & V\dual\otimes \O_{\PTqvX}[1] \ar@{=}[d]\\
\trunc\LL_{\PTqvX} \ar[r] & \trunc \LL_{\PTqvX/\cA_V}  \ar[r]^-{\KS} & \Omega^1_{\cA_V}|_{\PTqvX}[1]
}	
\end{equation}
for some dotted arrow $\phi^\red_V$, where the upper distinguished triangle is induced by \eqref{eq:EEredV}, the lower distinguished triangle is the truncation of \eqref{eq:S5.5}, and the right square commutes by Lemma \ref{lem:compatibilityofKS} below. 
By the long exact sequence associated to \eqref{eq:S5.6}, we can deduce that the map $\phi^\red_V:\EE^\red_V \to \trunc \LL_{\PTqvX}$ is an obstruction theory, in the sense of \cite{BF}.
Moreover since the lower distinguished triangle in \eqref{eq:S5.6} splits, the obstruction theory $\phi^\red_V :\EE^\red_V \to \trunc \LL_{\PTqvX}$ is uniquely determined by the diagram \eqref{eq:S5.6}. Thus the map $\phi^\red_V$ in \eqref{eq:S5.6} is equal to the composition \eqref{eq:S5.redSOT}
since the composition 
\[\EE = \EE_{\cX_V/\cA_V}\xrightarrow{\phi_{\cX_V/\cA_V}} \tau^{\geq -1} \LL_{\PTqvX / \cA_V} \to \tau^{\geq -1} \LL_{\PTqvX}\]
is the original (non-reduced) obstruction theory for $\PTqvX$.
\end{proof}

We need the following compatibility result between the Kodaira-Spencer maps to complete the proof of Proposition \ref{prop:redSOT}. Here we consider a more general version than needed for Proposition \ref{prop:redSOT}. This will be used later in the proof of deformation invariance in Theorem \ref{Thm:DefInv}. 

\begin{lemma}\label{lem:compatibilityofKS}
Let $f:\cX\to\cB$ be a smooth projective morphism of relative dimension $4$ over a quasi-projective scheme $\cB$. Let $\curP:=\PTqvXB$ be the relative moduli space of $\PT_q$ pairs on the fibres of $f:\cX\to\cB$. Let
\[\xymatrix{
\curP\times_{\cB}\cX \ar[d]^{\pi} \ar[r] & \cX \ar[d]^f \\
\curP:=\PTqvXB \ar[r] & \cB
}\]
be a fibre diagram. Then we have a commutative diagram
\begin{equation}\label{eq:S5.8}
\xymatrix{
\EE_{\cX/\cB}:=R\hom_{\pi}(\II\udot,\II\udot\otimes \omega_{\cX/\cB})_0[3] \ar[r]^-{\SR_{\cX/\cB}} \ar[d]^{\phi_{\cX/\cB}} & Rf_*(\Omega_{\cX/\cB}^1 \otimes \omega_{\cX/\cB})|_{\curP}[4] \ar[d]^{-\KS_{\cX/\cB}}\\
\LL_{\curP/\cB} \ar[r]^{\KS_{\curP/\cB}} & \LL_{\cB}|_{\curP}[1]
}	
\end{equation}
where the relative obstruction theory $\phi_{\cX/\cB}$ and the relative cosection $\SR_{\cX/\cB}$ are given by the two components
\begin{align*}
	&\At_{\curP\times_\cB \cX/\cX}(\II\udot) : \II\udot \xrightarrow{\At_{\curP\times_\cB \cX/\cB}(\II\udot)} \II\udot \otimes \LL_{\curP \times_{\cB} \cX/\cB}[1] \xrightarrow{\mathrm{pr}_1} \II\udot \otimes \LL_{\curP/\cB}[1]\\
	&\At_{\curP\times_\cB \cX/\curP}(\II\udot) : \II\udot \xrightarrow{\At_{\curP\times_\cB \cX/\cB}(\II\udot)} \II\udot \otimes \LL_{\curP \times_{\cB} \cX/\cB}[1] \xrightarrow{\mathrm{pr}_2} \II\udot\otimes \Omega^1_{\cX/\cB}[1]
\end{align*}
of the Atiyah class $\At_{\curP\times_\cB \cX/\cB}(\II\udot)$ of the universal complex $\II\udot$.
For $\phi_{\cX/\cB}$, we use the adjunction $R \pi_*(-) \dashv L\pi^*(-) \otimes \omega_{\cX/\cB}[4]$.
\end{lemma}

\begin{proof}
The proof is similar that that of Lemma \ref{Lem:Obstructions}.
Consider the diagram
\begin{equation}\label{eq:S5.7}
\xymatrix@C+4pc{
& \LL_{\cX/\cB}|_{\curP\times_\cB \cX} \ar@<.8ex>[d] \ar[rd]^{\KS_{\cX/\cB}} & \\
R\hom_{\curP\times_\cB \cX}(\II\udot,\II\udot)[-1] \ar[r]|-{\At_{\curP\times_\cB \cX/\cB}(\II\udot)} \ar@{.>}[ru]^-{\At_{\curP\times_\cB \cX/\curP}(\II\udot)} \ar@{.>}[rd]_-{\At_{\curP\times_\cB \cX/\cX}(\II\udot)} & \LL_{\curP\times_\cB \cX/\cB} \ar[r]|-{\KS_{\curP\times_\cB \cX/\cB}} \ar@{.>}@<.8ex>[u]^-{\mathrm{pr}_2}\ar@{.>}@<-.8ex>[d]_-{\mathrm{pr}_1}
 & \LL_{\cB}|_{\curP\times_\cB \cX}[1]\\
& \LL_{\curP/\cB}|_{\curP\times_\cB \cX} \ar@<-.8ex>[u] \ar[ru]_{\KS_{\curP/\cB}}
}	
\end{equation}
where the left two triangles commute with the dotted arrows, and the right two triangles commute with the solid arrows.
Since $\At_{\curP\times_\cB\cX/\cB}(\II\udot)$ factors through $\At_{\curP\times_\cB\cX}(\II\udot)$, the composition of the two horizontal arrows in \eqref{eq:S5.7} vanishes
\[\KS_{\curP\times_\cB\cX/\cB} \circ \At_{\curP\times_\cB\cX/\cB}(\II\udot)=0.\]
We deduce that the big square in \eqref{eq:S5.7} commutes (up to sign) 
\begin{equation}\label{eq:S5.9}
\KS_{\cX/\cB} \circ \At_{\curP\times_\cB\cX/\curP}(\II\udot) = - \KS_{\curP/\cB}\circ \At_{\curP\times_{\cB}\cX/\cX}(\II\udot).	
\end{equation}
The commutativity of \eqref{eq:S5.8} follows from \eqref{eq:S5.9} by adjunction.
\end{proof}
%


\begin{remark}\label{Rem2}
The main limitation of the algebraic twistor family approach in this subsection is that the reduced obstruction theory is given as a {\em composition}, but not as a {\em factorization},
\begin{equation}\label{Eq:4.22.1}
\xymatrix{
\EE^\red_V \ar[r] \ar[rd]_{\phi^\red_V} \ar@/^.6cm/[rr]^{\id} & \EE \ar[d]^{\phi} \ar[r]& \EE^\red_V \ar@{.>}[ld]^{?}\\
& \trunc \LL_{\PTqvX} .& 
}    
\end{equation}
In classical reduced theory, e.g.~\cite{KT1}, this was sufficient to construct reduced Behrend-Fantechi virtual cycles \cite{BF}.
However, in DT4 theory, this is not sufficient to construct reduced Oh-Thomas virtual cycles \cite{OT} since we need to know whether the reduced obstruction theory $\phi^\red_V$ satisfies the {\em isotropic condition}
\[\xymatrix{
& \fQ(\EE) \ar@{.>}[r]^-{\not\exists} \ar@{^{(}->}[d] & \fQ(\EE^\red_V) \ar@{^{(}->}[d] \\
\fC_{\PTqvX} \ar@{^{(}->}[r] \ar@/_0.6cm/@{^{(}->}[rr]_{\text{algebraic twistor}}\ar@{.>}[rru]|>>>>>>>>>{?} \ar@{^{(}->}[ru]& \fC(\EE) \ar@{->>}[r] & \fC(\EE^\red_V).
}\]
Thus we additionally need the cone reduction property in the previous subsection (Proposition \ref{prop:conereduction}) to construct a reduced virtual cycle from $\phi^\red_V$. 

To conclude, the cone reduction property in Proposition \ref{prop:conereduction} and the reduced obstruction theory in Proposition \ref{prop:redSOT} give us a commutative diagram
\begin{equation}\label{Eq:4}
\xymatrix@C+4pc{
(\fC_{\PTqvX})_\red \ar@{^{(}->}[d] \ar@{^{(}->}[r]^{\text{cone reduction}} & \fQ(\EE^\red_V) \ar@{^{(}->}[d] \\
\fC_{\PTqvX} \ar@{^{(}->}[r] \ar@{^{(}->}[r]_{\text{algebraic twistor}}  & \fC(\EE^\red_V).
}    
\end{equation}
This proves Theorem \ref{thm:4}.

In Appendix \ref{Appendix:ReductionviaDAG}, we will construct a diagonal arrow $\fC_{\PTqvX} \hookrightarrow \fQ(\EE^\red_V)$ in commutative diagram \eqref{Eq:4} using derived algebraic geometry, see Remark \ref{Rem:A16}.
\end{remark}



\begin{remark}\label{Rem:ReducedDerived}
By Proposition \ref{prop:P(X)=P(X/A)}, the standard derived enhancement on ${\curly P}^{(q)}_v(\cX_V/\cA_V)$ gives us a {\em reduced derived enhancement} on $\PTqvX$.
However, this derived enhancement is not necessarily homotopically finitely presented over $\Spec(\C)$ and does not give us the reduced virtual cycle in Theorem \ref{Thm:RVFC}.
In Appendix \ref{Appendix:ReductionviaDAG}, we will construct a different reduced derived enhancement on $\PTqvX$ that has a $(-2)$-shifted symplectic structure and induces the desired reduced virtual cycle.
\end{remark}

\subsection{Reduced virtual dimension}\label{ss:rvd}
In this subsection, we present various methods to compute the number $\rho_\gamma$ and the reduced virtual dimension.

\noindent

\hfill

\noindent \textbf{Local complete intersection surfaces.}
Let $X$ be a Calabi-Yau 4-fold. We take $i : S \hookrightarrow X$ a \emph{local complete intersection} (lci) surface, $v = \ch(\O_S) = (0,0,\gamma,\beta,\chi(\O_S)-\gamma \cdot \td_2(X)) \in H^*(X)$. Our goal is to make the key diagram from Proposition \ref{prop:ob=srdual} more explicit and determine the reduced virtual dimension $\mathrm{rvd} = \chi(\O_S)-\tfrac{1}{2}\gamma^2+\tfrac{1}{2}\rho_\gamma$.
The normal sheaf $N_{S/X}$ is locally free because $S$ is lci. Since $X$ has trivial canonical bundle, we have
$$K_S\cong N_{S/X} \wedge N_{S/X}$$
and the Serre duality pairing gives 
a non-degenerate symmetric bilinear pairing 
\begin{equation*}\label{eqn:SD_S}
    \mathsf{SD}_S : H^1(S,N_{S/X}) \cong H^1(S,N_{S/X})^\vee.
\end{equation*}

The space $H^1(S,N_{S/X})$ maps to the obstruction space of $I_{S/X}$. From the short exact sequence
\begin{equation}\label{idealses}
    0\to I_{S/X}\to \O_X\to \O_S\to 0\,,
\end{equation}
we get an exact triangle
\[R\Hom_X(I_{S/X},\O_S)\to R\Hom_X(I_{S/X},I_{S/X})[1]\to R\Hom_X(I_{S/X},\O_X)[1]\,.\]
By taking cones of maps from $R\Gamma(X,\O_X)[1]$ to the second and the third factor, we get an exact triangle
\begin{equation}\label{eqn:4.6.1}
    R\Hom_X(I_{S/X},\O_S)\to R\Hom_X(I_{S/X},I_{S/X})_0[1]\to R\Hom_X(\O_S,\O_X)[2]\,.
\end{equation}
Consider the map
\begin{equation}\label{eqn:4.6.2}
   \theta_S: H^1(S,N_{S/X})\to\Ext^1_X(I_{S/X},\O_S)\to \Ext^2_X(I_{S/X},I_{S/X})_0\,,
\end{equation}
where the first map is defined by adjunction. 

We give local descriptions of the obstruction map \eqref{eq:ob} and the semi-regularity map \eqref{eq:sr} for local complete intersection surfaces.
The map defined by 
\begin{equation*}\label{eqn:res}
    \sfob_{S/X} : H^1(X,T_X)\to H^1(S,T_X|_S)\to H^1(S,N_{S/X})
\end{equation*}
is the obstruction map for deforming $S\subset X$ along a deformation of $X$ \cite[Prop.~2.6]{Blo}. Let $d_S: \O_S\to I_{S/X}[1]$ be the map induced by \eqref{idealses}. Let $\alpha:\Ext^1_X(I_{S/X},\O_S)\to H^3(X,\Omega_X^1)$ be the map defined by the trace of 
\[I_{S/X}\to \O_S[1] \xrightarrow{\At(\O_S)} \O_S\otimes \Omega_X^1[2] \xrightarrow{d_S\otimes \id} I_{S/X}\otimes \Omega_X^1[3]\,.\] 
The local semi-regularity map is defined by
\begin{equation}\label{eqn:Lsr}
    \sr_{S/X}: H^1(S,N_{S/X})\to \Ext_X^1(I_{S/X},\O_S)\xrightarrow{\alpha}H^3(X,\Omega_X^1).
\end{equation}

The main result of this subsection consists of the following compatibilities.
\begin{proposition}\label{prop:4.6.1}
Let $i : S\hookrightarrow X$ be a local complete intersection surface inside a Calabi-Yau 4-fold $X$.
\begin{enumerate}
    \item[$(1)$] The map \eqref{eqn:4.6.2} is an isomorphism between two vector spaces with non-degenerate symmetric bilinear forms 
    \[\theta_S : \left(H^1(S,N_{S/X}),\SD_S\right)  \xrightarrow{\cong} \left(\Ext_X^2(I_{S/X}, I_{S/X})_0,\sfY\right).\]
    \item[$(2)$] The following diagram commutes
    \begin{equation*}
    \xymatrix@R=1pc@C=1.2pc{
    H^1(X,T_X)\ar[rr]^{\sfob_{S/X}}\ar[ddrr]_{\sfob} & & H^1(S,N_{S/X})\ar[rr]^{\sr_{S/X}}\ar[dd]^{\theta_S}_{\cong} \ar @{} [dl]^{(2')} \ar @{} [dr]_{(2'')}  & & H^3(X,\Omega_X^1).\\
    &&&&&\\
    & & \Ext^2_X(I_{S/X},I_{S/X})_0\ar[uurr]_{\sr} &&
    }
    \end{equation*}
\end{enumerate}
\end{proposition}
\begin{proof}To simplify the notation, denote $I:=I_{S/X}$. We first prove that the map $\theta_S$ is an isomorphism. The compatibility of the two bilinear forms will be proved at the end. The exact triangle \eqref{eqn:4.6.1} gives a long exact sequence
\begin{align}
\begin{split} \label{eqn:les1}
0 &\to \underbrace{\Ext^2_X(\O_S,\O_X)}_{\cong H^0(S,K_S)} \xrightarrow{\textup{(i)}} \Ext^1_X(I,\O_S) \xrightarrow{\textup{(ii)}} \Ext^2_X(I,I)_0 \\
&\to \underbrace{\Ext^3_X(\O_S,\O_X)}_{\cong H^1(S,K_S)} \xrightarrow{\textup{(iii)}} \Ext^2_X(I,\O_S) \xrightarrow{\textup{(iv)}} \Ext^3_X(I,I)_0 \\
&\to \Ext^4_X(\O_S,\O_X)\xrightarrow{\textup{(v)}} \Ext^3_X(I,\O_S)\to 0
\end{split}
\end{align}
since $\Hom_X(I,\O_S)\to \Ext^1_X(I,I)_0$ is an isomorphism (Theorem \ref{Thm:PairtoPerf}).
The map (v) is Serre dual to \[\Ext^1_X(\O_S,I)\xrightarrow{\textup{(v')}} \Hom_X(\O_X,\O_S).\] Applying $\ext\udot_X(\O_S,-)$ to \eqref{idealses}, we have  
\[\hom_X(\O_S,I)=0, \quad \ext^1_X(\O_S,I)\cong \hom_X(\O_S,\O_S)\]
since $\hom_X(\O_S,\O_X)=\ext^1_X(\O_S,\O_X)=0$. Therefore
(v') can be identified with the map induced by the isomorphism $\hom_X(\O_S,\O_S)\cong \hom_X(\O_X,\O_S)$ and (iv) is surjective.

Consider the local-to-global spectral sequence 
\[H^p(X,\ext^q_X(I,\O_S)) \Rightarrow \Ext^{p+q}_X(I,\O_S).\]
Since $S$ is lci, there are isomorphisms
$$
\ext^q_X(I,\O_S) \cong \ext^{q+1}_X(\O_S,\O_S) \cong \bigwedge^{q+1} N_{S/X}, \quad \forall q\geq 0
$$
by \cite[Prop.VII.2.5]{SGA6}.
Since $\ext^q_X(I,\O_S)=0$, for all $q\geq 2$, the local-to-global spectral sequence collapses to
the long exact sequence
\begin{align}
\begin{split}  \label{eqn:les2}
0 &\to H^1(S,N_{S/X}) \xrightarrow{} \Ext^1_X(I,\O_S) \xrightarrow{(a)} \underbrace{H^0(S,N_{S/X} \wedge N_{S/X})}_{H^0(S,K_S)} \\
& \to H^2(S,N_{S/X}) \xrightarrow{} \Ext^2_X(I,\O_S) \xrightarrow{(b)} \underbrace{H^1(S,N_{S/X} \wedge N_{S/X})}_{H^1(S,K_S)} \to 0.
\end{split}
\end{align}
Combined with \eqref{eqn:les1}, we obtain morphisms
\begin{align*}
&H^0(S,K_S) \xrightarrow{\textup{(i)}} \Ext^1_X(I,\O_S) \xrightarrow{(a)} H^0(S,K_S), \\
&H^1(S,K_S) \xrightarrow{\textup{(iii)}} \Ext^2_X(I,\O_S) \xrightarrow{(b)} H^1(S,K_S).
\end{align*}
We show that the composition is an isomorphism. Since $\O_S$ is Cohen-Macaulay, we have $\ext^q_X(\O_S,\O_X)=0$, for $q=3,4$ (Lemma \ref{Lem.PT0=PT1}). Therefore we have $H^p(X,\ext^2_X(\O_S,\O_X))\cong \Ext^{p+2}_X(\O_S,\O_X)$ for all $p$. By functoriality of the local-to-global spectral sequence, the map $(a)\circ (i)$ is induced by
\[\xymatrix{
\Ext_X^2(\O_S,\O_X)\ar[r]^{\textup{(i)}} & \Ext^1_X(I,\O_S) \ar[d]^{(a)}\\
H^0(X,\ext^2_X(\O_S,\O_X))\ar[u]_{\cong}\ar[r]^{(\star)} & H^0(X,\ext^2_X(\O_S,\O_S))
}\]
where the map $(\star)$ is induced by applying $\ext\udot_X(\O_S,-)$ to \eqref{idealses}. By the fundamental local isomorphism, $(\star)$ is induced by 
\[\ext_X^2(\O_S,\O_X)\cong K_S\otimes i^*\O_X\cong \ext^2_X(\O_S,\O_S)\,.\]
Hence the composition is an isomorphism. Similar arguments show that the second composition is also an isomorphism. Consequently, \eqref{eqn:les1} and \eqref{eqn:les2} split into two short exact sequences and hence the composition
\[
\begin{tikzcd}
\theta_S: H^1(S,N_{S/X})\ar[r,hook] & \Ext^1_X(I,\O_S)\ar[r,twoheadrightarrow, "\textup{(ii)}"] &\Ext^2_X(I,I)_0
\end{tikzcd}
\]
is an isomorphism.

The commutativity $(2'')$ follows from the compatibility of $\At(I)$ and $\At(\O_S)$ along $d_S$. To prove the commutativity $(2')$, we first show that the following diagram commutes
\begin{equation}\label{eqn:Os}
    \begin{tikzcd}
    & \Ext^2_X(\O_S,\O_S) \\
    H^1(X,T_X) \ar[r,"\ob_{S/X}"]\ar[ur,"\sfob(\O_S)"] & H^1(S,N_{S/X}). \ar[u,"\theta_S'"] 
    \end{tikzcd}
\end{equation}
Here the map $\theta_S'$ is obtained similarly to \eqref{eqn:4.6.2}. Consider the composition
\[\theta_S'' : H^1(S, N_{S/X})\xrightarrow{\cong} \Ext^1_S(I/I^2,\O_S) \to \Ext^1_X(i_*I/I^2,i_*\O_S)\to \Ext^2_X(\O_S,\O_S)\,,\]
where the second map is the pushforward and the third map is induced by the exact sequence
\begin{equation}\label{eqn:I2}
    0\to I/I^2\to \O_X/I^2\to \O_S\to 0\,.
\end{equation}
The map $\theta_S'$ can be identified with $\theta''_S$ by comparing \eqref{idealses} and \eqref{eqn:I2}.
Therefore the commutativity of \eqref{eqn:Os} follows from the commutative diagram
\[
\begin{tikzcd}[column sep=large]
\O_S \ar[r,"e_S"] \ar[d,equal] & I/I^2[1]\ar[d,"\delta"]\\
\O_S \ar[r,"\At(\O_S)"] & \O_S\otimes\Omega^1_X[1]\,,
\end{tikzcd}
\]
where $e_S$ is induced from \eqref{eqn:I2} and $\delta$ is defined by the conormal sequence. By the naturality of Atiyah classes, the following diagram commutes
\[
\begin{tikzcd}[column sep=large]
H^1(X,T_X)\ar[r,"\sfob"] \ar[d,"\sfob(\O_S)"] & \Ext^2_X(I,I)_0\ar[d,"d_S"]\\
\Ext^2_X(\O_S,\O_S)\ar[r,"d_S"] & \Ext^3_X(\O_S,I)\,.
\end{tikzcd}
\]
The right vertical map is injective because it is dual to (ii). Therefore if suffices to check $(2')$ after post-composing with $d_S$. Since \eqref{eqn:Os} commutes, this follows from  the following diagram
\begin{equation*}\label{eqn:Os2}
    \begin{tikzcd}[column sep=large]
    H^1(S,N_{S/X})\ar[r,"\theta_S"]\ar[d,"\theta_S'"] & \Ext^2_X(I,I)_0\ar[d,"d_S"]\\
    \Ext^2_X(\O_S,\O_S) \ar[r,"d_S"] & \Ext^3_X(\O_S,I).
    \end{tikzcd}
\end{equation*}



Finally we prove that the map $\theta_S$ preserves pairings. We first show that $\theta_S'$ preserves pairings. Let $i^!$ be the right adjoint of $i_*$. Then $\theta_S'$ decomposes as
\begin{align*}
    \Ext^1_S(I/I^2,\O_S)\to \Ext^1_S(I/I^2,i^!i_*\O_S)\xrightarrow{\cong} \Ext^1_X(i_*I/I^2,i_*\O_S)\xrightarrow{e_S} \Ext^2_X(i_*\O_S,i_*\O_S)
\end{align*}
where the first map is the adjoint, the second map is Grothendieck duality, and the third map is induced by \eqref{eqn:I2}. By the functoriality of Serre duality, we are reduced to checking commutativity of the diagram
\[\xymatrix{
\Ext^1_S(I/I^2,i^!i_*\O_S)\ar[r]^\cong\ar[d]^{\SD_S} &\Ext^1_X(i_*I/I^2,i_*\O_S)\ar[d]^{\SD_X}\\
\Ext^3_S(Li^*i_*\O_S,I/I^2)^\vee &\Ext^3_X(i_*\O_S,i_*I/I^2)^\vee\ar[l]_\cong
}\]
where the bottom isomorphism is induced by the adjunction $Li^*\dashv i_*$.
Since $S$ is lci, we have $i^!=\SD_S\circ Li^*\circ \SD_X^{-1}$.  Therefore the above diagram commutes and $\theta_S'$ preserves pairings.
The pullbacks of quadratic forms on $\Ext^2_X(I,I)_0$ and $\Ext^2_X(\O_S,\O_S)$ to $\Ext^1_X(I,\O_S)$ coincide because $\tr(d_S\circ u \circ d_S \circ u) = \tr(u\circ d_S\circ u \circ d_S)$ for $u\in \Ext^1_X(I,\O_S)$. Since $\theta_S,\theta_S'$ both factor through $\Ext^1_X(I,\O_S)$, the compatibility of quadratic forms under pullback along $\theta_S$ follows from compatibility under pullback along $\theta_S'$.
\end{proof}
\begin{remark}
The map \eqref{eqn:Lsr} coincides with Bloch's semi-regularity map \cite[Prop.~8.2]{BF03}. Furthermore, we have a map 
$$
\sr' = \tr(\At(\O_S) \circ -) : \Ext_X^2(\O_S,\O_S) \to H^3(X,\Omega_X^1).
$$
By Proposition \ref{prop:4.6.1}$(2'')$,
compatibility of Atiyah classes, and cyclicity of trace, the following diagram commutes
\[\begin{tikzcd}
    \Ext^2_X(\O_S,\O_S)\ar[dr,"\sr'"] &\\
    H^1(S,N_{S/X}) \ar[u,"\theta_S'"]\ar[r,"\sr_{S/X}"] & H^3(X,\Omega_X^1).
\end{tikzcd}\]
\end{remark}

The following simple observation is useful for computing the reduced virtual dimension. 
\begin{remark} \label{rem:ATA}
Let $J$ be a complex invertible symmetric $m \times m$ matrix. In general, a complex $m \times n$ matrix $A$ does not satisfy $\rk(A^TJA)=\rk(A)$. This equality holds if and only if $J$ is non-degenerate on $\im(A)$.
\end{remark}
We now compute the reduced virtual dimension when $v=\ch(\O_S)$ for some lci surfaces.
\begin{corollary}\label{cor:4.6.4}
Let $i : S\hookrightarrow X$ be a local complete intersection surface inside a Calabi-Yau 4-fold $X$. If $v=\ch(\O_S)$ then
\[\mathrm{rvd} = \frac{1}{2}\chi(N_{S/X}) + \frac{1}{2} \mathrm{rk}(\sfob_{S/X}^\vee \circ \SD_S \circ \sfob_{S/X}).\]
In particular when $\sfob_{S/X}$ is surjective, i.e.~$S$ is semi-regular, we have
$$
\mathrm{rvd} = h^0(S,N_{S/X}).
$$
\end{corollary}
\begin{proof}
By the Hirzebruch-Riemann-Roch theorem we have
$$
\chi(N_{S/X}) = \int_S \ch(N_{S/X}) \, \td(S) = 2 \chi(\O_S) -  \int_S c_2(N_{S/X}) = 2 \chi(\O_S) - [S]^2,
$$
where we used $c_1(N_{S/X}) = c_1(\bigwedge^2 N_{S/X}) = c_1(K_S)$ and the self-intersection formula $i^* i_* [S] = c_2(N_{S/X})$. Thus the formula for the reduced virtual dimension follows from Proposition \ref{prop:4.6.1}. In the semi-regular case, we have $\mathrm{rk}(\sfob_{S/X}^\vee \circ \SD_S \circ \sfob_{S/X}) = \mathrm{rk}(\sfob_{S/X}) = h^1(N_{S/X})$ (Remark \ref{rem:ATA}). The result follows from the fact that $h^0(N_{S/X}) = h^2(N_{S/X})$ (see also Theorem \ref{Thm:SR=smoothofrvd}).
\end{proof}

\noindent
\textbf{Complete intersection Calabi-Yau 4-folds}
The following result is inspired by Steenbrink \cite{S87}.
\begin{corollary}\label{cor:cicy4}
Let $X\subset \PP^N$ be a complete intersection Calabi-Yau 4-fold defined by a regular section of $E_1=\oplus_{i=1}^{N-4} \O_{\PP^N}(e_i)$, $e_i\geq 2$. Suppose $S\subset X$ is a complete intersection surface of $\PP^N$ defined by a regular section of $E_2=\oplus_{j=1}^{N-2} \O_{\PP^N}(d_j)$, $d_j\geq 1$. Then the ideal sheaf $I_{S/X}$ is semi-regular. 
Moreover if $v=\ch(\O_S)$, then  
\[\rvd = \dim(\ker(H^0(E_2|_S)\to H^0(E_1|_S))\]
where the map is induced by the morphism $E_2|_S = N_{S/\PP^N} \to N_{X/\PP^N}|_S = E_1|_S$.
\end{corollary}
\begin{proof}
By the short exact sequence 
\begin{equation} \label{ses:N}
0\to N_{S/X}\to N_{S/\PP^N}\to N_{X/\PP^N}|_S\to 0
\end{equation}
and Proposition  \ref{prop:4.6.1}, it is enough to show that $\sfob_{S/X}$ is surjective. 
We note that the following diagram commutes
\[\begin{tikzcd}
H^0(\PP^N,E_1)\ar[r]\ar[dr] &H^0(S,E_1|_S)\ar[r] &H^1(S,N_{S/X})\\
& H^0(X,E_1|_X)\ar[r]\ar[u] & H^1(X,T_X).\ar[u,"\sfob_{S/X}"']
\end{tikzcd}\]
Therefore it suffices to show that two top horizontal maps are surjective. Consider a Koszul resolution $\bigwedge^{\bullet+1}E_2^\vee\to I_{S/\PP^N}$ and the associated spectral sequence \cite[Tag 07A9]{Sta} 
\[\Ext_{\PP^N}^p(E_1\otimes \bigwedge^{q+1}E_2^\vee,K_{\PP^N})\Rightarrow \Ext^{p+q}_{\PP^N}(E_1\otimes I_{S/\PP^N},K_{\PP^N})\,.\]
Since, on line bundles, $H^p(\PP^N,-)$ is nonzero only when $p=0$ or $N$, this spectral sequence degenerates and we have $0=\Ext^{N-1}_{\PP^N}(E_1\otimes I_{S/\PP^N},K_{\PP^N})\cong H^1(\PP^N,E_1\otimes I_{S/\PP^N})^\vee$. Therefore the first horizontal map is surjective. Similarly we have $H^2(\PP^N,E_2\otimes I_{S/\PP^N})= 0$. 
Consequently $H^1(S,E_2|_S)=H^1(S,N_{S/\PP^N})=0$ and hence the second horizontal map is surjective.
Therefore $\sfob_{S/X} = (\sr_{S/X})^\vee$ is surjective and $I_{S/X}$ is semi-regular. 
\end{proof}
Corollary~\ref{cor:cicy4} allows us to compute $\rvd$ and $\rho_\gamma$ for semi-regular cases. 
\begin{example} 
Suppose $X:=X_{e_1, \ldots, e_{N-4}} \subset \PP^N$ is a complete intersection Calabi-Yau 4-fold.  The possible cases are $X_6$, $X_{2,5}$, $X_{3,4}$, $X_{2,2,4}$, $X_{2,3,3}$, $X_{2,2,2,3}$, $X_{2,2,2,2,2}$. Consider complete intersection surfaces $S \subset \PP^N$ of type $(d_1,\ldots, d_{N-2})$ such that $S \subset X$. By Corollary \ref{cor:cicy4} the ideal sheaf $I_{S/X}$ is semi-regular.
\begin{enumerate}
    \item Suppose for all $i$ we have
\begin{equation} \label{ineq:rigid}
- d_i + \sum_{j=1}^{N-2} d_j < N+1.
\end{equation}
Then we claim $S \subset X$ is rigid, i.e.~$h^0(N_{S/X})=0$, so in particular $\rvd = 0$. This can be seen by showing that $H^0(S,N_{S/X})^\vee \cong H^2(S,N_{S/X}) = 0$ as follows. Using \eqref{ses:N} it suffices to show that $H^1(S,E_1|_S) = H^2(S,E_2|_S) = 0$. This follows by showing $H^2(\PP^N,E_1 \otimes I_{S/X}) = H^3(\PP^N,E_2 \otimes I_{S/X}) = 0$ by a Koszul calculation as in the proof of Corollary \ref{cor:cicy4}. 
E.g.~for $N=5$ all solutions to \eqref{ineq:rigid} are (up to permutations)
\[(1,1,1), (1,1,2), (1,1,3), (1,2,2); (1,2,3), (2,2,2), (1,1,4); (2,2,3).\]
In cases 1--4, $S$ is del Pezzo. In cases 5--7, $S$ is K3. In case 8, $S$ is general type. Using the formula $\rvd = \chi(\O_S) - \tfrac{1}{2} \gamma^2 + \tfrac{1}{2} \rho_\gamma$, and determining $\chi(\O_S)$, $\gamma^2$ in terms of $d_1,d_2,d_3$, this produces the following values for $\rho_\gamma$ in the above cases
\[19, 32, 37, 54; 62, 92, 32; 106.\]
\item Suppose $X=X_6$ is defined by a homogeneous polynomial $f \in R:=\C[x_0,...,x_5]$ of degree 6 and $S \subset X$ is defined by homogeneous polynomials $(s_1,s_2,s_3)$ of degrees $(d_1,d_2,d_3)$, then $f = s_1t_1+s_2t_2+s_3t_3$ for some $t_1,t_2,t_3$. 
Using the Koszul resolution of $I_{S/\PP^5}$, as in proof of Corollary \ref{cor:cicy4}, it is easy to see that $H^1(\PP^5,I_{S/\PP^5}(n)) = 0$ for all $n$. If the ideal $(s_1,s_2,s_3) \subset R$ is saturated then $H^0(S,N_{S/X})$ is the kernel of the following map
\begin{align*}
H^0(E_2|_S) \cong \bigoplus_{i=1}^{3} \Big(R / (s_1,s_2,s_3) \Big)_{(d_i)} &\to H^0(E_1|_S) \cong \Big(R / (s_1,s_2,s_3) \Big)_{(6)} \\
(u_1,u_2,u_3) &\mapsto u_1t_1+u_2t_2+u_3t_3
\end{align*}
where $(-)_{(n)}$ denotes the degree $n$ part.
This can be implemented in {\tt{Maccaulay2}}. E.g.~for the Fermat sextic and certain $S$ of degrees $(1,3,3), (2,3,3),(3,3,3)$, we obtain $\rvd = h^0(N_{S/X}) = 1,1,3$ respectively. Using the formula $\rvd = \chi(\O_S) - \tfrac{1}{2} \gamma^2 + \tfrac{1}{2} \rho_\gamma$, we obtain that $\rho_\gamma$ is $71,122,141$ respectively. 
\end{enumerate}

The value of $\rho_\gamma$ can also be computed by Griffiths' residue calculus \cite{G68}. In the cases of the sextic 4-fold listed above, our values for $\rho_\gamma$ reproduce \cite[Table 17.3]{Mov}. 
Corollary~\ref{cor:cicy4} gives an algebraic method for determining $\rho_\gamma$.


\end{example}

\noindent
\textbf{Holomorphic symplectic 4-folds} When a Calabi-Yau 4-fold $X$ has a non-degenerate holomorphic 2-form, the rank $\rho_\gamma$ of $\sfB_\gamma$ is easy to determine. 
The following computation is inspired by Voisin \cite{Voi89}.
\begin{corollary}\label{cor:4.6.5}
Suppose $X$ has a non-degenerate holomorphic 2-form $\sigma\in H^0(X,\Omega_X^2)$. Let $i:S\hookrightarrow X$ be a (possibly disconnected) smooth lagrangian surface, i.e., $i^*\sigma=0$. If $v=\ch(\O_S)$ then
\[\rvd = h^{1,0}(S) - \frac{1}{2} \dim \coker(i^*)\]
where $i^*: H^1(X,\Omega_X^1)\to H^1(S,\Omega_S^1)$.
\end{corollary}
\begin{proof}
We prove the case when $S$ is connected. The disconnected case then follows easily. 
Since $S$ is a lagrangian, we have an isomorphism
\[\xymatrix{
0\ar[r] & N_{S/X}^*\ar[r]\ar[d]_\cong &\Omega^1_X|_S \ar[r]\ar[d]^\sigma_\cong & \Omega_S^1 \ar[r]\ar[d]_\cong & 0\\
0 \ar[r] &T_S \ar[r] &T_X|_S \ar[r] &N_{S/X} \ar[r] & 0. 
}\]
Therefore $\sfob_{S/X}$ can be identified with $i^*$. 
Under this identification, $\mathsf{SD}_S$ corresponds to the cup product. By the Hodge index theorem, the cup product is non-degenerate on the image of $i^*$. Therefore Remark \ref{rem:ATA} and Corollary \ref{cor:4.6.4} imply that the reduced virtual dimension equals
\begin{equation*}
\rvd =\frac{1}{2}\chi(\Omega^1_S) + \frac{1}{2} \mathrm{rk}(i^*). \qedhere
\end{equation*}
\end{proof}
\begin{example}
Suppose $X$ has a non-degenerate holomorphic 2-form. Consider a lagrangian plane $\PP^2 \cong S \subset X$ and let $v = \ch(\O_S)$. Since $i^*$ is non-trivial and $h^{1,1}(S) = 1$, $S$ is semi-regular. Moreover, $N_{S/X} \cong \Omega_{S}^1$ implies that $S$ is rigid and $\mathrm{rvd} = 0$. When $\curP_v^{(q)}(X)$ only contains lagrangian planes, $\deg ([\curP_v^{(q)}(X)]^{\red})$ is a (signed) count of lagrangian planes in $X$ in class $v$.
\end{example}


\begin{example}
In \cite{Sch} Schoen constructed minimal surfaces $S$ of general type related to the Hodge conjecture for abelian 4-folds. These surfaces satisfy $h^{1,0}(S)=4$, $h^{2,0}(S) = 5$, and $h^{1,1}(S) = 12$.
Let $S$ be a Schoen surface, then its Albanese variety $X=\mathrm{Alb}(S)$ is an abelian 4-fold. The Albanese map $i:S\to X$ is an embedding of $S$ as a lagrangian surface (for some non-degenerate $\sigma \in H^0(X,\Omega_X^2)$) \cite{CMLR}, and the cohomology class $[S] \in H^{2,2}(X) \cap H^4(X,\Q)$ is not the product of divisor classes on $X$ \cite{Sch}. Let $v=\ch(\O_S)$. From the constructions in \cite{Sch}, one can see that there exist semi-regular Schoen surfaces $S \subset X$. Then $i^*$ is surjective. 
Therefore $\rvd=4$ and $\rho_{[S]}=12$.
We revisit this example in future work.
\end{example}

\begin{remark} \label{rem:critnonsr}
Combining Theorem \ref{Thm:SR=smoothofrvd} and Corollary \ref{cor:4.6.5}, we obtain a numerical criterion for a smooth lagrangian surface $i: S \hookrightarrow X$ to be non-semi-regular. Consider $i^*: H^{1,1}(X)\to H^{1,1}(S)$. Indeed if $\rk(i^*)$ is \emph{odd} or $b_1(S) < \dim \coker(i^*)$, then $S$ is not semi-regular.

For example, let $\pi : X \to B$ be a hyperk\"ahler 4-fold with lagrangian fibration over a smooth base.  
The general fibre $S = F$ is a lagrangian abelian surface and the base $B$ is isomorphic to $\PP^2$ (\cite{Hwa, HX}). For $v=\ch(\O_F)$ and any $q\in\{-1,0,1\}$, we have a morphism
$$
B \to \PTqvX, \quad P \mapsto \pi^* I_{P/B}.
$$
This is a bijection on $\C$-points and using the fact that $R \pi_* \O_X \cong \O_B \oplus \Omega_B^1[-1] \oplus \Omega_B^2[-2]$ \cite{Mat1}, one can also show that it is an isomorphism on tangent spaces at $\C$-points. Thus, since $B$ is smooth, the map is an isomorphism.
Let $\Delta \subset B \times B$ denote the diagonal and $\pi_B : X \times B \to B$ projection.
Using again $R \pi_* \O_X \cong \O_B \oplus \Omega_B^1[-1] \oplus \Omega_B^2[-2]$, we obtain the following obstruction bundle
$$
\ext^2_{\pi_B}(\pi^* I_\Delta, \pi^* I_\Delta)_0 \cong T_B \otimes \Omega_B^1 \cong \mathcal{E}{\it{nd}}(T_B).
$$
The natural pairing $T_B \otimes \Omega_B^1\to \O_B$ is the unique cosection (up to scaling). Since $i^*$ is non-zero, we deduce $\rk i^* = 1$. A direct computation of $\rk i^* = 1$ can be found in \cite{SY,Mat2}. We deduce $\rvd = \tfrac{1}{2}$ so $S$ is not semi-regular.


\end{remark}
\begin{example} 
Let $C \subset \mathrm{K3}$ be a smooth curve of genus $g$ on a K3 surface. Then the Hilbert square $X = \mathrm{K3}^{[2]}$ is a hyperk\"ahler 4-fold and the symmetric product $S = C^{[2]}$ is a smooth lagrangian surface. Since $S$ is cut out transversally by a tautological section of the vector bundle $\O(C)^{[2]}$, we have
$$
\chi(\Omega^1_S) = \int_{S} \ch(\Omega^1_S) \cdot \td(S) = \int_X c_2(\O(C)^{[2]}) \ch(\Omega^1_X - \O(C)^{[2]\vee}) \, \frac{\td(X)}{\td(\O(C)^{[2]})}.
$$
This tautological integral can be easily determined e.g.~using \cite{EGL}. The result is $-(C^2)^2/4 = -(g-1)^2$. Moreover, $\rk i^* = 1$ when $g=0$ and $\rk i^* = 2$ otherwise (\cite[Sect.~3]{Voi89}). 
By Corollary~\ref{cor:4.6.5} we find
$$
\mathrm{rvd} = -\frac{1}{2}(g-1)^2 + \left\{ \begin{array}{cc} \frac{1}{2} & \textrm{if} \ g=0  \\ 1 & \textrm{otherwise.} \end{array} \right.
$$
This number is only non-zero for $g=0,1,2$, then $\mathrm{rvd} = 0,1,\tfrac{1}{2}$ respectively. In particular, for $g  \geq 2$ we deduce that $S$ is not semi-regular (Remark \ref{rem:critnonsr}). For $g=0$, $S \cong \PP^2$ is a lagrangian plane. 
\end{example}
%
When $X$ is a product of two $K3$ surfaces, we determine $\rho_\gamma$ for diagonal classes, see also Example \ref{ex:k3k3}.

\begin{proposition}\label{prop:k3k3}
Let $X$ be a product of two K3 surfaces $S_1$ and $S_2$. If $\gamma =\eta_1\cup\eta_2$ for some non-zero $\eta_1\in H^{1,1}(S_1)$ and $\eta_2\in H^{1,1}(S_2)$, then we have $\rho_\gamma=2$.
\end{proposition}
\begin{proof}
Let $\sigma_1\in H^{2,0}(S_1)$ and $\sigma_2\in H^{2,0}(S_2)$ be two nowhere vanishing holomorphic 2-forms. Take $\omega=\sigma_1\cup \sigma_2 \in H^0(X,\Omega_X^4)$. Consider the isomorphism \[H^1(X,T_X) \cong H^1(S_1, T_{S_1}) \oplus H^1(S_2, T_{S_2}).\]
The symmetric bilinear form $\sfB_\gamma$ on $H^1(X,T_X)$ can be related to the two rank one maps
\[
\mathsf{A_\ell} : H^{1}(S_\ell,T_{S_\ell})\to\C,\, \quad \xi \mapsto \int_{S_\ell}  \eta_\ell \cup \iota_\xi(\sigma_\ell) = - \int_{S_\ell}  \iota_\xi(\eta_\ell) \cup  \sigma_\ell\,, \quad \ell=0,1\,.
\]
Namely, if $\xi_1 \in H^{1}(S_1,T_{S_1})$ and $\xi_2  \in H^{1}(S_2,T_{S_2})$, then
\begin{align*}
    \sfB_\gamma(\xi_1,\xi_2) &= \int_X \iota_{\xi_1}\iota_{\xi_2} (\eta_1\cup\eta_2) \cup \sigma_1 \cup \sigma_2\\
    &=\int_{S_1}\eta_1\cup\iota_{\xi_1}(\sigma_1)\cdot\int_{S_2}\eta_2\cup\iota_{\xi_2}(\sigma_2) \\ 
    &= \mathsf{A_1}(\xi_1) \cdot \mathsf{A_2}(\xi_2)
\end{align*}
and for $\xi_1,\xi_2\in H^{1}(S_1,T_{S_1})$ or $\xi_1,\xi_2\in H^{1}(S_2,T_{S_2})$, we have $\sfB_\gamma(\xi_1,\xi_2) = 0$. Therefore the rank of $\sfB_\gamma$ is two.
\end{proof}

\noindent
\textbf{Vanishing of $\rho_\gamma$.} 
Let $X$ be an arbitrary Calabi-Yau 4-fold and $\gamma \in H^{2,2}(X)$.
There are interesting cases for which $\rho_\gamma = 0$ and hence the non-reduced virtual cycle equals the reduced virtual cycle.
\begin{proposition}
Assume that $H^2(X,\O_X)=0$. If $\gamma = \eta_1\cup\eta_2$ for some $(1,1)$-classes $\eta_1$ and $\eta_2$ on $X$, then we have $\rho_\gamma=0$.
\end{proposition}
\begin{proof}
For $\xi\in H^1(X,T_X)$, we have
$\iota_{\xi}(\eta_1\cup\eta_2) = \iota_{\xi}(\eta_1) \cup \eta_2 + \eta_1 \cup \iota_{\xi}(\eta_2) =0$
because $\iota_{\xi}(\eta_1), \iota_{\xi}(\eta_2) \in H^2(X,\O_X)$. Hence $\rho_\gamma=0$ by definition.
\end{proof}
\begin{proposition}
Let $Y\subset X$ be a smooth divisor with $H^2(Y,\O_Y)=0$. If $\gamma=[S]$ for some effective divisor $S\subset Y$ with $H^1(Y,\O_Y(S))=0$, then $\rho_\gamma=0$.
\end{proposition}
\begin{proof}
By Proposition \ref{prop:4.6.1} it is enough to show that $\ob_{S/X}$ is zero. Consider the commutative diagram
\[\begin{tikzcd}
 & H^1(T_Y) \ar[r,"\sfob_{S/Y}"]\ar[d,twoheadrightarrow] & H^1(N_{S/Y})\ar[d]\\
\sfob_{S/X}:H^1(T_X)\ar[r] & H^1(T_X|_Y) \ar[r] & H^1(N_{S/X})
\end{tikzcd}\]
where the left vertical map is surjective by the vanishing of $H^1(N_{Y/X}) \cong H^1(\O_Y(Y)) \cong H^1(K_Y)\cong H^2(\O_Y)^\vee$. By the vanishing of $H^1(Y,\O_Y(S))$, we have $H^1(N_{S/Y}) = 0$, therefore $\ob_{S/X}=0$.
\end{proof}
For example, the hypotheses of the previous proposition are satisfied for $\PP^3 \cong Y \subset X$ and any effective divisor $S \subset Y$.

\section{Deformation invariance}\label{sec:deformationinvariance}

In this section, we prove the deformation invariance of reduced virtual cycles.
As a corollary, we prove the variational Hodge conjecture for surface classes that carry non-zero reduced virtual cycles. 

\subsection{Deformation invariance}

In this subsection, we prove that the reduced virtual cycles on the moduli spaces of $\PT_q$ pairs are deformation invariant along the Hodge loci.

We first fix some notation. 
Let
$f:\cX\to\cB$ 
be a morphism of schemes satisfying the following conditions:
\begin{enumerate}
\item $f:\cX\to \cB$ is a smooth projective morphism of relative dimension $4$ with connected fibres;
\item $\cB$ is a smooth connected affine scheme;
\item $\omega_{\cX/\cB}$ is trivial, i.e., there exists an isomorphism $w: \O_{\cX}\cong \omega_{\cX/\cB}$.
\end{enumerate}
We denote by $\cX_b:=\cX\times_\cB \{b\}$ the fibre of $f:\cX\to\cB$ over a point $b \in \cB$.

Choose a horizontal section 
\[\tv = (\tv_p)_{p\geq0} \in \bigoplus_{p \geq 0} F^pH^{2p}_{DR}(\cX/\cB) \]
in the sense Definition \ref{Def:horizontalsection}.
We may regard $\tv$ as a locally constant family of cohomology classes
$\tv_b \in H^*(\cX_b,\C)$ by Remark \ref{Rem:horizontal}.
Let 
\[\tgamma \in \Gamma(\cB, R^2f_*\Omega^2_{\cX/\cB})\]
be the family of $(2,2)$-classes induced by $\tv_2 \in F^2H^4_{DR}(\cX/\cB)$, see \eqref{Eq:HdgFilt}.

Consider a vector bundle
\[\cT:=R^1f_*(T_{\cX/\cB})\]
on $\cB$ with a canonical symmetric bilinear form
\[\sfB_{\tgamma} : \cT \otimes \cT \xrightarrow{\id\otimes\tgamma} \cT \otimes \cT \otimes R^2f_*(\Omega^2_{\cX/\cB}) \xrightarrow{} R^4f_*(\O_{\cX}) \xrightarrow{w} R^4f_*\Omega^4_{\cX/\cB} \cong \O_{\cB}.\]

If the function
$b\in \cB \mapsto \rho_{\tgamma_b} \in \Z$
is constant, then we can form an induced orthogonal bundle of rank $\rho:=\rho_{\tgamma_b}$ as
\[\cT_{\tgamma}:={\cT}/{\ker(\sfB_{\tgamma})}.\]

Let $\curP:=\PTqvXB$ be the moduli space of $\PT_q$ pairs on the fibres of $f:\cX \to\cB$, defined in Theorem \ref{Thm:RelativeModuli}. We can form a fibre diagram
\begin{equation}\label{Eq:62}
\xymatrix{
\curP_b \ar@{^{(}->}[r]^{j_b} \ar[d] & \curP \ar[d]^{p} \\
\{b\} \ar@{^{(}->}[r]^{i_b} & \cB
}	
\end{equation}
where $\curP_b:=\PTqvXb$ is the moduli space of $\PT_q$ pairs on the fibre $\cX_b$.
By Theorem \ref{Thm:RelativeModuli}(2) and \cite{HT}, we have a relative symmetric obstruction theory
\begin{equation}\label{Eq:RelSOT}
\phi : \EE \to \trunc\LL_{\curP/\cB}	
\end{equation}
for $p:\curP \to \cB$. Moreover, the induced map
\[\phi_b : \EE_b:=L j_b^*\EE \xrightarrow{Lj_b^*\phi}  L j_b^*(\trunc \LL_{\curP/\cB}) \to \trunc \LL_{\curP_b}\]
is the standard symmetric obstruction theory for the fibre $\cX_b$ in \eqref{Eq:SOT}.

Now we can state our main theorem in this section.

\begin{theorem}\label{Thm:DefInv}
Let $f:\cX\to\cB$ be a smooth projective morphism of relative dimension $4$ with connected fibres to a smooth connected affine scheme $\cB$ such that $\omega_{\cX/\cB} \cong \O_{\cX}$. Let $\tv\in \bigoplus_{p}F^p H^{2p}_{DR}(\cX/\cB)$ be a horizontal section, $\tgamma \in \Gamma(\cB,R^2f_*\Omega^2_{\cX/\cB})$ the section induced by $\tv_2$, and $q \in \{-1,0,1\}$.
Assume that the function $b\in \cB \mapsto \rho_{\tgamma_b}\in \Z$ is constant.
For every orientation
\[o_{1}:\O_{\PTqvXB} \cong \det(\EE),\qquad o_{2}:\O_{\cB} \cong \det(\cT_{\tgamma}) \]
of the symmetric complex $\EE$ on $\PTqvXB$ and the orthogonal bundle $\cT_{\tgamma}$ on $\cB$, respectively, there exists a cycle class
\[\left[\PTqvXB\right]^\red \in A_{\rvd+\dim(\cB)}\left(\PTqvXB\right)\]
such that
\[\left[\PTqvXb\right]^\red = i_b^!\left[\PTqvXB\right]^\red \in A_{\rvd}\left(\PTqvXb\right)\]
for all $b \in \cB$. Here $i_b^!: A_*(\PTqvXB) \to A_*(\PTqvXb)$ denotes the refined Gysin pullback \cite{Ful} with respect to the fibre diagram \eqref{Eq:62} and $[\PTqvXb]^\red$ denotes the reduced virtual cycle in Theorem \ref{Thm:RVFC} with respect to the orientations induced by $o_1$ and $o_2$.
\end{theorem}

We briefly explain how we prove Theorem \ref{Thm:DefInv}.
We use the cosection localization technique of Kiem-Li \cite{KL13} as in subsection \ref{ss:VFC.CL}. 
Most of the arguments are just straightforward generalizations of the proof of Theorem \ref{Thm:RVFC} to the relative setting.
However there is one key difference. Kiem-Li's cone reduction lemma (cf.~Lemma \ref{Lem:KLconereduction}) fails in the relative setting.
Thus as in \cite{KP20}, we need to modify the relative obstruction theory/cosection for $\curP \to \cB$ to an absolution obstruction theory/cosection for $\curP$. In general, we need an additional condition on the cosection to do this. 
In our case, this additional condition can be deduced from the fact that the Hodge locus of $\tgamma \in H^4_{DR}(\cX/\cB)$ in $\cB$ is equal to $\cB$ itself.

\begin{proof}[Proof of Theorem \ref{Thm:DefInv}]
Form a fibre diagram
\[\xymatrix{
\curP \times_{\cB} \cX \ar[r]^-{q} \ar[d]^-{\pi} & \cX \ar[d]^-{f} \\
\curP \ar[r]^-{p} & \cB
}\]
where $\pi$ and $q$ denote the projection maps. Let $\II\udot$ denote the universal complex on $\curP\times_B \cX$.

We have a relative version of the cosections in Proposition \ref{prop:cosection},
\[\SR : \EE\dual[1]=R\hom_\pi(\II\udot,\II\udot)_0[2] \longrightarrow \cT|\dual_{\curP}.\]
Indeed, consider the composition
\begin{equation}\label{Eq:63}
R\hom_\pi(\II\udot,\II\udot)_0[2] \xrightarrow{\At_{\curP\times_{\cB} \cX/\curP}(\II\udot)} R\hom_\pi(\II\udot,\II\udot\otimes \Omega^1_{\cX/\cB})[3]\xrightarrow{\tr^3}  (R^3f_* \Omega^1_{\cX/\cB})|_{\curP}
\end{equation}
where the third trace map is defined as 
\[\tr^3 : R\hom_\pi(\II\udot,\II\udot\otimes \Omega^1_{\cX/\cB})[3] \stackrel{\tr}{\to} (Rf_*\Omega^1_{\cX/\cB}[3])|_{\curP} \to (R^3f_*\Omega^1_{\cX/\cB})|_{\curP}.\]
Here the last arrow uses Deligne's decomposition theorem \cite[Thm.~6.1]{D68},
\[R f_* \Omega^i_{\cX/\cB} \cong \bigoplus_j R^jf_* \Omega^i_{\cX/\cB}[-j],\] 
which replaces footnote \ref{footnote:decomposition} in the relative setting.
The map $\SR$ is then defined by composing the above map \eqref{Eq:63} with the Serre duality isomorphism $R^3f_*\Omega^1_{\cX/\cB} \cong (R^1f_*T_{\cX/\cB})\dual = \cT\dual$.


Note that we can form a relative version of the commutative triangle \eqref{eq:triangle.obsrBgamma'} in Proposition \ref{prop:cosection}, i.e., we have a commutative triangle
\begin{equation}\label{Eq13}
\xymatrix{
\cT|_{\curP} \ar[rr]^{\SR\dual} \ar[rd]_{\sfB_{\tgamma}|_{\curP}}&& \EE[-1] \ar[ld]^{\SR} \\
& \cT|_{\curP}\dual &
}	
\end{equation}
where we identified $\EE[-1] \cong \EE\dual[1]$ via the symmetric form of $\EE$. 
Here we used that  
\[\ch_2(\II\udot) = p^*(\tgamma) \in \Gamma(\curP, R^2\pi_*(\Omega^2_{\curP\times_{\cB} \cX/\curP})) \cong \Gamma(\curP, p^* (R^2f_* (\Omega^2_{\cX/\cB}))) \]
since both sections extend to horizontal sections $\widetilde{\ch_2}(\II\udot)$ and $p^*(\tv_2)$ of the vector bundle $F^2\cH^4_{DR}(\curP\times_\cB \cX/\curP)$ by Lemma \ref{lem:chishorizontal} (see our convention \eqref{Eq:ch}). Then the commutativity of \eqref{Eq13} follows from \cite[Prop.~4.2]{BF03} as in the proof of Proposition \ref{prop:cosection}.

Since $\cB$ is affine, we can find a subbundle $\cV \subseteq \cT$ such that the composition
\begin{equation}\label{Eq17}
\xymatrix{
\cV \ar@{^{(}->}[r] \ar@/_0.4cm/[rr]_{\cong} & \cT \ar@{->>}[r]  & \cT_{\tgamma}}	
\end{equation}
is an isomorphism. Then the restriction
\[\sfB_{\tgamma}|_{\cV} : \cV \otimes \cV \to \O_{\cB}\]
is non-degenerate. Thus we may regard $\cV$ as an orthogonal bundle over $\cB$. Moreover, the orientation of $\cV$ is determined by the orientation $o_2$ of $\cT_{\tgamma}$ via the isomorphism $\cV \cong \cT_{\tgamma}$ in \eqref{Eq17}.
We may view $\cV \subseteq \cT$ as a family of maximal non-degenerate subspaces of $H^1(\cX_b,T_{\cX_b})$.

As in Definition \ref{Def:EEredV}, the composition $\SR_{\cV}:\EE\dual[1] \xrightarrow{\SR} \cT|_{\curP} \dual \to \cV|_{\curP} \dual$ induces a direct sum decomposition of symmetric complexes
\[ \EE = \EE^\red_{\cV} \oplus \cV|_{\curP}[1].\]
More precisely, we let $\EE^\red_{\cV} := \cone(\SR\dual_{\cV}[1])$ and consider the symmetric form induced by the symmetric form of $\EE$.

Note that the relative symmetric obstruction theory $\phi:\EE \to \trunc\LL_{\curP/\cB}$ in \eqref{Eq:RelSOT} satisfies the isotropic condition. More precisely, we have 
\[\fC_{\curP/\cB} \subseteq \fQ(\EE)\]
as substacks of $\fC(\EE)$.
Indeed, this follows from the relative version of the Darboux theorem \cite{BBJ,BG} for the relative $(-2)$-shifted symplectic form for $p:\curP \to \cB$ in \cite{PTVV} since $\tgamma \in F^2H_{DR}^4(\cX/\cB)$. See \cite{Par2} and Remark \ref{Rem:RelativeDarboux} below. 


We claim that the relative cone reduction property holds for $p:\curP \to \cB$. More precisely, we claim that we can form a commutative diagram
\begin{equation}\label{Eq16}
\xymatrix@C+1pc{
& \fQ(\EE^\red_{\cV}) \ar@{^{(}->}[d] \ar@{^{(}->}[r] & \fC(\EE^\red_{\cV}) \ar[r] \ar@{^{(}->}[d] & \curP \ar@{^{(}->}[d]^0 \\
 (\fC_{\curP/\cB})_\red \ar@{^{(}->}[r]  \ar@{.>}[ru]& \fQ(\EE) \ar@{^{(}->}[r] & \fC(\EE) \ar[r]^{\fC(\SR_{\cV}\dual[1])} & \cV|_{\curP}^\vee
}	
\end{equation}
for some dotted arrow, where the embedding $(\fC_{\curP/\cB})_{\red} \hookrightarrow \fQ(\EE)$ is given by the obstruction theory $\phi$ as in the previous paragraph.

We first modify the relative obstruction theory $\phi:\EE \to \trunc \LL_{\curP/\cB}$ into an absolute one. Indeed, we can form a morphism 
\begin{equation}\label{Eq:51}
\xymatrix@C+1pc{
\EE^\abs \ar[r]^m \ar@{.>}[d]^{\phi^\abs}& \EE \ar[r] \ar[d]^{\phi} & \Omega^1_{\cB}|_{\curP}[1] \ar@{=}[d] \\
\trunc \LL_{\curP} \ar[r] & \trunc  \LL_{\curP/\cB} \ar[r]^-{\KS_{\curP/\cB}} & \Omega^1_{\cB}|_{\curP}[1]
}	
\end{equation}
of distinguished triangles on $\curP$, where the lower triangle is exact since $\cB$ is smooth.
The associated long exact sequence implies that $\phi^\abs :\EE^\abs \to \trunc\LL_\curP$ is an obstruction theory. Moreover, by considering the associated abelian cone stacks, we have a cartesian diagram
\begin{equation}\label{Eq14}
\xymatrix@C+2pc{
\fC_{\curP/\cB} \ar@{->>}[d] \ar@{^{(}->}[r]  & \fC(\EE) \ar@{->>}[d]^{\fC(m)} \\
\fC_{\curP} \ar@{^{(}->}[r]  & \fC(\EE^\abs)
}	
\end{equation}
where the horizontal arrows are closed embeddings given by the obstruction theories $\phi$, $\phi^\abs$, and the vertical arrows are $T_{\cB}$-torsors.

We then modify the relative cosection $\SR : \EE\dual[1] \to \cT|_{\curP}\dual$ into an absolute one. We will find a map $\SR^\abs:(\EE^\abs)\dual[1] \to \cT|_{\curP}\dual$ that fits into the commutative diagram
\begin{equation}\label{Eq19}
\xymatrix{
\EE\dual[1]  \ar[r]^-{\SR} \ar[d]_{m\dual[1]} & \cT|_{\curP}\dual \\
(\EE^\abs)\dual[1] \ar@{.>}[ru]_{\SR^\abs} .&
}    
\end{equation}
We need to show that the composition 
\begin{equation}\label{Eq:50}
\cT|_{\curP} \xrightarrow{\SR\dual} \EE[-1]\xrightarrow{\phi[-1]} \trunc \LL_{\curP/\cB}[-1] \xrightarrow{\KS_{\curP/\cB}} \Omega^1_{\cB}|_{\curP} 	
\end{equation}
vanishes.

The vanishing of \eqref{Eq:50} is the key part of the proof of Theorem \ref{Thm:DefInv}.
By the commutative triangle \eqref{Eq13} and the compatibility result in Lemma \ref{lem:compatibilityofKS}, we have a commutative diagram
\[\xymatrix@C+1pc{
\cT|_{\curP} \ar[r]^{\SR\dual} \ar[rd]_{\sfB_{\tgamma}|_{\curP}}& \EE[-1] \ar[d]^{\SR} \ar[r]^-{\phi[-1]} &  \trunc\LL_{\curP/\cB}[-1] \ar[d]^{-\KS_{\curP/\cB}} \\
&  \cT|_{\curP}\dual \ar[r]^{\KS_{\cX/\cB}|_{\curP}} & \Omega^1_{\cB}|_{\curP}
}\]
where $\KS_{\cX/\cB}$ denotes the Kodaira-Spencer map. Thus it suffices to show that the composition
\[\cT \xrightarrow{\sfB_{\tgamma}} \cT\dual \xrightarrow{\KS_{\cX/\cB}} \Omega^1_{\cB}\]
vanishes. Since $\cB$ is smooth, it suffices to show that the composition
\begin{equation}\label{Eq18}
T_{\cB,b} \xrightarrow{\KS_{\cX/\cB,b}} H^1(\cX_b,T_{\cX_b}) \xrightarrow{\sfB_{\tgamma_b}} H^1(\cX_b,T_{\cX_b})\dual	
\end{equation}
vanishes for all $b \in \cB$.
By \cite[Prop.~4.2]{Blo} (see also Lemma \ref{Lem:Obstructions}(1)), the composition \eqref{Eq18} indeed vanishes since the horizontal section $\tv_2$ is contained in $F^2 H^4_{DR}(\cX/\cB)$. 

Now we can apply Kiem-Li's cone reduction lemma \cite{KL13} for the absolute versions of the obstruction theory $\phi^{\abs}$ in \eqref{Eq:51} and cosection $\SR^{\abs}$ in \eqref{Eq19}.
Indeed, by Lemma \ref{Lem:KLconereduction}, we have a commutative diagram
\begin{equation}\label{Eq15}
\xymatrix{
& \fC((\EE^\abs)^\red_{\cV}) \ar@{^{(}->}[d] \ar[rr] && \curP \ar@{^{(}->}[d]^0\\
(\fC_{\curP})_\red \ar@{^{(}->}[r] \ar@{.>}[ru] & \fC(\EE^\abs) \ar[r]^-{\fl_{\SR^{\abs}}} &  \cT|_{\curP}\dual \ar@{->>}[r] & \cV|_{\curP}\dual
}	
\end{equation}
for some dotted arrow, where $\fl_{\SR^\abs}:=\fC((\SR^\abs)\dual[1])$ and
\[(\EE^\abs)^\red_{\cV}:=\cone(\cV|_{\curP}[1] \to \cT|_{\curP}[1] 
\xrightarrow{(\SR^\abs)\dual[1]} \EE^\abs).\] 
The existence of the dotted arrow in \eqref{Eq15} tells us that the lower horizontal composition in \eqref{Eq15} is zero.
Since the diagrams \eqref{Eq14} and \eqref{Eq19} are commutative, 
the lower horizontal composition in \eqref{Eq16} is also zero, which gives us the desired dotted arrow in \eqref{Eq16}.

We define the relative reduced virtual cycle as
\[[\curP]^\red := \sqrt{0^!_{\fQ(\EE^\red_{\cV})}}[\fC_{\curP/\cB}] \in A_*(\curP),\]
where the square root Gysin pullback $\sqrt{0^!_{\fQ(\EE^\red_{\cV})}}$ is defined as in the proof of Theorem \ref{Thm:NRVFC}.
By Vistoli's rational equivalence \cite[Lem.~3.16]{Vist},\footnote{In \cite{Vist}, it is stated for unramified maps between DM stacks, but the result holds for arbitrary DM type morphism between Artin stacks. This can be deduced from the proof of \cite[Prop.~4]{Kre2} (cf. \cite[Prop.~2.5]{Qu}).}
we have
\[i_b^![\fC_{\curP/\cB}] = [\fC_{\curP_b}] \in A_* (\fC_{\curP/\cB}|_{\curP_b})\]
where we denote $\curP_b :=\PTqvXb$.
Since Oh-Thomas's localized square root Euler classes \cite{OT} are bivariant classes by \cite[Lem.~4.4(3)]{KP20}, the square root Gysin pullback $\sqrt{0^!_{\fQ(\EE^\red_{\cV})}}$ commutes with the ordinary Gysin pullback $i^!_b$.
Therefore, we have
\[i_b^![\curP]^\red = [\curP_b]^\red,\]
which completes the proof. 
\end{proof}

\begin{remark}\label{Rem:RelativeDarboux}
The relative version of the Darboux theorem \cite{BBJ,BG} works only under an additional assumption.
If $\XX$ is a derived scheme over a scheme $\cB$ with a $(-2)$-shifted symplectic form $\theta \in \HN^{-4}(\XX/\cB)(2) \cong \cA^{2,\cl}(\XX/\cB;-2)$, we have a global Darboux chart only if $\theta$ maps to zero under the canonical map
\begin{equation}\label{Eq:HNtoHP}
\HN^{-4}(\XX/\cB)(2) \to \HochP^{-4}(\XX/\cB)(2) 
\end{equation}
to the periodic cyclic homology. 



There is a simple example where there is no Darboux chart for a $(-2)$-shifted symplectic derived scheme. Consider the origin in the affine plane
\[\XX:=\Spec(\C[t_1,t_2]/(t_1,t_2)) \subseteq \cB:=\Spec(\C[t_1,t_2]) \cong \bbA^2.\]
Using the quasi-free resolution $(\C[t_1,t_2][y_1,y_2], dy_i=t_i) \xrightarrow{qis} \C[t_1,t_2]/(t_1,t_2)$, we can find a $(-2)$-symplectic form
\[\theta:=d_{DR}y_1 \wedge d_{DR}y_2 \in \HN^{-4}(\XX/\cB)(2).\]
The isotropic condition fails in this case since $\fC_{\XX/\cB}=\fC(\LL_{\XX/\cB})=\bbA^2$ for the induced symmetric obstruction theory $\LL_{\XX/\cB}\xrightarrow{\id}\LL_{\XX/\cB}$.
This implies that there is no Darboux chart.
Moreover, we have $\theta \mapsto t_1t_2\neq0 \in \HochP^{-4}(\XX/\cB)(2) \cong \C[[t_1,t_2]]$ under the canonical map \eqref{Eq:HNtoHP}.

However, we still have the desired isotropic condition in the situation of Theorem \ref{Thm:DefInv}.
This will be shown in \cite{Par2}, and we briefly summarize the results here.
Let $f:\cX\to\cB$ be a smooth projective morphism of relative dimension $4$ with $\omega : \O_{\cX} \xrightarrow{\cong} \Omega^4_{\cX/\cB}$.
Let $\cM$ be any moduli space of perfect complexes on the fibres of $f: \cX\to\cB$.
The isotropic condition can be detected infinitesimally so we may assume that $\cB=\Spec(A)$ for a local Artinian $\C$-algebra $A$.
Let $\gamma=\ch_2(E) \in H^4(\cX_0,\C)$ for any $[E] \in \cM$ and let $\tgamma \in H^4_{DR}(\cX/\cB)$ be the horizontal lift.
Then the $(-2)$-symplectic form on the derived enhancement of $\cM$ \cite{PTVV,ToVa} maps to zero under the map \eqref{Eq:HNtoHP} if and only if
\begin{equation}\label{Eq:No(0,4)terms}
\tgamma \in F^1H^4_{DR}(\cX/\cB)=\HH^4(\cX,\Omega^{\geq1}_{\cX/\cB}) \subseteq H^4_{DR}(\cX/\cB)=\HH^4(\cX,\Omega^\bullet_{\cX/\cB}).
\end{equation}
In the situation of Theorem \ref{Thm:DefInv}, the condition \eqref{Eq:No(0,4)terms} is satisfied since we already assumed $\tgamma \in F^2H^4_{DR}(\cX/\cB)$. 
Moreover, condition \eqref{Eq:No(0,4)terms} implies the existence of a local Darboux chart which induces the desired isotropic condition.

Heuristically, the result in \cite{Par2} tells us that the relative isotropic condition is satisfied along the locus in the base where the deformations of $\ch_2(E)$ do not have a $(0,4)$-Hodge piece.
This condition is compatible with the gauge-theoretical point of view since the holomorphic bundles are the solutions of
\[\int_{\cX_0}\alpha(E)\cup \omega=0\]
where $\alpha(E)=2\ch_2(E)$, see \cite[Appendix A]{OT}. 
We thank Richard Thomas for the explanation of the gauge-theory perspective and the expectation that \eqref{Eq:No(0,4)terms} is the necessary additional condition to obtain deformation invariance.
\end{remark}

Consequently, we have deformation invariance of numerical invariants.

\begin{corollary}\label{Cor:DefInvInvariants}
In the situation of Theorem \ref{Thm:DefInv}, 
the function
\[b\in \cB \mapsto \int_{[\PTqvXb]^\red}{j_b^*(\star)} \in \Q\]
is constant for any insertion $\star \in A^{\rvd}(\curP^{(q)}_{\tv}(\cX/\cB))$. 
\end{corollary}


There are two technical assumptions in Theorem \ref{Thm:DefInv}:
\begin{enumerate}
\item [A1)] The function $b \in \cB \mapsto \rho_{\tgamma_b} \in \Z$ is constant.
\item [A2)] The symmetric complex $\EE$ on $\PTqvXB$ and the orthogonal bundle $\cT_{\tgamma}$ on $\cB$ are orientable.
\end{enumerate}
We first discuss the assumption A1 on $\rho_{\tgamma_b}$.

\begin{remark}\label{Rem:B0}
The assumption A1 is {\em generically} satisfied. 
Indeed, the function $b \in \cB \mapsto \rho_{\tgamma_b} \in \Z$ is a lower semi-continuous function. 
Hence if we let $\rho_{\max}$ be the maximal value, then
\[\cB^{\circ} := \{b\in\cB :\rho_{\tgamma_b}=\rho_{\max}\} \]
is an open dense subscheme of $\cB$. 
Consequently, we can form an orthogonal bundle of rank $\rho_{\max}$ on $\cB^\circ$ as
\[\cT_{\tgamma}:=(\cT|_{\cB^{\circ}})/\ker(\sfB_{\tgamma}).\]

On the other hand, for the {\em special} cases when assumption A1 is not satisfied, 
we have a vanishing result by Proposition \ref{Prop:vanishingforjumpingrho} below.
\end{remark}

\begin{proposition}\label{Prop:vanishingforjumpingrho}
Let $f:\cX\to\cB$ be a smooth projective morphism of relative dimension $4$ with connected fibres to a smooth connected affine scheme $\cB$ such that $\omega_{\cX/\cB} \cong \O_{\cX}$. Let $\tv\in \bigoplus_{p}F^p H^{2p}_{DR}(\cX/\cB)$ be a horizontal section and let $\tgamma \in \Gamma(\cB,R^2f_*\Omega^2_{\cX/\cB})$ be the section induced by $\tv_2$.
Let $\cB^{\circ} := \{b\in\cB :\rho_{\tgamma_b}=\rho_{\max}\}$ as in Remark \ref{Rem:B0}. 
If $b\in\cB$ is a point such that $b \not\in \cB^\circ$, then we have
\[\left[\PTqvXb\right]^\red = 0 \in A_*\left(\PTqvXb\right)\]
for any orientation.
\end{proposition}

Before we prove Proposition \ref{Prop:vanishingforjumpingrho}, we discuss the assumption A2 on orientations.
The following simple observation shows that the orientability of the orthogonal bundle $\cT_{\tgamma}$ is a minor issue.

\begin{remark}\label{Rem:orientationofTgamma}
There exists a canonical 2-fold {\'e}tale cover $u:\widetilde{\cB} \to \cB^{\circ}$ such that $\cT_{\tgamma}|_{\widetilde{\cB}}$ is orientable.
Indeed, we can form a fibre diagram
\[\xymatrix{
\widetilde{\cB} \ar[r] \ar[d] & \Spec(\C) \ar[d] \\
\cB^{\circ} \ar[r]^{\det(\cT_{\tgamma})} & B\mu_2
}\]
where the horizontal map $\cB^{\circ} \to B\mu_2$ corresponds to the pair of the line bundle $\det(\cT_{\tgamma})$ and the isomorphism $\O_{\cB^{\circ}} \cong \det(\cT_{\tgamma})^{\otimes2}$ induced by $\sfB_{\tgamma}$.
Then we have a canonical bijection
\[\{\text{orientations of $\cT_{\tgamma}$}\} \cong \{\text{sections of $\widetilde{\cB} \to \cB^{\circ}$}\}.\]
Consequently, there exists a canonical orientation of the pullback $\cT_{\tgamma}|_{\widetilde{\cB}}$.
\end{remark}

On the other hand, the orientability of the symmetric complex $\EE$ is more subtle. 
The orientability result of Cao-Gross-Joyce \cite{CGJ} for the fibres does not imply the orientability for the relative case.
We suggest two strategies:
\begin{enumerate}[label=\alph*)]
\item Use the canonical orientation on the canonical 2-fold {\'e}tale cover of the {\em moduli space} $\curP$ as in Remark \ref{Rem:orientationofTgamma}.
\item Find an orientation after a {\em base change} by a finite {\'e}tale cover $\cB' \to \cB$.
\end{enumerate}

Strategy a) is sufficient for applications to the variational Hodge conjecture in the next subsection.
However, if we want to obtain the deformation invariance of numerical invariants (Corollary \ref{Cor:DefInvInvariants}),
strategy a) is not enough and we need strategy b). We expect that the following conjecture holds.

\begin{conjecture}\label{conj:familyorientation}
Let $f:\cX\to\cB$ be a smooth projective morphism with connected fibres of relative dimension $4$ to a smooth connected affine scheme $\cB$ such that $\omega_{\cX/\cB} \cong \O_{\cX}$. Let $\tv\in \bigoplus_{p}F^p H^{2p}_{DR}(\cX/\cB)$ be a horizontal section and let $q \in \{-1,0,1\}$. 
\begin{enumerate}
\item[$(1)$] Then there exists a finite {\'e}tale cover $u:\cB'  \to \cB$ such that $\EE|_{{\curly P}^{(q)}_{\tv'}(\cX'/\cB')}$ is orientable, where $\cX'=\cX\times_\cB \cB'$ and $\tv'=u^*(\tv)$.
\item[$(2)$] Moreover, we can further assume that the degree of the {\'e}tale cover $u:\cB'\to \cB$ in (1) divides 
\[|K_{\top}^0(\cX_b^\an)_{\mathrm{tor}}|!\]
where $K_{\top}^0(\cX_b^\an)_{\mathrm{tor}} \subseteq K^0_{\mathrm{top}}(\cX_b^\an)$ is the torsion subgroup of the complex topological $K$-theory of $\cX_b^\an$, which is finite. 
\end{enumerate}
\end{conjecture}

We will provide evidence for Conjecture \ref{conj:familyorientation} in Appendix \ref{Appendix:RelativeOrientation}.
Indeed, we will prove that the topological analog of Conjecture \ref{conj:familyorientation} is true (see Corollary \ref{Cor:Or.2}).

\begin{remark}
Assuming Conjecture \ref{conj:familyorientation}, $\EE$ is orientable in the following cases:
\begin{enumerate}
\item $\cB^\an$ is simply-connected.
\item $H^*(\cX_b^\an,\Z)$ has no torsion elements for some $b \in \cB$, e.g., all complete intersection Calabi-Yau $4$-folds in smooth projective toric varieties.
\end{enumerate}
In case (2), we have $K^0_{\mathrm{top}}(\cX^\an_b)_\mathrm{tor}=0$ by \cite[Cor.~2.5(i)]{AH}.
\end{remark}



We now prove Proposition \ref{Prop:vanishingforjumpingrho}.

\begin{proof}[Proof of Proposition \ref{Prop:vanishingforjumpingrho}]


We may assume that
\[\cB^{\circ} = \cB - \{b\}.\]
Indeed, we can find a smooth connected affine curve $C$ with a map $g:C \to \cB$ such that $g^{-1}(\cB^{\circ}) = C-\{c\}$ for some point $c\in C$. Replace $\cB$ by $C$.

We claim that there exists a non-degenerate subbundle
\[\cV \subseteq \cT\]
of rank $\rho_{\tgamma_b}$,
after replacing $\cB$ by some open neighborhood of $b \in \cB$. 
Indeed, choose a maximal non-degenerate subspace $V \subseteq H^1(\cX_b,T_{\cX_b})$. Then we can extend this to a map
$\cV:=V \otimes \O_{\cB} \to \cT$
of vector bundles.

Since its dual $\cT\dual \to \cV\dual$ is surjective over $b \in \cB$, by shrinking $\cB$, we may assume that $\cV$ is a subbundle of $\cT$.
Since $\sfB_{\tgamma}|_{\cV}$ is non-degenerate over $b \in\cB$, by shrinking $\cB$, we may assume that $\sfB_{\tgamma}|_{\cV}$
is non-degenerate on $\cB$. 

Let $\tcurP \to \curP$ be the canonical 2-fold {\'e}tale cover induced by the line bundle $\det(\EE)\otimes \det(\cV)\dual$ and the symmetric forms of $\EE$ and $\cV$ (as in Remark \ref{Rem:orientationofTgamma}). Form a fibre diagram
\[\xymatrix{
\tcurP_{b} \ar@{^{(}->}[r]^j \ar[d] & \tcurP \ar[d]^{} & \tcurP^\circ \ar@{_{(}->}[l]_k \ar[d] \\
\curP_{b} \ar@{^{(}->}[r] \ar[d] \ar@{.>}@/^0.4cm/[u]^{t_b}& \curP \ar[d]^p & \curP^\circ \ar[d]^{} \ar@{_{(}->}[l] \\
\{b\} \ar@{^{(}->}[r]^{i_b} & \cB & \cB^{\circ}. \ar@{_{(}->}[l]
}\]
Since the symmetric complex $\EE_b$ on the fibre $\curP_b$ is orientable by \cite{CGJ} and the vector space $\cV_b$ is always orientable, we have
\[\tcurP_b = \curP_b \sqcup \curP_b.\]
Let $t_b : \curP_b \hookrightarrow \tcurP_b$ denote the inclusion to the first component.

Although the rank of $\mathsf{B}_{\widetilde{\gamma}}$ jumps at $b$, 
the rank of $\mathsf{B}_{\widetilde{\gamma}}|_{\cV}$ does not jump,
so the argument in Theorem \ref{Thm:DefInv} can indeed be used.
Hence there exists a cycle class
\[[\tcurP]^\red_{\cV} \in A_*(\tcurP)\]
such that
\begin{equation}\label{Eq:61}
i_{b}^! [\tcurP]^\red_{\cV} = \left( [\curP_{b}]^\red, -[\curP_{b}]^\red \right) \in A_*(\tcurP_{b}) = A_*(\curP_{b})\oplus A_*(\curP_{b}).
\end{equation}
Indeed, we have a decomposition of symmetric complexes
\[\EE\cong\EE^\red_{\cV} \oplus (\cV|_{\curP}[1])\]
for some symmetric complex $\EE^\red_{\cV}$ and the relative cone reduction property 
\[(\fC_{\curP/\cB})_\red \subseteq \fQ(\EE^\red_{\cV}).\]
We can define the desired cycle class as
\[[\tcurP]^\red_{\cV}:=\sqrt{0^!_{\fQ\left(\EE^\red_\cV|_{\tcurP}\right)}}\left[\fC_{\tcurP/\cB}\right] \in A_*(\tcurP)\]
since the pullback $\EE^\red_\cV|_{\tcurP}$ is orientable by the definition of $\tcurP$.

Since there are additional cosections on $\curP^\circ$, we claim that
\begin{equation}\label{Eq:V1}
k^*[\tcurP]^\red_{\cV} = 0 \in A_*(\tcurP^\circ).
\end{equation}
Indeed, 
the orthogonal bundle $\cT_{\tgamma}$ is well-defined on $\cB^{\circ}$, and the composition
\[\cV|_{\cB^{\circ}} \hookrightarrow \cT|_{\cB^{\circ}} \twoheadrightarrow \cT_{\tgamma}\]
is injective on the fibres. Thus by considering a section of the surjection
\[\frac{\cT|_{\cB^{\circ}}}{\cV|_{\cB^{\circ}}} \twoheadrightarrow \frac{\cT_{\tgamma}}{\cV|_{\cB^{\circ}}}\]
we can find a non-degenerate subbundle $\cW$ of $\cT|_{\cB^{\circ}}$ properly containing $\cV|_{\cB^{\circ}}$,
\[\cV|_{\cB^{\circ}} \subsetneq \cW \subseteq \cT|_{\cB^{\circ}}.\]
Then we have $\cW = \cV|_{\cB^{\circ}} \oplus \cU$ for the orthogonal complement $\cU$ of $\cV|_{\cB^\circ}$ in $\cW$. Moreover, we have a decomposition of symmetric complexes
\[\EE|_{\curP^{\circ}} = \EE^\red_{\cW} \oplus (\cW|_{\curP^{\circ}}[1]) = \EE^{\red}_{\cW} \oplus (\cV|_{\curP^{\circ}}[1]) \oplus (\cU|_{\curP^{\circ}}[1])\]
for some symmetric complex $\EE^\red_{\cW}$. Hence we have
\[(\EE^{\red}_{\cV})|_{\curP^{\circ}} \cong \EE^{\red}_{\cW} \oplus (\cU|_{\curP^{\circ}}[1]).\]
By shrinking $\cB$ further, we find a section $w \in \Gamma(\cB^\circ, \cU)$ such that $w^2$ is nowhere vanishing. 
Consequently, we can find a cosection
\[\sigma : \EE^\red_{\cV}|_{\curP^{\circ}}[1] \to \O_{\curP^{\circ}}\]
such that $\sigma^2$ is nowhere vanishing. Lemma \ref{Lem:Vanishing} below proves the claim.

By the localization sequence
\[\xymatrix{
A_*(\tcurP_b) \ar[r]^{j_*} & A_*(\tcurP) \ar[r]^{k^*} & A_*(\tcurP^{\circ}) \ar[r] & 0,
}\]
and the vanishing \eqref{Eq:V1}, we have
\[[\tcurP]^\red_{\cV} = j_*(\alpha)\]
for some $\alpha \in A_*(\tcurP_b)$. By the self-intersection formula, we have
\[i^!_b[\tcurP]^\red_{\cV} = i_b^!\circ j_*(\alpha) = e(N_{b/\cB})\cdot\alpha = 0.\]
Then the formula \eqref{Eq:61} proves the desired vanishing result.
\end{proof}

We need the following lemma to complete the proof of Proposition \ref{Prop:vanishingforjumpingrho}.

\begin{lemma}[cf. {\cite[Thm.~8.12]{KP20}}]\label{Lem:Vanishing}
Let $\EE$ be any symmetric complex on a quasi-projective scheme $\sX$. Let $\sigma :\EE\dual[1] \to \O_{\sX}$ be a cosection such that $\sigma^2\in \Gamma(\sX,\O_{\sX})$ is nowhere vanishing. Consider a fibre diagram
\[\xymatrix{
\fQ(\EE_{\sigma}) \ar[r] \ar@{^{(}->}[d]^{l} &  \sX \ar@{^{(}->}[d]^0\\
\fQ(\EE) \ar[r]^-{\fl_{\sigma}} & \bbA^1_{\sX}
}\]
where $\EE_{\sigma}:=\cone(\sigma\dual[1] : \O_{\sX} \to \EE)$.
Then the composition
\[
A_*(\fQ(\EE_\sigma)) \xrightarrow{l_*} A_*(\fQ(\EE)) \xrightarrow{\sqrt{0^!_{\fQ(\EE)}}} A_*(\sX)\]
vanishes.
\end{lemma}

\begin{proof}
We omit the proof. See \cite[Sect.~8.3]{KP20}.
\end{proof}

\subsection{Variational Hodge conjecture}

As an application of the deformation invariance in the previous subsection, we prove that the {\em variational Hodge conjecture} holds for surface classes with non-zero reduced virtual cycles.

Fix a smooth projective variety $X$ over $\C$.

\begin{definition}
Let $\gamma \in H^{2p}(X,\C)$ be a cohomology class on $X$.
\begin{enumerate}
\item We say that $\gamma$ is a {\em Hodge class} if 
\[\gamma \in Hdg^p(X):=H^{2p}(X,\Q) \cap H^p(X,\Omega^p_X) \subset H^{2p}(X,\C).\]
\item We say that $\gamma$ is an {\em algebraic} cohomology class if $\gamma$ contained in the image of the cycle class map,
\[\gamma \in \im(\cl : A^p(X)_\Q \to Hdg^p(X)).\]
\end{enumerate}
\end{definition}

The {\em Hodge conjecture} predicts that every Hodge class on $X$ is algebraic.
Grothendieck's {\em variational Hodge conjecture} \cite[footnote 13]{Gr66} (cf.~\cite{Blo}) predicts that a deformation of an algebraic cohomology class as a Hodge class is algebraic. 

Thus we make the following definition:

\begin{definition}\label{Def:VHC}
Let $X$ be a smooth projective variety and $\gamma\in H^{2p}(X,\C)$ an algebraic cohomology class.
We say that the {\em variational Hodge conjecture} holds for the pair $(X,\gamma)$ if the following condition is satisfied:
for any smooth projective morphism $f: \cX \to \cB$ to a smooth connected scheme $\cB$
and any horizontal section $\tv_p \in \Gamma(\cB,F^p\cH^{2p}_{DR}(\cX/\cB))$
such that $\cX_0 \cong X$ and $(\tv_p)_0=\gamma$ for some $0 \in \cB$, the cohomology classes $(\tv_p)_b \in H^{2p}_{DR}(\cX_b)$ are algebraic for all $b \in \cB$.
%
\end{definition}

\begin{remark}
In some of the literature, the variational Hodge conjecture is stated for any horizontal section of the Hodge bundle $\cH_{DR}^{2p}(\cX/\cB)$ instead of the filtration $F^p\cH_{DR}^{2p}(\cX/\cB)$. Deligne's global invariant cycle theorem \cite[(4.1.1)]{D71} implies that the two versions are equivalent (see also \cite[Prop.~11.3.5]{CS}).
\end{remark}

We now state the main theorem of this subsection. 
\begin{theorem}\label{Thm:VHC}
Let $X$ be a Calabi-Yau 4-fold and let $\gamma \in H^2(X,\Omega^2_X)$.
If
\[[\PTqvX]^\red \neq 0 \in A_{\rvd}(\PTqvX)\]
for some $v \in H^*(X,\C)$ with $v_2=\gamma$ and $q \in \{-1,0,1\}$, then the variational Hodge conjecture holds for $(X,\gamma)$.
\end{theorem}

For Calabi-Yau 4-folds, we would like to regard Theorem \ref{Thm:VHC} as a {\em virtual} generalization of the results of Bloch and Buchweitz-Flenner \cite{Blo,BF03} since the semi-regular case is heuristically the {\em ideal} case (see Theorem \ref{Thm:SR=smoothofrvd}). In particular, for Calabi-Yau 4-folds, we recover the result of Buchweitz-Flenner as an immediate corollary.

\begin{corollary}
Let $X$ be a Calabi-Yau 4-fold and let $\gamma\in H^2(X,\Omega^2_X)$. If there exists a $\PT_q$ pair $(F,s)$ on $X$ such that $\ch_2(F)=\gamma$, for some $q \in \{-1,0,1\}$, and the associated complex $I\udot:=[\O_X \xrightarrow{s}F]$ is semi-regular, 
then the variational Hodge conjecture holds for $(X,\gamma)$.
\end{corollary}

\begin{proof}
By Theorem \ref{Thm:SR=smoothofrvd}, there exists a non-empty smooth open subset $U \subseteq \PTqvX$ such that
$[\PTqvX]^\red|_U = \pm [U] \in A_*(U).$
Hence $[\PTqvX]^\red\neq0$ and Theorem \ref{Thm:VHC} completes the proof.
\end{proof}

We now prove our main theorem in this subsection. 

\begin{proof}[Proof of Theorem \ref{Thm:VHC}]
Let $f:\cX\to\cB$ be a smooth projective morphism to a smooth connected scheme $\cB$ and $\tgamma$ be a horizontal section of $F^2\cH_{DR}^{4}(\cX/\cB)$. 
Assume that $\cX_0 \cong X$ and $\tgamma_0=\gamma$ for some point $0\in \cB$. 

Heuristically, Theorem \ref{Thm:VHC} follows from the deformation invariance in Theorem \ref{Thm:DefInv}. However, to obtain a rigorous proof, we need to overcome the following technical issues:
\begin{enumerate}
\item existence of a horizontal lift $\widetilde{\beta} \in \Gamma(\cB,F^3\cH_{DR}^6(\cX/\cB))$ of $\beta=v_3$;
\item $f:\cX\to\cB$ may not be relatively Calabi-Yau (i.e., $\omega_{\cX/\cB}\cong \O_{\cX}$);
\item the function $b\in \cB \mapsto \rho_{\tgamma_b}$ may not be constant;
\item orientability of the orthogonal bundle $\cT_{\tgamma}$ on $\cB$;
\item orientability of the virtual cotangent complex $\EE$ of $\PTqvXB \to \cB$;
\item the reduced virtual cycles $[\PTqvXb]^\red$ live in the Chow groups of different schemes for each $b\in\cB$.
\end{enumerate}

We first resolve issue (1). Denote by $P_v(t)$ the Hilbert polynomial determined by $v$ with respect to some polarization on $X$. We will consider a small variation on the moduli space $\PTqvXB$, namely
\[\curP:= \curP^{(q)}_{\widetilde{\gamma}, P_v}(\cX/\cB) =
\left\{(F,s,b) : \begin{matrix}
   \text{$(F,s)$ is a $\PT_q$ pair on the fibre $\cX_b$ of $b \in \cB$}\\ \text{ such that $\ch_2(F)=\tgamma_b$ and $P_F(t) = P_v(t)$}
\end{matrix}\right\}.\]
The basic properties of $\PTqvXB$ in Theorem \ref{Thm:RelativeModuli} also holds for $\curP$. More precisely, the projection map
\[p : \curP \to \cB\]
is projective and the canonical map
\[\curP \hookrightarrow \cPerf(\cX/\cB)_{\O_{\cX}}^\spl\]
is an open embedding. Consequently, we have a symmetric obstruction theory
\[\phi:\EE \to \trunc \LL_{\curP/\cB}\]
for $p:\curP\to\cB$ that satisfies the isotropic condition. To prove Theorem \ref{Thm:VHC}, it suffices to show that the projection map $p:\curP \to \cB$ is surjective. Since $p:\curP \to \cB$ is projective, it suffices to show that the image of $p:\curP\to\cB$ contains a non-empty open subset of $\cB$.


Secondly, we resolve the issue (2). We claim that there exists a nonempty open neighborhood $\cB'$ of $0\in\cB$ such that 
\[\omega_{\cX\times_{\cB}\cB'/\cB'}\cong \O_{\cB'}.\] 
Indeed, by \cite[Thm.~5.5]{D68}, the pushforward $f_*(\omega_{\cX/\cB})$ is a vector bundle that commutes with base change. Since the fibre $\cX_0$ is Calabi-Yau, $f_*(\omega_{\cX/\cB})$ is a line bundle. Consider the canonical map of line bundles
\[f^*f_*(\omega_{\cX/\cB}) \to \omega_{\cX/\cB}\]
given by adjunction. Since it is an isomorphism over $0 \in \cB$, it is an isomorphism over an open subscheme of $\cX$ containing the fibre $\cX_0$. Since $f:\cX\to\cB$ is projective, by shrinking $\cB$, we may assume that the map $f^*f_*(\omega_{\cX/\cB}) \to \omega_{\cX/\cB}$ is an isomorphism. Moreover, by shrinking $\cB$ again, we may assume that $f_*(\omega_{\cX/\cB}) \cong \O_{\cB}$, and thus $\omega_{\cX/\cB} \cong f^*f_*(\omega_{\cX/\cB}) \cong \O_{\cX}$. It proves the claim.

Next, the vanishing result in Proposition \ref{Prop:vanishingforjumpingrho} resolves issue (3). Indeed, since $[\curP^{(q)}_{v_0}(\cX_0)]^\red\neq 0$ by assumption, we have $0 \in \cB^{\circ}$. Replace $\cB$ by $\cB^{\circ}$. Then we may assume that the function $b\in \cB \mapsto \rho_{\tgamma_b}$ is constant.

Issue (4) is resolved by simply replacing $\cB$ by a 2-fold {\'e}tale cover $\cB' \to \cB$ such that $\cT_{\tgamma}|_{\cB'}$ is orientable (Remark \ref{Rem:orientationofTgamma}).

We resolve issue (5) via an infinitesimal trick. 
(Alternatively, we can also use a $2$-fold {\'e}tale cover of $\cP$ as in the proof of Proposition \ref{Prop:vanishingforjumpingrho}.) 
Assume that $p:\curP \to \cB$ is not surjective. Since the image is closed, there exists a 1-dimensional integral closed subscheme $C \subseteq \cB$ such that $0 \in C$ and $C \not\subseteq \im(p:\curP \to \cB)$.
By replacing $\cB$ by the normalization $\widetilde{C} \to C$ of $C$, we may assume that $\cB$ is a smooth quasi-projective curve, and the (set-theoretical) image $\im(p:\curP \to \cB)$ consists of finitely many points.
By shrinking $\cB$ further, we may assume that the image contains a single point $0 \in \cB$.
Then we have $(\curP)_\red = (\curP_0)_\red$ and hence $\EE$ is orientable by Lemma \ref{lem:nilp.orienation} below.

Finally, we explain how to avoid issue (6). We follow the situation in the previous paragraph so that $\cB$ is a smooth curve and the set-theoretical image of $p:\curP \to \cB$ contains a single point $0 \in \cB$.
Form a fibre diagram
\[\xymatrix{
\curP_0 \ar@{^{(}->}[r]^{j} \ar[d] & \curP \ar[d]^p\\
\{0\} \ar@{^{(}->}[r]^{i_0} & \cB .
}\]
Since $\curP_0 \hookrightarrow \curP$ is a nilpotent thickening and the Chow groups are invariant under nilpotent thickenings, we may write
\[[\curP]^\red = j_*(\alpha)\]
for some $\alpha \in A_*(\curP_0)$. By Theorem \ref{Thm:DefInv} and the self-intersection formula, we have
\[[\curP_0]^\red  = i_0^! [\tcurP]^\red = i_0^! j_*(\alpha) = e(N_{0/\cB})\cap\alpha = 0\]
since the normal bundle $N_{0/\cB}$ is a trivial line bundle. This contradicts the assumption $[\curP_v^{(q)}(X)]^{\red} \neq 0$, and hence we conclude that $p:\curP \to \cB$ is surjective.
\end{proof}

We need the Lemma \ref{lem:nilp.orienation} to complete the proof of Theorem \ref{Thm:VHC}.

\begin{lemma}\label{lem:nilp.orienation}
Let $\sX \hookrightarrow \sX'$ be a nilpotent embedding of quasi-projective schemes. Let $E$ be an orthogonal bundle on $\sX'$. If $E|_{\sX}$ is orientable, then $E$ is orientable.
\end{lemma}

\begin{proof}
Consider the canonical short exact sequence of algebraic groups
\[\xymatrix{1 \ar[r] & SO(n) \ar[r] & O(n) \ar[r]^{\det} & \mu_2 \ar[r] & 1}\]
where $n$ is the rank of $E$.
Note that the orthogonal bundle $E$ defines a function $\sX' \xrightarrow{E} BO(n)$ to the classifying stack $BO(n):=[\Spec(\C)/O(n)]$ of $O(n)$. Consider the fibre diagram
\[\xymatrix{
U|_{\sX} \ar[r] \ar[d] & U \ar[r] \ar[d]& BSO(n) \ar[r] \ar[d] & \Spec(\C) \ar[d]\\
\sX \ar[r] & \sX' \ar[r]^-{E} & BO(n) \ar[r]^{\det} & B\mu_2.
}\]
Since $E|_{\sX}$ is orientable, the principal $\mu_2$-bundle $U|_{\sX} \to \sX$ is trivial. Hence $U|_{\sX} = U_1 \bigsqcup U_2$ for some open subschemes such that the compositions $U_1,U_2 \hookrightarrow U|_{\sX} \to \sX$ are isomorphism. Since $U$ and $U|_{\sX}$ have the same Zariski topology, we can decompose $U = \widetilde{U_1}\bigsqcup \widetilde{U_2}$ by disjoint open subschemes $\widetilde{U_1}, \widetilde{U_2}$ extending $U_1,U_2$. Then $\widetilde{U_1}, \widetilde{U_2} \to  \sX'$ are both finite \'etale bijective maps. Hence they are both isomorphisms. This means that the $\mu_2$-bundle $U \to \sX'$ has a section. Hence $E$ is orientable.
\end{proof}

The following example illustrates Theorem \ref{Thm:VHC} in some non-semi-regular cases.
\begin{example}\label{ex:k3k3}
Let $X=S_1\times S_2$ be a product of two K3 surfaces. Let $\beta_1\in H^{1,1}(S_1)$, $\beta_2\in H^{1,1}(S_2)$ be two irreducible effective curve classes with $\beta_1^2=2h-2$, $h\geq 2$ and $\beta_2^2=0$. 
Let $v:=\ch(\O_{C_1\times C_2})$ for any $C_1\in |\O_{S_1}(\beta_1)|$ and $C_2\in |\O_{S_2}(\beta_2)|$.
Then we have an isomorphism
\begin{equation}\label{eq:k3k31}
    \phi : |\O_{S_1}(\beta_1)|\times \PP^1\to \curP^{(1)}_v(X)\,.
\end{equation}
Under this isomorphism we have
\[[\curP^{(1)}_v(X)]^\red = h^0(\O_{S_1}(\beta_1))\cdot[\pt \times\PP^1]\]
for some orientation.

We briefly illustrate the argument. The proof will appear in one of the sequels to this paper.
By the choice of $\beta_2$, there exists an elliptic fibration $p:S_2\to\PP^1$ such that $\beta_2$ is the class of a fibre.  By Proposition~\ref{prop:k3k3} the reduced virtual dimension of $\curP^{(1)}_v(X)$ is one. 
Let $\II\udot=[\O_{|\O_{S_1}(\beta_1)|\times S_1}\twoheadrightarrow \O_{\mathcal{C}}]$ be the universal pair on $|\O_{S_1}(\beta_1)|\times S_1$. Consider the fibre diagram
\[\xymatrix{
& S_1 \times S_2 \ar@{^{(}->}[r]^-{j} \ar[d] \ar[ld]_{\mathrm{pr}} & S_1\times \PP^1 \times S_2 \ar[d]^{\id\times\id\times p} \\
S_1 & S_1\times\PP^1 \ar@{^{(}->}[r]^-{\id\times\Delta} & S_1\times\PP^1\times\PP^1
}\]
where $\Delta:\PP^1\to\PP^1\times\PP^1$ is the diagonal. Then the universal pair $j_*\textup{pr}^*\II\udot$ on $|\O_{S_1}(\beta_1)|\times\PP^1 \times X$ induces the morphism \eqref{eq:k3k31}. Using \cite[Prop.~1.18]{Par}, we can show that $\phi$ is an isomorphism and the reduced obstruction theory of $\curP^{(1)}_v(X)$ reduces to a 2-term obstruction theory on $|\O_{S_1}(\beta_1)|\times\PP^1$.

By Theorem~\ref{Thm:SR=smoothofrvd}, the points of $\curP^{(1)}_v(X)$ are \emph{not} semi-regular. Since the reduced virtual cycle of $\curP^{(1)}_v(X)$ is non-zero, Theorem~\ref{Thm:VHC} applies to the class $\gamma = \beta_1\cup \beta_2$.  
\end{example}

\begin{remark}
The original version of the variational Hodge conjecture by Grothendieck \cite{Gr66} is stated for all {\em geometric points}, while Definition \ref{Def:VHC} only considers $\C$-points. In the situation of Theorem \ref{Thm:VHC}, since we have shown that the map $p: \curP^{(q)}_{\tgamma,P_v}(\cX/\cB) \to \cB$ is surjective on $\C$-points, it is also surjective for all geometric points. Hence we also deduce the original form of the variational Hodge conjecture.
\end{remark}

\begin{remark}
In the situation of Theorem \ref{Thm:VHC}, we further speculate that  $[\PTqvX]^\red\neq0$ implies {\em smoothness} of the Hodge locus of $\gamma$ in the moduli space of Calabi-Yau $4$-folds near $X$.\footnote{For simplicity we consider Calabi-Yau $4$-folds with $h^{0,1}=h^{0,2}=h^{0,3}=0$. Then there exists a moduli space of Calabi-Yau $4$-folds as a smooth separated Deligne-Mumford stack with quasi-projective coarse moduli space by Viehweg \cite{Vie} (see also Koll\'ar \cite[Cor.~8.23]{Kol2}) and the Bogomolov-Tian-Todorov theorem \cite{Bog,Tian,Tod}.}
We note that not much is known about singularities of the Hodge locus.
At least when the scheme-theoretical Hodge locus is reduced near $X$ and $[\PTqvX]^\red\neq0$, then Proposition \ref{Prop:vanishingforjumpingrho} proves smoothness near $X$ since the dimension of the tangent space is constant near $X$.
We hope to pursue this direction in a future work.
\end{remark}





\appendix

\section{Reductions via \texorpdfstring{$(-1)$}{}-shifted 1-forms}\label{Appendix:ReductionviaDAG}

In this section, we construct a reduced $(-2)$-shifted symplectic derived enhancement on $\PTqvX$. 
In particular, this construction shows that the reduced obstruction theory constructed from the algebraic twistor family in section \ref{ss:VFC.ATF} satisfies the isotropic condition.

In subsection \ref{ss:ThreeReductions}, we present general reduction procedures by cosections that come from $(-1)$-shifted {\em closed}/{\em exact} $1$-forms. 
In subsection \ref{ss:shifteclosed1forms}, we show that the cosections on $\PTqvX$ induced by the semi-regularity map can be enhanced to $(-1)$-shifted {\em closed} $1$-forms.

We use the following conventions/notations in this section:
\begin{enumerate}
\item We use the language of {\em $\infty$-categories} in the sense of Lurie \cite{Lur1,Lur2}. 
\begin{itemize}
\item We write $\Map_{\mathcal{C}}(-,-)$ for the mapping space of an $\infty$-category $\mathcal{C}$.
\end{itemize}
We consider the following $\infty$-categories:
\begin{itemize}
\item Let $\cdga$ be the $\infty$-category of commutative differential graded $\C$-algebras and let $\cdga^{\leq0} \subseteq \cdga$ be the full subcategory of non-positively graded objects.
\item Let $\dSt$ be the $\infty$-category of {\em derived stacks}, i.e., $\infty$-sheaves on $\cdga^{\leq0}$ with respect to the {\'e}tale topology.

\item Let $\epsilon\textrm{-}\dg^\gr$ (resp. $\dg^\gr$) be the symmetric monoidal stable $\infty$-category of graded mixed complexes (resp. graded complexes), see \cite{PTVV}.

\noindent We use the weight-shift notation $E((m)):=\oplus_p E(p+m)$ for any $E=\oplus_p E(p) \in \epsilon\textrm{-}\dg^\gr$ and $m \in \Z$.
\end{itemize}

\item We use the language of {\em derived Artin stacks} of To\"en-Vezzosi \cite{ToVe}.
\begin{itemize}
\item 
All derived Artin stacks are assumed to be {\em of finite type}, i.e., their classical truncation are (higher) Artin stacks of finite type.
\item Let $\sfD(\MM)$ be the symmetric monoidal stable $\infty$-category of quasi-coherent sheaves on a derived Artin stack $\MM$.
Let $f^*:\sfD(\MM_2) \to \sfD(\MM_1)$ be the pullback functor for a morphism $f:\MM_1 \to \MM_2$ of derived Artin stacks and let $f_*$ be the right adjoint.
\item A derived Artin stack $\MM$ is {\em homotopically finitely presented} if the cotangent complex $\LL_\MM\in \sfD(\MM)$ is perfect (see \cite[Prop.~2.2.2.4]{ToVe}).
\item A derived Artin stack $\MM$ is a {\em derived scheme} 
if the classical truncation $\MM_\cl$ is a scheme. 
\end{itemize}

\item Let $\DR(-): \dSt^\op \to \epsilon\textrm{-}\cdga^\gr$ be the $\infty$-functor of {\em derived de Rham complexes} of Pantev-To\"en-Vaqui\'e-Vezzosi \cite{PTVV}, defined as follows:
\begin{itemize}
\item Let $\epsilon\textrm{-}\cdga^\gr$ (resp.~$\cdga^\gr$) be the $\infty$-category of graded mixed commutative differential graded $\C$-algebras (resp. graded commutative differential graded $\C$-algebras), see \cite{CPTVV}. 

\item Define $\DR$ as the right Kan extension of the $\infty$-functor $\DR_\aff|_{\cdga^{\leq0}}:\cdga^{\leq0} \to \epsilon\textrm{-}\cdga^\gr$, where $\DR_\aff$ is the left adjoint of the $\infty$-functor $\epsilon\textrm{-}\cdga^\gr \to \cdga :  E\mapsto E(0)$.
\end{itemize}
Abusing notation, we use the same symbol $\DR(-)$ to denote the compositions $\dSt^{\op} \xrightarrow{\DR} \epsilon\textrm{-}\cdga^\gr \to \cdga^\gr$ and $\dSt^{\op} \xrightarrow{\DR} \epsilon\textrm{-}\cdga^\gr \to \dg^\gr$.
\item
Let $\NC(-) : \epsilon\textrm{-}\cdga^\gr \to \cdga^\gr$ (resp. $\CC(-): \epsilon\textrm{-}\dg^\gr \to \dg^\gr$) be the $\infty$-functor of {\em weighted negative cyclic complexes} (resp.~{\em weighted cyclic complexes}), defined as the right adjoint (resp.~left adjoint) of the $\infty$-functor of trivial mixed structures, $\cdga^\gr \to \epsilon\textrm{-}\cdga^\gr$ (resp. $\dg^\gr \to \epsilon\textrm{-}\dg^\gr$) $: E \mapsto (E,\epsilon=0)$.
\begin{itemize}
\item Abusing notation, we use the same symbols $\NC,\CC$ to denote the compositions $\NC\circ\DR,\CC\circ\DR$.
\end{itemize}
\end{enumerate}

\subsection{Three reductions of derived schemes}\label{ss:ThreeReductions}
In this subsection, we revisit Kiem-Li's cone reduction \cite{KL13} via {\em derived algebraic geometry}.
We will see that there are three levels of reductions for derived schemes.

We consider the following hierarchy of structures
\[\xymatrix{
\{\text{derived schemes}\} \ar[d]^-{(a)}\\
\left\{
\text{schemes with obstruction theories}
\right\} \ar[d]^-{(b)}\\
\left\{
{\begin{matrix}
\text{schemes with closed embedding of}\\
\text{their intrinsic normal cone into an abelian cone stack}
\end{matrix}}\right\}.
}\]
More precisely, the above two arrows can be given as follows:
\begin{enumerate}[label=(\alph*)]
\item For any homotopically finitely presented derived scheme $\XX$, there is an induced obstruction theory
\[\phi : \EE:=\LL_{\XX}|_\sX \to \LL_\sX \to \trunc \LL_\sX\]
on the classical truncation $\sX:=\XX_{\cl}$ by \cite[Prop.~1.2]{STV}. 
\item For any scheme $\sX$ with an obstruction theory $\phi:\EE \to \trunc \LL_\sX$, there is an induced closed embedding
\[\iota:\fC_{\sX} \hookrightarrow \fC(\EE)\]
of the intrinsic normal cone $\fC_{\sX}$ into the abelian cone stack $\fC(\EE)$.
\end{enumerate}

In derived algebraic geometry, {\em $(-1)$-shifted $1$-forms} are the natural analogs of {\em cosections}. We have a similar hierarchy for them:
\begin{enumerate}[label=(\alph*)]
\item Let $\alpha : \O_{\XX} \to \LL_{\XX}[-1]$ be a $(-1)$-shifted $1$-form on a homotopically finitely presented derived scheme $\XX$.
Then we have an induced {\em cosection}
\[\sigma:=\alpha|_{\sX}\dual : \EE\dual[1] \to \O_{\sX}\]
for the induced obstruction theory $\phi : \EE:=\LL_\XX|_{\sX} \to \trunc \LL_{\sX}$ on the classical truncation $\sX:=\XX_{\cl}$. 
\item Let $\sigma : \EE\dual[1] \to \O_{\sX}$ be a cosection for an obstruction theory $\phi:\EE\to\trunc\LL_{\sX}$ on a scheme $\sX$. Then we have an induced {\em linear function}
\[\fl_{\sigma}:=\fC(\sigma\dual[1]) : \fC(\EE) \to \bbA^1_{\sX}\]\
on the associated abelian cone stack $\fC(\EE)$.
\end{enumerate}

We now state the main result in this subsection.
We refer to Appendix \ref{Appendix:Artinstacks} for the generalizations of intrinsic normal cones, obstruction theories, and abelian cone stacks to (higher) Artin stacks.

\begin{theorem}\label{Thm:DerivedCL}
Let $\XX$ be a homotopically finitely presented derived Artin stack.
Let $\phi:\EE:=\LL_{\XX}|_{\sX} \to \trunc \LL_{\sX}$ be the induced obstruction theory on the classical truncation $\sX:=\XX_{\cl}$ and
let $\iota : \fC_{\sX} \hookrightarrow \fC(\EE)$ be the induced closed embedding.
\begin{enumerate}
\item[$(1)$] {\bf (Cone reduction, \cite{KL13})}
For any $(-1)$-shifted $1$-form $\alpha \in \cA^1(\XX;-1)$, we have a commutative diagram of cone stacks
\[\xymatrix{
&& \fC(\EE_\sigma) \ar@{^{(}->}[d] \ar[r] & \sX \ar@{^{(}->}[d]^0\\
(\fC_{\sX})_\red \ar@{.>}[rru]^{\iota^{\red}} \ar@{^{(}->}[r] & \fC_{\sX} \ar@{^{(}->}[r]^{\iota} & \fC(\EE) \ar[r]^{\fl_{\sigma}} & \bbA^1_{\sX}
}\]
for a unique dotted arrow where $\sigma:=\alpha|_{\sX}\dual$ and $\EE_{\sigma}:=\cone(\sigma\dual[1])$. 
\item[$(2)$] {\bf (Obstruction theory reduction)}
For any $(-1)$-shifted {\em closed} $1$-form $\talpha \in \cA^{1,\cl}(\XX;-1)$, we have a commutative diagram of complexes
\[\xymatrix{
\O_{\sX}[1] \ar[r]^-{\sigma\dual[1]} & \EE \ar[r] \ar[d]_{\phi} & \EE_\sigma \ar@{.>}[ld]^-{\phi^{\red}} \\
& \trunc \LL_{\sX}&
}\]
for a unique dotted arrow where $\alpha:=\talpha_0$ is the underlying $(-1)$-shifted $1$-form, $\sigma:=\alpha|_{\sX}\dual$, and $\EE_{\sigma}:=\cone(\sigma\dual[1])$.
\item[$(3)$] {\bf (Derived reduction, \cite{STV})} For any $(-1)$-shifted {\em exact} $1$-form $\ttalpha \in \cA^{1,\mathrm{ex}}(\XX;-1)\simeq \cA^0(\XX;-1)$ (see Remark \ref{Rem:closed/exactforms} below),
we have a homotopy commutative diagram of derived Artin stacks
\[\xymatrix{
& \XX^{\red} \ar[r] \ar[d] & \XX \ar[d]^0 \\
\sX \ar[r] \ar[ru] & \XX \ar[r]^-{\ttalpha} & \bbA^1_{\XX}[-1]
}\]
for some $\XX^\red$ where the square is homotopy cartesian and the triangle induces isomorphisms $\sX \cong (\XX^{\red})_{\cl} \cong \XX_{\cl}$.
\end{enumerate}
\end{theorem}

In Theorem \ref{Thm:DerivedCL}, we use the following definitions of shifted {\em closed}/{\em exact} forms.

\begin{remark}\label{Rem:closed/exactforms}
Let $\XX$ be a homotopically finitely presented derived Artin stack.
\begin{enumerate}
\item The space of {\em $d$-shifted $p$-forms} is defined as the mapping space
\begin{align*}
\cA^p(\XX;d):=\,&\Map_{\dg^\gr}(\C((-p)), \DR(\XX)[d-p]) \\\cong\, & \Map_{\sfD(\XX)}(\O_{\XX},\wedge^p\LL_\XX[d]).	
\end{align*}
\item The space of {\em $d$-shifted closed $p$-forms} is defined as the mapping space
\begin{align*}
\cA^{p,cl}(\XX;d):=\,&\Map_{\dg^\gr}(\C((-p)),\NC(\XX)[d-p]) \\ \cong\,& \Map_{\epsilon\textrm{-}\dg^\gr}(\C((-p)),\DR(\XX)[d-p]).
\end{align*}
\item [(3)] The space of {\em $d$-shifted exact $p$-forms} is defined as
\[\cA^{p,\mathrm{ex}}(\XX;d):=\,\Map_{\dg^\gr}(\C((-p+1)),\CC(\XX)[d-p+1]).\]
\end{enumerate}
Then we have canonical maps of spaces
\[\xymatrix{
\cA^{p,\mathrm{ex}}(\XX;d) \ar[r] & \cA^{p,cl}(\XX;d) \ar[r]^-{(-)_0} & \cA^{p}(\XX;d).
}\]
%
\end{remark}

Before we prove Theorem \ref{Thm:DerivedCL}, we explain how the three reductions in Theorem \ref{Thm:DerivedCL} are related.

\begin{remark}\label{Rem:RelationbetweenReductions}
In the situation of Theorem \ref{Thm:DerivedCL}, we have the following:
\begin{enumerate}
\item The obstruction theory reduction in Theorem \ref{Thm:DerivedCL}(2) is equivalent to the {\em scheme-theoretical cone reduction}, i.e., there exists a commutative diagram of cone stacks
\[\xymatrix{
& \fC(\EE_\sigma) \ar@{^{(}->}[d] \ar[r] & \sX \ar@{^{(}->}[d]^0\\
\fC_{\sX} \ar@{.>}[ru] \ar@{^{(}->}[r]^{\iota} & \fC(\EE) \ar[r]^{\fl_{\sigma}} & \bbA^1_{\sX}
}\]
for some dotted arrow. 
Hence the obstruction theory reduction in Theorem \ref{Thm:DerivedCL}(2) clearly implies the cone reduction in Theorem \ref{Thm:DerivedCL}(1).
\item The derived reduction in Theorem \ref{Thm:DerivedCL}(3) implies the obstruction theory reduction in Theorem \ref{Thm:DerivedCL}(2). Indeed, the commutative diagram of derived Artin stacks induces a commutative diagram of cotangent complexes
\[\xymatrix@C+1pc{
\O_{\sX}[1] = \LL_{\XX^\red/\XX}|_{\sX}[-1] \ar[r]^-{d_{DR}(\ttalpha)_0} &\LL_{\XX}|_{\sX} \ar[r] \ar[d] & \LL_{\XX^\red}|_{\sX} \ar@{.>}[ld] \\
& \LL_{\sX} &
}\]
where $\LL_{\XX^\red/\XX}|_{\sX} = \LL_{\XX/\bbA^1_{\XX}[-1]}|_{\sX} = \O_{\sX}[2]$. By composing with the canonical map $\LL_{\sX} \to \trunc \LL_{\sX}$, we obtain the desired obstruction theory reduction.
\end{enumerate}
\end{remark}

Note that the cone reduction in Theorem \ref{Thm:DerivedCL}(1) was shown by Kiem-Li \cite{KL13} (see Lemma \ref{Lem:KLconereduction}) and the derived reduction in Theorem \ref{Thm:DerivedCL}(3) is trivial. (This approach was first introduced by Sch\"urg-To\"en-Vezzosi \cite{STV}.) 
Thus the essential part of Theorem \ref{Thm:DerivedCL} is the obstruction theory reduction in Theorem \ref{Thm:DerivedCL}(2). The following {\em derived Poincar\'e lemma} (which was essentially shown by Brav-Bussi-Joyce \cite{BBJ}) shows that $(-1)$-shifted closed $1$-forms are always exact and hence Theorem \ref{Thm:DerivedCL}(2) follows from Theorem \ref{Thm:DerivedCL}(3).

\begin{proposition}\label{Prop:DerivedPoincare}
Let $\XX$ be a homotopically finitely presented derived Artin stack.
Then the canonical map
\[\cA^{1,\mathrm{ex}}(\XX;-1) \xrightarrow{} \cA^{1,\cl}(\XX;-1)\]
is an equivalence.
\end{proposition}

\begin{proof}
Since the statement is local, we may assume that $\XX$ is a derived affine scheme such that the associated analytic space of the classical truncation is connected.

Let $\PC(\XX)$ be the cofibre of the map $B: \CC(\XX)((-1))[1] \to \NC(\XX)$ in $\dg^\gr$ induced by the map $\DR(\XX)((-1))[1] \to \NC(\XX)$ in \cite[Rem.~1.6]{PTVV} via adjunction.
Consider the induced morphism of cofibre sequences
\begin{equation}\label{Eq:A4.2}
\xymatrix{
\NC(\XX)(p) \ar[r]^-{I}\ar[d] & \PC(\XX)(p) \ar[r]^-{S}\ar[d] & \CC(\XX)(p-1)[2] \ar@{=}[d]\\
\DR(\XX)(p) \ar[r]^-{I} & \CC(\XX)(p) \ar[r]^-{S} & \CC(\XX)(p-1)[2]
}	
\end{equation}
in $\dg$.
Note that we have canonical equivalences
\begin{equation}\label{Eq:A4.1}
\PC(\XX)(p) \cong \PC(\sX)(p) \cong R\Gamma(\widehat{\sY},\hOmega_{\sY}^{\bullet})[2p] \cong R\Gamma(\sX^\an,\C)[2p]
\end{equation}
in $\dg$ where $\hOmega_{\sY}^{\bullet}$ is Hartshorne's algebraic de Rham complex \cite{Har} for an embedding $\sX \hookrightarrow \sY$ into a smooth scheme $\sY$ and $\sX^\an$ is the analytic space associated to $\sX$.
Based on the HKR isomorphism \cite{TV11,BZN} (see also \cite{Hoy}),
the first equivalence follows from Goodwillie's nil-invariance \cite[Thm.~IV.2.1]{Goo1}, 
the second equivalence follows from the Feigin-Tsygan theorem \cite[Thm.~5]{FT} (see also \cite[Thm.~2.2]{Emm} and \cite[Thm.~3.4]{W4}), 
and the third equivalence follows from Hartshorne's comparison theorem \cite[Thm.~IV.1.1]{Har}.

We first claim that the canonical map given by the upper sequence in \eqref{Eq:A4.2}
\begin{equation}\label{Eq:A4.3}
\Map_{\dg}(\C,\CC(\XX)(0)[-1]) \xrightarrow{B} \Map_{\dg}(\C,\NC(\XX)(1)[-2])	
\end{equation}
is an equivalence. Indeed, since $H^k(-)$ of \eqref{Eq:A4.1} vanishes for $p=1$ and $k \leq -3$, it suffices to show that
\[H^{-2}(\PC(\XX)(1)) \xrightarrow{B} H^{0}(\CC(\XX)(0))\]
is injective.
Then we have a commutative square
\[\xymatrix{
H^{-2}(\PC(\XX)(1)) \ar[r]^S \ar[d]^{\cong} & H^0(\CC^0(\XX)(0))\ar[d]\\
H^{-2}(\PC(\XX))(1) \ar@{^{(}->}[r]^S & H^0(\CC^0(\XX)(0))
}\]
where the left vertical arrow is bijective by \eqref{Eq:A4.1} and the lower horizontal arrow is injective by \cite[Prop.~2.6]{Emm}.
Hence we have the desired injectivity.

Then the canonical map given by the lower sequence in \eqref{Eq:A4.2}
\begin{equation}\label{Eq:A4.4}
\Map_{\dg}(\C,\DR(\XX)(0)[-1]) \xrightarrow{I} \Map_{\dg}(\C,\CC(\XX)(0)[-1])	
\end{equation}
is an equivalence since $\CC(\XX)(-1) \cong 0$.
By combining the two equivalences \eqref{Eq:A4.3} and \eqref{Eq:A4.4}, we can deduce that the canonical map
\[\Map_{\dg}(\C,\DR(\XX)(0)[-1]) \xrightarrow{d_{DR}} \Map_{\dg}(\C,\NC(\XX)(1)[-2])\]
is an equivalence as desired.
\end{proof}

\begin{remark}
The derived Poincar\'e lemma also appears in \cite[Lem.~6.11]{AKLPR}.
We refer readers who are interested in a generalization of Kiem-Li's cosection localization \cite{KL13}, using derived algebraic geometry and homotopical methods in intersection theory, to \cite{AKLPR} (we only consider the cone reduction part in this paper).
\end{remark}

\subsection{Shifted closed $1$-forms on moduli of perfect complexes}\label{ss:shifteclosed1forms}

In this subsection, we prove that the cosections given by the semi-regularity maps can be enhanced to shifted {\em closed} $1$-forms. 
This is achieved by generalizing the {\em integration map} of \cite{PTVV} and defining the {\em hat product} for derived mapping stacks.

We first generalize the {\em integration map} in \cite{PTVV} for Calabi-Yau varieties to arbitrary smooth projective varieties.

\begin{theorem}\label{Prop:Integration}
Let $X$ be a smooth projective variety of dimension $n$.
Let $\MM$ be a derived Artin stack.
Then there exists a natural map 
\[\int_{\pi}(-) : \DR(X\times\MM) \to \DR(\MM)((-n))\]
in $\epsilon\textrm{-}\dg^\gr$ where $\pi: X\times\MM\to\MM$ is the projection map. 
 Consequently, we have a natural homotopy-commutative square
\[\xymatrix{
\cA^{p,\cl}(X\times\MM,i) \ar[r]^-{\int_{\pi}(-)} \ar[d] & \cA^{p-n,\cl}(\MM,i-n) \ar[d]\\
\cA^p(X\times\MM,i) \ar[r]^-{\int_{\pi}(-)} & \cA^{p-n}(\MM,i-n).
}\]
\end{theorem}


\begin{proof}
By the K\"unneth formula in Lemma \ref{Lem:Kunneth} below, we have a canonical equivalence
\[\DR(\MM) \otimes \DR(X) \xrightarrow{\cong} \DR(X\times \MM) \]
in $\epsilon\textrm{-}\cdga^\gr$.
Hence it suffices to construct a natural integration map for $X$,
\[\int_X (-) : \DR(X) \to \C((-n))\]
in $\epsilon\textrm{-}\dg^\gr$. Equivalently, it suffices to construct a natural map
\[\CC(X)(n) \to \C\]
in $\dg$ since we have an adjunction $\CC(-)(n) \dashv (-)((-n))$.
The canonical equivalence
\[H^0(\CC(X)(n)) \cong \HH^{2n}(X,\Omega^{\leq n}_X) \xleftarrow{\cong} H^n(X,\Omega^n_X) \xrightarrow{\cong} \C\]
gives us the desired map.
\end{proof}

We provide a K\"unneth formula for $\DR$ in the following lemma. This lemma is required to complete the proof of Theorem \ref{Prop:Integration}.

\begin{lemma}\label{Lem:Kunneth}
Let $X$ be a smooth projective variety. 
Let $\MM$ be a derived Artin stack.
Then there is a canonical equivalence
\[\DR(X) \otimes \DR(\MM) \to \DR(X\times \MM)\]
in $\epsilon\textrm{-}\cdga^\gr$.
\end{lemma}

\begin{proof}

Form a homotopy-cartesian square of derived Artin stacks
\begin{equation}\label{Eq:Kun1}
\xymatrix{
X \times \MM \ar[r]^-{q} \ar[d]^{\pi} & X \ar[d] \\
\MM \ar[r] & \Spec(\C).
}\end{equation}
Consider the canonical map 
\begin{equation}\label{Eq:Kun2}
q^*(-)\cup \pi^*(-)	: \DR(X) \otimes \DR(\MM) \to \DR(X\times\MM) 
\end{equation}
in $\epsilon\textrm{-}\cdga^\gr$ given by the universal property of the coproduct.
It suffices to show that the map \eqref{Eq:Kun2} induces an equivalence in $\dg^\gr$ since the forgetful functor
\begin{equation*}\label{Eq:Kun3}
\epsilon\textrm{-}\cdga^\gr \to \cdga^\gr \to \dg^\gr	
\end{equation*}
is conservative.

Note that we have a canonical equivalence
\begin{equation}\label{Eq:Kun4}
\DR(\MM) \cong \bigoplus_p R\Gamma(\MM,\wedge^p\LL_\MM[p])	
\end{equation}
in $\cdga^\gr$, if we forget the mixed structure in $\DR(\MM)$ (see \cite[Rem.~2.4.4]{CPTVV}).
Moreover, since we have a canonical equivalence $\LL_{X\times\MM}\cong q^*(\LL_X)\oplus \pi^*(\LL_\MM)$ in $\sfD(X\times\MM)$, we also have an induced canonical equivalence
\[\wedge^p\LL_{X\times\MM} \cong \bigoplus_{p=p_1+p_2}q^*(\wedge^{p_1}\LL_X) \otimes \pi^*(\wedge^{p_2}\LL_\MM)\]
in $\sfD(X\times\MM)$.
Applying \eqref{Eq:Kun4} for $X,\MM,X\times\MM$ to the canonical map \eqref{Eq:Kun2}, it suffices to show that the canonical map
\begin{equation}\label{Eq:Kun5}
R\Gamma(X,\wedge^{p_1} \LL_X) \otimes R\Gamma(\MM,\wedge^{p_2}\LL_\MM) \to R\Gamma(X\times\MM,q^*(\wedge^{p_1}\LL_X) \otimes \pi^*(\wedge^{p_2}\LL_\MM))	
\end{equation}
in $\dg$ is an equivalence.

We will factor the map \eqref{Eq:Kun5} as follows: Firstly, consider the canonical map
\begin{equation}\label{Eq:Kun6}
\pi_*q^*(\wedge^{p_1}\LL_X) \otimes \wedge^{p_2}\LL_\MM \to \pi_*(q^*(\wedge^{p_1}\LL_X)	\otimes \pi^*(\wedge^{p_2}\LL_\MM))
\end{equation}
in $\sfD(\MM)$ given by the adjunction $\pi^*\dashv \pi_*$.
Secondly, consider the canonical base change map
\begin{equation}\label{Eq:Kun7}
R\Gamma(X,\wedge^{p_1}\LL_X) \otimes \O_\MM \to \pi_*q^*(\wedge^{p_1}\LL_X)
\end{equation}
in $\sfD(\MM)$ given by the cartesian square \eqref{Eq:Kun1}. 
Finally, consider the canonical map
\begin{equation}\label{Eq:Kun8}
R\Gamma(X,\wedge^{p_1}\LL_X)\otimes R\Gamma(\MM,\wedge^{p_2}\LL_M) \to R\Gamma(\MM,R\Gamma(X,\wedge^{p_1}\LL_X) \otimes \wedge^{p_2}\LL_\MM)
\end{equation}
in $\dg$ given by the adjunction $(-)\otimes\O_\MM\dashv R\Gamma(\MM,-)$.
Then the map \eqref{Eq:Kun5} is the composition of \eqref{Eq:Kun8}, $R\Gamma(\MM,(-)\otimes \wedge^{p_2}\LL_\MM)$ of the map \eqref{Eq:Kun7}, and $R\Gamma(\MM,-)$ of the map \eqref{Eq:Kun6}.
Hence it suffices to show that the three maps \eqref{Eq:Kun6},\eqref{Eq:Kun7},\eqref{Eq:Kun8} are equivalences.

The map \eqref{Eq:Kun6} is the projection formula for $\pi:X\times\MM\to\MM$
and the map \eqref{Eq:Kun7} is the base change formula for the square \eqref{Eq:Kun1}.
Both of these maps are equivalences by \cite[Prop.~3.10]{BZFN} since $X \to \Spec(\C)$ is a perfect morphism in the sense of \cite[Def.~3.2]{BZFN} by \cite[Cor.~3.23,~Prop.~3.19]{BZFN}.

The map \eqref{Eq:Kun8} 
is an equivalence since $R\Gamma(X,\wedge^{p_1}\LL_X)$ is a perfect complex in $\dg$ 
and both the functors $(-)\otimes R\Gamma(\MM,\wedge^{p_2}\LL_\MM), R\Gamma(\MM,(-)\otimes\wedge^{p_2}\LL_\MM)$ commute with finite colimits.
\end{proof}

We compare the integration map in Theorem \ref{Prop:Integration} and the one in \cite{PTVV}.

\begin{remark}
Let $X$ be a Calabi-Yau $n$-fold with fixed Calabi-Yau $n$-form $\omega : \O_X \xrightarrow{\cong} \Omega^n_X$.
Let $\MM$ be a derived Artin stack.
Then the integration map in \cite{PTVV} can be expressed as
\begin{equation}\label{Eq:PTVV}
\int_{\pi,\PTVV} (-) \cong \int_{\pi} ((-) \cup q^*\omega) : \DR(X\times\MM) \to \DR(\MM)[-n] 	
\end{equation}
where $q:X\times\MM \to X$ and $\pi : X\times\MM\to\MM$ are the projection maps.
More precisely, the right hand side of the equivalence \eqref{Eq:PTVV} is defined as the composition
\begin{multline*}
\DR(X\times\MM)\xrightarrow{\id\otimes \omega} \DR(X\times\MM) \otimes \DR(X)((n))[-n]\\
\xrightarrow{\id\otimes q^*}\DR(X\times\MM) \otimes \DR(X\times\MM)((n))[-n]\\
\xrightarrow{\cup} \DR(X\times\MM)((n))[-n] \xrightarrow{\int_\pi}\DR(\MM)[-n]
\end{multline*}
where $\omega \in H^0(X,\Omega^n_X)\cong\HH^n(X,\Omega^{\geq n}_X)$ is regarded as an element in
\[ \pi_0\Map_{\dg}(\C,\NC(X)(n)[-n]) \cong \pi_0\Map_{\epsilon\textrm{-}\dg^\gr}(\C,\DR(X)((n))[-n])\]
and the cup product $\cup$ is given by the multiplicative structure of $\DR(X\times\MM)$.

We now explain how we can obtain the equivalence \eqref{Eq:PTVV}. Recall that the integration map in \cite{PTVV} is based on the construction of the canonical map
\[\kappa : \DR(X\times\MM) \to R\Gamma(X,\O_X)\otimes  \DR(\MM) \]
in $\epsilon\textrm{-}\dg^\gr$.
This can be reinterpreted as the composition
\[\DR(X\times\MM) \xrightarrow{\cong} \DR(X)\otimes \DR(\MM) \to R\Gamma(X,\O_X)\otimes \DR(\MM)\]
where the first arrow is given by the K\"unneth formula in Lemma \ref{Lem:Kunneth} and the second arrow is induced by the canonical map $\DR(X) \to R\Gamma(X,\O_X)$ that corresponds to the canonical isomorphism $\CC(X)(0) \xrightarrow{\cong} R\Gamma(X,\Omega^{\leq0}_X)$ via adjunction.
Hence it suffices to show that
\[\xymatrix{
\int_{X,\PTVV}{(-)} \cong \int_X {(-) \cup \omega} :  \DR(X) \to \C[-n]
}\]
holds in $\epsilon\textrm{-}\dg^\gr$. This follows from the homotopy-commutative diagram
\[\xymatrix@C-1pc{
\CC(X)(0) \ar@{}[r]|-{\cong} \ar[d]^{(-)\cup\omega}& R\Gamma(X,\Omega^{\leq 0}_X) \ar[d]^{(-)\cup\omega}\ar[r]^-{\omega}
& R\Gamma(X,\Omega^n_X) \ar[r] \ar[d] & H^n(X,\Omega^n_X)[-n] \ar[r]^-{\cong} \ar[d]^-{\cong} & \C[-n] \ar@{=}[d]\\
\CC(X)(n)[-n] \ar@{}[r]|-{\cong} & R\Gamma(X,\Omega^{\leq n}_X[n])\ar@{=}[r] &R\Gamma(X,\Omega^{\leq n}_X[n])\ar[r] & \HH^{2n}(X,\Omega^{\leq n}_X)[-n] \ar[r]^-{\cong} & \C[-n]
}\]
in $\dg$ via adjunction.
\end{remark}

We note that the integration map in Theorem \ref{Prop:Integration} without the closing structures can be described quite simply as follows.

\begin{remark}\label{Rem:Integrationwithoutepsilon}
Let $X$ be a smooth projective variety of dimension $n$.
Let $\MM$ be a derived Artin stack.	
Then the integration map
\[\int_{\pi}(-) : \cA^{p}(X\times\MM;i) \to \cA^{p-n}(X;i-n)\]
can be described as follows: Consider the projection map
\[\wedge^p\LL_{X\times\MM} \cong \bigoplus_{p=p_1+p_2}q^*(\wedge^{p_1}\LL_X) \otimes \pi^*(\wedge^{p_2}\LL_\MM) \to \pi^*(\wedge^{p-n}\LL_\MM)\otimes q^*(\Omega^n_X)\]
in $\sfD(X\times\MM)$. The adjunction $\pi_* \dashv \pi^!\cong \pi^* \otimes q^*(\Omega^n_{X})[n]$ gives us a canonical map
\[\pi_*(\wedge^p\LL_{X\times\MM}) \to \wedge^{p-n}\LL_\MM[-n]\]
in $\sfD(\MM)$. Applying $\Map_{\sfD(\MM)}(\O_{\MM}[-i],-)$, we obtain the above integration map.
\end{remark}

Based on the generalized integration map in Theorem \ref{Prop:Integration}, we can define the {\em hat products} for {\em derived mapping stacks}. 
We first collect some basic facts on derived mapping stacks.

\begin{remark}\label{Rem:DerivedMappingspace}
Let $X$ be a smooth projective variety.
Let $\YY$ be a derived Artin stack.
The derived mapping stack $\UMap(X,\YY)$ is the derived stack determined by the adjunction
\[\Map_{\dSt}(\UU,\UMap(X,\YY)) \cong \Map_{\dSt}(X\times\UU,\YY)\]
for any $\UU \in \dSt$.
By \cite{HLP}, $\UMap(X,\YY)$ is a derived Artin stack when $\YY$ has a quasi-affine diagonal or the classical truncation of $\YY$ is a $1$-Artin stack with affine stabilizers.

The cotangent complex can be computed as follows. For simplicity, assume that $\YY$ is homotopically finitely presented and has quasi-affine diagonal. Let
\[\ev : X\times\UMap(X,\YY) \to \YY\]
be the evaluation map. By taking differential, we have a canonical map
\[\ev_* : \TT_{X\times{\UMap(X,\YY)}} \cong  q^*(T_X) \oplus \pi^*(\TT_{\UMap(X,\YY)})\to \ev^*(\TT_{\YY})\]
where $\pi,q$ are the projection maps. 
The second component of $\ev_*$ induces an isomorphism
\[(\ev_*)_2 : \TT_{\UMap(X,\YY)} \cong \pi_*\ev^*(\TT_\YY).\]
in $\sfD({\UMap(X,\YY)})$ by \cite[Prop.~5.1.10]{HLP}.

We note that the first component of $\ev_*$ gives us 
\[R\Gamma(X,T_X) \otimes \O_{{\UMap(X,\YY)}} \cong \pi_*q^*T_X \xrightarrow{(\ev_*)_1} \pi_*\ev^*(\TT_\YY) \xleftarrow[\cong]{(\ev_*)_2} \TT_{{\UMap(X,\YY)}},\]
and thus we obtain a canonical map
\begin{equation}\label{Eq:xihat}
\widehat{-} : H^i(X,T_X) \to \HH^i({\UMap(X,\YY)},\TT_{\UMap(X,\YY)})	
\end{equation}
between shifted vector fields.
\end{remark}

We now define the {\em hat product} for a derived mapping stack.

\begin{definition}\label{Def:HatProduct}
Let $X$ be a smooth projective variety of dimension $n$. 
Let $\YY$ be a derived Artin stack.
Assume that $\UMap(X,\YY)$ is representable by a derived Artin stack.
We define the {\em hat product}
\[\widehat{-\cdot-} : \DR(X) \otimes \DR(\YY) \to \DR(\UMap(X,\YY))((-n))\]
as the composition
\[\DR(X)\otimes \DR(\YY) \xrightarrow{q^*(-) \cup \ev^*(-)}\DR(X\times\UMap(X,\YY)) \xrightarrow{\int_\pi}\DR(\UMap(X,\YY))((-n))\]
in $\epsilon\textrm{-}\dg^\gr$, where $\ev$ is the evaluation map and $\pi,q$ are the projection maps.
Consequently, we also have a homotopy-commutative square
\[\xymatrix{
\cA^{p,\cl}(X;i) \times \cA^{q,\cl}(\YY;j) \ar[r]^-{\widehat{-\cdot-}} \ar[d] & \cA^{p+q-n,\cl}(\UMap(X,\YY);i+j-n) \ar[d]\\
\cA^p(X;i)\times \cA^q(\YY,j) \ar[r]^-{\widehat{-\cdot-}} & \cA^{p+q-n}(\UMap(X,\YY);i+j-n).
}\]
\end{definition}

We can describe the shifted symplectic forms of \cite{PTVV} on the derived moduli spaces of perfect complexes in terms of the hat product.

\begin{remark}
Let $X$ be a Calabi-Yau $n$-fold with a fixed Calabi-Yau $n$-form $\omega :\O_X \cong \Omega^n_X$.

Let $R\cPerf$ be the derived moduli stack of perfect complexes (of fixed tor-amplitude $[a,b]$).
Then the cotangent complex can be computed as
\[\LL_{R\cPerf} \cong \II\dual \otimes \II[-1]\]
for the universal complex $\II$ on $R\cPerf$.
Moreover, the $2$-shifted non-degenerate $2$-form on $R\cPerf$ given by the duality 
$(\LL_{R\cPerf}[1])\dual \cong \LL_{R\cPerf}[1]$
can be enhanced to a $2$-shifted symplectic $2$-form
\[\Theta \in \cA^{2,\cl}(R\cPerf;2)\]
using the negative cyclic homology valued second Chern character of $\II$ \cite{TV15}.

Let $R\cPerf(X):=\UMap(X,R\cPerf)$ be the derived moduli stack of perfect complexes on $X$. 
Then
\[\theta:=\widehat{\omega\cdot\Theta} =\int_\pi q^*\omega \cup \ev^*\Theta \in \cA^{2,\cl}(R\cPerf(X),2-n)\]
is the $(2-n)$-shifted symplectic form in \cite{PTVV}.
\end{remark}

We now show that the cosections associated to the semi-regularity maps can be enhanced to $(-1)$-shifted closed $1$-forms.

\begin{proposition}\label{Prop:SRclosed}
Let $X$ be a Calabi-Yau $n$-fold with a fixed Calabi-Yau $n$-form $\omega \in H^0(X,\Omega^n_X)$.
For each $(n-3)$-shifted vector field $\xi \in H^{n-3}(X,T_X)$, we can form a $(-1)$-shifted closed $1$-form
\[\alpha_{\xi}\in \cA^{1,\cl}(R\cPerf(X);-1)\]
such that the induced cosection on the obstruction theory $(R\hom_\pi(\II,\II)[1])\dual \to \LL_{\cPerf(X)}$ of the classical truncation $\cPerf(X):=R\cPerf(X)_\cl$ is equal to the composition
\begin{multline*}
R\hom_\pi(\II,\II)[2] \xrightarrow{\At(\II)}R\hom_\pi(\II,\II\otimes \LL_{X\times\cPerf(X)})[3] \\\xrightarrow{} R\hom_\pi(\II,\II\otimes \Omega^1_X)[3]\xrightarrow{\tr} R\Gamma(X,\Omega^1_X)\otimes\O_{\cPerf(X)}[3]\\
\xrightarrow{} H^3(X,\Omega^1_X)\otimes \O_{\cPerf(X)} \xrightarrow{-\int_X(-)\cup\iota_{\xi}\omega}\O_{\cPerf(X)}
\end{multline*}
where $\II$ is the universal complex.
\end{proposition}

\begin{proof}
Note that we have
\[H^{2n-4}(X,\Omega^{\geq n-1}_X) = \bigoplus_{p\geq n-1}H^{2n-4-p}(X,\Omega^p_X) \subseteq H^{2n-4}_{DR}(X)\]
by the Hodge decomposition.
Hence $\iota_{\xi}\omega \in H^{n-3}(X,\Omega^{n-1}_X)$ induces an element
\[[\iota_\xi\omega]:=(\iota_\xi\omega,0,\ldots,0) \in H^{2n-4}(X,\Omega^{\geq n-1}_X).\]
We define the $(-1)$-shifted closed $1$-form as the following hat product in Definition \ref{Def:HatProduct}
\[\alpha_\xi := \widehat{[\iota_\xi\omega]\cdot \Theta} \in \cA^{1,\cl}(R\cPerf(X;-1)).\]

It remains to compute the induced cosection of $\alpha_\xi$ in the classical truncation. By Lemma \ref{Lem:Hatproduct} below, the underlying $(-1)$-shifted $1$-form can be written as
\[(\alpha_\xi)_0= \widehat{(\iota_\xi\omega)\cdot \Theta_0} =\iota_{\widehat{\xi}}(\widehat{\omega\cdot\Theta_0})=\iota_{\widehat{\xi}}(\theta_0) \in \cA^1(R\cPerf(X);-1)\]
where $\widehat{\xi}$ is defined by \eqref{Eq:xihat} in Remark \ref{Rem:DerivedMappingspace}.
Here $(-)_0$ denotes the underlying form after forgetting the closing structure as in Remark \ref{Rem:closed/exactforms}.
By \cite[Rem.~A.1]{STV}, we have
\[\ev_* =\At(\II) : \TT_{X\times R\cPerf(X)} \to R\hom(\II,\II)[1]\]
for the evaluation map $\ev: X\times R\cPerf(X) \to R\cPerf$.
Since $\widehat{\xi}$ is defined by the first composition of $\ev_*$, we have the desired formula by taking dual and classical truncation.
\end{proof}

We need the following technical lemma to complete the proof of Proposition \ref{Prop:SRclosed}.
This formula is inspired by \cite[Lem.~8]{Viz}.

\begin{lemma}\label{Lem:Hatproduct}
Let $X$ be a smooth projective variety of dimension $n$. 
Let $\YY$ be a homotopically finitely presented derived Artin stack. 
Assume that $\UMap(X,\YY)$ is a derived Artin stack.
Then we have
\[\iota_{\widehat{\xi}}(\widehat{\alpha\cdot\beta}) \cong -\widehat{(\iota_{\xi}\alpha)\cdot\beta} \in \cA^{p+q-n-1}(\UMap(X,\YY);i+j-n+k)\]
for $\alpha \in \cA^{p}(X;i)$, $\beta \in \cA^q(\YY;j)$, $\xi \in H^k(X,T_X)$, where $\widehat{\xi}$ is defined by \eqref{Eq:xihat} in Remark \ref{Rem:DerivedMappingspace}.
\end{lemma}

\begin{proof}
For simplicity of notation,
let $\MM:=\UMap(X,\YY)$ and
\[\txi:=(q^*\xi,-\pi^*\hxi) \in \HH^k(X\times\MM,\TT_{X\times\MM}) =H^k(X\times\MM,q^*\TT_X)\oplus \HH^k(X\times\MM,\pi^*\TT_\MM).\]
By the definition of $\hxi$ in \eqref{Eq:xihat}, we have
\[\ev_*(\txi) = 0 \in \HH^k(X\times \MM,\ev^*(\TT_\YY)).\]

Firstly, we have
\begin{align*}
\iota_{\txi}(q^*(\alpha)\cup\ev^*(\beta)) 
&= (\iota_{q_*\txi}q^*(\alpha)) \cup \ev^*(\beta) \pm q^*(\alpha)\cup (\iota_{\ev_*\txi}\ev^*(\beta))	\\
&=q^*(\iota_{\xi}(\alpha))\cup \ev^*(\beta)
\end{align*}
in $\HH^{i+j+k}(X\times\MM,\wedge^{p+q-1}\LL_{X\times\MM})$.

Secondly, we have
\[\int_\pi\iota_{\txi}(q^*(\alpha)\cup\ev^*(\beta)) = - \iota_{\hxi}\int_\pi(q^*(\alpha)\cup\ev^*(\beta))\]
in $\HH^{i+j+k-n}(\MM, \wedge^{p+q-n-1}\LL_\MM)$ (see Remark \ref{Rem:Integrationwithoutepsilon}).

Combining these two formulas, we have
\[\widehat{(\iota_{\xi}\alpha)\cdot\beta}  = \int_\pi\iota_{\txi}(q^*(\alpha)\cup\ev^*(\beta))  = -\iota_{\hxi}\int_\pi(q^*(\alpha)\cup\ev^*(\beta)) =  -\iota_{\widehat{\xi}}(\widehat{\alpha\cdot\beta}),\]
which completes the proof.
\end{proof}

Now we have the desired reduced obstruction theory as follows:

\begin{corollary}\label{Cor:ReducedOTviaDAG}
Let $X$ be a Calabi-Yau $4$-fold, $v \in H^*(X,\Q)$, and $q \in \{-1,0,1\}$. 
Let $\curP := \PTqvX$ and let $\phi:\EE\to\trunc\LL_{\curP}$ be the standard symmetric obstruction theory.
Choose a non-degenerate subspace $V \subseteq H^1(X,T_X)$ with respect to the symmetric bilinear form $\sfB_{\gamma}$ for $\gamma:=v_2 \in H^2(X,\Omega^2_X)$.
Then we have a reduced obstruction theory $\phi^{\red}$ that fits into the commutative diagram
\[\xymatrix{
V\otimes\O_{\curP}[1] \ar[r]^-{\SR_V\dual[1]} & \EE \ar[r] \ar[d]_{\phi} & \EE_V^\red \ar@{.>}[ld]^-{\phi^{\red}} \\
& \trunc \LL_{\curP} .&
}\]
Moreover, this reduced obstruction theory $\phi^\red$ comes from a canonical $(-2)$-shifted symplectic derived enhancement.
\end{corollary}

\begin{proof}
Choose a basis $\xi_1,\cdots,\xi_{\rho_\gamma}\in V$.
For the first statement, by Theorem \ref{Thm:DerivedCL}(2), it suffices to show that the associated cosections $\SR(\xi_i): \EE\dual[1] \xrightarrow{\SR} H^1(X,T_X)\dual\otimes\O_{\curP} \xrightarrow{\xi} \O_{\curP}$ can be enhanced to $(-1)$-shifted closed $1$-forms on the standard derived enhancement $R\curP$ of $\curP$.
By Proposition \ref{Prop:SRclosed}, $\SR(\xi_i)$ are the underlying cosections of the $(-1)$-shifted closed $1$-forms $u^*(\alpha_{\xi_i})$, where $u:R\curP \to R\cPerf(X)$ is given by the universal complex on $X\times R\curP$. 

For the second statement, since the classical truncation $\curP$ is quasi-projective, by Proposition \ref{Prop:DerivedPoincare}, the closed $1$-forms $u^*(\alpha_{\xi_i})$ can be lifted to unique exact $1$-forms $\widetilde{u^*(\alpha_{\xi_i})} \in \cA^{1,\mathrm{ex}}(R\curP;-1) = \cA^{0}(R\curP;-1)$.
By taking the derived zero locus of these $(-1)$-shifted functions, we have a reduced derived scheme $R\curP^\red \subseteq R\curP$ which induces the reduced obstruction theory as above. Moreover, we can pullback the $(-2)$-shifted symplectic form $\theta \in \cA^{2,\cl}(R\curP;-2)$ via the inclusion map $R\curP^\red \hookrightarrow R\curP$, which is also non-degenerate.
\end{proof}

\begin{remark}\label{Rem:A16}
In the situation of Corollary \ref{Cor:ReducedOTviaDAG}, the existence of the reduced obstruction theory is equivalent to the scheme-theoretical cone reduction (see Remark \ref{Rem:RelationbetweenReductions}(1)).
Hence we have a commutative diagram
\begin{equation*}\label{Eq:A16}
\xymatrix{
&\fQ(\EE_V^\red) \ar@{^{(}->}[r] \ar@{^{(}->}[d] & \fC(\EE_V^\red) \ar@{^{(}->}[d]  \ar@/^1cm/[dd]^{\fq_{\EE^\red_V}}\\
\fC_{\curP} \ar@{.>}[ru]^{\exists} \ar@{^{(}->}[r] \ar[rru] & \fQ(\EE) \ar@{^{(}->}[r]\ar[d] & \fC(\EE) \ar[d]^{\fq_\EE} \\
& \curP \ar[r]^0 & \bbA^1_{\curP}.
}    
\end{equation*}
In particular, the existence of a dotted arrow implies the isotropic condition.

The main difference between Corollary \ref{Cor:ReducedOTviaDAG} and Proposition \ref{prop:redSOT} is that the reduced obstruction theory in Corollary \ref{Cor:ReducedOTviaDAG} is given as a {\em factorization}, whereas the one in Proposition \ref{prop:redSOT} is given as a {\em composition}.
However note that by \eqref{Eq:4.22.1}, the two actually coincide.
\end{remark}

\begin{remark}
The result in this section does not replace section \ref{sec:VFC}. Corollary \ref{Cor:ReducedOTviaDAG} only replaces Kiem-Li's cone reduction lemma in Lemma \ref{Lem:KLconereduction}. We still need the other arguments in section \ref{sec:VFC} to define an Oh-Thomas virtual cycle from Corollary \ref{Cor:ReducedOTviaDAG}. Indeed, we need to compute the square of $\SR$ to show that $\SR_V$ is a {\em non-degenerate} cosection (see Remark \ref{Rem:non-deg/isotropic.cosection}) and hence the induced symmetric form on $\EE^\red_V$ is non-degenerate. 
\end{remark}

\begin{remark}\label{Rem:BlochConj}
Corollary \ref{Cor:ReducedOTviaDAG} provides a simple alternative proof of the Bloch conjecture for surface classes on Calabi-Yau $4$-folds, which was shown in general by Pridham \cite{Pri} and Bandiera-Lepri-Manetti \cite{BLM} (and by Iacono-Manetti \cite{IM} when the normal bundle is extendable).

Let $X$ be a Calabi-Yau $4$-fold and $S\subseteq X$ be a local complete intersection surface.
Consider an extension $\cA \hookrightarrow \bcA$ of local Artinian $\C$-schemes such that $\m_{\cA}\cdot \J=0$ where $\J:=\I_{\cA/\bcA}$.
Let $\cS \subseteq X\times\cA$ be a $\cA$-flat extension of $S$.
Let
\[\sfob(\cS) \in H^1(X,N_{S/X}) \otimes \J \cong \Ext^2_X(\I_{S/X},\I_{S/X})_0 \otimes  \J\]
be the obstruction of deforming $\cS\subseteq X\times \cA$ to $X\times\bcA$ in \cite[Prop.~2.6]{Blo}.
By \cite{Blo}, the obstruction $\sfob(\cS)$ lies in the kernel of the semi-regularity map
\[\sr : H^1(X,N_{S/X}) \cong \Ext^2_X(\I_{S/X},\I_{S/X})_0 \to H^3(X,\Omega_X^1)\]
when the canonical map
\begin{equation}\label{Eq:A.19}
d_{DR} : \J \to \Omega^1_{\bcA}\otimes \O_{\cA}	
\end{equation}
is injective. The Bloch conjecture states that the same result holds without the assumption on the injectivity of \eqref{Eq:A.19}.
(See Proposition \ref{prop:4.6.1} for the comparison of the obstruction/semi-regularity maps for closed subschemes \cite{Blo} and ideal sheaves \cite{BF03}.)

Let $\curP:=\curP^{(-1)}_v(X)$ be the Hilbert scheme of subschemes of Chern character $v=\ch(\O_S)$.
Let $u : \cA \to \curP:=\curP^{(-1)}_v(X)$ be the canonical map given by $\cS$. By \cite[Thm.~4.5]{BF}, the obstruction $\sfob(\cS)$ is equivalent to the composition
\[u^*\EE \xrightarrow{u^*\phi} u^*(\trunc\LL_{\PTqvX}) \to \trunc\LL_{\cA} \to \trunc\LL_{\cA/\bcA} \cong \J[1]\]
where $\phi:\EE\to\trunc\LL_{\PTqvX}$ is the standard obstruction theory.
Instead of choosing a non-degenerate subspace $V$ in Corollary \ref{Cor:ReducedOTviaDAG}, we now consider all the cosections that come from $H^1(X,T_X)$.
By the arguments in Corollary \ref{Cor:ReducedOTviaDAG}, we can form a reduced obstruction theory $\phi'$ that fits into the diagram
\[\xymatrix{
H^1(X,T_X)\otimes\O_{\curP}[1] \ar[r]^-{\SR\dual[1]} & \EE \ar[r] \ar[d]_{\phi} & \EE' \ar@{.>}[ld]^-{\phi'} \\
& \trunc \LL_{\curP} .&
}\]
Hence the obstruction $\sfob(\cS)$ comes from the left term in the exact sequence,
\[\xymatrix@C-1.2pc{h^1(\EE'|_{[S]}\dual \otimes \J) \ar[r] & h^1(\EE|_{[S]}\dual \otimes \J) \cong \Ext^2_X(\I_{S/X},\I_{S/X})_0\otimes \J \ar[r]^-{\sr} & H^1(X,T_X)\dual \otimes \J,}\]
which implies $\sr(\sfob(\cS))=0$.
\end{remark}


\begin{remark}
The results in this section can also be applied to the reduced curve counting theory for Calabi-Yau $3$-folds with holomorphic $2$-forms in \cite{KT1}.
Indeed, Proposition \ref{Prop:SRclosed} implies that the cosections associated to the semi-regularity map and a $(2,0)$-form in \cite{KT1} can be enhanced to a $(-1)$-shifted closed $1$-form. This in particular provides reduced derived enhancements.

However, the reduced curve counting theory for $K3$ surfaces in \cite{MPT} does not require the result in this subsection \ref{ss:shifteclosed1forms} (but only the previous subsection \ref{ss:ThreeReductions}), and was already studied in \cite{STV}.
The main reason is that the cosections in \cite{MPT} do not involve the Atiyah classes and we can see directly that it comes from the de Rham differential of $(-1)$-shifted functions. 
\end{remark}


\section{Moduli stacks of semi-stable sheaves}\label{Appendix:Artinstacks}

In this section, we generalize the results in the previous sections to the moduli {\em stack} of  semi-stable sheaves.
In particular, we construct reduced Oh-Thomas virtual cycles on these moduli stacks and prove that they are deformation invariant along Hodge loci.


We will use the following notations/conventions in this section.
\begin{enumerate}	
\item We use the language of $\infty$-categories \cite{Lur1,Lur2} and derived algebraic geometry \cite{ToVe} as in Appendix \ref{Appendix:ReductionviaDAG}.
\begin{itemize}
\item An {\em algebraic stack} will always be a classical $1$-Artin stack. 
\end{itemize}
\item We use the following concepts developed by Aranha-Pstragowski \cite{AP}.
These are natural generalizations of analogous notions for DM stacks in \cite{BF,Man}.
\begin{itemize}
\item To any algebraic stack $\sX$ and a derived category object $\EE \in \sfD(\sX)$, 
we can associate a {\em (higher) abelian cone stack} \cite[Section 3]{AP}
\[\qquad\qquad\fC(\EE):=\Spec\Sym\EE[-1] :  T \in \mathsf{Sch} \mapsto \Map_{\sfD(T)}(\EE[-1]|_T,\O_T).\]
When $\trunc\EE \cong [D \to B\dual \to S\dual]$ for a coherent sheaf $D$ and vector bundles $B,S$, then we write $\fC(\EE):= \left[\sfrac{C(D)}{\left[\sfrac{B}{S}\right]}\right]$.
\item Let $\fC_{\sX/\sY} \subseteq \fC(\LL_{\sX/\sY})$ denote the {\em intrinsic normal cone} for a morphism of algebraic stacks $f:\sX \to \sY$, see \cite[Thm.~6.2, Thm.~6.3]{AP}.
\item An {\em obstruction theory}\footnote{In \cite[Def.~8.1]{AP}, the full cotangent complex $\LL_{\sX/\sY}$ is used in the definition of an obstruction theory, but here we use the truncated cotangent complex $\trunc\LL_{\sX/\sY}$. } is a morphism $\phi :\EE \to \trunc\LL_{\sX/\sY}$ in $\sfD(\sX)$ such that $h^0(\phi)$, $h^{1}(\phi)$ are bijective, and $h^{-1}(\phi)$ is surjective.
An obstruction theory $\phi :\EE \to \trunc\LL_{\sX/\sY}$ induces a closed embedding $\fC_{\sX/\sY}\hookrightarrow \fC(\EE)$ by \cite[Prop.~8.2]{AP}.
\end{itemize}
\item We consider the following generalizations of \cite{OT,Par}.
\begin{itemize}
\item We consider the {\em symmetric complexes} of tor-amplitude $[-3,1]$ as in \cite[Def.~C.1]{Par}. They are perfect complexes $\EE$ on algebraic stacks $\sX$ equipped with $(-2)$-shifted symmetric forms $\theta : \O_{\sX} \to \Sym^2(\EE[-1])$ such that $\iota(\theta):\EE\dual[2]\cong \EE$.

Note that there is a {\em canonical quadratic function} $\fq_\EE:\fC(\EE) \to \bbA^1_{\sX}$ associated to the symmetric form $\theta$.
We denote by $\fQ(\EE)\subseteq \fC(\EE)$ the zero locus of $\fq_\EE$.
\item An {\em orientation} of a symmetric complex $\EE$ of tor-amplitude $[-3,1]$ is an isomorphism of line bundles $o : \O_{\sX} \to \det(\EE)$ such that $\det(\iota(\theta))= (-1)^{\frac{r(r-1)}{2}}o\circ o\dual$ where $r=\rank(\EE)$.
\item A {\em symmetric obstruction theory} for a morphism $f:\sX\to\sY$ of algebraic stacks is a pair of a symmetric complex $\EE$ of tor-amplitude $[-3,1]$ on $\sX$ and an obstruction theory $\phi:\EE \to \trunc \LL_{\sX/\sY}$ for $f$.
\end{itemize}
\end{enumerate}

\subsection{Reduced virtual cycles for Artin stacks}

In this subsection, we extend Oh-Thomas virtual cycles \cite{OT} (more generally, square root virtual pullbacks \cite{Par}) and Kiem-Li's cone reduction lemma \cite{KL13} to Artin stacks (which are global quotients). 
Consequently, we have reduced Oh-Thomas virtual cycles for these Artin stacks.

We first generalize {\em Oh-Thomas virtual cycles} \cite{OT} (and {\em square root virtual pullbacks} \cite{Par}) to Artin stacks.
\begin{theorem}\label{Thm:OTforArtin}
Let $f : \sX \to \sY$ be a morphism of algebraic stacks.
Let $\phi : \EE \to \trunc \LL_{\sX/\sY}$ be a symmetric obstruction theory of tor-amplitude $[-3,1]$ such that
$(\fC_{\sX/\sY})_{\red} \subseteq \fQ(\EE)$ as substacks of $\fC(\EE)$. 
Let $o : \O_{\sX} \to \det(\EE)$ be an orientation.
Assume that $\sX$ is a global quotient stack, i.e., there exists a quasi-projective scheme $P$ with a linear action of a linear algebraic group $G$ such that $\sX \cong [P/G]$.
Then there is a canonical square root virtual pullback
\[\sqrt{f^!}: A_*(\sY) \to A_*(\sX).\]
Moreover, for any cartesian square
\[\xymatrix{
\sX' \ar@{^{(}->}[r] \ar[d]^{f'} & \sX\ar[d]^f\\
\sY' \ar@{^{(}->}[r]^i & \sY
}\]
where $i$ is a regular closed embedding, we have
\begin{equation}\label{Eq:DefInvArtin}
i^! \circ \sqrt{f^!} = \sqrt{(f')^!} \circ i^! : A_*(\sY) \to A_*(\sX').	
\end{equation}
\end{theorem}



\begin{remark}
In the situation of Theorem \ref{Thm:OTforArtin}, when $\sY=\Spec(\C)$,  we define the {\em Oh-Thomas} virtual cycle as
\[[\sX]^\vir:=\sqrt{f^!}[\sY] \in A_*(\sX).\]
Then the formula \eqref{Eq:DefInvArtin} gives us the deformation invariance.
\end{remark}

\begin{proof}[Proof of Theorem \ref{Thm:OTforArtin}]
We prove Theorem \ref{Thm:OTforArtin} in four steps.

\hfill

\noindent{\em Step 1.}
We first show that there exists a {\em symmetric resolution}
\begin{equation}\label{Eq:A.3}
\EE \cong [S \xrightarrow{e} B \xrightarrow{d} E \cong E\dual \xrightarrow{d\dual} B\dual \xrightarrow{e\dual} S\dual]	
\end{equation}
as in \cite[Prop.~4.1]{OT}.
More precisely, \eqref{Eq:A.3} is an equivalence in $\sfD(\sX)$ for some vector bundles $B$, $S$ and a special orthogonal bundle $(E,\fq_E)$ such that the symmetric form of $\EE$ is represented by 
\[\xymatrix{
\EE\dual[2] \ar[d]^{\iota(\theta)} \ar@{}[r] & S \ar[r]^{e} \ar@{=}[d] & B \ar[r]^{d}  \ar@{=}[d] & E \ar[r]^-{d\dual\circ\fq_E} \ar[d]^{\fq_E} & B\dual \ar[r]^{e\dual} \ar@{=}[d] & S\dual \ar@{=}[d] \\
\EE \ar@{}[r] & S \ar[r]^{e} & B \ar[r]^-{\fq_E\circ d} & E\dual \ar[r]^-{d\dual} & B\dual \ar[r]^{e\dual} & S\dual.
}\]
Since the global quotient stack $\sX$ has the resolution property \cite[Lem.~2.6]{Tho87}, we can find a resolution $A^{\bullet}:= [A^{-3}\xrightarrow{a^{-2}} A^{-2} \xrightarrow{a^{-1}} A^{-1} \xrightarrow{a^0} A^0 \xrightarrow{a^1} A^1] \cong \EE$ such that $\iota(\theta):\EE\dual[2] \to \EE$ is represented by a chain map $\theta^{\bullet} : (A^{\bullet})\dual[2] \to A^{\bullet}$ satisfying $(\theta^{\bullet})\dual[2]=\theta^{\bullet}.$ Let $S:=A^{-3}$ and
\[B:=\coker((A^1)\dual \xrightarrow{((a^1)\dual,\theta^{-3})} (A^0)\dual \oplus A^{-3}).\]
Then $B\dual=\ker(A^{0}\oplus (A^{-3})\dual \xrightarrow{(a^1,\theta^{1})} A^1)$ and we have a quasi-isomorphism
\[c^{\bullet} : B^{\bullet}:=[A^{-3}\xrightarrow{a^{-2}} A^{-2} \xrightarrow{a^{-1}} A^{-1} \xrightarrow{b^0} B\dual \xrightarrow{b^1} S\dual] \xrightarrow{\cong} A^{\bullet}\] 
where $b^0:=(a^0,0)$, $b^1:=\mathrm{pr}_2$, $c^{\leq-1}:=\id$, $c^0:=\mathrm{pr}_1$, and $c^1:=\theta^1$.
Also $\theta^{\bullet}$ factors through a self-dual map $\tau^{\bullet} : (B^{\bullet})\dual[2] \to B^{\bullet}$. Since $\tau^1$ is an isomorphism, we can apply \cite[Prop.~4.1]{OT} to the middle three terms of $B^{\bullet}$. Let 
\[E:=\coker(B \to (A^{-1})\dual\oplus A^{-2}) \cong \ker(A^{-1}\oplus (A^{-2})\dual \to B\dual)=E\dual,\]
then we have \eqref{Eq:A.3} as claimed.\footnote{We thank Arkadij Bojko for pointing out an omission in an earlier version of Step 1.}

\hfill

\noindent{\em Step 2.} We then construct $\sqrt{f^!}$ as follows:
consider the fibre diagram
\[\xymatrix{
C(D) \ar@{^{(}->}[r] \ar@{->>}[d] &  E \ar@{->>}[d] \\
\fC(\EE) \ar@{^{(}->}[r] & \left[\sfrac{E}{\left[\sfrac{B}{S}\right]}\right]
}\]   
where $D:=\coker(d:B \to E)$. The horizontal maps are closed embeddings and the vertical maps are $[\sfrac{B}{S}]$-torsors. Then we have
\[\fq_{\EE}|_{C(D)} = \fq_E|_{C(D)}\]
since the symmetric forms of $\EE$ and $E$ map to the same (degenerate) symmetric form on $[S\xrightarrow{e} B \xrightarrow{d} E \to 0 \to 0]$.
Let $\tau \in \Gamma(C(D),E|_{C(D)})$ be the tautological section.
Then we have a localized Edidin-Graham class 
\begin{equation}\label{Eq:A.1}
\sqrt{e}(E|_{C(D)},\tau) : A_*(C(D)) \to A_*(\sX)
\end{equation}
since $C(D)$ is a global quotient stack. 
See \cite[Rem.~3.3]{KP20} for the extension of the localized Edidin-Graham class \cite[Def.~3.2]{OT} to the equivariant Chow groups \cite{EG2,Tot2}. Here the orientation of $E$ is induced by the orientation of $\EE$ via the isomorphism $\det(E)\cong\det(\EE)$ induced by the symmetric resolution \eqref{Eq:A.3}.

Consider a global factorization
\begin{equation}\label{Eq:5}
\xymatrix{
& \sZ\ar[d]^{h}\\
\sX \ar[ru]^g \ar[r]_f & \sY
}\end{equation}
into a representable morphism $g$ and a smooth morphism $h$ for some algebraic stack $\sZ$. 
Such factorization exists since $\sX$ can be embedded into a smooth algebraic stack $\mathscr{W}$ by assumption and we may take $g:\sX \to \sZ:=\mathscr{W}\times\sY$. Let
\[\xymatrix{
\fC \ar[r]^s \ar[d]^r & C \ar@{^{(}->}[r]^j \ar[d] & C(D) \ar[d] \\
\fC_{\sX/\sZ} \ar[r] & \fC_{\sX/\sY} \ar@{^{(}->}[r] & \fC(\EE)
}\] 
be a fibre diagram of higher Artin stacks. Then we can form maps
\begin{equation}\label{Eq:A.2}
A_*(\sY) \xrightarrow{h^*} A_*(\sZ) \xrightarrow{\mathrm{sp}_{\sX/\sZ}} A_*(\fC_{\sX/\sZ}) \xrightarrow{r^*} A_*(\fC) \xrightarrow{(s^*)^{-1}} A_*(C) \xrightarrow{j_*} A_*(C(D))
\end{equation}
between the Chow groups of algebraic stacks(=$1$-Artin stacks).
Here the specialization map $\mathrm{sp}_{\sX/\sZ}$ is well-defined as in \cite{Man} since $g:\sX\to\sZ$ is representable, and the smooth pullbacks $r^*$, $s^*$ are isomorphisms by \cite[Prop.~4.3.2]{Kre} since $r$ is a $[B/S]$-torsor and $s$ is a $\mathrm{Tot}(\TT_{\sZ/\sY})$-torsor.
By composing the two maps \eqref{Eq:A.1} and \eqref{Eq:A.2}, we can define the square root virtual pullback
\begin{equation}\label{Eq:6}
\sqrt{f^!}: A_*(\sY) \to A_*(\sX)
\end{equation}

It is easy to show that $\sqrt{f^!}$ is independent of the choice of a global factorization \eqref{Eq:5}.
Indeed, we just have to show that \eqref{Eq:A.2} is independent of \eqref{Eq:5}, which can be seen by the arguments in \cite[Lem.~4.6]{KP19}.

\hfill

\noindent{\em Step 3.} Next we show that the square root virtual pullback $\sqrt{f^!}$ in \eqref{Eq:6} is independent of the choice of the symmetric resolution \eqref{Eq:A.3}, as in \cite[Sect.~4.2]{OT}.
This part is quite technical so we will only focus on the main difference with \cite{OT}.
By a deformation argument as in \cite{OT} and adding an acyclic symmetric complex of the form 
\[[T_1 \xrightarrow{\id} T_1 \to 0 \to T_1\dual \xrightarrow{\id}T_1\dual] \oplus [0 \to T_2 \xrightarrow{(1,0)} T_2\oplus T_2\dual \xrightarrow{(0,1)} T_2\dual \to 0],\]
it suffices to consider the following special case. 
Consider a commutative diagram 
\[\xymatrix{
0 \ar[r] & M \ar@{^{(}->}[r] \ar@{^{(}->}[d] & N \ar@{->>}[r] \ar@{^{(}->}[d] & K \ar[r] \ar@{^{(}->}[d] & 0 \\
 & S \ar[r]^e & B \ar[r]^d & E 
}\]
where the upper row is exact and the vertical arrows are injective. We will compare the two symmetric resolutions
\begin{align*}
\EE &\cong [S \xrightarrow{e} B \xrightarrow{d} E \xrightarrow{d\dual} B\dual \xrightarrow{e\dual} S\dual]\\
\EE &\cong [(S/M) \xrightarrow{e} (B/N) \xrightarrow{d} (K^{\perp}/K) \xrightarrow{d\dual} (B/N)\dual \xrightarrow{e\dual} (S/M)\dual].
\end{align*}
Consider the fibre diagram
\[\xymatrix{
C(D) \ar@{^{(}->}[r] \ar[d]^t & K^{\perp} \ar@{^{(}->}[r] \ar[d] & E \ar[dd]\\
C(G) \ar@{^{(}->}[r] \ar[d] & K^{\perp}/K \ar[d]  \\
\fC(\EE) \ar@{^{(}->}[r] & \left[\sfrac{(K^{\perp}/K)}{\left[\sfrac{(B/N)}{(S/M)}\right]}\right] \ar@{^{(}->}[r] & \left[\sfrac{E}{\left[\sfrac{B}{S}\right]}\right]
}\]
where $G:=\coker((B/N) \to (K^{\perp}/K))$. 
Let $\widetilde{\tau} \in \Gamma(C(G),K^{\perp}/K|_{C(G)})$ be the tautological section.
Then the formula
\[\sqrt{e}(E|_{C(D)},\tau) = \sqrt{e}(K^{\perp}/K|_{C(G)}, \widetilde{\tau}) \circ t^*\]
in \cite[Cor.~4.7]{KP20} proves the claim.

\hfill

\noindent{\em Step 4.}
Finally, we prove \eqref{Eq:DefInvArtin}. Consider the factorization $\sX'\to\sZ'\to\sY'$ induced by the base change of \eqref{Eq:5} along $i:\sY'\to\sY$. By Vistoli's rational equivalence \cite{Vist,Kre2} (cf. \cite[Prop.~2.5]{Qu}), we have a commutative square
\[\xymatrix{
A_*(\cZ) \ar[rr]^{\sp_{\sX/\sZ}} \ar[d]^{i^!} && A_*(\fC_{\sX/\sZ}) \ar[d]^{i^!} \\
A_*(\cZ') \ar[r]^{\sp_{\sX'/\sZ'}} & A_*(\fC_{\sX'/\sZ'}) \ar[r] & A_*(\fC_{\sX/\sZ}|_{\sX'}).\\
}\]
Since the localized Edidin-Graham class \eqref{Eq:A.1} commutes with the refined Gysin pullback $i^!$ by \cite[Lem.~4.4(3)]{KP20}, we obtain the desired formula \eqref{Eq:DefInvArtin}.
\end{proof}

Next, we generalize Kiem-Li's cone reduction lemma \cite{KL13} to Artin stacks.

\begin{theorem}\label{Thm:KLforArtin}
Let $\sX$ be an algebraic stack.
Let $\phi:\EE\to \trunc \LL_{\sX}$ be an obstruction theory. 
Let $\Sigma : \EE\dual[1] \to F$ be a map to a vector bundle $F$. 
Then
\[(\fC_{\sX})_{\red} \subseteq \fC(\EE_{\Sigma})\]
as substacks of $\fC(\EE)$ where $\EE_{\Sigma}:=\cone(\Sigma\dual[1] : F\dual[1] \to \EE)$.
\end{theorem}

\begin{proof}
Choose a smooth surjection $\tsX \to \sX$ from an affine scheme $\tsX$. We claim that we can form a morphism of distinguished triangles
\begin{equation}\label{Eq:A.4}
\xymatrix{
\EE|_{\tsX}\ar[r] \ar[d]^{\phi} \ar[r] & \tEE \ar[r] \ar[d]^{\tphi} & \Omega^1_{\tsX/\sX} \ar@{.>}[r] \ar@{=}[d] & \EE|_{\tsX}[1] \ar[d]^{\phi[1]}\\
\trunc (\LL_{\sX}|_{\tsX}) \ar[r] & \trunc \LL_{\tsX} \ar[r] & \Omega^1_{\tsX/\sX} \ar[r] & \trunc (\LL_{\sX}|_{\tsX}) [1]
}	
\end{equation}
in the underlying triangulated category $D(\sX):=[\sfD(\sX)]$.
Indeed, the dotted arrow exists since $\tsX$ is affine and $\cone(\phi[1]) \in D(\tsX)$ has amplitude $(-\infty,-3]$. 
Then $\tEE$ and $\tphi$ can be defined by the axioms of triangulated categories. 

The long exact sequence associated to \eqref{Eq:A.4} ensures that $\tphi$ is an obstruction theory.
Moreover we can define a map $\tSigma$ as the composition
\[\tSigma : \tEE\dual[1]\to \EE|_{\tsX}\dual[1] \xrightarrow{\Sigma|_{\tsX}} F|_{\tsX}.\]
Then we can apply Kiem-Li's cone reduction lemma (see Lemma \ref{Lem:KLconereduction}) for the scheme $\tsX$. Hence we have
\[(\fC_{\tsX})_{\red} \subseteq \fC(\tEE_{\tSigma})\]
as substacks of $\fC(\tEE)$, where $\tEE_{\tSigma}:=\cone(\tSigma\dual[1] : F|_{\tsX}\dual[1] \to \tEE)$.
By \eqref{Eq:A.4}, we have a fibre diagram
\[\xymatrix{
(\fC_{\tsX})_{\red} \ar@{^{(}->}[r] \ar@{->>}[d]& \fC_{\tsX} \ar@{^{(}->}[r] \ar@{->>}[d] & \fC(\tEE) \ar@{->>}[d] & \fC(\tEE_{\tSigma}) \ar@{_{(}->}[l] \ar@{->>}[d] \\
(\fC_{\sX}|_{\tsX})_\red \ar@{^{(}->}[r] &  \fC_{\sX}|_{\tsX} \ar@{^{(}->}[r] & \fC(\EE|_{\tsX}) & \fC((\EE_\Sigma)|_{\tsX}) \ar@{_{(}->}[l]
}\]
where the horizontal maps are closed embeddings and the vertical maps are $BT_{\tsX/\sX}$-torsors by \cite[Lem.~5.12(1),~Thm.~6.3(2)]{AP}. Hence we can deduce the desired property via descent.
\end{proof}

Finally, we generalize the construction of reduced Oh-Thomas virtual cycles for non-degenerate cosections in this paper (cf.~\cite{KP20}) to Artin stacks.

\begin{theorem}\label{Thm:RedOTforArtin}
In the situation of Theorem \ref{Thm:OTforArtin}, consider a map $\Sigma:\EE\dual[1] \to F$ to a special orthogonal bundle $(F,\fq_F,o_F)$ on $\sX$ such that $\Sigma^2=\fq_F$.
We further assume that $\sY$ is smooth scheme
and the composition
\begin{equation}\label{Eq:B.5}
F\dual \xrightarrow{\Sigma\dual} \EE[-1] \xrightarrow{\phi} (\trunc\LL_{\sX/\sY})[-1] \xrightarrow{\KS} \Omega^1_{\sY}|_
{\sX}     
\end{equation}
is zero.
Then there is a canonical reduced square root virtual pullback
\[\sqrt{f^!_{\red}}: A_*(\sY) \to A_*(\sX)\]
such that $\sqrt{f^!} = \sqrt{e}(F)\circ \sqrt{f^!_{\red}}$.
Moreover, for any cartesian square
\[\xymatrix{
\sX' \ar@{^{(}->}[r] \ar[d]^{f'} & \sX\ar[d]^f\\
\sY' \ar@{^{(}->}[r]^i & \sY
}\]
where $i$ is a closed embedding of smooth schemes, then we have
\begin{equation}\label{Eq:B.4}
i^! \circ \sqrt{f^!_\red} = \sqrt{(f')^!_\red} \circ i^! : A_*(\sY) \to A_*(\sX').	
\end{equation}
\end{theorem}

\begin{proof}
Let $\EE \cong \EE^\red \oplus F[1]$ be a decomposition of symmetric complexes induced by $\Sigma^2=\fq_F$.
By a standard argument as in Theorem \ref{Thm:DefInv}, we can modify the relative obstruction theory $\phi$ into an absolute obstruction theory. 
The vanishing of the composition \eqref{Eq:B.5} ensures that the relative cosection $\Sigma$ can be lifted to the absolute obstruction theory.
Hence we can apply the cone reduction lemma in Theorem \ref{Thm:KLforArtin}.
As in Theorem \ref{Thm:DefInv}, we now have a closed embedding $(\fC_{\sX/\sY})_\red \hookrightarrow \fQ(\EE^\red)$. Then we can define the reduced square root virtual pullback $\sqrt{f^!_\red}$ as in Theorem \ref{Thm:OTforArtin} using the data $(\fC_{\sX/\sY})_\red \hookrightarrow \fQ(\EE^\red)$. Here we consider the orientation of $\EE^\red$ induced by the canonical isomorphism $\det(\EE^\red) \cong \det(\EE)\otimes\det(F[1])\dual$.
The argument in Theorem \ref{Thm:OTforArtin} also proves formula \eqref{Eq:B.4}.
\end{proof}

\begin{remark}\label{Rem:globalquotientshypothesis}
We hope to remove the global quotient hypothesis in Theorem \ref{Thm:OTforArtin} and Theorem \ref{Thm:RedOTforArtin}.
At least when $\sX$ has reductive stabilizers and quasi-affine diagonal, then the authors expect that this will not be very difficult.
Indeed, we sketch the following argument.
By Alper-Hall-Rydh \cite{AHR}, there exists an {\'e}tale cover $\cX' \to \cX$ by a global quotient stack $\cX'$. 
By the Chow lemma \cite{BP} (cf.~\cite{LMB}), we can form a proper surjection $\cX'' \to \cX$ from a global quotient stack $\cX''$ using $\cX'$. 
Since we have a Kimura sequence \cite[Cor.~1.23]{AKLPR} for the Chow groups of higher Artin stacks developed by Khan \cite{Kha}, we can reduce the situation to the global quotient case.

If this expectation works, then we can apply Theorem \ref{Thm:RedOTforArtin} to any open substack of the moduli stack $\cPerf(X)$ of a perfect complexes on a Calabi-Yau $4$-fold $X$ which has a good moduli space \cite{Alp}.
\end{remark}

\subsection{Moduli stacks of semi-stable sheaves}

In this subsection, we apply the results in the previous subsection to the moduli {\em stack} of sheaves.

We first consider the {\em rigidified} version. Let $(X,H)$ be a polarized Calabi-Yau $4$-fold. Let $\cCoh(X)$ be the moduli stack of coherent sheaves on $X$. Then there exists a natural $B\GG_m$-action on $\cCoh(X)$. Let $\cCoh(X)\!\!\!\fatslash\GG_m:=[\cCoh(X)/B\GG_m]$ be the quotient stack.\footnote{Since the $B\GG_m$-action on $\cCoh(X)$ is {\em free}, i.e., the induced $\GG_m$-action on the stabilizers $\Aut(E)$ are free for all $E\in \cCoh(X)$, we can alternatively use the {\em rigidification} of classical $1$-Artin stacks, see \cite[Thm.~A.1]{AOV}.} Let
\[\cM^{H,ss}_v(X) \subseteq \cCoh(X)\!\!\!\fatslash \GG_m\]
be the open substack of (Gieseker) $H$-semi-stable sheaves with Chern character $v\in H^*(X,\Q)$. Here semi-stability is an open condition by \cite[Prop.~2.3.1]{HL}.

\begin{theorem}\label{Thm:semistable.rigidified}
Let $(X,H)$ be a polarized Calabi-Yau $4$-fold and $v\in H^*(X,\Q)$.
Then, for any choice of orientation, there exists a canonical {\em reduced} virtual cycle
\[[\cM^{H,ss}_v(X) ]^\red \in A_{1-\frac12\chi(v,v) +\frac12\rho_\gamma}(\cM^{H,ss}_v(X) )\]
where $\gamma:=v_2$ and $\chi(v,v):=\int_Xv\dual \cdot v \cdot \td(X)$.
Moreover, if $[\cM^{H,ss}_v(X)]^\red\neq 0$, then the variational Hodge conjecture holds for $(X,\gamma)$.
\end{theorem}

\begin{proof}
We first fix some notation. Let $\cM:=\cM^{H,ss}_v(X)$ and $\lambda:\tcM\to\cM$ be the $\GG_m$-gerbe induced by $\cCoh(X)\to\cCoh(X)\!\!\!\fatslash\GG_m$.
Let $\sE$ be the universal sheaf on $X \times\tcM$ and $\pi:X\times\tcM \to\tcM$ be the projection map.
Note that we have a weight decomposition
\[D(\tcM) \cong \prod_{w\in\Z}D(\cM)\]
by Bergh-Schn\"urer \cite{BS}, where $D(\tcM):=[\sfD(\tcM)]$ and $D(\cM):=[\sfD(\cM)]$ are the underlying triangulated categories of the stable $\infty$-categories.\footnote{For an $1$-Artin stack, the derived category defined via descent in the context of derived algebraic geometry \cite{ToVe} and the derived category defined as a full subcategory of sheaves of modules on the lisse-\'etale site in \cite{LMB,Ols} coincide by Hall-Rydh \cite[Prop.~1.3]{HR}.}
Moreover,
\[\lambda^* : D(\cM) \to D(\tcM)\]
is fully faithful and the essential image consists of weight zero objects.

We describe the canonical symmetric obstruction theory on $\tcM$ as follows. Let $R\cCoh(X)$ be the derived moduli stack of coherent sheaves on $X$. Then the canonical map $\tcM \xrightarrow{\sE} R\cCoh(X)$ induces an open embedding on the classical truncations. Hence we have an induced symmetric obstruction theory
\[\tphi : \tEE:=\LL_{R\cCoh(X)}|_{\tcM}=(R\hom_\pi(\sE,\sE)[1])\dual \xrightarrow{\At(\sE)} \trunc \LL_{\tcM},\]
which satisfies the isotropic condition by \cite{BBBBJ} and is orientable by \cite{CGJ}.
Since $\tEE \in D(\tcM)$ is a weight zero object, we have $\lambda^*(\lambda_*\tEE) \cong \tEE$ for $\lambda_*\tEE\in D(\cM)$.

We note that the canonical obstruction theory on $\cM$ given by the standard derived enhancement is {\em not} symmetric.
Indeed, the $B\GG_m$-action on $\cCoh(X)$ extends to $R\cCoh(X)$ and the canonical map $\cM \to [R\cCoh(X)/B\GG_m]$ induces an open embedding on the classical truncations. Hence we have an induced obstruction theory
\[\phi_1 : \EE_1:=\LL_{[R\cCoh(X)/B\GG_m]}|_{\cM}\to \trunc \LL_{\cM}\]
where the virtual cotangent complex $\EE_1$ fits into the distinguished triangle
\[\xymatrix{
\LL_{[R\cCoh(X)/B\GG_m]}|_{\tcM} \ar[r]\ar@{=}[d] & \LL_{R\cCoh(X)}|_{\tcM} \ar[r] \ar@{=}[d]& \LL_{B\GG_m}|_{\tcM}\ar@{=}[d]\\
\lambda^*\EE_1 \ar[r] & \tEE \ar[r]^-{t} & \O_{\tcM}[-1].
}\]
Here $t\dual : \O_{\tcM}[1] \to \tEE\dual=R\hom_{\pi}(\sE,\sE)[1]$ is given by the identity $\id_{\sE}:\sE\to\sE$ since the induced $\GG_m$-action on the derived loop space $\cL_{R\cCoh(X),E} \subseteq \mathrm{Tot}(R\Hom_X(E,E))$ at $[E] \in R\cCoh(X)$ is the homothety.

We then {\em symmetrize} the obstruction theory $\phi_1$ via reduction operation in \cite[Lem.~C.2]{Par}.
Indeed, since
\[\Hom_{D(\cM)}(\O_{\cM}[3],\O_{\cM}[-2])=0,\]
we can form a reduction diagram (in the sense of \cite[Lem.~C.2]{Par})
\begin{equation}\label{Eq:Rigid.1}
\xymatrix{
\O_{\cM}[3] \ar@{=}[r] \ar[d] & \O_{\cM}[3] \ar[d]^{\lambda_*t\dual[2]}\\
\EE_1 \ar[r]^{\alpha\dual[2]} \ar[d]^{\beta\dual[2]} & \lambda_*\tEE \ar[r]^-{\lambda_*t} \ar[d]^{\alpha} & \O_{\cM}[-1] \ar@{=}[d]\\
\EE \ar[r]^{\beta} & \EE_1\dual[2] \ar[r] & \O_{\cM}[-1]
}	
\end{equation}
in the triangulated category $D(\cM)$ where $\EE$ is a {\em symmetric} complex. Since
\[\Hom_{D(\cM)}(\O_{\cM}[3],\trunc \LL_{\cM})=\Hom_{D(\cM)}(\O_{\cM}[4],\trunc \LL_{\cM})=0,\]
there exists a unique symmetric obstruction theory
\[\phi : \EE \to \trunc\LL_\cM\]
which fits into the commutative diagram
\[\xymatrix{
\O_{\cM}[3] \ar[r] & \EE_1 \ar[r]^{\beta\dual[2]} \ar[rd]_{\phi_1}&  \EE \ar[r] \ar@{.>}[d]^{\exists\phi}  & \O_{\cM}[4] \\
&& \trunc \LL_{\cM}
}\]
as the dotted arrow.
The isotropic condition for $\phi$ follows from the one for $\tphi$. Indeed, we can form a commutative diagram of (higher) cone stacks on $\tcM$
\[\xymatrix{
\fC_{\tcM} \ar@{^{(}->}[r] \ar[d] & \fC(\lambda^*\EE_1\dual[2]) \ar@{^{(}->}[r]_-{\cong}^-{\fC(\lambda^*\alpha)}\ar[d]^{\fC(\lambda^*\beta)} & \fC(\tEE) \ar[d]\\
\fC_{\cM}|_{\tcM} \ar@{^{(}->}[r] & \fC(\lambda^*\EE) \ar@{^{(}->}[r]_-{\cong} & \fC(\lambda^*\EE_1)
}\]
where the horizontal maps are closed embeddings and the vertical maps are $\mathrm{Tot}(\TT_{\tcM/\cM})\cong B\GG_a$-torsors. From the reduction diagram \eqref{Eq:Rigid.1}, we can deduce that the two quadratic functions $\fq_{\EE}$ and $\fq_{\tEE}$ restrict to the same function on $\fC_{\tcM}$. Since the composition
\[\fC_{\tcM} \twoheadrightarrow \fC_{\cM}|_{\tcM} \twoheadrightarrow \fC_{\cM}\]
is smooth surjective, we obtain the desired isotropic condition via descent.

Moreover, an orientation of $\EE$ can be induced by one of $\tEE$ since we have
\[\det(\EE) \cong \lambda_*\det(\tEE)\]
by the reduction diagram \eqref{Eq:Rigid.1}.

We will now form a cosection $\SR: \EE\dual[1] \to H^1(X,T_X)\dual \otimes \O_{\cM}$ for the obstruction theory $\phi$ on $\cM$. Firstly, define a cosection for the obstruction theory $\tphi$ on $\tcM$ as the composition
\begin{multline*}
\widetilde{\SR} : \tEE\dual[1]=R\hom_\pi(\sE,\sE)[2] \xrightarrow{\At(\sE)} R\hom_\pi(\sE,\sE\otimes \LL_{X\times\tcM})[3] \\\xrightarrow{} R\hom_\pi(\sE,\sE\otimes \Omega^1_{X})[3] \xrightarrow{\tr} R\Gamma(X, \Omega^1_X)\otimes \O_{\tcM}[3] \\
\xrightarrow{} H^3(X,\Omega^1_X)\otimes \O_{\tcM} \xrightarrow{\cong} H^1(X,T_X)\dual\otimes \O_{\tcM}.
\end{multline*}
Then we have $\widetilde{\SR}^2=\sfB_\gamma\otimes 1_{\O_{\tcM}}$ by the argument in Proposition \ref{prop:cosection}. Then there exists a unique cosection
\[\SR : \EE\dual[1] \to H^1(X,T_X)\dual\otimes \O_{\cM}\]
which fits into the commutative diagram
\[\xymatrix{
\O_{\cM}[3] \ar[r] & \EE_1 \ar[r]^{\beta\dual[2]} \ar[d]_{\alpha\dual[2]}&  \EE \ar[r] \ar@{.>}[d]^{\exists\SR}  & \O_{\cM}[4] \\
& \lambda_*\tEE \ar[r]^-{\lambda_*\widetilde{\SR}} & H^1(X,T_X)\dual\otimes \O_{\cM}[1]
}\]
since
\[\Hom_{\cM}(\O_{\cM}[3], H^1(T_X)\dual\otimes \O_{\cM}[1]) = \Hom_{\cM}(\O_{\cM}[4], H^1(T_X)\dual\otimes \O_{\cM}[1])=0.\]
Moreover, we have $\SR^2=\lambda_*\widetilde{\SR}^2 =\sfB_\gamma \otimes 1_{\O_{\cM}}$ by the reduction diagram \eqref{Eq:Rigid.1}.

Therefore, if we choose a maximal non-degenerate subspace $V \subseteq H^1(T_X)$, then we have a non-degenerate cosection $\SR_V : \EE\dual[1] \xrightarrow{\SR} H^1(T_X)\dual\otimes\O_{\cM}[1]\to V\dual\otimes\O_{\cM}[1]$ for the symmetric obstruction theory $\phi:\EE\to\trunc\LL_{\cM}$. Hence we can apply Theorem \ref{Thm:RedOTforArtin} and define a reduced Oh-Thomas virtual cycle
\[[\cM]^\red \in A_*(\cM)\]
where $\cM$ is a global quotient stack \cite{HL}.

Finally, we show that $[\cM]^\red$ is deformation invariant and $[\cM]^\red\neq0$ implies the variational Hodge conjecture. Let $f:\cX\to\cB$ be a family of Calabi-Yau $4$-folds (as in Theorem \ref{Thm:DefInv}). Let $\tgamma \in F^2H^4_{DR}(\cX/\cB)$ be a horizontal section and let $P_v(t)\in \Q[t]$ be the Hilbert polynomial determined by $v$. Then we have a relative moduli stack $\cM^{ss}(\cX/\cB)$ of semi-stable sheaves $E$ on the fibres $\cX_b$ of $f:\cX\to\cB$ over $b \in \cB$ with $\ch_2(E)=\tgamma_b$ and $P_E(t)=P_v(t)$ as a global quotient open substack of $\cCoh(\cX/\cB)$ \cite{HL}.
Then all the arguments above also work for the relative setting so that we have a relative symmetric obstruction theory $\phi_{\cX/\cB} : \EE_{\cX/\cB} \to \trunc\LL_{\cM(\cX/\cB)}$ and a cosection $\SR_{\cX/\cB}:\EE_{\cX/\cB}\dual[1] \to (R^1f_*T_{\cX/\cB})|_{\cM(\cX/\cB)}\dual$.
As in Lemma \ref{lem:compatibilityofKS}, we have an analogous commutative square
\[\xymatrix@C+2pc{
\EE_{\cX/\cB} \ar[r]^-{\SR_{\cX/\cB}} \ar[d]^{\phi_{\cX/\cB}}  &  (R^1f_*T_{\cX/\cB})|_{\cM(\cX/\cB)}\dual[1] \ar[d]^{\KS_{\cX/\cB}}\\ 
\trunc \LL_{\cM(\cX/\cB)} \ar[r]^-{\KS_{\cM(\cX/\cB)/\cB}} & \Omega^1_{\cB}|_{\cM(\cX/\cB)}[1]
}\]
since $\Hom_{\cM(\cX/\cB)}(\O_{\cM(\cX/\cB)}[4], \Omega^1_\cB|_{\cM(\cX/\cB)}[1])=0$. Then, based on Theorem \ref{Thm:RedOTforArtin}, all the remaining arguments in Theorem \ref{Thm:DefInv} work in this setting.
Moreover, the analog of Theorem \ref{Thm:VHC} also works in this setting, since there is a good moduli space of $\cM(\cX/\cB)$ which is proper over $\cB$.
\end{proof}

\begin{remark}
In the proof of Theorem \ref{Thm:semistable.rigidified}, if we restrict to the moduli {\em scheme} $\cM^{H,st}_v(X)$ of stable sheaves, then the reduction in \eqref{Eq:Rigid.1} is just the truncation
\[\tau^{[-2,0]}(R\hom_\pi(\sE,\sE)[1])\dual,\]
which is used in \cite{OT,KP20}.
\end{remark}

\begin{remark}
Theorem \ref{Thm:semistable.rigidified} also holds for the non-rigidified moduli stack $\widetilde{\cM}^{H,ss}_v(X) \subseteq \cCoh(X)$.
However, we will not use this approach since this will give us a weaker statement for the application to the variational Hodge conjecture since $[\widetilde{\cM}^{H,ss}_v(X)]^\red = \lambda^*[\cM^{H,ss}_v(X)]^\red$.
\end{remark}

\begin{remark}\label{Rem:goodmoduli.rigidified}
Assuming the expectation in Remark \ref{Rem:globalquotientshypothesis}, then the first part of Theorem \ref{Thm:semistable.rigidified} holds for any $1$-Artin open substack $\cM(X) \subseteq \cPerf(X)\!\!\!\fatslash \GG_m$ with good moduli space. Moreover, if we can form a $1$-Artin open substack $\cM(\cX/\cB)\subseteq\cPerf(\cX/\cB)\!\!\!\fatslash \GG_m$ with proper good moduli space for any family of Calabi-Yau $4$-folds $\cX\to\cB$, then the second part of Theorem \ref{Thm:semistable.rigidified} also holds.
\end{remark}

Next, we consider the {\em fixed determinant} version.
Let $(X,H)$ be a polarized Calabi-Yau $4$-fold.
Note that there is a canonical determinant map $\det: \cCoh(X) \to \cPic(X)$ to the moduli stack  $\cPic(X)$ of line bundles on $X$. Let  $\cCoh(X)_{\O_X}:=\cCoh(X)\times_{\det,{\curly Pic}(X),\O_X}\Spec(\C)$ be the fibre over the trivial line bundle. Let
\[\cM^{H,ss}_v(X)_{\O_X} \subseteq \cCoh(X)_{\O_X}\]
be the open substack of $H$-semi-stable sheaves with Chern character $v \in H^*(X,\Q)$.

\begin{theorem}\label{Thm:semistable.fixeddet}
Let $(X,H)$ be a polarized Calabi-Yau $4$-fold and $v\in H^*(X,\Q)$.
Assume that $v_0\neq0$ and $v_1=0$.
Then, for any choice of orientation, there exists a canonical {\em reduced} virtual cycle
\[[\cM^{H,ss}_v(X)_{\O_X} ]^\red \in A_{\frac12\chi(\O_X)-\frac12\chi(v,v) +\frac12\rho_\gamma}(\cM^{H,ss}_v(X)_{\O_X} )\]
where $\gamma:=v_2$.
Moreover, if $[\cM^{H,ss}_v(X)_{\O_X}]^\red\neq 0$, then the variational Hodge conjecture holds for $(X,\gamma)$.
\end{theorem}

\begin{proof}
This follows from the arguments in Theorem \ref{Thm:RVFC}, Theorem \ref{Thm:DefInv}, and Theorem \ref{Thm:VHC}, based on  Theorem \ref{Thm:RedOTforArtin}.
This is possible by the existence of the relative moduli stack of semi-stable sheaves as a global quotient stack with proper good moduli space over the base.
We omit the details.
\end{proof}

\begin{remark}
Assuming the expectation in Remark \ref{Rem:globalquotientshypothesis}, we can generalize Theorem \ref{Thm:semistable.fixeddet} to any $1$-Artin open substack of $\cPerf(X)_{\O_X}$ with (proper) good moduli space, as in Remark \ref{Rem:goodmoduli.rigidified}.
\end{remark}

\section{Orientations for relative moduli spaces}\label{Appendix:RelativeOrientation}
In this subsection, we prove that family orientations exist on the topological versions of the moduli spaces. We regard this result as heuristic evidence for Conjecture \ref{conj:familyorientation}.

\subsection{Orientations for fibrations of topological spaces}

We consider the following purely topological statement.

\begin{proposition}\label{Prop:Or.1}
Let $p : \cC \to \cB$ be a (locally trivial) fibration of topological spaces. 
Assume that the base $\cB$ is path connected and locally path connected, and the fibre $\cC_b:=\cC\times_{\cB}\{b\}$ over a point $b \in \cB$ has finitely many connected components.
Let $O \to \cC$ be a principal $\Z_2$-bundle. 
If the fibres $O|_{\cC_b}$ are trivial for all $b \in \cB$, then there exists a finite-sheeted covering $u:\cB' \to \cB$ such that the pullback $O|_{\cC\times_{\cB}\cB'}$ is trivial. 

Moreover, we can further assume that $\deg(u)=|\pi_0(\cC_b)|!$.
\end{proposition}

\begin{proof}
Fix a point $c \in \cC$ and let $b:=p(c) \in \cB$. 
Consider the fibre diagram	
\[\xymatrix{
O_{c} \ar[r] \ar[d] & O_b \ar[r] \ar[d] & O\ar[d]\\
\{c\} \ar[r] & \cC_b \ar[r] \ar[d] & \cC \ar[d]^p \\
& \{b\} \ar[r] & \cB
}\]
where $O_b:=O|_{\cC_b}$.
Since the two maps $O \to \cC $ and $O \to \cB$ are  fibrations, there are monodromy actions
\begin{equation}\label{Eq:Or.2}
m: \pi_1(\cC,c) \lefttorightarrow \pi_0(O_c), \qquad n : \pi_1(\cB,b) \lefttorightarrow \pi_0(O_b).
\end{equation}
Moreover, the canonical map
\[\pi_0(O_c) \to \pi_0(O_b)\]
is compatible with the above two monodromy actions in \eqref{Eq:Or.2}. Since we have $O_b = \cC_b \bigsqcup \cC_b$ by assumption, there is a canonical bijection
\begin{equation}\label{Eq:Or.1}
\pi_0(O_b) \cong \pi_0(\cC_b) \times \pi_0(O_c).
\end{equation}
We can form a commutative diagram of groups
\begin{equation}\label{Eq:Or.3}
\xymatrix{
\pi_1(\cC,c) \ar[r]^{p_*} \ar[d]_m& \pi_1(\cB,b) \ar[d]^n \\
\Aut(\pi_0(O_c)) \ar[r] & \Aut(\pi_0( O_b) )
}\end{equation}
where the vertical maps are given by the monodromy actions in \eqref{Eq:Or.2}, the upper horizontal map is the canonical pushforward for $p$, and the lower horizontal map is given by the decomposition \eqref{Eq:Or.1}.

To show that the $\Z_2$-bundle $O \to \cC$ is trivial, it suffices to show that the left vertical map $m:\pi_1(\cC,c) \to \Aut(\pi_0(O_c))$ in \eqref{Eq:Or.3} is trivial. Since the lower horizontal map in \eqref{Eq:Or.3}, given by \eqref{Eq:Or.1}, is injective, it suffices to show that the right vertical map $n:\pi_1(\cB,b) \to \Aut(\pi_0(O_b))$ in \eqref{Eq:Or.3} is trivial. 

Note that $\pi_0(O_b)$ is finite since $\pi_0(\cC_b)$ is finite by assumption and we have \eqref{Eq:Or.1}. Hence the kernel of the map $n:\pi_1(\cB,b) \to \Aut(\pi_0(O_b))$ is a normal subgroup of $\pi_1(\cB,b)$ of finite index. Hence there exists a covering map
\[u : \cB' \to \cB\]
such that the composition
\[\pi_1(\cB',b') \xrightarrow{u_*} \pi_1(\cB,b) \xrightarrow{n} \Aut(\pi_0(O_b))\]
is trivial and the degree of $u$ divides the order of the finite group $\Aut(\pi_0(O_b))$.
Then the pullback $O|_{\cC\times_{\cB}\cB'}$ is trivial since the commutative square \eqref{Eq:Or.3} is compatible with base change.

Finally, replacing $\cB'$ by its covering, we can further assume that the degree of $u:\cB'\to\cB$ is exactly the order of the group $\Aut(\pi_0(O_b))$.
\end{proof}

\subsection{Orientations for relative topological moduli spaces}

Let $f:\cX\to\cB$ be a smooth projective morphism of smooth quasi-projective schemes.
Consider the mapping space
\[\cC(\cX/\cB):=\mathrm{Map}_{\mathrm{Top}_{/\cB^{\mathrm{top}}}}(\cX^{\mathrm{top}}, \cB^{\mathrm{top}} \times BU \times \Z),\]
where $(-)^{\mathrm{top}}$ denotes the underlying topological space (with the analytic topology). Here we used the modified compact-open topology on the mapping space $\mathrm{Map}_{\mathrm{Top}/\cB^{\mathrm{top}}}(-,-)$ which satisfies adjunction (\cite{BB}).
Then we have a homotopy fibre diagram
\[\xymatrix{
\cC(\cX_b) \ar[r] \ar[d] & \cC(\cX/\cB) \ar[d] \\
\{b\} \ar[r] & \cB
}\]
for each $b \in \cB$, where $\cC(\cX_b):=\cC(\cX_b/\{b\})$. Note that
\[\pi_0(\cC(\cX_b)) = K^0_{\mathrm{top}}(\cX_b)\]
is the topological $K$-theory of $\cX_b$.

Let $\tv$ be a horizontal section of $\bigoplus_p \cH^{2p}_{DR}(\cX/\cB)$.
Then $\tv^\mathrm{top}$ can be regarded as a section of the $\C$-local system $\bigoplus_p R^{2p} f^{\mathrm{top}}_* \C$ (Remark~\ref{Rem:horizontal}(1)).
Hence we can consider an open subspace
\[\cC(\cX/\cB,\tv) \subseteq \cC(\cX/\cB)\]
which fits into the fibre diagram
\[\xymatrix{
\cC(\cX_b,\tv_b) \ar[r] \ar[d] & \cC(\cX/\cB,\tv) \ar[d] \\
\{b\} \ar[r] & \cB
}\]
for each $b \in \cB$, where $\cC(\cX_b,\tv_b)=\cC(\cX_b/\{b\},\tv_b) \subseteq \cC(\cX_b)$ is the union of connected components $\alpha \subseteq \cC(\cX_b)$ whose corresponding $K$-theory classes $[\alpha]$ map to $\tv_b \in H^*(\cX_b,\C)$ under the topological Chern character map
\[\ch^\mathrm{top} : K^0_\mathrm{top}(\cX_b) \to H^*(\cX_b,\Q).\]

\begin{corollary}\label{Cor:Or.2}
Let $f:\cX\to\cB$ be a smooth projective morphism of smooth quasi-projective schemes.
Assume that $\cB$ is connected.
Let $\tv$ be a horizontal section of $\bigoplus_p \cH^{2p}_{DR}(\cX/\cB)$.
Let $O \to \cC(\cX/\cB,\tv)$ be a principal $\Z_2$-bundle. 
If the fibres $O|_{\cC(\cX_b,\tv_b)}$ are trivial for all $b \in \cB$, then there exists a finite {\'e}tale cover $u:\cB' \to \cB$ such that the pullback $O|_{\cC(\cX'/\cB',\tv')}$ is trivial, where $\cX':=\cX\times_{\cB}\cB'$ and $\tv':=u^*(\tv)$.

Moreover, we can further assume that $\deg(u) = |K_{\mathrm{top}}^0(\cX_b)_{\mathrm{tor}}|!$ where $(-)_{\mathrm{tor}}$ denotes the torsion subgroup.
\end{corollary}

\begin{proof}
By Ehresmann's theorem $\cC(\cX/\cB,\tv) \to \cB$ is a locally trivial fibration.
To apply Proposition \ref{Prop:Or.1}, it suffices to show that
\[\pi_0(\cC(\cX_b,\tv_b)) = (\ch^\mathrm{top})^{-1}(\tv_b) \]
is finite. This set is bijective to $\ker(\ch^\mathrm{top})=K^0_{\mathrm{top}}(\cX_b)_{\mathrm{tor}}$. Since $K^0_{\mathrm{top}}(\cX_b)$ is a finitely generated abelian group (by the Atiyah-Hirzebruch spectral sequence), 
$\pi_0(\cC(\cX_b,\tv_b)) \cong K^0_{\mathrm{top}}(\cX_b)_{\mathrm{tor}}$ is also finite.

Finally, there is a natural equivalence of categories
\[\{\text{finite {\'e}tale covers of $\cB$}\} \cong \{\text{finitely-sheeted topological coverings of $\cB^\top$}\}\]
by \cite[Expos\'e XII, Cor.~5.2]{SGA1}.
Hence the finite covering in Proposition \ref{Prop:Or.1} comes from a finite {\'e}tale map.
\end{proof}

\begin{remark}
Let $f:\cX\to\cB$ be a smooth projective morphism of smooth quasi-projective schemes of relative dimension $4$ with connected fibres. Assume that the relative canonical bundle $\omega_{\cX/\cB}$ is trivial and $\cB$ is connected.
Let $\tv$ be a horizontal section of $\bigoplus_p \cH^{2p}_{DR}(\cX/\cB)$.
Let $\cPerf(\cX/\cB,\tv)$ be the moduli stack of all perfect complexes on the fibres of $f$ (which exists as a higher Artin stack) whose Chern characters are pullbacks of $\tv$. Then there exists a canonical $\Z_2$-bundle $O \to \cPerf(\cX/\cB,\tv)$ given by the determinant line bundle $\det(\EE)$ of the restriction $\EE$ of the cotangent complex of the derived enhancement of $\cPerf(\cX/\cB,\tv)$. We want to know whether the $\Z_2$-bundle $O$ is trivial.

If the canonical map 
\[p^{\mathrm{top}} : \cPerf(\cX/\cB,\tv)^{\mathrm{top}} \to \cB^{\mathrm{top}}\]
(given by the topological realization functor \cite{Sim,Bla}) were a locally trivial fibration, then this will follow from Proposition \ref{Prop:Or.1}.
However, $p^{\mathrm{top}}$ is {\em not} locally trivial in general.

What we can use instead is the topological mapping space. There is a canonical map
\[\cPerf(\cX/\cB,\tv)^{\mathrm{top}} \to \cC(\cX/\cB,\tv)\]
and the map
\[c : \cC(\cX/\cB,\tv) \to \cB^{\mathrm{top}}\]
which is a locally trivial fibration.

In view of Corollary \ref{Cor:Or.2}, the remaining ingredients required to prove Conjecture \ref{conj:familyorientation} are as follows:
\begin{enumerate}
\item construct a natural $\Z_2$-bundle on $\cC(\cX/\cB)$ whose fibres on $\cC(\cX_b)$ are the ones defined by Joyce-Tanaka-Upmeier \cite{JTU}, and
\item compare the $\Z_2$ bundle in (1) with the natural orientation on $\cPerf(\cX/\cB)$.
\end{enumerate}

We note that this strategy is exactly the same as the fibrewise result of Cao-Gross-Joyce \cite{CGJ}. They showed the orientability on $\cPerf(\cX_b)$ by (1) constructing an orientation bundle on $\cC(\cX_b)$, (2) showing the compatibility of orientations of $\cPerf(\cX_b)$ and $\cC(\cX_b)$, and (3) showing the orientability of $\cC(\cX_b)$. In their setting, (3) was the most difficult part using topological facts on $5$-manifolds. In our situation, we only need to generalize parts (1) and (2) to the relative setting. 
\end{remark}

\noindent {\tt{younghan.bae@math.ethz.ch, m.kool1@uu.nl, hyeonjun93@snu.ac.kr}}
\end{document}